\documentclass[12pt,reqno]{amsbook}

\DeclareRobustCommand{\SkipTocEntry}[4]{}
\makeindex

\usepackage{amsfonts,amssymb,amsthm}
\usepackage{amsmath,amscd}
\usepackage{pstricks}
\usepackage{pstricks,pst-node}
\usepackage{mathrsfs}
\usepackage[all]{xy}

\numberwithin{section}{chapter}

\renewcommand{\subsection}[1]{\vspace{.18in}\par\noindent\addtocounter{subsection}{1}\setcounter{equation}{0}{\bf\thesubsection.\hspace{5pt}#1}}

\newtheorem{theorem}{Theorem}[section]
\numberwithin{equation}{theorem}

\theoremstyle{definition}
\newtheorem{Def}[theorem]{Definition}
\newtheorem{Conj}[theorem]{\bf Conjecture}
\newtheorem{Prob}[theorem]{\bf Problem}
\newtheorem{Example}[theorem]{Example}
\newtheorem{Rem}[theorem]{Remark}

\newtheorem{Rems}[theorem]{Remarks}
\theoremstyle{plain}
\newtheorem{Prop}[theorem]{Proposition}
\newtheorem{Thm}[theorem]{Theorem}

\newtheorem{Lem}[theorem]{Lemma}
\newtheorem{Coro}[theorem]{Corollary}

\newcommand{\SrC}{\sS(n,r)_\mbc}
\newcommand{\UglC}{U({\frak{gl}}_n)_{\mathbb C}}
\def\ttv{{z}}
\newcommand{\afHrC}{\sH_{\vtg}(r)_\mathbb C}
\newcommand{\afUslC}{\text{\rm U}_{\mathbb C}(\widehat{\frak{sl}}_n)}
\newcommand{\afJnn}{I_{\vtg}(n,n)_0}
\newcommand{\OgnC}{\Og_{n,\mathbb C}}
\newcommand{\Det}{\mathrm{Det}}
\newcommand{\OgC}{\Og_{\mathbb C}}
\newcommand{\ms}{\mathscr}
\newcommand{\afHoC}{\sH_{\vtg}(1)_\mathbb C}
\newcommand{\HoC}{\sH(1)_{\mathbb C}}
\newcommand{\afHnC}{\sH_{\vtg}(n)_\mathbb C}
\newcommand{\HnC}{\sH(n)_{\mathbb C}}
\newcommand{\HrC}{\sH(r)_{\mathbb C}}

\newcommand{\afSnC}{{\mathcal S}_{\vtg}(n,n)_\mathbb C}
\newcommand{\afSkC}{{\mathcal S}_{\vtg}(n,k)_\mathbb C}
\newcommand{\afSlC}{{\mathcal S}_{\vtg}(n,l)_\mathbb C}
\newcommand{\afSiC}{{\mathcal S}_{\vtg}(n,i)_\mathbb C}
\newcommand{\afSklC}{{\mathcal S}_{\vtg}(n,k+l)_\mathbb C}

\newcommand{\bin}{\bigcup}

\newcommand{\han}{\subseteq}

\newcommand{\scr}[1]{\mathscr #1}
\def\bds{\boldsymbol}

\newcommand{\bse}{\boldsymbol{e}}
\newcommand{\bsa}{\boldsymbol{a}}

\newcommand{\bsz}{\boldsymbol{z}}
\newcommand{\bsH}{{\boldsymbol{\sH}}}

\newcommand{\lan}{\langle}
\newcommand{\ran}{\rangle}
\def\lr#1{\langle #1\rangle}

\def\fS{{\frak S}}
\def\fL{{\frak L}}

\def\fA{{\frak A}}

\def\fC{{\frak C}}
\def\fka{{\frak a}}
\def\fkd{{\frak d}}
\def\scrA{{\scr A}}
\def\scrB{{\scr B}}

\def\sfx{{\mathsf x}}
\def\sfy{{\mathsf y}}
\def\sfz{{\mathsf z}}
\def\sfs{{\mathsf s}}

\def\sfk{{\mathsf k}}
\def\sfF{{\mathsf F}}
\def\sfG{{\mathsf G}}

 \newcommand{\bfOg}{{\bf\Omega}}
\def\wt{\widetilde}

\newcommand{\msP}{\mathscr P}

\newcommand{\msD}{\mathscr D}\newcommand{\afmsD}{{\mathscr D}^\vtg}
\newcommand{\fSr}{\fS_r}
\newcommand{\affSr}{{\fS_{\vtg,r}}}

\newcommand{\Hr}{{\sH(r)}}
\newcommand{\afHr}{{\sH_\vtg(r)}}
\newcommand{\bsHr}{{\boldsymbol{\mathcal H}}(r)}
\newcommand{\afbfHr}{{\boldsymbol{\mathcal H}_\vtg(r)}}

\def\sA{{\mathcal A}}
\def\sB{{\mathcal B}}

\def\sE{{\mathcal E}}
\def\sF{{\mathcal F}}

\def\sH{{\mathcal H}}
\def\sI{{\mathcal I}}
\def\sJ{{\mathcal J}}
\def\sK{{\mathcal K}}
\def\sL{{\mathcal L}}

\def\sO{{\mathcal O}}\def\ttO{{\mathfrak O}}
\def\sP{{\mathcal P}}
\def\sQ{{\mathcal Q}}

\def\sS{{\mathcal S}}
\def\sT{{\mathcal T}}
\def\sU{{\mathcal U}}

\def\sX{{\mathcal X}}
\def\sY{{\mathcal Y}}
\def\sZ{{\mathcal Z}}
\def\rmL{{\mathrm L}}

\newcommand{\vtg}{{\!\vartriangle\!}}
\newcommand{\Hall}{{{\mathfrak H}_\vtg(n)}}

\newcommand{\HallL}{{\tt H}}
\newcommand{\bfHall}{{\boldsymbol{\mathfrak H}_\vtg(n)}}
\newcommand{\Hallpi}{\bfHall^{\geq 0}}
\newcommand{\Hallmi}{\bfHall^{\leq 0}}

\newcommand{\dHallr}{{\boldsymbol{\mathfrak D}_\vtg}(n)}
\newcommand{\dHallrr}{{\boldsymbol{\mathfrak D}_\vtg}(r)}
\newcommand{\comp}{{{\mathfrak C}_\vtg(n)}}
\newcommand{\bfcomp}{{\boldsymbol{\mathfrak C}_\vtg(n)}}

\def\field{{\mathbb F}}
\def\bffke{\boldsymbol{\frak e}}
\def\bffkf{\boldsymbol{\frak f}}
\def\bffkk{\boldsymbol{\frak k}}

\def\deg{{\rm dg}}
\def\scrX{{\scr X}}
\def\scrY{{\scr Y}}
\def\rep{{\rm {\bf Rep}}}

\def\ggp#1#2{\left[\kern-3.2pt\left[{#1\atop #2}\right]\kern-3.2pt\right]}

\newcommand{\dleb}{\left[\!\!\left[}
\newcommand{\leb}{\left[}
\newcommand{\drib}{\right]\!\!\right]}
\newcommand{\rib}{\right]}

\newcommand{\nmod}{\!\!\!\!\!\mod\!}
\newcommand{\mnmod}{\!\!\!\mod\!}
\newcommand{\bfdim}{{\mathbf{dim\,}}}

\newcommand{\Thnr}{\Th_r(n)}

\newcommand{\tA}{{}^t\!A}
\newcommand{\tB}{{}^t\!B}

\newcommand{\tri}{\triangle(n)}

\newcommand{\afsl}{\widehat{\frak{sl}}_n}
\newcommand{\afgl}{\widehat{\frak{gl}}_n}

\newcommand{\afg}{\mathfrak n}
\newcommand{\afFn}{{\mathscr F_\vtg}}
\newcommand{\afBr}{{\mathscr B_\vtg}}
\newcommand{\afE}{E^\vartriangle}

\newcommand{\afbfU}{\bfU_\vtg(n)}

\newcommand{\cycn}{{{n}}}
\newcommand{\afSr}{{\mathcal S}_{\vtg}(\cycn,r)}
\newcommand{\afSrp}{{\mathcal S}_{\vtg}(\cycn,r)^+}
\newcommand{\afSrm}{{\mathcal S}_{\vtg}(\cycn,r)^-}
\newcommand{\afSrz}{{\mathcal S}_{\vtg}(\cycn,r)^0}
\newcommand{\bfSr}{{\boldsymbol{\mathcal S}}(n,r)}

\newcommand{\afbsK}{{\boldsymbol{\mathcal K}}_\vtg(n)}
\newcommand{\afhbsK}{\widehat{{\boldsymbol{\mathcal K}}}_\vtg(n)}
\newcommand{\afhsKq}{\widehat{{\mathcal K}}_\vtg(n)_\mbq}

\newcommand{\afsK}{{{\mathcal K}_\vtg(n)}}
\newcommand{\afbfSr}{{\boldsymbol{\mathcal S}}_\vtg(\cycn,r)}
\newcommand{\afbfS}{{\boldsymbol{\mathcal S}}_\vtg}

\newcommand{\afal}{{\boldsymbol\alpha}^\vartriangle}
\newcommand{\afbt}{{\boldsymbol\beta}^\vartriangle}
\newcommand{\afbse}{\boldsymbol e^\vartriangle}
\newcommand{\afmbnn}{\mathbb N_\vtg^{\cycn}}
\newcommand{\afmbzn}{\mathbb Z_\vtg^{\cycn}}

\newcommand{\afLa}{\Lambda_\vartriangle}

\newcommand{\afLanr}{\Lambda_\vtg(\cycn,r)}\newcommand{\afLarr}{\Lambda_\vtg(r,r)}
\newcommand{\afThn}{\Theta_\vtg(\cycn)}

\newcommand{\afThnpm}{\Theta_\vtg^\pm(\cycn)}
\newcommand{\afThnp}{\Theta_\vtg^+(\cycn)}
\newcommand{\afThnm}{\Theta_\vtg^-(\cycn)}

\newcommand{\afThnr}{\Theta_\vtg(\cycn,r)}
\newcommand{\afTh}{\Theta_\vartriangle}
\newcommand{\afMnq}{M_{\vtg,\cycn}(\mathbb Q)}
\newcommand{\afMnz}{M_{\vtg,\cycn}(\mathbb Z)}
\newcommand{\afMnc}{M_{\vtg,\cycn}(\mathbb C)}
\newcommand{\afInr}{I_\vtg(n,r)}

\newcommand{\ttp}{\mathtt{p}}

\newcommand{\ttk}{\mathtt{k}}
\newcommand{\tth}{\mathtt{h}}

\def\leq{\leqslant}\def\geq{\geqslant}
\def\le{\leqslant}\def\ge{\geqslant}

\newcommand{\Th}{\Theta}
 \newcommand{\dt}{\delta}
 
 \newcommand{\Dt}{\Delta}
 \newcommand{\lm}{\longmapsto}
 \newcommand{\map}{\mapsto}
 
  \newcommand{\Og}{\Omega}
  \def\bdOg{{\boldsymbol \Og}}
  
 \newcommand{\og}{\omega}
 
  \newcommand{\vi}{\varphi}

 \newcommand{\up}{v}
 \newcommand{\nup}{v}
 \newcommand{\ep}{\varepsilon}
 \newcommand{\al}{\alpha}
 \newcommand{\bt}{\beta}
 \newcommand{\h}{\widehat}
 \newcommand{\ti}{\widetilde}

\newcommand{\afzrp}{\zeta_r^+}
\newcommand{\afzrm}{\zeta_r^-}

\newcommand{\sg}{\sigma}

\def\th{\theta}

\newcommand{\p}{\prec}
\newcommand{\pr}{\preccurlyeq}
\newcommand{\bop}{\bigoplus}
\newcommand{\op}{\oplus}
\newcommand{\ot}{\otimes}

\newcommand{\bfe}{\mathbf{e}}

\newcommand{\bfl}{\mathbf{0}}

\newcommand{\ol}{\overline}
\newcommand{\lra}{\longrightarrow}
\newcommand{\ra}{\rightarrow}
 \newcommand{\la}{{\lambda}}
 \newcommand{\tila}{{\widehat{\lambda}}}
 \newcommand{\timu}{{\widehat{\mu}}}
 \newcommand{\La}{\Lambda}

 \newcommand{\mbn}{\mathbb N}
 \newcommand{\mbq}{\mathbb Q}
 \newcommand{\mbc}{\mathbb C}
 \newcommand{\mbz}{\mathbb Z}
 
 \newcommand{\bfi}{{\mathbf{i}}}
  \newcommand{\bfd}{{\mathbf{d}}}
 
 \newcommand{\bfj}{{\mathbf{j}}}
 
 \newcommand{\bfs}{{\mathbf{s}}}

\newcommand{\bft}{{\mathbf{t}}}
\newcommand{\bfa}{{\mathbf{a}}}
\newcommand{\bfb}{{\mathbf{b}}}
\newcommand{\bfm}{{\mathbf{m}}}
 
 \newcommand{\bfU}{{\mathbf{U}}}
  \newcommand{\bfL}{{\mathbf{L}}}

\newcommand{\ga}{{\gamma}}

\newcommand{\Aut}{\operatorname{Aut}}
\newcommand{\End}{\operatorname{End}}
\newcommand{\Hom}{\operatorname{Hom}}
\newcommand{\Ext}{\operatorname{Ext}}

\newcommand{\spann}{\operatorname{span}}
\newcommand{\diag}{\operatorname{diag}}

\def\rad{\mbox{\rm rad}\,}

\newcommand{\afUnrC}{U_{\vtg}(n,r)_\mbc}
\newcommand{\afUnC}{U_{\vtg}(n)_\mbc}
\newcommand{\Lcp}{\bar L}
\newcommand{\Mcp}{\bar M}
\newcommand{\Icp}{\bar I}

\def\ro{\text{\rm ro}}
\def\co{\text{\rm co}}
\def\wh{\widehat}

\def\smat#1{{\begin{smallmatrix}#1\end{smallmatrix}}}
\def\fD{{\mathfrak D}}
\newcommand{\ul}{\underline}
\newcommand{\bfsg}{{\boldsymbol\sigma}}

\def\sSq{{\mbc_G(\scrY\times\scrY)}}
\def\sHq{{\mbc_G(\scrX\times\scrX)}}
\def\sTq{{\mbc_G(\scrY\times\scrX)}}
\def\scF{{{\mathscr F}}}
\newcommand{\afThrr}{\Theta_\vtg(r,r)}
\def\afsygr{{\fS_{\vtg,r}}}
\newcommand{\afJnr}{I_\vtg(n,r)_0}
\def\rd{{\text{\rm rd}}}

\def\ttx{{\tt x}}
\def\ttg{{\tt g}}

\def\ttk{{\tt k}}

\newcommand{\bsA}{{\boldsymbol{\mathfrak A}}}
\def\hmod{{\text-}{\mathsf{mod}}}
\def\hMod{{\text-}{\mathsf{Mod}}}
\def\id{{\operatorname{id}}}
\newcommand{\afSrC}{{\mathcal S}_{\vtg}(\cycn,r)_{\mathbb{C}}}

\newcommand{\afUglC}{\text{\rm U}_{\mathbb C}(\widehat{\frak{gl}}_n)}
\newcommand{\DC}{{\fD}_{\vtg,\mathbb C}}

\newcommand{\zrC}{\xi_{r,\mathbb C}}
\newcommand{\bfP}{\mathbf{P}}
\newcommand{\bfQ}{\mathbf{Q}}

\def\Thnr{{\Theta_{\!\vtg\!}(\!\begin{smallmatrix} n\\ r\end{smallmatrix}\!)}}
\def\Thrr{{\Theta_{\!\vtg\!}(\!\begin{smallmatrix} r\\ r\end{smallmatrix}\!)}}
\def\Thnrr{{\Theta_{\!\vtg\!}(\!\begin{smallmatrix} n\\ r\end{smallmatrix},r\!)_\og}}

\def\bfone{{\mathfrak l}}

\begin{document}
\frontmatter

\title{A Double Hall Algebra Approach to Affine Quantum Schur--Weyl Theory}

\author{Bangming Deng}
\address{School of Mathematical Sciences, Beijing Normal University,
Beijing 100875,  China.}
\email{dengbm@bnu.edu.cn}
\author{Jie Du}
\address{School of Mathematics, University of New South Wales,
Sydney 2052, Australia.} \email{j.du@unsw.edu.au}
\author{Qiang Fu$^\dagger$}

\address{Department of Mathematics, Tongji University, Shanghai, 200092, China.}
\email{q.fu@hotmail.com}
\address{\rm$^\dagger$Corresponding author.}


\begin{abstract} We investigate the structure of the double Ringel--Hall algebras
associated with cyclic quivers and its connections with quantum loop
algebras of $\mathfrak{gl}_n$, affine quantum Schur algebras and
affine Hecke algebras. This includes their Drinfeld--Jimbo type
presentation, affine quantum Schur--Weyl reciprocity,
representations of affine quantum Schur algebras, and connections
with various existing works by Lusztig, Varagnolo--Vasserot,
Schiffmann, Hubery, Chari--Pressley, Frenkel--Mukhin, etc. We will
also discuss conjectures on a realization of
Beilinson--Lusztig--MacPherson type and Lusztig type integral forms
for double Ringel--Hall algebras.
\end{abstract}

\date{\today}

\maketitle

\setcounter{page}{6}
\tableofcontents

\mainmatter

 \sloppy \maketitle
\chapter*{Introduction}

Quantum Schur--Weyl theory refers to a three-level duality relation.
At Level I, it investigates a certain double centralizer property,
{\it the quantum Schur--Weyl reciprocity}, associated with some
bimodules of quantum $\frak{gl}_n$ and the Hecke algebra (of type
$A$)---the tensor spaces of the natural representation of quantum
$\frak{gl}_n$ (see \cite{Ji}, \cite{Du92}, \cite{DPS}). This is the
quantum version of the well-known Schur--Weyl reciprocity which was
beautifully used in H. Weyl's influential book \cite{We}. The key
ingredient of the reciprocity is a class of important finite
dimensional endomorphism algebras, {\it the quantum Schur algebras},
whose classical version was introduced by I. Schur over a hundred
years ago (see \cite{Sc1}, \cite{Sc2}). At Level II, it establishes
a certain {\it Morita equivalence} between quantum Schur algebras
and Hecke algebras. Thus, quantum Schur algebras are used to bridge
representations of quantum $\frak{gl}_n$ and Hecke algebras. More
precisely, they link polynomial representations of quantum
$\frak{gl}_n$ with representations of Hecke algebras via the Morita
equivalence. The third level of this duality relation is motivated
by two simple questions associated with the structure of
(associative) algebras. If an algebra is defined by generators and
relations, the {\it realization problem} is to reconstruct the
algebra as a vector space with hopefully explicit multiplication
formulas on elements of a basis; while, if an algebra is defined in
term of a vector space such as an endomorphism algebra, it is
natural to seek their generators and defining relations.

As one of the important problems in quantum group theory, the
realization problem is to construct a quantum group in terms of a
vector space and certain multiplication rules on basis elements.
This problem is crucial to understand their structure and
representations (see \cite[p.~xiii]{Kac2} for a similar problem for
Kac--Moody Lie algebras). Though the Ringel--Hall algebra
realization of the $\pm$-part of quantum enveloping algebras
associated with symmetrizable Cartan matrices is an important
breakthrough in the early nineties of the last century, especially
for the introduction of the geometric approach to the theory, the
same problem for the entire quantum groups is far from completion.
However, Beilinson--Lusztig--MacPherson (BLM) \cite{BLM} solved the
problem for quantum $\frak{gl}_n$ by exploring further properties
coming out from the quantum Schur--Weyl reciprocity. On the other
hand, as endomorphism algebras and as homomorphic images of quantum
$\frak{gl}_n$, it is natural to look for presentations for quantum
Schur algebras via the presentation of quantum $\frak{gl}_n$. This
problem was first considered in \cite{DG} (see also \cite{DP03}).
Thus, as a particular feature in the type $A$ theory, realizing
quantum $\frak{gl}_n$ and  presenting quantum Schur algebras form
the Level III of this duality relation. For a complete account of
the quantum Schur--Weyl theory and further references, see Parts 3
and 5 of \cite{DDPW}.

There are several developments in the establishment of an affine
analogue of the quantum Schur--Weyl theory. Soon after BLM's work,
Ginzburg and Vasserot \cite{GV, V} used a geometric and
$K$-theoretic approach to investigate {\it affine quantum Schur algebras}
\footnote{Perhaps, it should be called quantum affine Schur algebras.
Since our purpose is to establish an affine analogue of the quantum Schur--Weyl theory, this terminology seems more appropriate to reflect this.}
as homomorphic images of {\it quantum loop algebra}
 $\bfU(\afgl)$ of $\frak{gl}_n$ in the sense of Drinfeld's new
presentation \cite{Dr88}, called quantum affine $\frak{gl}_n$ (at
level 0) in this paper. This establishes at Level I the first
centralizer property for the affine analogue of the quantum
Schur--Weyl reciprocity. Six years later, investigations around
affine quantum Schur algebras focused on their different definitions
and hence different applications. For example, Lusztig \cite{Lu99}
generalized the fundamental multiplication formulas \cite[3.4]{BLM}
for quantum Schur algebras to the affine case and showed that
(extended) quantum affine $\frak{sl}_n$ does not map onto affine
quantum Schur algebras; Varagnolo--Vasserot \cite{VV99} investigated
Ringel--Hall algebra actions on tensor spaces and described the
geometrically-defined affine quantum Schur algebras in terms of the
endomorphism algebras of tensor spaces. Moreover, they proved that
the tensor space definition coincides with Green's definition
\cite{Gr99} via $q$-permutation modules. Some progress on the second
centralizer property has also been made recently by Pouchin
\cite{GP}. The approaches used in these works  are mainly geometric.
However, like the non-affine case, there would be more favorable
algebraic and combinatorial approaches.

At Level II, representations at non-roots-of-unity of quantum affine
$\frak{sl}_n$ and $\frak{gl}_n$ over the complex number field
$\mbc$, including classifications of finite dimensional simple
modules, have been thoroughly investigated by Chari--Pressley
\cite{CP91,CPbk,CP95}, and Frenkel--Mukhin \cite{FM} in terms of
Drinfeld polynomials. Moreover, an equivalence between the module
category of the Hecke algebra $\afHrC$ and a certain full
subcategory of quantum affine $\frak{sl}_n$ (resp., $\frak{gl}_n$)
has also been established algebraically by Chari--Pressley \cite{CP}
(resp., geometrically by Ginzburg--Reshetikhin--Vasserot \cite{GRV})
under the condition $n>r$ (resp., $n\geq r$). Note that the approach
in \cite{GRV} uses intersection cohomology complexes. It would be
interesting to know how affine quantum Schur algebras would play a
role in these works.

Much less progress has been made at Level III. When $n>r$,
Doty-Green \cite{DG} and McGerty \cite{Mc07} have found a
presentation for affine quantum Schur algebras, while the last two
authors of the paper have investigated the realization problem in
\cite{DF09}, where they first developed an approach without using
the stabilization property, a key property used in the BLM approach,
and presented an ideal candidate for the realization of quantum
affine $\frak{gl}_n$.

This paper attempts to establish the affine quantum Schur--Weyl
theory as a whole and is an outcome of {\it algebraically}
understanding the works mentioned above.

First, building on Schiffmann \cite{Sch} and Hubery \cite{Hub2}, our
starting point is to interpret Drinfeld's quantum affine
$\frak{gl}_n$ in the sense of \cite{Dr88} as the double Ringel--Hall
algebra $\dHallr$ over $\mbq(\up)$ of the cyclic quiver $\tri$ with
$n$ vertices, and present the latter in terms of Chevalley type
generators. In this way, we easily obtain an action on the tensor
space which upon restriction coincides with the Ringel--Hall algebra
action defined geometrically by Varagnolo--Vasserot \cite{VV99} and
commutes with the affine Hecke algebra action.

Second, by a thorough investigation of a BLM type basis for affine
quantum Schur algebras, we introduce certain triangular relations
for the corresponding structure constants and, hence, a triangular
decomposition for affine quantum Schur algebras. With this
decomposition, we establish explicit algebra epimorphisms
$\xi_r=\xi_{r,\mbq(\up)}$ from the double Ringel--Hall algebra
$\dHallr$ to affine quantum Schur algebras
$\afbfSr:=\afSr_{\mbq(\up)}$ for all $r\geq0$. This algebraic
construction has several nice applications, especially at Levels II
and III. For example, the homomorphic image of commutator formulas
for semisimple generators gives rise to a beautiful polynomial
identity whose combinatorial proof remains mysterious.

Like the quantum Schur algebra case, we will establish for $n\geq r$
a Morita equivalence between affine quantum Schur algebras
$\afSr_\field$ and affine Hecke algebras $\afHr_\field$ of type $A$
over a field $\field$ with a non-root-of-unity parameter. As a
by-product, we prove that every simple $\afSr_\field$-module is
finite dimensional. Thus, applying the classification of simple
$\afHrC$-modules by Zelevinsky \cite{Zelevinsky} and Rogawski
\cite{Rogawski} yields a classification of simple $\afSrC$-modules.
Hence, inflation via $\zrC$ gives many finite dimensional simple
$\afUglC$-modules. We will also use the epimorphisms $\zrC$ together
with the action on tensor spaces and a result of Chari--Pressley to
prove that finite dimensional simple polynomial representations of
$\afUglC$ are all inflations of simple $\afSrC$-modules. In this
way, we can see the bridging role played by affine quantum Schur
algebras between representations of quantum affine $\frak{gl}_n$ and
those of affine Hecke algebras. Moreover, we obtain a classification
of simple $\afSrC$-modules in terms of Drinfeld polynomials and,
when $n>r$, we identify them with those arising from simple
$\afHrC$-modules.

Our findings also show that, if we regard the category $\afSrC\hMod$
of $\afSrC$-modules as a full subcategory of $\afUslC$-modules, this
category is quite different from the category $\mathcal
C^{hi}\cap\mathcal C'$ considered in \cite[\S6.2]{Lu93}. For
example, the latter is completely reducible and simple objects are
usually infinite dimensional, while $\afSrC\hMod$ is unlikely to be
completely reducible and all simple objects are finite dimensional.
As observed in \cite[Rem.~9.4(2)]{DF} for quantum $\frak{gl}_\infty$
and infinite quantum Schur algebras, this is another kind of
phenomenon of infinite type in contrast with the finite type case.


The discussion of the realization and presentation problems is also
based on the algebra epimorphisms $\xi_r$ and relies on the use of
semisimple generators and indecomposable generators for $\dHallr$
which are crucial to understand the integral structure and
multiplication formulas. We first use the new presentation for
$\dHallr$ to give a decomposition for $\afbfSr=\bfU_\vtg(n,r)\mathbf
Z_\vtg(n,r)$ into a product of two subalgebras, where $\mathbf
Z_\vtg(n,r)$ is a central subalgebra and $\bfU_\vtg(n,r)$ is the
homomorphic image of $\bfU_\vtg(n)$, the extended quantum affine
$\frak{sl}_n$. By a close look at this structure, we manage to get a
presentation for ${\boldsymbol\sS}_\vtg(r,r)$ for all $r\geq 1$ and
acknowledge that the presentation problem is very complicated in the
$n<r$ case. On the other hand, we formulate a realization conjecture
suggested by the work \cite{DF09} and prove the conjecture in the
classical ($\up=1$) case.

We remark that, unlike the geometric approach in which the ground
ring must be a field or mostly the complex number field $\mbc$, the
algebraic, or rather, the representation-theoretic approach we used
in the paper works largely over a ring or mostly the integral
Laurent polynomial ring $\mbz[\up,\up^{-1}]$.

We organize the paper as follows.

In the first preliminary chapter,
we introduce in \S1.4 three different types of generators and their
associated monomial bases for the Ringel--Hall algebras of cyclic
quivers, and display in \S1.5 the Green--Xiao Hopf structure on the extended
version of these algebras.

Chapter 2 introduces a new presentation using Chevalley generators
for Drinfeld's quantum loop algebra $\bfU(\afgl)$ of
$\mathfrak{gl}_n$. This is achieved by
 constructing the presentation for the double Ringel--Hall
algebra $\dHallr$ associated with cyclic quivers  (Theorem
\ref{presentation dHallAlg}), based on works
 of Schiffmann and Hubery, and by lifting
Beck's algebra monomorphism from the quantum $\afsl$ with a
Drinfeld--Jimbo presentation into $\bfU(\afgl)$ to obtain an
isomorphism between $\dHallr$ and $\bfU(\afgl)$  (Theorem \ref{iso
afgln dHallr}).

Chapter 3 investigates the structure of affine quantum Schur
algebras. We first recall the geometric definition by
Ginzburg--Vasserot and Lusztig, the Hecke algebra definition by R.
Green and the tensor space definition by Varagnolo--Vasserot. Using
the Chevalley generators, we easily obtain an action on the
$\mbq(\up)$-space $\bfOg$ with a basis indexed by $\mathbb Z$, and
hence, an action on $\bfOg^{\ot r}$ (\S3.5). We prove that this
action commutes with the affine Hecke algebra action defined in
\cite{VV99}. Moreover, we show that the restriction of the action to
the $\pm$-part of $\dHallr$ (i.e., to the corresponding Ringel--Hall
algebra) coincides with the Ringel--Hall algebra action
geometrically defined by Varagnolo--Vasserot (Theorem \ref{xirl}).
As an application of this coincidence, the commutator formula
associated with semisimple generators, arising from the skew-Hopf
pairing, gives rise to certain polynomial identity associated with a
pair of elements $\la,\mu\in\mbn_\vtg^n$ (Corollary
\ref{polyidentity}). The main result of the chapter is an elementary
proof of the surjective homomorphism $\xi_r$ from the double
Ringel--Hall algebra $\dHallr$, i.e., the quantum loop algebra
$\bfU(\afgl)$, onto the affine quantum Schur algebra $\afbfSr$
(Theorem \ref{surjective-dHall-aff}). The approach we used is the
establishment of a triangular decomposition of $\afbfSr$ (Theorem
\ref{PBW basis of affine q-Schur algebras}) through an analysis of
the BLM type bases.

In Chapter 4, we discuss the representation theory of affine quantum
Schur algebras over $\mbc$ and its connection to polynomial
representations of quantum affine $\frak{gl}_n$ and representations
of affine Hecke algebras. We first establish a category equivalence
between the module categories $\afSrC\hmod$ and $\afHrC\hmod$ for
$n\geq r$ (Theorem \ref{category equiv}). As an application, we will
reinterpret Chari--Pressley's category equivalence
(\cite[Th.~4.2]{CP}) between (level $r$) representations of
$\afUslC$ and those of affine Hecke algebras $\afHrC$, where $n>r$,
in terms of representations of $\afSrC$ (Proposition
\ref{Chari-Presslry}). We then develop two approaches to the
classification of simple $\afSrC$-modules. In the so-called upwards
approach, we use the classification of simple $\afHrC$-modules of
Zelevinsky and Rogawski to classify simple $\afSrC$-modules (Theorem
\ref{weakly category equiv}), while in the downwards approach, we
determine the classification of simple $\afSrC$-modules (Theorem
\ref{representation}) in terms of simple polynomial representations
of $\afUglC$. When $n>r$, we prove an identification theorem
(Theorem \ref{n>r representation}) for the two classifications.
Finally, in \S4.6, a classification of simple $\afUnrC$-modules is
also completed and its connection to finite dimensional simple
$\afUslC$-modules and finite dimensional simple (polynomial)
$\afUglC$-modules are also discussed.

We move on to look at the presentation and realization problems in
Chapter 5. Since $\dHallr\cong \bfU_\vtg(n)\otimes {\mathbf
Z}_\vtg(n)$, it follows that $\afbfSr=\bfU_\vtg(n,r){\mathbf
Z}_\vtg(n,r)$, where $\bfU_\vtg(n,r)$ and ${\mathbf Z}_\vtg(n,r)$
are homomorphic images of $\bfU_\vtg(n)$ and ${\mathbf Z}_\vtg(n)$,
respectively. Thus, ${\mathbf Z}_\vtg(n,r)\subseteq\bfU_\vtg(n,r)$
if and only if $n>r$. A presentation for $\bfU_\vtg(n,r)$ is given
in \cite{Mc07} (see also \cite{DGr} for the case $n>r$). Building on
McGerty's presentation, we first give a Drinfeld--Jimbo type
presentation for the subalgebra $\bfU_\vtg(n,r)$ (Theorem \ref{2nd presentation for im(afzr)}). We
then describe a presentation for the central subalgebra ${\mathbf
Z}_\vtg(n,r)$ as a Laurent polynomial ring in one indeterminate over
a polynomial ring in $r-1$ indeterminates over $\mbq(\up)$.
 We manage to describe a presentation for $\boldsymbol\sS_\vtg(r,r)$ for all $r\geq1$
 (Theorem \ref{n,n case}) by adding an extra generator (and its inverse) together with an
additional set of relations on top of the relations given in Theorem
\ref{2nd presentation for im(afzr)}. What we will see from this case
is that the presentation for $\afbfSr$ with $r>n$ can be very
complicated.

We discuss the realization problem from \S5.4 onwards. We first
describe the modified BLM approach developed in \cite{DF09}. With
some supporting evidence, we then formulate the realization
conjecture (Conjecture \ref{realization conjecture}) as suggested in
\cite[5.5(2)]{DF09}, and state its classical ($\up=1$) version.  We
end the chapter with a closer look at Lusztig's transfer maps
\cite{L00} by displaying some explicit formulas for their action on
the semisimple generators for $\afbfSr$ (Corollary \ref{5.5.2}). These formulas show also
that the homomorphism from $\bfU(\afsl)$ to $\underset\longleftarrow\lim\boldsymbol\sS_\vtg(n,n+m)$
induced by the transfer maps cannot either be extended to the double Ringel--Hall algebra $'\dHallr$.
(Lusztig already pointed out that it cannot be extended to $\bfU_\vtg(n)$.)
 This somewhat justifies the realization conjecture in
which a direct product is used instead of the inverse limit.

In the last Chapter 6, we prove the realization conjecture for the
classical ($\up=1$) case. The key step to the proof is the
establishment of more multiplication formulas (Proposition
\ref{MFforSBE}) between homogeneous indecomposable generators and an
arbitrary BLM type basis element. As a by-product, we display a
basis for the universal enveloping algebra of the loop algebra of
$\mathfrak{gl}_n$ (Theorem \ref{classical BLM basis}) together with
explicit multiplication formulas between generators and arbitrary
basis elements (Corollary \ref{classical MFs}).

There are two Appendixes \S\S3.10,6.4 which collect a number of
lengthy calculations used in some proofs.

{\bf Conjectures and problems.} There are quite few conjectures and
problems throughout the paper. The conjectures are mostly natural
generalizations to the affine case, for example, the realization
conjecture \ref{realization conjecture} and the conjectures in \S3.8
on an integral form for double Ringel--Hall algebras and the second
centralizer property in the affine quantum Schur--Weyl reciprocity.
Some problems are designed to seek further solutions to certain
questions such as ``quantum Serre relations'' for semisimple
generators (Problem \ref{semisimple presentation}), Affine Branching
Rules (Problem \ref{Prob-Branching-Rules}), and further
identification of simple modules from different classifications
(Problem \ref{Prob-Identif-Thm}). There are also problems for
seeking different proofs. Problems \ref{Problem for MF} and
\ref{Prob-realization} form a key step towards the proof of the
realization conjecture.


{\bf Notational scheme.} For most of the notations throughout the
paper, if it involves a subscript $_\vtg$ or a superscript $^\vtg$,
it indicates that the same notation without $\vtg$ has been used in
the non-affine case, say, in \cite{BLM}, \cite{Gr80}, etc. Here the
triangle $\vtg$ depicts the cyclic Dynkin diagram of affine type
$A$.

For a ground ring $\sZ$ and a $\sZ$-module (or a $\sZ$-algebra)
$\mathcal A$, we often use the notation $\mathcal A_\field:=\mathcal
A\otimes \field$ to  represent the object obtained by {\it base
change} to a field $\field$, which itself is a $\sZ$-module. In
particular, if $\sZ=\mbz[\up,\up^{-1}]$, then we write ${\boldsymbol
\sA}$ for $\sA_{\mbq(\up)}$.


{\bf Acknowledgements.} The main results of the paper have been
presented by the authors at the following conferences and workshops.  We would like to thank the organizers for the opportunities of presenting our work.

 \noindent (Presented by B. Deng)
\begin{itemize}
\item the International Conference on Cohomology
and Representation Theories held at the Center of Mathematical
Sciences, Hangzhou, July 2009;
\item the Conference
on Perspectives in Representation Theory, Cologne, September 2009,
\end{itemize}
\noindent (Presented by J. Du)
\begin{itemize}
\item the International Workshop on
Quantized Algebras and Physics held at the Chern Institute of
Mathematics, July 2009;
\item the Workshop on Representation Theory
and Combinatorics held at the Morningside Center of Mathematics,
August 2009;
\item the International Workshop on Combinatorial and
Geometric Approach to Representation Theory, Seoul National
University, September 2009;
\item the 5th International Conference on Representation Theory, Xian, August 2010;
 \item the 2010 ICM Satellite Conference, Bangalore, August 2010,
\end{itemize}

\noindent (Presented by Q. Fu)
\begin{itemize}
\item the Eleventh National
Conference on Lie algebras, Guizhou Normal University, July 2009;
\item Shanghai Workshop on Representation Theory and Complex Analysis, a joint conference of ECNU, Tongji and Nagoya Universities, Shanghai, November 2009;
\item the 12th National Algebra Conference, Lanzhou, June 2010.
\end{itemize}

The research was partially supported by the Australian Research Council (DP0665124, 2006--2008), the 111 Programme of China, the
Natural Science Foundation of China and the Doctoral Program of
Higher Education. The first three and fifth chapters were written
while Deng and Fu were visiting the University of New South Wales at
the various times. The hospitality and support of UNSW are
gratefully acknowledged.

The second author would like to thank Alexander Kleshchev and Arun
Ram for helpful comments on  the Affine Branching Rule
\eqref{Branching}, and V. Chari for several discussions and
explanations on the paper \cite{CP} and some related topics. He
would also like to thank East China Normal University, Universities
of Mainz, Virginia and Auckland for their hospitality
during his sabbatical leave in 2009.

\vfill
\noindent
Sydney\\
22 October 2010

\chapter{Preliminaries}

We start with the loop algebra of $\mathfrak{gl}_n(\mbc)$ and its
interpretation in terms of matrix Lie algebra. We use the subalgebra
of integer matrices of the latter to introduce several important
index sets which will be used throughout the paper. Ringel--Hall
algebras $\bfHall$ associated with cyclic quivers $\tri$ and their
geometric construction are introduced  in \S1.2. In \S1.3, we discuss the composition
subalgebra $\bfcomp$ of $\bfHall$ and relate it to the quantum loop algebra
$\bfU(\afsl)$. We then describe in
\S1.4 three types of generators for $\bfHall$, which consist of
all simple modules together with, respectively, the Schiffmann--Hubery
central elements, homogeneous semisimple modules, and homogeneous
indecomposable modules, and their associated monomial bases (Corollaries \ref{monomial basis with central elements}
and \ref{monomial bases with semisimple/indecom}). These
generating sets will play different roles in the sequel. Finally,
extended Ringel--Hall algebras and their Hopf structure are
discussed in \S1.5.

\section{The loop algebra $\afgl$ and some notations}

For a positive integer $n$, let $\frak{gl}_n(\mbc)$ be the complex general linear Lie algebra, and let
$$\afgl(\mbc):={\frak{gl}_n}(\mbc)\ot\mbc[t,t^{-1}]$$
be the loop algebra associated to $\frak{gl}_n(\mbc)$;\index{loop
algebra of $\frak{gl}_n$} see \cite{Kac2}. Thus, $\afgl(\mbc)$ is
spanned by $E_{i,j}\ot t^m$ for all $1\le i,j\le n$ and $m\in\mbz$,
where $E_{i,j}$ is the matrix $(\delta_{k,i}\delta_{j,l})_{1\le
k,l\le n}$. The (Lie) multiplication is the bracket product
associated with the multiplication
$$(E_{i,j}\ot t^m)(E_{k,l}\ot t^{m'})=\delta_{j,k}E_{i,l}\ot t^{m+m'}.$$

We may interpret the Lie algebra $\afgl(\mbc)$ as a matrix Lie algebra.
Let
 $\afMnc$  be the set of all $\mbz\times\mbz$ complex matrices
$A=(a_{i,j})_{i,j\in\mbz}$ with $a_{i,j}\in\mbc$ such that
\begin{itemize}
\item[(a)]$a_{i,j}=a_{i+n,j+n}$ for $i,j\in\mbz$; \item[(b)] for
every $i\in\mbz$, the set $\{j\in\mbz\mid a_{i,j}\not=0\}$ is finite.
\end{itemize}
Clearly, conditions (a) and (b) imply that there are only finitely
many nonzero entries in each column of $A$. For $A,B\in\afMnc$, let
$[A,B]=AB-BA$. Then ($\afMnc,[\ ,\ ]$) becomes a Lie algebra over
$\mbc$.

Denote by $M_{n,\bullet}(\mbc)$ the set of
$n\times\mbz$ matrices $A=(a_{i,j})$ over $\mbc$ satisfying (b).
Then there is a bijection
\begin{equation}\label{flat1}
\flat_1:\afMnc\lra M_{n,\bullet}(\mbc),\quad
(a_{i,j})_{i,j\in\mbz}\longmapsto (a_{i,j})_{1\leq i\leq
n,j\in\mbz}.
\end{equation}

For $i,j\in\mbz$, let $\afE_{i,j}\in\afMnc$ be the matrix
$(e^{i,j}_{k,l})_{k,l\in\mbz}$ defined by
\begin{equation*}e_{k,l}^{i,j}=
\begin{cases}1&\text{if $k=i+sn,l=j+sn$ for some $s\in\mbz$,}\\
0&\text{otherwise}.\end{cases}
\end{equation*}
The set $\{\afE_{i,j}|1\leq i\leq n,j\in\mbz\}$ is a $\mbc$-basis of
$\afMnc$. Since
$\afE_{i,j+ln}\afE_{p,q+kn}=\delta_{j,p}\afE_{i,q+(l+k)n},$ for all
$i,j,p,q,l,k\in\mbz$ with $1\leq j,p\leq n$, it follows that the map
$$\afMnc\lra\afgl(\mbc),\,\,\,\afE_{i,j+ln}\longmapsto E_{i,j}\ot t^l, \,1\le i,j\le n,l\in\mbz $$
is a Lie algebra isomorphism. We will identify the loop algebra
$\afgl(\mbc)$ with $\afMnc$ in the sequel.

In Chapter 6, we will consider the loop algebra
$\afgl:=\afgl(\mbq)=\afMnq$ defined over $\mbq$ and its universal
enveloping algebra $\sU(\afgl)$ and triangular parts $\sU(\afgl)^+$,
$\sU(\afgl)^-$ and $\sU(\afgl)^0$. Here $\sU(\afgl)^+$ (resp.,
$\sU(\afgl)^-$, $\sU(\afgl)^0$) is the subalgebra of $\sU(\afgl)$
generated by $\afE_{i,j}$ for all $i<j$ (resp., $\afE_{i,j}$ for all
$i>j$, $\afE_{i,i}$ for all $i$). We will also relate these algebras in \S6.1
with the specializations at $\up=1$ of the Ringel--Hall algebra $\Hall$ and the double Ringel--Hall algebra $\fD_{\vtg}(n)$.

We now introduce some {\bf notations} which will be used throughout the paper.

Consider the subset $\afMnz$ of $\afMnc$ consisting of matrices with integer entries.
For each $A\in\afMnz$, let
$$\ro(A)=\bigl(\sum_{j\in\mbz}a_{i,j}\bigr)_{i\in\mbz},\qquad
\co(A)=\bigl(\sum_{i\in\mbz}a_{i,j}\bigr)_{j\in\mbz}.$$\index{$\ro(A)$,
row sum vector}\index{$\co(A)$, column sum vector} We obtain
functions
$$\ro,\co:\afMnz\lra\afmbzn,$$
where
$$\afmbzn:=\{(\la_i)_{i\in\mbz}\mid
\la_i\in\mbz,\,\la_i=\la_{i-n}\ \text{for}\ i\in\mbz\}.$$
 \index{$\afmbzn$, set of infinite sequences of integers with $n$ periodicity}
For $\la=(\la_i)_{i\in\mbz}\in\afmbzn$, $A\in\afMnz$ and
$i_0\in\mbz$, let
$$\sg(\la)=\sum_{i_0+1\leq i\leq i_0+n}\la_i,\qquad
\sg(A)=\sum_{i_0+1\leq i\leq i_0+n\atop
j\in\mbz}a_{i,j}=\sum_{i_0+1\leq j\leq i_0+n\atop
 i\in\mbz}a_{i,j}. \index{$\sg(\la)$, sum of $n$ consecutive components of $\la$} \index{$\sg(A)$, sum of $n$ consecutive rows of
 $A$} $$
 Clearly, both $\sg(\la)$ and $\sg(A)$ are defined and
independent of $i_0$. We sometimes identify $\afmbzn$ with $\mbz^n$
via the following bijection
\begin{equation}\label{flat2}
\flat_2:\afmbzn\lra\mbz^n,\quad \la\longmapsto
\flat_2(\la)=(\la_1,\ldots,\la_n).
\end{equation}
For example, we define a ``dot product'' on $\afmbzn$ by
$\la\centerdot\mu:=\flat_2(\la)\centerdot\flat_2(\mu)=\sum_{i=1}^n\la_i\mu_i$,
and define the order relation $\leq$ \index{$\leq$, order on $\afmbzn$ or $\mbz^n$} on $\afmbzn$ by setting
\begin{equation}\label{order on afmbzn}
\la\leq\mu \iff \flat_2(\la)\leq\flat_2(\mu)\iff\la_i\leq \mu_i\text{ for all $1\leq i\leq n$}.
\end{equation}
 Also, let $\boldsymbol e_i^\vtg\in\afmbzn$ be defined by
$\flat_2(\boldsymbol e_i^\vtg)=\boldsymbol
e_i=(0,\ldots,0,\underset{(i)}{1},0,\ldots,0)$.

Let
$$\aligned
\afThn:=&\{A=(a_{i,j})\in\afMnz\mid a_{i,j}\in\mbn
\}=M_{\vtg,n}(\mbn),\index{$\afThn$, set of $\mbz\times\mbz$-matrices over $\mbn$ with $n$ periodicity} \\
\afmbnn:=&\{(\la_i)_{i\in\mbz}\in \afmbzn\mid \la_i\ge0\},\\
\endaligned$$
and for $r\geq0$,
$$\aligned
\afThnr:=&\{A\in\afThn\mid\sg(A)=r\}\index{$\afThnr$, set of matrices $A\in\afThn$ with $\sigma(A)=r$}\\
\afLanr:=&\{\la\in\afmbnn\mid\sg(\la)=r\}.\\
\endaligned$$
\index{$\afLanr$, set of sequences $\la\in\afmbzn$ with
$\sg(\la)=r$}

The set $M_n(\mbz)$ can be naturally regarded as a subset of
$M_{n,\bullet}(\mbz)$ by sending $(a_{i,j})_{1\leq i,j\leq n}$ to
$(a_{i,j})_{1\leq i\leq n,j\in\mbz}$, where $a_{i,j}=0$ if
$j\in\mbz\backslash[1,n]$. Thus, (the inverse of) $\flat_1$ induces
an embedding
\begin{equation}\label{flat1'}
\flat_1':M_n(\mbz)\lra\afMnz.
\end{equation}
By removing the subscripts $\vtg$, we define similarly the subsets
$\Theta(n)$, $\Theta(n,r)$ of $M_n(\mbz)$ and subset $\La(n,r)$ of
$\mbn^n$, etc. Note that $\flat_2(\La_\vtg(n,r))=\La(n,r)$.
\index{$\La(n,r)$, set of compositions of $r$ into $n$ parts}

Let $\sZ=\mbz[\up,\up^{-1}]$,\index{$\sZ=\mbz[\up,\up^{-1}]$,
Laurent polynomial ring in indeterminate $\up$} where $\up$ is an
indeterminate, and let $\mbq(\up)$ be the fraction field of $\sZ$.
For integers $N,t$ with $t\geq 0$, let
\begin{equation}\label{Gauss poly1}
\leb{N\atop t}\rib=\prod\limits_{1\leq
i\leq t}\frac{\up^{N-i+1}-\up^{-(N-i+1)}}{\up^{i}-\up^{-i}}\in\sZ
\,\,\text{ and }\,\,\leb{N\atop 0}\rib=1.
\end{equation}
If we put $[m]=\frac{\up^m-\up^{-m}}{\up-\up^{-1}}=\leb{m\atop 1}\rib$ and
$[N]^!:=[1][2]\cdots[N]$, then $\leb{N\atop
t}\rib=\frac{[N]^!}{[t]^![N-t]^!}$ for all $1\leq
t\leq N$.

When counting occurs, we often use
$$\dleb{N\atop t}\drib:=\up^{t(N-t)}\leb{N\atop t}\rib$$
to denote the Gaussian polynomials in $\up^2$.

Also, for any $\mathbb Q(\up)$-algebra $\scr A$ and an invertible
element $X\in\scr A$, let
\begin{equation}\label{Gauss poly2}
\leb{X;a\atop t}\rib=\prod_{s=1}^{t}\frac{Xv^{a-s+1}-X^{-1}v^{-a+s-1}}{v^s-v^{-s}}\,\,\text{ and }\,\,\leb{X;a\atop 0}\rib=1.
\end{equation}
for all $a,t\in\mathbb Z$ with $t\geq 1$.


\section{Representations of cyclic quivers and Ringel--Hall algebras}

 Let $\tri$ ($n\geq 2$) be
the cyclic quiver\index{$\tri$, cyclic quiver}\index{cyclic quiver}
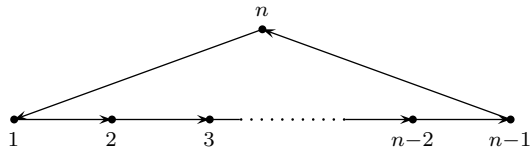
\begin{center}
\begin{pspicture}(-3,-.6)(3.6,1.6)
\psset{linewidth=0.5pt, arrowsize=3.5pt}
\psdot*(-3,0) \psdot*(-1.7,0) \psdot*(-.4,0) \psdot*(2.3,0)
\psdot*(3.6,0)\psdot*(.3,1.2) \uput[u](.3,1.2){$_{n}$}
\uput[d](-3,0){$_1$} \uput[d](-1.7,0){$_2$} \uput[d](-.4,0){$_3$}
\uput[d](2.3,0){$_{n-2}$} \uput[d](3.6,0){$_{n-1}$}
\psline{->}(-3,0)(-1.7,0) \psline{->}(-1.7,0)(-.4,0)
\psline(-.4,0)(0,0)
\psline[linestyle=dotted,linewidth=1pt](0,0)(1.4,0)\psline{->}(1.4,0)(2.3,0)
\psline{->}(2.3,0)(3.6,0) \psline{->}(3.6,0)(.3,1.2)
\psline{->}(.3,1.2)(-3,0)
\end{pspicture}
\end{center}
with vertex set $I=\mbz/n\mbz=\{1,2,\ldots,n\}$ and arrow set
$\{i\to i+1\mid i\in I\}$. Let $\field$ be a field. By
$\rep^0\!\!\tri=\rep_\field^0\tri$ we denote the category of
finite-dimensional {\it nilpotent representations}\index{nilpotent
representation} of $\tri$ over $\field$, i.e., representations
$V=(V_i,f_i)_{i\in I}$ of $\tri$ such that all $V_i$ are finite
dimensional and the composition $f_n\cdots f_2f_1:V_1\ra V_1$ is
nilpotent. The vector $\bfdim V=(\dim_\field V_i)\in\mbn I=\mbn^n$
is called the dimension vector of $V$. (We shall sometimes identify
$\mbn I$ with $\mbn^n_\vtg$ under \eqref{flat2}.) For each vertex
$i\in I$, there is a one-dimensional representation $S_i$ in
$\rep^0\!\!\tri$ satisfying $(S_i)_i=\field$ and $(S_i)_j=0$ for
$j\neq i$. It is known that the $S_i$ form a complete set of simple
objects in $\rep^0\!\!\tri$. Hence, each {\it semisimple
representation} \index{semisimple representation} in
$\rep^0\!\!\tri$ is given by $S_{\bfa}=\oplus_{i\in I}a_iS_i$, where
$\bfa=(a_1,\ldots,a_n)\in\mbn I$. A semisimple representation
$S_{\bf a}$ is called {\it sincere}\index{sincere} if $\bf a$ is
sincere, namely, all $a_i$ are positive. In particular, the vector
$$\dt:=(1,\ldots,1)\in\mbn^n$$ will be often used.

Moreover, up to isomorphism, all indecomposable representations in
$\rep^0\!\!\tri$ are given by $S_i[l]$ ($i\in I$ and $l\geq 1$) of
length $l$ with top $S_i$. Thus, the isoclasses of representations
in $\rep^0\!\!\tri$ are indexed by multisegments $\pi=\sum_{i\in
I,\,l\geq 1}\pi_{i,l}[i;l)$, where the representation $M(\pi)$
corresponding to $\pi$ is defined by
$$M(\pi)=M_\field(\pi)=\bop_{1\leq i\leq n,l\geq 1}\pi_{i,l}S_i[l].$$

Since the set $\Pi$ of all multisegments can be identified with the
set
$$\afThnp=\{A=(a_{i,j})\in\afThn\mid a_{i,j}=0\text{ for
$i\geq j$}\}$$
 \index{$\afThnp$, set of strictly upper triangular matrices in
 $\Theta_\vtg(n)$}
 of all strictly upper triangular matrices via
\begin{equation*}
\flat_3:\afThnp\lra\Pi,\qquad A=(a_{i,j})_{i,j\in\mbz}\longmapsto
\sum_{i<j,1\leq i\leq n}a_{i,j}[i;j-i),
\end{equation*}
 we will use $\afThnp$ to index the finite dimensional nilpotent representations.
In particular, for any $i,j\in\mbz$ with $i<j$,  we have
$$M^{i,j}:=M(\afE_{i,j})=S_i[j-i], \text{ and }M^{i+n,j+n}=M^{i,j}.$$ Thus,
for any $A=(a_{i,j})\in\afThnp$ and $i_0\in\mbz$,
$$M(A)=M_\field(A)=\bop_{1\leq i\leq n,i<j}a_{i,j}M^{i,j}=\bop_{i_0+1\leq i\leq
i_0+n,i<j}a_{i,j}M^{i,j}.$$
 For $A=(a_{i,j})\in \afThnp$, set
$$\fkd(A)=\sum_{i<j,1\leq i\leq n}a_{i,j}(j-i).$$
 Then $\dim_\field M(A)=\fkd(A)$. \index{$\fkd(A)$, dimension of $M(A)$}Moreover, for each $\la=(\la_i)\in\mbn I$, set
$A_\la=(a_{i,j})$ with $a_{i,j}=\dt_{j,i+1}\la_i$, i.e.,
$A_\la=\sum_{i=1}^n \la_i E_{i,i+1}^\vtg$.\index{$A_\la$, matrix
defining the semisimple representation $S_\la$} Then
\begin{equation}\label{semisimple module}
M(A_\la)=\bigoplus_{1\leq i\leq n} \la_iS_i=:S_{\la}
\end{equation}\index{$S_\la$, semisimple representation}
 is semisimple. Also, for $A\in\afThnp$, we write $\bfd(A)=\bfdim
 M(A)\in\mbz I$, the dimension vector of $M(A)$. \index{$\bfd(A)$, dimension vector of $M(A)$} Hence, $\mbz I$ is identified
 with the Grothendieck group of $\rep^0\!\!\tri$.

A matrix $A=(a_{i,j})\in\afThnp$ is called {\it aperiodic}\index{aperiodic matrix} if for
each $l\geq1$, there exists $i\in\mbz$ such that $a_{i,i+l}=0$.
Otherwise, $A$ is called  {\it periodic}.\index{periodic matrix} A nilpotent representation
$M(A)$ is called aperiodic (resp., periodic) if $A$ is aperiodic
(resp., periodic).\index{nilpotent representation!aperiodic $\sim$}\index{nilpotent representation!periodic $\sim$}

It is well known that there exist Auslander--Reiten sequences in
$\rep^0\!\!\tri$; see \cite{ARS}. More precisely, for each $i\in I$
and each $l\geq 1$, there is an Auslander--Reiten sequence
$$0\lra S_{i+1}[l]\lra S_i[l+1]\oplus S_{i+1}[l-1]\lra S_i[l]\lra
0,$$
 where we set $S_{i+1}[0]=0$ by convention. The $S_{i+1}[l]$ is
called the Auslander--Reiten translate of $S_i[l]$, denoted by
$\tau S_i[l]$. In this case, $\tau$ indeed defines an equivalence
from $\rep^0\!\!\tri$ to itself, called the Auslander--Reiten
translation. For each $A=(a_{i,j})\in\afThnp$, we define $\tau
(A)\in\afThnp$ by $M(\tau(A))=\tau M(A)$. Thus, if we write
$\tau(A)=(b_{i,j})\in\afThnp$, then $b_{i,j}=a_{i-1,j-1}$ for all
$i,j$.

We now introduce the degeneration order on $\afThnp$ and generic
extensions of nilpotent representations. These notions play an
important role in the study of bases for both the Ringel--Hall
algebra $\bfHall$ of $\tri$ and its composition subalgebra
$\bfcomp$; see, for example, \cite{DD05,DDX}. For two nilpotent
representations $M,N$ in $\rep^0\!\!\tri$ with $\bfdim M=\bfdim N$,
define
\begin{equation}\label{partial order leq on the set afThnp}
N\leq_{\deg} M\Longleftrightarrow
\dim_\field\Hom(X,N)\geq\dim_\field\Hom(X,M)\text{ for all }
X\in\rep^0\!\!\tri;
\end{equation}
 see \cite{Zwa}.\index{$\leq_\deg$, degeneration order}\index{degeneration order}  This gives rise to a partial ordering on the set of
isoclasses of representations in $\rep^0\!\!\tri$, called the {\it
degeneration order}. Thus, it also induces a partial order on
$\afThnp$ by letting
$$A\leq_\deg B\Longleftrightarrow M(A)\leq_{\deg} M(B).$$
 By \cite{Re1} and \cite[\S3]{DD05}, for any two nilpotent representations
$M$ and $N$, there exists a unique extension $G$ (up to isomorphism)
of $M$ by $N$ with minimal $\dim\End(G)$. This representation $G$ is
called the {\it generic extension} \index{$M*N$, generic extension}
of $M$ by $N$ and will be denoted by $M*N$ in the sequel. Moreover,
for nilpotent representations $M_1,M_2,M_3$,
$$(M_1*M_2)*M_3\cong M_1*(M_2*M_3).$$
Also, taking generic extensions preserves the degeneration order.
More precisely, if $N_1\leq_{\deg} M_1$ and $N_2\leq_{\deg} M_2$ ,
then $N_1*N_2\leq_{\deg} M_1*M_2$. For $A,B\in\afThnp$, let
$A*B\in\afThnp$ be defined by $M(A*B)\cong M(A)*M(B)$.

As above, let $\sZ=\mbz[\up,\up^{-1}]$ be the Laurent polynomial
ring in indeterminate $\up$.\index{$\sZ=\mbz[\up,\up^{-1}]$, Laurent
polynomial ring in indeterminate $\up$} By \cite{Ri93} and
\cite{Guo95}, for $A,B_1,\ldots,B_m\in\afThnp$, there is a
polynomial $\vi^{A}_{B_1,\ldots, B_m}\in\mbz[\up^2]$ in $\up^2$,
called the {\it Hall polynomial},\index{$\vi^{A}_{B_1,\ldots, B_m}$,
Hall polynomial}  such that for any finite field $\field$ of $q$
elements, $\vi^{A}_{B_1,\ldots, B_m}|_{\up^2=q}$ (the evaluation of
$\vi^{A}_{B_1,\ldots, B_m}$ at $\up^2=q$) equals to the number
$F_{M_\field(B_1),\ldots, M_\field(B_m)}^{M_\field(A)}$ of the
filtrations
$$0=M_m\han M_{m-1}\han\cdots \han M_1\han M_0=M_\field(A)$$
such that $M_{t-1}/M_t\cong M_\field(B_t)$ for all $1\leq t\leq m$.

Moreover, for each $A=(a_{i,j})\in \afThnp$, there is a polynomial
$\fka_A=\fka_A(\up^2)\in\sZ$ in $\up^2$ such that, for each finite
field $\field$ with $q$ elements,
$\fka_A|_{\up^2=q}=|\Aut(M_{\field}(A))|$;\index{$\fka_A$, order of
$\text{Aut}(M(A))$} see, for example, \cite[Cor.~2.1.1]{Peng}. For
later use, we give an explicit formula for $\fka_A$. Let $m_A$
denote the dimension of $\rad\End(M_\field(A))$, which is known to
be independent of the field $\field$. We also have
$$\End(M_\field(A))/\rad\End(M_\field(A))\cong \prod_{1\leq i\leq
n,\,a_{i,j}>0}M_{a_{i,j}}(\field),$$
 where $M_{a_{i,j}}(\field)$ denotes the full matrix algebra of
$a_{i,j}\times a_{i,j}$ matrices over $\field$. Hence,
$$|\Aut(M_\field(A))|=|\field|^{m_A}\prod_{1\leq i\leq
n,\,a_{i,j}>0}|GL_{a_{i,j}}(\field)|.$$
 Consequently, the polynomial
\begin{equation}\label{poly-auto-group}
\fka_A=v^{2m_A}\prod_{1\leq i\leq
n,\,a_{i,j}>0}(v^{2a_{i,j}}-1)(v^{2a_{i,j}}-\up^2)\cdots(v^{2a_{i,j}}-v^{2a_{i,j}-2}
).
\end{equation}
In particular, if $z\in\mbc$ is not a root of unity, then
${\fka_A}|_{\up^2=z}\neq0$.

Let $\Hall$ be the {\it (generic) Ringel--Hall algebra}
\index{$\Hall$, Ringel--Hall algebra of $\tri$} of the cyclic quiver
$\tri$, which is by definition the free $\sZ$-module over with basis
$\{u_A=u_{[M(A)]}\mid A\in\afThnp\}.$\index{$u_A=u_{[M(A)]}$, basis
elements of $\Hall$} The multiplication is given by
$$u_{A}u_{B}=\up^{\lan \bfd(A),\bfd(B)\ran}\sum_{C\in\afThnp}\vi^{C}_{A,B}u_{C}$$
for $A,B\in \afThnp$, where
\begin{equation}\label{Euler form}
\lan
\bfd(A),\bfd(B)\ran=\dim\Hom(M(A),M(B))-\dim\Ext^1(M(A),M(B))
\end{equation} is
the {\it Euler form} \index{Euler form}\index{$\lan\;,\;\ran$, Euler
form}  associated with the cyclic quiver $\tri$. If we write
$\bfd(A)=(a_i)$ and $\bfd(B)=(b_i)$, then
\begin{equation}\label{Euler form 2}
\lan \bfd(A),\bfd(B)\ran=\sum_{i\in I}a_ib_i-\sum_{i\in
I}a_ib_{i+1}.\end{equation}

Since both $\dim_\field\End(M_\field(A))$
and $\dim_\field M_\field(A)=\fkd(A)$ are independent of the
ground field, we put for each $A\in\afThnp$,
\begin{equation}\label{tilde u_A}
d_A'=\dim_\field\End(M_\field(A))-\dim_\field M_\field(A)\;\;\text{and}\;\;\ti u_A=v^{d_A'}u_A;
\end{equation}
cf. \cite[(8.1)]{DDX}.\index{$d_A'$, as in $\tilde
u_A:=\up^{-d_A'}u_A$}\footnote{There is another notation, called
$d_A$, which will be defined in \eqref{sqAdA} for all $A\in\afThn$.}
As seen in \cite{DDX}, it is convenient sometimes to work with the
PBW type basis $\{\ti u_A\mid A\in\afThnp\}$ of $\Hall$. \index{PBW
type basis}

The degeneration ordering gives rise to the following ``triangular
relation'' in $\Hall$: for $A_1,\ldots, A_t\in \afThnp$,
\begin{equation} \label{mult-order-formula}
u_{A_1}\cdots u_{A_t}=v^{\sum_{1\leq r<s\leq
t}\lr{\bfd(A_r),\bfd(A_s)}}\sum_{B\leq_\deg A_1*\cdots
*A_t}\vi^B_{A_1,\ldots,A_t}u_B.
\end{equation}

There is a natural
$\mbn I$-grading on $\Hall$:
\begin{equation}\label{grading}
\Hall=\bigoplus_{\bfd\in\mbn I}\Hall_\bfd,
\end{equation}
where $\Hall_\bfd$ is spanned by all $u_A$ with $\bfd(A)=\bfd$. Moreover,
we will frequently consider in the sequel the algebra $\bfHall=\Hall\otimes_\sZ\mbq(\up)$ obtained
by base change to the fraction field $\mbq(\up)$.

In order to relate Ringel--Hall algebras $\Hall$ of cyclic quivers
with affine quantum Schur algebras later on, we recall Lusztig's
geometric construction of $\Hall$ (specializing $v$ to a square root
of a prime power) developed in \cite{Lu91, Lu93} (cf. also
\cite{Lu98}).

Let $\field=\field_q$ be the finite field of $q$ elements and
$\bfd\in\mbn I$. Fix an $I$-graded $\field$-vector space
$V=\oplus_{i\in I}V_i$ of dimension vector $\bfd$, i.e.,
$\dim_\field V_i=d_i$ for all $i\in I$. Then each element in
$$E_V=\bigl\{(f_i)\in \bigoplus_{i\in I}\Hom_\field(V_i,V_{i+1})\mid f_n\cdots f_1\;\text{is nilpotent}\bigr\}$$
 can be viewed as a nilpotent representation of $\tri$ over $\field$ of dimension vector $\bfd$.
The group $G_V=\prod_{i\in I}GL(V_i)$ acts on $E_V$ by
conjugation. Then there is a bijection between the $G_V$-orbits in
$E_V$ and the isoclasses of nilpotent representations of $\tri$ of
dimension vector $\bfd$. For each $A\in\afThnp$ with
$\bfd(A)=\bfd$, we will denote by $\ttO_A$ the orbit in $E_V$
corresponding to the isoclass of $M(A)$.

Define $\HallL_{\bfd}=\mbc_{G_V}(E_V)$ to be the vector space of
$G_V$-invariant functions from $E_V$ to $\mbc$. Now let
$\bfa,\bfb\in\afmbnn$ with $\bfd=\bfa+\bfb$ and fix $I$-graded
$\field$-vector spaces $U=\oplus_{i\in I}U_i$ and $W=\oplus_{i\in
I}W_i$ with dimension vectors $\bfa$ and $\bfb$, respectively. Let
$E$ be the set of triples $(x,\phi,\psi)$ such that $x\in E_V$ and
the sequence
$$\begin{CD}
0 @>>> W @>\phi>> V @>\psi>> U @>>> 0
\end{CD}$$
of $I$-graded spaces is exact and $\phi(W)$ is stable by $x$, and
let $F$ be the set of pairs $(x,W')$, where $x\in E_V$ and $W'\han
V$ is an $x$-stable $I$-graded subspace of dimension vector
$\bfb$. Consider the diagram
$$\begin{CD}
E_U\times E_W @<p_1<< E @>p_2>> F @>p_3>> E_V,
\end{CD}$$
 where $p_1,p_2,p_3$ are projections defined in an obvious way.
Given $f\in\mbc_{G_U}(E_U)=\HallL_\bfa$ and
$g\in\mbc_{G_W}(E_W)=\HallL_\bfb$, define the convolution product of
$f$ and $g$ by
$$fg=q^{-\frac{1}{2}m(\bfa,\bfb)}(p_3)_{!}h\in\mbc_{G_V}(E_V)=\HallL_\bfd,$$
 where
$h\in\mbc(F)$ is the function such that $p_2^*h=p_1^*(fg)$ and
$$m(\bfa,\bfb)=\sum_{1\leq i\leq n}a_ib_i+\sum_{1\leq i\leq
n}a_ib_{i+1}.$$
 Consequently, $\HallL =\bop_{\bfd\in\afmbnn}\HallL_{\bfd}$ becomes
an associative algebra over $\mbc$.\footnote{The algebra $\HallL$
defined also geometrically in \cite[3.2]{VV99} has a multiplication
opposite to the one for $\HallL$ here.} For each $A\in\afThnp$, let
$\chi_{_{\ttO_{A}}}$ be the characteristic function of $\ttO_{A}$
and put
$$\lan\ttO_{A}\ran=q^{-\frac{1}{2}\dim\ttO_{A}}\chi_{_{\ttO_{A}}}.$$
By \cite[8.1]{DDX}), there is an algebra isomorphism
\begin{equation}\label{Hall algebra defined using function}
\Hall\ot_{\sZ}\mbc\lra\HallL, \;\; u_A\longmapsto
v^{\fkd(A)-\bfd(A)\centerdot\bfd(A)}\chi_{_{\ttO_{A}}},\;\;\text{for
$A\in\afThnp$,}
\end{equation}
where $\mbc$ is viewed as a $\sZ$-module by specializing  $\up$ to
$q^{\frac{1}{2}}$. In particular, this isomorphism takes $\ti u_A$
to $\lan\ttO_{A}\ran$.

\section{The quantum loop algebra $\bfU(\afsl)$}
As mentioned in the introduction, an important breakthrough for
the structure of quantum groups associated with semisimple complex
Lie algebras is Ringel's Hall algebra realization of  the
$\pm$-part of the quantum enveloping algebra associated with the
same quiver; see \cite{R90, R932}. For the Ringel--Hall algebra
$\Hall$ associated with a cyclic quiver, it is known from
\cite{Ri93} that a subalgebra, called the {\it composition algebra},\index{composition algebra} is
isomorphic to the $\pm$-part of a quantum affine $\mathfrak
{sl}_n$. We now describe this algebra and use it to display a
certain monomial basis.

The $\sZ$-subalgebra $\comp$ of $\Hall$ generated by $u_{[mS_i]}$
($i\in I$ and $m\geq 1$) is called the {\it composition
algebra}\index{$\comp$, composition algebra of $\tri$} of $\tri$. By
\eqref{grading}, $\comp$ inherits an $\mbn I$-grading by dimension
vectors:
$$\comp=\bigoplus_{\bfd\in\mbn I}\comp_\bfd,$$
where $\comp_\bfd=\comp\cap\Hall_\bfd$.
Define the
divided power
$$u_i^{(m)}=\frac{1}{[m]^!}u_i^m\in\comp$$
for $i\in I$ and $m\geq 1$. In fact,
$$u_i^{(m)}=v^{m(m-1)}u_{[mS_i]}\in\comp\subset\Hall.$$

We now use the strong monomial basis property developed in \cite{DD05,DDX} to
 construct an explicit monomial basis
of $\comp$. For each $A=(a_{i,j})\in \afThnp$, define
$$\ell=\ell_A={\rm max}\{j-i\mid a_{i,j}\not=0\}.$$
 In other words, $\ell$ is the Loewy length of the representation $M(A)$.

Suppose now $A$ is aperiodic. Then there is $i_1\in[1,n]$ such that
 $a_{i_1,i_1+\ell}\not=0$, but $a_{i_1+1,i_1+1+\ell}=0$.
 If there are some $a_{i_1+1,j}\not=0$, we let $p\geq 1$ satisfy
that $a_{i_1+1,i_1+1+p}\not=0$ and $a_{i_1+1,j}=0$ for all
$j>i_1+1+p$; if $a_{i_1+1,j}=0$ for all $j>i_1+1$, let $p=0$. Thus, $\ell>p$. Now set
$$t_1=a_{i_1,i_1+1+p}+\cdots +a_{i_1,i_1+\ell}$$
and define $A_1=(b_{i,j})\in \afThnp$ by letting
$$b_{i,j}=\begin{cases}  0,\;\;& \text{if $i=i_1, j\geq i_1+1+p$},\\
                         a_{i_1+1,j}+a_{i_1,j},& \text{if $i=i_1+1<j, j\geq i_1+1+p$},\\
                         a_{i,j},  \;\; & \text{otherwise}.\end{cases}$$
Then, $A_1$ is again aperiodic. Applying the above
process to $A_1$, we get $i_2$ and $t_2$. Repeating the above
process (ending with the zero matrix), we finally get two
sequences $i_1,\ldots,i_m$ and $t_1,\ldots,t_m$. This gives a word
$$w_A=i_1^{t_1}i_2^{t_2}\cdots i_m^{t_m},$$
 where $i_1,\ldots,i_m$ are viewed as elements in $I=\mbz/n\mbz$, and
define the monomial
$$u^{(A)}=u_{i_1}^{(t_1)}u_{i_2}^{(t_2)}\cdots u_{i_m}^{(t_m)}\in\comp.$$

The algorithm above can be easily modified to get a similar algorithm
for quantum $\mathfrak{gl}_n$. We illustrate the algorithm with an example in this case.
\begin{Example} If
$A=\smat{1&2&3&4\\ &5&0&0\\
&&6&0\\
&&&7\\}$, then
$\ell=4,i_1=1, p=1$ and $t_1=2+3+4=9$. Here we ignore all zero entries on and below the diagonal for simplicity. Then
$$\aligned
A_1=\smat{1&0&0&0\\ &7&3&4\\
&&6&0\\
&&&7\\}, &\quad\ell=3,i_2=2, p=1\text{ and }t_2=3+4=7\\
A_2=\smat{1&0&0&0\\ &7&0&0\\
&&9&4\\
&&&7\\}, &\quad\ell=2,i_3=3, p=1\text{ and }t_3=4\\
A_3=\smat{1&0&0&0\\ &7&0&0\\
&&9&0\\
&&&11\\}, &\quad\ell=1,i_4=4, p=0\text{ and }t_4=11\\
\endaligned
$$
and $i_5=3, p=0, t_5=9$, $i_6=2,p=0,t_6=7$ and $i_7=1,p=0,t_7=1$. Hence,
$$u^{(A)}=u_{1}^{(9)}u_{2}^{(7)}u_{3}^{(4)}u_{4}^{(11)}u_{3}^{(9)}u_{2}^{(7)}u_{1}^{(1)}.$$
\end{Example}

\begin{Prop} \label{integral-monomial-basis-comp} The set
$$\{u^{(A)}\mid A\in \afThnp \;\text{aperiodic}\}$$
is a $\sZ$-basis of $\comp$.
\end{Prop}
\begin{proof} Let $A\in \afThnp$ be aperiodic and let
$w_A=i_1^{t_1}i_2^{t_2}\cdots i_m^{t_m}$ be the corresponding word
constructed as above. By \cite[Th.~5.5]{DD05}, $w_A$ is
distinguished, that is, $\varphi^A_{A_1,\ldots,A_m}=1$, where
$A_s=t_sE_{i_s,i_s+1}^\vtg$ for $1\leq s\leq m$. By
\cite[Th.~7.5(i)]{DDX}, the $u^{(A)}$ with $A\in \afThnp$ aperiodic
form a $\sZ$-basis of $\comp$.
\end{proof}

We now define the quantum enveloping algebra of the loop algebra (the {\it quantum loop algebra}
for short)
of $\frak{sl}_n$.
Let $C=C_{\tri}=(c_{i,j})_{i,j\in I}$\index{$C=C_{\tri}$, Cartan matrix of a cyclic quiver}
\index{Cartan matrix}\index{Cartan matrix!$\sim$ of a cyclic quiver} be the generalized Cartan matrix of
type $\ti A_{n-1}$, where $I=\mbz/n\mbz$. We always assume that if
$n\geq 3$, then $c_{i,i}=2$, $c_{i,i+1}=c_{i+1,i}=-1$ and
$c_{i,j}=0$ otherwise. If $n=2$, then $c_{1,1}=c_{2,2}=2$ and
$c_{1,2}=c_{2,1}=-2$. In other words,
\begin{equation}\label{CMforAn-1}
C=\begin{pmatrix}2&-2\\-2&2\\
\end{pmatrix}\text{ or }C=\begin{pmatrix}2&-1&0&\cdots&0&-1\\
                 -1&2&-1&\cdots&0&0\\
                 0&-1&2&\cdots&0&0\\
                 \vdots&\vdots&\vdots&\vdots&\vdots&\\
                 0&0&0&\cdots&2&-1\\
                 -1&0&0&\cdots&-1&2\\
\end{pmatrix}(n\geq3).
\end{equation}
The quantum group associated to $C$
is denoted by $\bfU(\afsl)$.

\begin{Def}\label{affine quantum group}
Let $n\geq 2$ and $I=\mbz/n\mbz$. The affine quantum loop algebra  $\bfU(\afsl)$  is the algebra
 over $\mathbb Q(\upsilon)$ presented by
generators
$$E_i,\ F_i,\ \ti K_i,\ \ti K_i^{-1},\,\,i\in I,$$
and relations: for $i,j\in I$
\begin{itemize}
\item[(QSL0)] $\ti K_1\ti K_2\cdots\ti K_n=1$
 \item[(QSL1)] $\ti K_{i}\ti K_{j}=\ti K_{j}\ti K_{i},\ \ti K_{i}\ti K_{i}^{-1}=1;$
  \item[(QSL2)] $\ti K_{i}E_j=\upsilon^{c_{i,j}}E_j\ti K_{i};$
 \item[(QSL3)] $\ti K_{i}F_j=\upsilon^{-c_{i, j}} F_j\ti K_{i};$
 \item[(QSL4)] $E_iE_j=E_jE_i,\ F_iF_j=F_jF_i\text{ if } i\not=j\pm 1;$
  \item[(QSL5)] $E_iF_j-F_jE_i=\delta_{i,j}\frac
       {\ti K_{i}-\ti K_{i}^{-1}}{\upsilon-\upsilon^{-1}};$
  \item[(QSL6)] $E_i^2E_j-(\upsilon+\upsilon^{-1})E_iE_jE_i+E_jE_i^2=0\
 \text{ if } i=j\pm 1$ and $n\geq 3;$
  \item[(QSL7)] $F_i^2F_j-(\upsilon+\upsilon^{-1})F_iF_jF_i+F_jF_i^2=0\
 \text{ if } i=j\pm 1$ and $n\geq 3;$
 \item[(QSL6$'$)] $E_i^3E_j-(\up^2+1+\up^{-2})E_i^2E_jE_i+(\up^2+1+\up^{-2})E_iE_jE_i^2-E_jE_i^3=0\
 \text{ if } i\not=j$ and $n=2;$
  \item[(QSL7$'$)] $F_i^3F_j-(\up^2+1+\up^{-2})F_i^2F_jF_i+(\up^2+1+\up^{-2})F_iF_jF_i^2-F_jF_i^3=0\
 \text{ if } i\not=j$ and $n=2.$\index{defining relations!({\rm QSL0})--({\rm QSL7}), ({\rm QSL6$'$}), ({\rm QSL7$'$})}
 \end{itemize}

 For later use in representation theory, let $\afUslC$ be the quantum loop algebra defined by the same generators and relations
(QSL0)--(QSL7) with $\up$ replaced by a non-root of unity $z\in
\mbc$ and $\mbq(\up)$ by $\mbc$.
 \end{Def}

A new presentation for $\bfU(\afsl)$ and $\afUslC$, known as Drinfeld's new presentation,
will be discussed in \S2.3.

In this paper, {\it quantum affine $\mathfrak{sl}_n$}\index{quantum affine
$\frak{sl}_n$} always refers to the {\it quantum loop (Hopf) algebra}
$\bfU(\afsl)$.\footnote{If (QSL0) is dropped, it also defines a
quantum affine $\mathfrak{sl}_n$ with the {\it central extension}; see,
e.g., \cite{CP}.}\index{quantum loop algebra!$\sim$ of $\mathfrak
{sl}_n$, $\bfU(\afsl)$} We will mainly work with $\bfU(\afsl)$ or quantum groups defined over $\mbq(\up)$
and mention from time to time a parallel theory over $\mbc$.

Let $\bfU(\afsl)^+$ (resp., $\bfU(\afsl)^-$, $\bfU(\afsl)^0$) be the
positive (resp., negative, zero) part of the quantum enveloping
algebra $\bfU(\afsl)$. In other words, $\bfU(\afsl)^+$ (resp.,
$\bfU(\afsl)^-$, $\bfU(\afsl)^0$) is a $\mbq(\up)$-subalgebra
generated by $E_i$ (resp., $F_i$, $\ti K^{\pm1}$), $i\in I$.

Let
$$\bfcomp=\comp\otimes_\sZ\mbq(\up).$$
 Thus, $\bfcomp$ identifies with the $\mbq(\up)$-subalgebra $\bfHall$
generated by $u_i:=u_{[S_i]}$ for $i\in I$.

\begin{Thm} {\rm (\cite{Ri93})} \label{CompAlg-affine sl_n}
There are $\mbq(\up)$-algebra isomorphisms
$$\bfcomp\lra \bfU(\afsl)^+,\,u_i\longmapsto E_i\,\text{ and }\,\bfcomp\lra
\bfU(\afsl)^-,\,u_i\longmapsto F_i.$$
\end{Thm}

By this theorem and the triangular decomposition
$$\bfU(\afsl)=\bfU(\afsl)^+\otimes
 \bfU(\afsl)^0\otimes\bfU(\afsl)^-,$$
the basis displayed in Proposition
\ref{integral-monomial-basis-comp} gives rise to a monomial basis
for the  $\bfU(\afsl)$.

\section{Three types of generators and associated monomial bases}

In this section, we display three distinct minimal sets of generators for
$\bfHall$, each of which contains the generators $\{u_i\}_{i\in I}$ for  $\bfcomp$.
 We also describe their associated monomial bases for $\bfHall$ in
the respective generators.

The first minimal set of generators contains simple modules and
certain central elements. These generators are convenient for a
presentation for the double Ringel--Hall algebra over
 $\mbq(\up)$ (or a specialization at a non-root of unity)  associated to cyclic quivers (see Chapter 2). \index{generators!Schiffmann--Hubery $\sim$}

 In \cite{Sch} Schiffmann first described the structure of $\bfHall$ as a tensor product of $\bfcomp$
and a polynomial ring in infinitely many indeterminates.  Later
Hubery explicitly constructed these central elements in \cite{Hub1}.
More precisely, for each $m\geq 1$, let\footnote{The element $c_m$
is the same element as $c_{n,m}$ defined in \cite[\S2.3]{Hub2}.}
\begin{equation}\label{central elements}
 c_m=(-1)^mv^{-2nm}\sum_{A}(-1)^{{\rm dim}\,{\rm End}(M(A))}\fka_A u_A \in\bfHall,
\end{equation}
 where the sum is taken over all $A\in \afThnp$ such that $\bfd(A)=\bfdim M(A)=m\dt$
 with $\dt=(1,\ldots,1)\in\mbn^n$, and
${\rm soc}\,M(A)$ is square-free, i.e., $\bfdim\, {\rm
soc}\,M(A)\leq \dt$. Note that in this case, ${\rm soc}\,M(A)$ is
square-free if and only if ${\rm top}\, M(A):=M(A)/\rad M(A)$ is
square-free. The following result is proved in \cite{Sch,Hub1}.

\begin{Thm} \label{Schiffmann-Hubery}  The elements $c_m$ are central in $\bfHall$.
Moreover, there is a decomposition
$$\bfHall= \bfcomp \otimes_{\mbq(\up)}{\mbq}(v)[c_1,c_2,\ldots],$$
 where ${\mbq}(v)[c_1,c_2,\ldots]$ is the polynomial algebra in
$c_m$ for $m\geq 1$. In particular, $\bfHall$ is generated by
$u_i$ and $c_m$ for $i\in I$ and $m\geq 1$.
\end{Thm}

We will call the central elements $c_i$ the {\it Schiffmann--Hubery
generators}.

Let $A=(a_{i,j})\in \afThnp$. For each $s\geq 1$, define
\begin{equation}\label{definition-A'}
m_s=m_s(A)={\rm min}\{a_{i,j}\mid j-i=s\}\;\;\text{and}\;\;
A'=A-\sum_{1\leq i\leq n,\,i<j}m_{j-i}E_{i,j}^\vtg.\end{equation}
 Then $A'$ is aperiodic. Moreover, for $A, B\in \afThnp$,
 $$A=B\Longleftrightarrow
 A'=B'\;\;\text{and}\;\;m_s(A)=m_s(B),\;\forall\;s\geq1.$$
Finally, for $A\in\afThnp$, define
$$c^A=\prod_{s\geq 1}c_s^{m_s}\in\bfHall,\;\text{where $m_s=m_s(A)$}.$$
 The next corollary is a direct consequence of Theorem \ref{Schiffmann-Hubery}
and Proposition \ref{integral-monomial-basis-comp}.

\begin{Coro} \label{monomial basis with central elements}
The set
$$\{u^{(A')}c^A\mid A\in \afThnp\}$$
is a $\mbq(\up)$-basis of $\bfHall$.
\end{Coro}

Next, we look at the minimal set of generators consisting of
simple modules and {\it homogeneous} semisimple
modules.\index{generators!homogeneous semisimple $\sim$} It is
known from \cite[Prop.~3.5]{VV99} (or \cite[Th.~5.2(i)]{DDX}) that
$\bfHall$ is also generated by $u_\bfa=u_{[S_\bfa]}$ for $\bfa\in
\mbn I$; see also \eqref{monomial sum} below.  If $\bfa$ is not
sincere, say $a_i=0$, then
\begin{equation}\label{non-sincere elements}
u_\bfa=\prod_{j\in
I,\,j\not=i}\frac{\up^{a_j(1-a_j)}}{[a_j]^!}\,u_{i-1}^{a_{i-1}}\cdots
u_1^{a_1}u_n^{a_n}\cdots u_{i+1}^{a_{i+1}}\in \bfcomp.
\end{equation}
Thus, $\bfHall$ is generated by $u_i$ and $u_\bfa$ for $i\in I$
and $\bfa\in\mbn I$ sincere. Indeed, this result can be
strengthened as follows; see also \cite[p.~421]{Sch}.

\begin{Prop} \label{generators-RH-alg}
The Ringel--Hall algebra $\bfHall$ is generated by $u_i$ and
$u_{m\dt}$ for $i\in I$ and $m\geq 1$. \index{semisimple generators}
\end{Prop}

\begin{proof} Let $\bds{\frak H'}$ be the $\mbq(\up)$-subalgebra generated
by $u_i$ and $u_{m\dt}$ for $i\in I$ and $m\geq 1$. To show
$\bds{\frak H'}=\bfHall$, it suffices to prove
$u_\bfa=u_{[S_\bfa]}\in\bds{\frak H}'$ for all $\bfa\in \mbn I$.

Take an arbitrary $\bfa\in \mbn I$. We proceed induction on
$\sg(\bfa)=\sum_{i\in I}a_i$ to show $u_\bfa\in\bds{\frak H'}$. If
$\sg(\bfa)=0$ or $1$, then clearly $u_\bfa\in \bds{\frak H}'$. Now
let $\sg(\bfa)>1$. If $\bfa$ is not sincere, then by
\eqref{non-sincere elements}, $u_\bfa\in \bds{\frak H}'$. So we may
assume $\bfa$ is sincere. The case where $a_1=\cdots=a_n$ is
trivial. Suppose now there exists $i\in I$ such that $a_i\not=
a_{i+1}$. Define $\bfa'=(a_j'),\bfa''=(a_j'')\in\mbn I$ by
$$a_j'=\begin{cases} a_i-1, & \text{if $j=i$};\\
                         a_i, & \text{otherwise},\end{cases}
\;\;\;\text{and}\;\;\;
a_j''=\begin{cases} a_{i+1}-1, & \text{if $j=i+1$};\\
                         a_i, & \text{otherwise}.\end{cases}$$
Then, we have in $\bfHall$
$$u_iu_{\bfa'}=v^{a_i-a_{i+1}-1}(u_X+v^{a_i-1}[a_i]u_\bfa)\;\;\;\text{and}$$
$$u_{\bfa''}u_{i+1}=v^{a_{i+1}-a_i-1}(u_X+v^{a_{i+1}-1}[a_{i+1}]u_\bfa),$$
where $X\in \afThnp$ is given by
$$M(X)\cong \bigoplus_{j\not=i,i+1}a_jS_j\oplus (a_i-1)S_i\oplus
(a_{i+1}-1)S_{i+1}\oplus S_i[2].$$
 Therefore,
$$u_iu_{\bfa'}-v^{2a_i-2a_{i+1}}u_{\bfa''}u_{i+1}=
v^{a_i-a_{i+1}-1}(v^{a_i-1}[a_i]-v^{a_{i+1}-1}[a_{i+1}])u_\bfa.$$
 The inequality $a_i\not=a_{i+1}$ implies
 $v^{a_i-1}[a_i]-v^{a_{i+1}-1}[a_{i+1}]\not=0$. Thus, we obtain
$$u_\bfa=\frac{v^{a_{i+1}-a_i+1}}{v^{a_i-1}[a_i]-v^{a_{i+1}-1}[a_{i+1}]}u_iu_{\bfa'}
-\frac{v^{a_i-a_{i+1}+1}}{v^{a_i-1}[a_i]-v^{a_{i+1}-1}[a_{i+1}]}u_{\bfa''}u_{i+1}.$$
Since $\sg(\bfa')=\sg(\bfa'')=\sg(\bfa)-1$, we have by the
induction hypothesis that both $u_{\bfa'}$ and $u_{\bfa''}$ belong
to $\boldsymbol{\frak H'}$. Hence, $u_\bfa\in\boldsymbol{\frak
H'}$. This finishes the proof.
\end{proof}

\begin{Rem}We will see that semisimple modules as generators are convenient
 for the description of Lusztig type integral form. First, by \cite[Th.~5.2(ii)]{DDX}, they
  generate the {\it integral} Ringel--Hall algebra $\Hall$ over $\sZ$.
  Second, there are in \S2.4 explicit commutator formulas between semisimple generators
  in the double Ringel--Hall algebra. Thus, a natural candidate for the Lusztig type form of quantum affine
   $\frak{gl}_n$ is proposed in \S3.7.
\end{Rem}

Finally, we introduce a set of generators for $\bfHall$ consisting
of simple and {\it homogeneous} indecomposable modules in
$\rep^0\!\!\tri$.\index{generators!homogeneous indecomposable
$\sim$} Since indecomposable modules correspond to simplest
non-diagonal matrices, these generators are convenient for deriving
explicit multiplication formulas; see \S3.4, \S5.4 and \S6.2.

For each $A\in\afThnp$, consider the radical filtration of $M(A)$
$$M(A)\supseteq \rad M(A) \supseteq\cdots\supseteq \rad^{t-1}M(A)\supseteq
\rad^tM(A)=0$$
 where $t$ is the Loewy length of $M(A)$. For $1\leq s\leq t$, we
write $$\rad^{s-1}M(A)/\rad^s M(A)=S_{\bfa_s}\qquad \text{for some
 $\bfa_s\in\mbn I$.}$$ Write ${\frak m}_A=u_{\bfa_1}\cdots u_{\bfa_t}$.
Applying \eqref{mult-order-formula} gives that
$${\frak m}_A=\sum_{B\leq_\deg A}f(B)u_B,$$
 where $f(B)\in\mbq(\up)$ with $f(A)=v^{\sum_{l<s}\lr{\bfa_l,\bfa_s}}$. In other words,
$$u_A=f(A)^{-1} {\frak m}_A-\sum_{B<_\deg A}f(A)^{-1}f(B)u_B.$$
 Repeating the above construction for maximal $B$ with $B<_\deg A$ and
continuing this process, we finally get that
\begin{equation}\label{monomial sum}
u_A=\sum_{B\leq_\deg A} \phi^B_A{\frak m}_B,
\end{equation}
 where $\phi_A^B\in\mbq(\up)$ with $\phi_A^A=f(A)^{-1}\not=0$.


\begin{Prop} \label{indecomposable basis} For each $s\geq 1$,
choose $i_s\in \mbz^+$. Then $\bfHall$ is generated by $u_i$ and
$u_{E_{i_s,i_s+sn}^\vtg}$ for $i\in I$ and $s\geq 1$.
\end{Prop}

\begin{proof} For each $m\geq 1$, let $\bfHall^{(m)}$ be the
$\mbq(\up)$-subalgebra of $\bfHall$ generated by $u_i$ and
$u_{s\dt}$ for $i\in I$ and $1\leq s\leq m$. By convention, we set
$\bfHall^{(0)}=\bfcomp$. Clearly, each $\bfHall^{(m)}$ is also
$\mbn I$-graded. By Theorem \ref{Schiffmann-Hubery},
$\bfHall^{(m)}$ is also generated by $u_i$ and $c_s$ for $i\in I$
and $1\leq s\leq m$. Moreover,
$\bfHall^{(m-1)}\subsetneqq\bfHall^{(m)}$ for all $m\geq 1$.

We use induction on $m$ to show the following

\medskip

\noindent{\bf Claim}: $\bfHall^{(m)}$ is generated by $u_i$ and
$u_{E_{i_s,i_s+sn}^\vtg}$ for $i\in I$ and $1\leq s\leq m$.

\medskip

Let $m\geq 1$ and suppose the claim is true for $\bfHall^{(m-1)}$.
Applying \eqref{monomial sum} to $E=E_{i_m,i_m+mn}^\vtg$
gives
\begin{equation}\label{monomial sum for u_E}
u_E=\sum_{B\leq_\deg E} \phi^B_E{\frak m}_B=\phi_E^E{\frak m}_E
+\sum_{B<_\deg E\atop B\not=C}\phi^B_E{\frak m}_B
+\phi^{C}_E{\frak m}_C,\end{equation}
 where $C=A_{m\dt}=\sum_{i\in I}mE_{i,i+1}^\vtg$, i.e., $M(C)\cong
 S_{m\dt}$ and ${\frak m}_C=u_{m\dt}$. For each $B$ with
 $B\leq_\deg E$ and $B\not=C$, the Loewy length of $M(B)$ is
 strictly greater than $1$. Using an argument similar to the
proof of Proposition \ref{generators-RH-alg}, we have ${\frak
m}_B\in \bfHall^{(m-1)}$. We now prove that $\phi_E^C\not=0$.
Suppose $\phi_E^C=0$. Then $u_E\in \bfHall^{(m-1)}$. Furthermore,
for each $j\in I$, there is $1\leq p\leq n$ such that
$E_{j,j+mn}^\vtg=\tau^p(E)$, where $\tau$ is the Auslander--Reiten
translation defined in \S1.2. Thus, applying $\tau^p$ to
\eqref{monomial sum for u_E} gives that
$$u_{E_{j,j+mn}^\vtg}=u_{\tau^p(E)}=\phi_E^E{\frak m}_{\tau^pE}
+\sum_{B<_\deg E\atop B\not=C}\phi^B_E{\frak m}_{\tau^pB}.$$
 Using a similar argument as above, we get $u_{E_{j,j+mn}^\vtg}\in
 \bfHall^{(m-1)}$. By definition, we have
$$\bfHall^{(m-1)}_\bfd=\bfHall_\bfd\;\;\text{for all $\bfd\in \mbn I$ with $\sg(\bfd)<mn$}.$$
In particular, $u_{E_{i,j}^\vtg}\in \bfHall^{(m-1)}$ for all $i<j$
with $j-i<mn$. Consequently, we get
$$u_{E_{i,j}^\vtg}\in \bfHall^{(m-1)}\;\;\text{for all $i<j$ with $j-i\leq mn$.}.$$
 By \cite[Th.~3.1]{GP}, each element in $\bfHall$ can be written as
a linear combination of products of $u_A$'s with $M(A)$
indecomposable. This together with the above discussion implies
that $\bfHall^{(m-1)}_{m\dt}=\bfHall_{m\dt}$. Thus, $u_{m\dt}\in
\bfHall^{(m-1)}$. This contradicts the fact that
$\bfHall^{(m-1)}\subsetneqq\bfHall^{(m)}$. Therefore,
$\phi_A^C\not=0$. We conclude from \eqref{monomial sum for u_E}
that
$$u_E-\phi^{C}_E u_{m\dt}=\phi_E^E{\frak m}_E
+\sum_{B<_\deg E\atop B\not=C}\phi^B_E{\frak m}_B\in
\bfHall^{(m-1)},$$
 which shows the claim for $\bfHall^{(m)}$.
This finishes the proof.
\end{proof}

By the proof of the above proposition, we see that for $m\geq 1$,
each of the following three sets
$$\aligned
&\{u_i,c_s\mid i\in I,1\leq s\leq m\},\;\;
\{u_i,u_{s\dt}\mid i\in I,1\leq s\leq m\},\\
&\qquad\{u_i,u_{E_{i_s,i_s+sn}^\vtg}\mid i\in I,1\leq s\leq m\}
\endaligned$$
 generates $\bfHall^{(m)}$. Hence, for each $m\geq 1$, there are
 nonzero elements $x_m,y_m\in\mbq(\up)$ such that
$$c_m\equiv x_m u_{m\dt}\;{\rm mod}\;\bfHall^{(m-1)},\;\;
c_m\equiv y_m u_{E_{i_s,i_s+sn}^\vtg}\;{\rm
mod}\;\bfHall^{(m-1)}.$$
 This together with Corollary \ref{monomial basis with central elements}
 gives the following result.

\begin{Coro}\label{monomial bases with semisimple/indecom} The set
$$\bigl\{u^{(A')}\prod_{s\geq 1}(u_{s\dt})^{m_s(A)}\mid A\in \afThnp\bigr\}$$
 is a $\mbq(\up)$-basis of $\bfHall$, where $A'$ is defined in \eqref{definition-A'}. For each $s\geq 1$,
choose $i_s\in \mbz^+$. Then the set
$$\bigl\{u^{(A')}\prod_{s\geq 1}(u_{E_{i_s,i_s+sn}^\vtg})^{m_s(A)}\mid A\in \afThnp\bigr\}$$
 is also a $\mbq(\up)$-basis of $\bfHall$.
\end{Coro}

\section{Hopf structure on extended Ringel--Hall algebras}

It is known that {\it generic} Ringel--Hall algebras exist only for
Dynkin or cyclic quivers. In the finite type case, these algebras
give a realization for the $\pm$-part of quantum groups. It is
natural to expect that this is also true for cyclic quivers. In
other words, we look for a quantum group such that $\bfHall$ is
isomorphic to its $\pm$-part. We will see in Chapter 2 that this
quantum group is the quantum loop algebra of $\mathfrak {gl}_n$ in
the sense of Drinfeld \cite{Dr88}, which in fact is isomorphic to
the so-called double Ringel--Hall algebra defined as the Drinfeld
double of two Hopf algebras $\Hallpi$ and $\Hallmi$ together with a
skew-Hopf paring. In this section, we first introduce the pair
$\Hallpi$ and $\Hallmi$.

We need some preparation. If we define the {\it symmetrization} of
the Euler form \eqref{Euler form} \index{Euler form!symmetric
$\sim$, $(\; ,\; )$}\index{symmetrization} by
$$(\al,\bt)=\lan\al,\bt\ran +\lan\bt,\al\ran,$$ then $I$ together
with $(\,\,,\,\,)$ becomes a Cartan datum in the sense of
\cite[1.1.1]{Lu93}.\index{Cartan datum} To a Cartan datum, there are associated {\it root
data} in the sense of \cite[2.2.1]{Lu93} which play an important
role in the theory of quantum groups.\index{root datum} We shall fix the following
root datum throughout the paper.
\begin{Def}\label{root datum}
Let $X=\mbz^n$, $Y=\Hom(X,\mbz)$ and let
$\langle\,\,,\,\,\rangle^\rd:Y\times X\to\mbz$ be the natural
perfect pairing. If we denote the standard basis of $X$ by
$\ep_1,\ldots,\ep_n$ and the dual basis by
$\kappa_1,\ldots,\kappa_n$, then
$\langle\kappa_i,\ep_j\rangle^\rd=\delta_{i,j}$. Thus, the
embeddings
\begin{equation}\label{embeddings}
I\lra Y, i\longmapsto \tilde i:=\kappa_i-\kappa_{i+1}\,\text{ and
}\,I\lra X, i\longmapsto i'=\ep_i-\ep_{i+1}
\end{equation}
 with $\ep_{n+1}=\ep_1$ and $\kappa_{n+1}=\kappa_1$ define a root datum $(Y,X,\langle\,,\,\rangle^\rd,\ldots)$.
\end{Def}
For notational simplicity,  we shall identify both $X$ and $Y$ with $\mbz I$ by setting $\ep_i=i=\kappa_i$ for
 all $i\in I$. Under this identification, the form $\langle\,\,,\,\,\rangle^\rd:\mbz I\times\mbz I\to \mbz$
  becomes a symmetric bilinear form, which is different from
the Euler forms $\langle\,\,,\,\,\rangle$ and its symmetrization $(\,\,,\,\,)$. However, they are related
as follows.

\begin{Lem}\label{rdEf} For $\bfa=\sum a_ii\in\mbz I$,
if we put $\tilde\bfa=\sum a_i\tilde i$ and
$\bfa'=\sum a_ii'$, then
$$(1)\quad(\bfa,\bfb)=\lr{\tilde\bfa,\bfb'}^\rd,\,\text{ and }\quad(2)\quad\,\lr{\bfa,\bfb}=\lr{\bfa',\bfb}^\rd,
\text{ for all }\bfa,\bfb\in\mbz I.$$
\end{Lem}
\begin{proof} Since all forms are bilinear, (1) follows from
$(i,j)=\langle\tilde i,j'\rangle^\rd$ for all $i,j\in I$ and (2)
from $\lr{i,j}=\delta_{i,j}-\delta_{i+1,j}=\langle i',j\rangle^\rd$
\end{proof}
We may use the following commutative diagrams to describe the two
relations:
 \begin{center}
\begin{pspicture}(-1,-0.1)(10,2.2)
\psset{xunit=.8cm,yunit=.7cm} \uput[u](1,2){$\mbz I\times\mbz I$}
\uput[d](1,0.8){$\mbz I\times\mbz I$}
\uput[r](4.5,1.4){$\mbz$,}\uput[l](1.1,1.4){$\tilde{(\;)}\times(\;)'$}
\psline{->}(1.1,2)(1.1,0.75) \psline{->}(1.95,2.45)(4.3,1.5)
\psline{->}(1.95,0.35)(4.3,1.2) \uput[u](3.2,1.9){$(\,,\,)$}
\uput[d](3.2,0.9){$\langle\,,\,\rangle^\rd$}

\uput[u](8.5,2){$\mbz I\times\mbz I$} \uput[d](8.5,0.8){$\mbz
I\times\mbz I$}
\uput[r](12,1.4){$\mbz$}\uput[l](8.6,1.4){$(\;)'\times1$}
\psline{->}(8.6,2)(8.6,0.75) \psline{->}(9.45,2.45)(11.8,1.5)
\psline{->}(9.45,0.35)(11.8,1.2) \uput[u](10.7,1.9){$\lr{\,,\,}$}
\uput[d](10.7,0.9){$\langle\,,\,\rangle^\rd$}
\end{pspicture}
\end{center}

\vspace{.3cm}
We also record the following fact which will be used below and in \S2.1.

Let $\field$ be a field. A Hopf algebra $\scrA$ over $\field$ is an
$\field$-vector space together with multiplication $\mu_\scrA$, unit
$\eta_\scrA$, comultiplication $\Dt_\scrA$, counit $\ep_\scrA$ and
antipode $\sg_\scrA$ which satisfy certain axioms see, e.g.,
\cite[\S5.1]{DDPW}.

\begin{Lem}\label{opposite Hopf} If $\scr A=(\scr A,\mu,\eta,\Dt,\ep,\sg)$ is a
Hopf algebra with multiplication $\mu$, unit $\eta$,
comultiplication $\Dt$, counit $\ep$ and antipode $\sg$, then ${\scr
A}^{\text{\rm op}}=({\scr A},\mu^{\text{\rm op}},\eta,\Dt^{\text{\rm
op}},\ep,\sg)$ is also a Hopf algebra. This is called the {\sf
opposite} Hopf algebra of $\scr A$. Moreover, if $\sg$ is
invertible, then both $({\scr A},\mu^{\text{\rm
op}},\eta,\Dt,\ep,\sg^{-1})$ and $({\scr A},\mu,\eta,\Dt^{\text{\rm
op}},\ep,\sg^{-1})$ are also Hopf algebras, which are called {\sf
semi-opposite} Hopf algebras.
\end{Lem}\index{opposite Hopf algebra}

Let 
\begin{equation}\label{nonnegative-tri-decom}
 \Hallpi=\bfHall\ot_{\mbq(\up)} \mbq(\up)[K_1^{\pm 1},\ldots,K_n^{\pm
 n}].\index{extended Ringel-Hall algebra}
\end{equation}
Putting $x=x\ot1$ and $y=1\ot y$ for $x\in \bfHall$ and
$y\in\mbq(\up)[K_1^{\pm 1},\ldots,K_n^{\pm n}]$, $\Hallpi$ is a
$\mbq(\up)$-space with basis $\{u_A^+K_{\al}\mid \al\in\mbz I,A\in
\afThnp\}$. We are now ready to introduce the Ringel--Green--Xiao Hopf
structure on $\Hallpi$.\index{Ringel--Hall algebra!extended $\sim$, $\Hallpi$}
Let $\afThnp_1:=\afThnp\backslash\{0\}$.

\begin{Prop}\label{Green Xiao}
The $\mbq(\up)$-space $\Hallpi$ with basis $\{u_A^+K_{\al}\mid
\al\in\mbz I,A\in \afThnp\}$ becomes a Hopf algebra with the
following algebra, coalgebra and antipode structures.
\begin{itemize}
\item[(a)] Multiplication and unit:
\begin{equation*}
\begin{split}
u_A^+u_B^+&=\sum_{C\in\afThnp}\nup^{\lan \bfd(A),\bfd(B)\ran}\vi_{A,B}^C u_C^+, \text{\quad for all $A,B\in\afThnp$},\\
K_\al u_A^+&=\nup^{\lr{\bfd(A),\al}}u_A^+K_\al, \text{\quad for all $\al\in\mbz I$, $A\in\afThnp$},\\
K_{\al} K_{\bt}&=K_{\al+\bt}, \text{\quad for all $\al,\bt\in\mbz
I$}.
\end{split}
\end{equation*}
with unit $1=u_0^+=K_0$.

\item[(b)] Comultiplication and counit {\rm(Green \cite{Gr95})}:
\begin{equation*}
\begin{split}
\Dt(u_C^+)&=\sum_{A,B\in\afThnp}\nup^{\lan
\bfd(A),\bfd(B)\ran}\frac{\fka_A\fka_B}{\fka_C}
\vi_{A,B}^C u_B^+\ot u_A^+\ti K_{\bfd(B)},\\
\Dt(K_\al)&=K_\al\ot K_\al,\;\; \text{where $C\in\afThnp$ and
$\al\in\mbz I$},
\end{split}
\end{equation*}
with counit $\ep$ satisfying $\ep(u_C^+)=0$ for all $C\in\afThnp_1$
and $\ep(K_\al)=1$ for all $\al\in\mbz I$. Here, for each
$\al=(a_i)\in\mbz I$, $\ti K_\al$ denotes $(\ti K_1)^{a_1}\cdots
(\ti K_n)^{a_n}$ with $\ti K_i=K_iK_{i+1}^{-1}$.\index{$\ti K_\al$,
$\ti K_i$}

\item[(c)] Antipode {\rm(Xiao \cite{X97})}:
\begin{equation*}
\begin{split}
&\sg(u_C^+)=\dt_{C,0}+\sum_{m\geq 1}(-1)^m\sum_{D\in\afThnp\atop
C_1,\ldots,C_m\in\afThnp_1}\frac{\fka_{C_1}\cdots
\fka_{C_m}}{\fka_C}\vi_{C_1,\ldots,C_m}^C \vi_{C_m,\ldots,C_1}^D
u_D^+\ti K_{-\bfd(C)},
\end{split}
\end{equation*}
for all $C\in\afThnp$, and $\sg(K_{\al})=K_{-\al}$, for all $\al\in\mbz
I$.
\end{itemize}

Moreover, the inverse of $\sg$ is given by
\begin{equation*}
\begin{split}
\sg^{-1}(u_C^+)&=\dt_{C,0}+\sum_{m\geq 1}(-1)^m\sum_{D\in\afThnp\atop C_1,\cdots,C_m\in\afThnp_1}
\up^{2\sum_{i<j}\lan\bfd(C_i),\bfd(C_j)\ran}\frac{\fka_{C_1}\cdots \fka_{C_m}}{\fka_C}\\
&\qquad\qquad\cdot \vi_{C_1,\ldots,C_m}^C\vi_{C_1,\ldots, C_m}^D \ti K_{-\bfd(C)}u_D^+\text{\quad for $C\in\afThnp$},\\
\sg^{-1}(K_{\al})&=K_{-\al} \text{\quad for all $\al\in\mbz I$}.
\end{split}
\end{equation*}
\end{Prop}

\begin{proof}The Hopf structure on $\Hallpi$
is almost identical to the Hopf algebra $H'$ defined in the proof of
\cite[Prop.~4.8]{X97} except that we used $\ti K_\al$ instead of
$K_\al$ in the comultiplication and antipode. Thus, the
comultiplication of $\Hallpi$ defined here is opposite to that
defined in \cite[Th.~4.5]{X97}, while the antipode is the inverse.
Hence,  by Lemma \ref{opposite Hopf}, $\Hallpi$ is the Hopf algebra
semi-opposite to a variant of the Hopf algebras considered in
\cite[loc cit]{X97}. (Of course, one can directly check by mimicking
the proof of \cite[Th.~4.5]{X97} that $\Hallpi$ with the operations
defined above satisfies the axioms of a Hopf algebra.)
\end{proof}

\begin{Rems}\label{remark for Hallpi}
  (1) Because of Lemma \ref{rdEf}, we are able to make the root
datum used in the second relation in (a) invisible in the definition
of $\Hallpi$.

(2) Besides the modification of changing $K_\al$ to $\ti K_\al$ in
comultiplication and antipode, we also used the Euler form
$\lr{\,,\,}$ (or rather the form $\lr{\,,\,}^\rd$ for the root datum
given in Definition \ref{root datum}) instead of the symmetric Euler
form $(\,,\,)$ used in Xiao's definition of $H'$
(\cite[p.~129]{X97}) for the commutator formulas between $K_\al$ and
$u^+_A$. This means that $H'$ is not isomorphic to $\Hallpi$.
However, there is a Hopf algebra homomorphism from $H'$ to $\Hallpi$
by sending $u^+_AK_\al $ to $u^+_A\ti K_\al $ whose image is the
(Hopf) subalgebra $'\!\Hallpi$ generated by $\ti K_i$ and $u_A^+$
for all $i\in I$ and $A\in\afThnp$. Note that it sends the central
element
 $K_1\cdots K_n\neq1$ in $H'$ to $\ti K_1\cdots \ti K_n=1$ in $\Hallpi$.

(3) The above modifications are necessary for the compatibility with
Lusztig's construction for quantum groups in \cite{Lu93} and with
the corresponding relations in affine quantum Schur algebras.
\end{Rems}

It is clear that the subalgebra of $\Hallpi$ generated by $u_A^+$
($A\in\afThnp$) is isomorphic to $\bfHall$. The subalgebra generated
by $K_\al$, $\al\in\mbz I$, is isomorphic to the Laurent polynomial
ring $\mbq(\up)[K_1^{\pm 1},\ldots,K_n^{\pm n}]$.

\begin{Coro}\label{Green Xiao 2}
The $\mbq(\up)$-space $\Hallmi$\index{Ringel--Hall algebra!extended
$\sim$, $\Hallmi$} with basis $\{K_\al u_A^- \mid \al\in\mbz I, A\in
\afThnp\}$ becomes a Hopf algebra with the following algebra,
coalgebra and antipode structures.\index{extended Ringel-Hall
algebra}

\begin{itemize}
\item[(a$'$)] Multiplication and unit:
\begin{equation*}
\begin{split}
u_A^-u_B^-&=\sum_{C\in\afThnp}\nup^{\lan \bfd(B),\bfd(A)\ran}\vi_{B,A}^C u_C^- \text{\quad for all $A,B\in\afThnp$},\\
u_A^-K_\al&=\nup^{\lr{\bfd(A),\al}}K_\al u_A^- \text{\quad for all $\al\in\mbz I$, $A\in\afThnp$},\\
K_{\al} K_{\bt}&=K_{\al+\bt} \text{\quad for all $\al,\bt\in\mbz
I$}.
\end{split}
\end{equation*}
with unit $1=u_0^-=K_0$.

\item[(b$'$)] Comultiplication and counit:
\begin{equation*}
\begin{split}
\Dt(u_C^-)&=\sum_{A,B\in\afThnp}\nup^{-\lan
\bfd(B),\bfd(A)\ran}\frac{\fka_A\fka_B}{\fka_C}
\vi_{A,B}^C \ti K_{-\bfd(A)}u_B^-\ot u_A^-\text{\quad for all $C\in\afThnp$},\\
\Dt(K_\al)&=K_\al\ot K_\al \text{\quad for all $\al\in\mbz I$},
\end{split}
\end{equation*}
with counit $\ep$ satisfying $\ep(u_C^-)=0$ for $C\in\afThnp_1$
and $\ep(K_\al)=1$ for all $\al\in\mbz I$.

\item[(c$'$)] Antipode:
\begin{equation*}
\begin{split}
&\sg(u_C^-)=\dt_{C,0} +\sum_{m\geq 1}(-1)^m\sum_{D\in\afThnp\atop
C_1,\ldots,C_m\in\afThnp_1}
\up^{2\sum_{i<j}\lan\bfd(C_i),\bfd(C_j)\ran}\frac{\fka_{C_1}\cdots
\fka_{C_m}}{\fka_C}\\
&\hspace{6.5cm}\times\vi_{C_1,\ldots,C_m}^C
\vi_{C_1,\ldots,C_m}^D\ti K_{\bfd(C)} u_D^-,\\
&\text{\quad for $C\in\afThnp$ and}\;\;\sg(K_{\al})=K_{-\al}
\text{\quad for all $\al\in\mbz I$}.
\end{split}
\end{equation*}
\end{itemize}
\end{Coro}

\begin{proof} Let $\boldsymbol{\mathfrak H}'$ be the $\mathbb Q(\up)$-space with
basis $\{K_\al'(u_A^-)'\mid \al\in\mbz I,A\in \afThnp\}$ and define
the following operations on $\boldsymbol{\mathfrak H}'$:
\begin{equation*} (a'')
\begin{split}
(u_A^-)'(u_B^-)'&=\sum_{C\in\afThnp}\nup^{\lan \bfd(B),\bfd(A)\ran}\vi_{B,A}^C(u_C^-)'
\text{\quad for all $ A,B\in\afThnp$},\\
(u_A^-)'K_\al' &=\nup^{-\lan\bfd(A),\al\ran}K_\al'(u_A^-)'
\text{\quad for all $\al\in\mbz I$, $A\in\afThnp$},\\
 K_{\al}' K_{\bt}'&=K_{\al+\bt}' \text{\quad for all $\al,\bt\in\mbz I$}.
\end{split}
\end{equation*}
with unit $1=(u_0^-)'=K_0'$.
\begin{equation*}(b'')
\begin{split}
\Dt((u_C^-)')&=\sum_{A,B\in\afThnp}\nup^{\lan\bfd(A),\bfd(B)\ran}\frac{\fka_A\fka_B}{\fka_C}
\vi_{A,B}^C(u_B^-)'\ot \ti K_{-\bfd(B)}'(u_A^-)' \text{\quad for all $C\in\afThnp$},\\
\Dt(K_\al')&=K_\al'\ot K_\al' \text{\quad for all $\al\in\mbz I$}
\end{split}
\end{equation*}
with counit  $\ep((u_A^-)')=0$ for $A\in\afThnp_1$ and
$\ep(K_\al')=1$ for all $\al\in\mbz I$, where $\ti
K_i'=K_i'(K_{i+1}')^{-1}$.
\begin{equation*}(c'')
\begin{split}
\sg((u_C^-)')&=\dt_{C,0}+\sum_{m\geq 1}(-1)^m\sum_{D\in\afThnp\atop C_1,\ldots,C_m\in\afThnp_1}
\up^{2\sum_{i<j}\lan\bfd(C_i),\bfd(C_j)\ran}\frac{\fka_{C_1}\cdots \fka_{C_m}}{\fka_C}\\
&\qquad\qquad\cdot \vi_{C_1,\ldots,C_m}^C\vi_{C_1,\ldots, C_m}^D (u_D^-)'\ti K_{\bfd(C)}'\text{\quad for $C\in\afThnp$},\\
\sg(K_{\al}')&=K_{-\al}' \text{\quad for all $\al\in\mbz I$}.
\end{split}
\end{equation*}

By Lemma \ref{opposite Hopf}, if we replace the multiplication of
$\Hallpi$ by its opposite one and $\sg$ by $\sg^{-1}$ and keep other
structure maps unchanged, then we obtain the semi-opposite Hopf
algebra $\Hallpi_1$ of $\Hallpi$. It is clear that the $\mathbb
Q(\up)$-linear isomorphism $\Hallpi_1\to \boldsymbol{\mathfrak H}'$
taking $u_A^+K_\al\mapsto (u_A^-)'K_{-\al}'$ preserves all the
operations. Thus, $\boldsymbol{\mathfrak H}'$ is a Hopf algebra with
the operations (a$''$)--(c$''$).

Now, for $\al\in\mbz I$ and $A\in \afThnp$, set in
$\boldsymbol{\mathfrak H}'$,
\begin{equation}\label{ulam}
K_{\al}:=K_{-\al}'\;\;\;\text{and}\;\;\;
u_A^-:=\up^{-\lan\bfd(A),\bfd(A)\ran}\ti K_{\bfd(A)}'(u_A^-)'.
\end{equation}
Then $\{K_\al u_A^- \mid \al\in\mbz I, A\in \afThnp\}$ is a new
basis of $\boldsymbol{\mathfrak H}'$. It is easy to check that
applying the operations (a$''$)--(c$''$) to the basis elements
$K_\al u_A^-$ gives (a$'$)--(c$'$). Consequently, $\Hallmi$ with the
operations (a$'$)--(c$'$) is a Hopf algebra.
\end{proof}

The proof above shows that $\Hallmi$ is the semi-opposite Hopf
algebra of $\Hallpi$ in the sense that multiplication and antipode
are replaced by the opposite and inverse ones, respectively. Note
that the inverse $\sg^{-1}$ of $\sg$ in $\Hallmi$ is defined by
\begin{equation*}
\begin{split}
\sg^{-1}(u_C^-)&=\dt_{C,0}\\
\qquad &+\sum_{m\geq 1}(-1)^m\sum_{D\in\afThnp\atop
C_1,\ldots,C_m\in\afThnp_1} \frac{\fka_{C_1}\cdots
\fka_{C_m}}{\fka_C}\vi_{C_1,\ldots,C_m}^C \vi_{C_m,\ldots,C_1}^D
u_D^-\ti K_{\bfd(C)} .
\end{split}
\end{equation*}


By (a$'$), the subalgebra of $\Hallmi$ generated by $u_A^-$
($A\in\afThnp$) is isomorphic to $\bfHall^{\rm op}$. Moreover,
there is a $\mbq(\up)$-vector space isomorphism
\begin{equation}\label{nonpositive-tri-decom}
 \Hallmi= \mbq(\up)[K_1^{\pm 1},\ldots,K_n^{\pm n}]\ot_{\mbq(\up)}\bfHall^{\rm op}.
\end{equation}

\begin{Rem}\label{mathcal A} If we put $\mathcal A=\sZ[(\up^m-1)^{-1}]_{m\geq1}$, then \eqref{poly-auto-group}
guarantees that extended Ringel--Hall algebras $\Hall^{\geq0}_\sA$
and $\Hall^{\leq0}_\sA$ over $\sA$ are well-defined. Thus, if $\up$
is specialized to $z$ in a field (or a ring) $\field$ which is not a
root of unity, then extended Ringel--Hall algebras
$\Hall^{\geq0}_\field:=\Hall^{\geq0}_\sA\ot\field$ and
$\Hall^{\leq0}_\field$ over $\field$ are  defined.
\end{Rem}

\chapter{Double Ringel--Hall algebras of cyclic quivers}

A Drinfeld double refers to a construction of gluing two Hopf
algebras via a skew-Hopf paring between them to obtain a new Hopf
algebra. We apply this construction in \S 2.1 to the extended Ringel--Hall
algebras $\Hallpi$ and $\Hallmi$ discussed in \S1.5 to obtain double Ringel--Hall
algebras $\dHallr$ of cyclic quivers.

The algebras $\dHallr$ possess
a rich structure. First, by using the Schiffmann--Hubery generators
and the connection with the quantum enveloping algebra associated
with a Borcherds--Cartan matrix, we obtain a presentation for
$\dHallr$ (Theorem \ref{presentation dHallAlg}). Second, by using Drinfeld's new presentation for the
quantum loop algebra $\mathbf U(\h{\mathfrak{gl}}_n)$, we extend
Beck (and Jing)'s embedding of quantum affine $\h{\mathfrak {sl}}_n$ into
$\mathbf U(\h{\mathfrak{gl}}_n)$ to obtain an isomorphism between $\dHallr$
and $\mathbf U(\h{\mathfrak{gl}}_n)$ (Theorem \ref{iso afgln dHallr}). Finally, applying the
skew-Hopf pairing to semisimple generators yields certain commutator
relations (Theorem \ref{commutator-formula}). Thus, we propose a possible presentation using semisimple
generators; see Problem \ref{semisimple presentation}.

\section{Drinfeld doubles and the Hopf algebra $\dHallr$}

In this section, we first recall from \cite{Jo95} the notion of a
skew-Hopf pairing and define the associated Drinfeld double. We then
apply this general construction to obtain the Drinfeld double
$\dHallr$ of the Ringel--Hall algebra $\bfHall$; see \cite{X97} for
a general construction.


 Let $\scrA=(\scr A,\mu_\scrA,\eta_\scrA,\Dt_\scrA,\ep_\scrA,\sg_\scrA)$ and $\scrB=(\scr B,\mu_\scrB,\eta_\scrB,\Dt_\scrB,\ep_\scrB,\sg_\scrB)$
be two Hopf algebras over a field $\field$. A {\it skew-Hopf
pairing} of $\scrA$ and $\scrB$ is an $\field$-bilinear form
$\psi:\scrA\times \scrB\to \field$ satisfying: \index{skew-Hopf
pairing}
\begin{itemize}
\item[(HP1)] $\psi(1,b)=\ep_\scrB(b)$, $\psi(a,1)=\ep_\scrA(a)$, for
all $a\in \scrA$, $b\in \scrB$,

\item[(HP2)] $\psi(a,bb')=\psi(\Dt_\scrA(a),b\ot b')$, for all $a\in
\scrA$, $b,b'\in \scrB$,

\item[(HP3)] $\psi(aa',b)=\psi(a\ot a',\Dt_\scrB^{\rm op}(b))$, for
all $a,a'\in \scrA$, $b,b'\in \scrB$,

\item[(HP4)] $\psi(\sg_\scrA(a),b)=\psi(a,\sg_\scrB^{-1}(b))$, for
all $a\in \scrA$, $b,b'\in \scrB$,
\end{itemize}
where $\psi(a\ot a',b\ot b')=\psi(a,b)\psi(a',b')$, and
$\Dt_\scrB^{\rm op}$ is defined by $\Dt_\scrB^{\rm op}(b)=\sum
b_2\otimes b_1$ if $\Dt_\scrB(b)=\sum b_1 \otimes b_2$. Note that
we have assumed here that $\sg_\scrB$ is invertible.

Let $\scrA*\scrB$ be the free product of $\field$-algebras $\scrA$
and $\scrB$ with identity. Thus, for any fixed bases $B_\scrA$ and
$B_\scrB$ for $\scrA,\scrB$, respectively, where both $B_\scrA$ and
$B_\scrB$ contain the identity element, $\scrA*\scrB$ is the
$\field$-space spanned by the basis consisting of all words
$b_1b_2\cdots b_m$ ($b_i\in (B_\scrA\backslash\{1\})\cup
(B_\scrB\backslash\{1\})$) of any length $m\geq0$ such that
$b_ib_{i+1}$ is not defined (in other words, $b_i,b_{i+1}$ are not
in the same $\scrA$ or $\scrB$) with multiplication given by
``contracted juxtaposition''
$$(b_1\cdots b_m)*(b_1'\cdots b'_{m'})=\begin{cases}
b_1\cdots b_mb_1'\cdots b'_{m'}, &\text{ if $b_mb_1'$ is not defined;}\\
b_1\cdots b_{m-1}cb_2'\cdots b'_{m'},&\text{ if $b_mb_1'$ is defined, $b_mb_1'\neq0$;}\\
0, &\text{ otherwise.}\\
\end{cases}$$
Note that, since $c=\sum_k\la_ka_{i_k}$ ($a\in\{b,b'\}$) is a linear
combination of basis elements, the element $b_1\cdots
b_{m-1}cb_2'\cdots b'_{m'}$ is a linear combination of words and is
defined inductively. Thus, $1$ is replaced by the empty word.

Let $\sI=\sI_{\scrA,\scrB}$ be the ideal of $\scrA*\scrB$
generated by
\begin{equation}\label{D-double-rel}
\sum (b_2*a_2)\psi(a_1,b_1)-\sum (a_1*b_1)\psi(a_2,b_2) \;\;(a\in
\scrA,b\in \scrB),
\end{equation}
where $\Dt_\scrA(a)=\sum a_1\ot a_2$ and $\Dt_\scrB(b)=\sum b_1\ot
b_2$. Moreover, $\sI$ is also generated by
\begin{equation}\label{D-double-rel-2}
b*a-\sum\psi(a_1,\sg_\scrB(b_1))(a_2* b_2)\psi(a_3,b_3)\;\;(a\in
\scrA,b\in \scrB),\end{equation}
 where $\Dt_\scrA^2(a)=\sum a_1\ot a_2\ot a_3$ and $\Dt_\scrB^2(b)=\sum
b_1\ot b_2\ot b_3$ (see, for example, \cite[p.~72]{Jo95}).

The Drinfeld double\index{Drinfeld double} of the pair $\scrA$ and
$\scrB$ is by definition the quotient algebra
$$D(\scrA,\scrB):=\scrA*\scrB/{\sI}.$$
 By \eqref{D-double-rel-2}, each element in $D(\scrA,\scrB)$ can
 be expressed as a linear combination of elements of the form
 $a*b+\sI$ for $a\in\scrA$ and $b\in\scrB$. Note that there is an $\field$-vector space isomorphism
$$D(\scrA,\scrB)\lra \scrA\ot_{\field}\scrB,\;
a*b+\sI\lm a\ot b;$$
 see \cite[Lem.~3.1]{SV2}. For notational simplicity, we
write $a*b+\sI$ as $a*b$. Both $\scrA$ and $\scrB$ can be viewed
as subalgebras of $D(\scrA,\scrB)$ via $a\mapsto a*1$ and
$b\mapsto 1*b$, respectively. Thus, if $x$ and $y$ lie in $\scrA$
or $\scrB$, we write $xy$ instead of $x*y$.

 The algebra $D(\scrA,\scrB)$ admits a Hopf algebra structure induced by
those of $\scrA$ and $\scrB$; see \cite[3.2.3]{Jo95}. More
precisely, comultiplication, counit and antipode in
$D(\scrA,\scrB)$ are defined by
\begin{equation}\label{coalg of Drinfeld double}
\aligned\Dt(a*b)&=\sum(a_1*b_1)\ot (a_2*b_2),\\
\ep(a*b)&=\ep_\scrA(a)\ep_\scrB(b),\\
\sg(a*b)&=\sg_\scrB(b)*\sg_\scrA(a),
\endaligned\end{equation}
where $a\in\scrA$, $b\in\scrB$, $\Dt_\scrA(a)=\sum a_1 \otimes
a_2$ and $\Dt_\scrB(b)=\sum b_1 \otimes b_2$.

By modifying \cite[Lem.~3.2]{SV2}, we obtain that the above ideal
$\sI$ can be generated by the elements in certain generating sets
of $\scrA$ and $\scrB$ as described in the following result.

\begin{Lem}\label{reduce lemma}
Let $\scrA,\scrB$ be Hopf algebras over $\field$ and let
$\psi:\scrA\times \scrB\ra \field$ be a skew-Hopf pairing. Assume
that $X_\scrA\han \scrA$ and $X_\scrB\han \scrB$ are generating sets
of $\scrA$ and $\scrB$, respectively. If
$\Dt_\scrA(X_\scrA)\han\spann_{\field}X_\scrA\ot\spann_{\field}X_\scrA$
and
$\Dt_\scrB(X_\scrB)\han\spann_{\field}\{X_\scrB\}\ot\spann_{\field}\{X_\scrB\}$,
then the ideal $\sI=\sI_{\scrA,\scrB}$ is generated by the following
elements
\begin{equation}\label{D-double-rel-generators}
\sum (b_2*a_2)\psi(a_1,b_1)-\sum (a_1*b_1)\psi(a_2,b_2)\;\;\text{for
all $a\in X_\scrA,b\in X_\scrB$}.
\end{equation}
\end{Lem}

\begin{proof} For $a\in\scrA$ and $b\in\scrB$, put
$$h_{a,b}:=\sum(b_2*a_2)\psi(a_1,b_1)-\sum (a_1*b_1)\psi(a_2,b_2).$$
Let $\sI'$ be the ideal of $\scrA*\scrB$ generated by $h_{a,b}$ for
all $a\in X_\scrA$, $b\in X_\scrB$, and set ${\scr
H}=\scrA*\scrB/\sI'$. We need to show $\sI=\sI'$. Since
$\sI'\subseteq \sI$, it remains to show that for all $a\in\scrA$ and
$b\in\scrB$, $h_{a,b}\in\sI'$, or equivalently, $h_{a,b}=0$ in $\scr
H$.

First, suppose $a\in X_\scrA$ and $b=y_1\cdots y_t$ with $y_j\in
X_\scrB$. We proceed by induction on $t$ to show that
$h_{a,b}\in\sI'$. Let $b'=y_1\cdots y_{t-1}$ and $b''=y_t$. Write
$\Dt_\scrA(a)=\sum a_1\ot a_2$, $\Dt_\scrA(a_1)=\sum a_{1,1}\ot
a_{1,2}$ and $\Dt_\scrA(a_2)=\sum a_{2,1}\ot a_{2,2}$. The
coassociativity of $\Dt_\scrA$ implies that
$$\sum a_{1,1}\ot a_{1,2}\ot a_2=\sum \Dt_\scrA(a_1)\ot a_2
=\sum a_1\ot \Dt_\scrA(a_2)=\sum a_1\ot a_{2,1}\ot a_{2,2}.$$
 Further, write $\Dt_\scrB(b')=\sum b_1'\ot b_2'$ and
$\Dt_\scrB(b'')=\sum b_1''\ot b_2''$. Then $\Dt_\scrB(b)=\sum
b_1'b_1''\ot b_2'b_2''$. Thus, we obtain in $\scr H$ that
 \begin{equation*}
 \begin{split}
&\sum(b_2'b_2''*a_2)\psi(a_1,b_1'b_1'')\\
=&\sum b_2'(b_2''*a_2)\psi(a_{1,1},b_1')\psi(a_{1,2},b_1'')  \\
=&\sum b_2'(b_2''*a_{2,2})\psi(a_1,b_1')\psi(a_{2,1},b_1'')  \\
=&\sum(b_2'*a_{2,1})*b_1''\psi(a_1,b_1')\psi(a_{2,2},b_2'')
\quad(\text{since $a_2\in \spann_{\field}X_\scrA$ and $b''\in X_\scrB$}) \\
=&\sum(b_2'*a_{1,2})*b_1''\psi(a_{1,1},b_1')\psi(a_2,b_2'')\\
=&\sum(a_{1,1}*b_1')b_1''\psi(a_{1,2},b_2')\psi(a_2,b_2'') \ \ \ \quad(\text{by induction for $a_1$ and $b'$})\\
=&\sum(a_1*b_1')b_1''\psi(a_{2,1},b_2')\psi(a_{2,2},b_2'')\\
=&\sum(a_1*b_1'b_1'')\psi(a_2,b_2'b_2''),
 \end{split}
 \end{equation*}
 that is, $h_{a,b}\in\sI'$.

Now suppose $a=x_1\cdots x_s\in\scrA$ with $x_i\in X_\scrA$ and
$b\in \scrB$. We proceed by induction on $s$. The case $s=1$ has
been already treated above. So assume $s>1$. Let $a'=x_1\cdots
x_{s-1}$ and $a''=x_s$. Then
\begin{equation*}
\begin{split}
&\sum(b_2*a_2'a_2'')\psi(a_1'a_1'',b_1)\\
=&\sum(b^{(2)}_3*a_2')a_2''\psi(a_1',b^{(2)}_2)\psi(a_1'',b^{(2)}_1)\\
=&\sum a_1'*(b^{(2)}_2*a_2'')\psi(a_2',b^{(2)}_3)\psi(a_1'',b^{(2)}_1)\quad(\text{by induction})\\
=&\sum a_1'a_1''*b^{(2)}_1\psi(a_2',b^{(2)}_3)\psi(a_2'',b^{(2)}_2)
\qquad\quad(\text{since $a''\in X_\scrA$})\\
=&\sum a_1'a_1''*b_1\psi(a_2'a_2'',b_2),
\end{split}
\end{equation*}
where $\Dt^2_\scrB(b)=\sum b^{(2)}_1\ot b^{(2)}_2\ot b^{(2)}_3$.
Hence, $h_{a,b}\in\sI'$. This completes the proof.
\end{proof}

A skew-Hopf pairing can be passed on to opposite Hopf algebras
(see Lemma \ref{opposite Hopf}). The following lemma can be checked directly.

\begin{Lem}\label{twist Hopf structure}
 Let
$\scrA=(\scrA,\mu_\scrA,\eta_\scrA,\Dt_\scrA,\ep_\scrA,\sg_\scrA)$,
$\scrB=(\scrB,\mu_\scrB,\eta_\scrB,\Dt_\scrB,\ep_\scrB,\sg_\scrB)$
be two Hopf algebras over a field $\field$ together with a skew-Hopf
pairing $\psi:\scrA\times \scrB\ra \field$. Assume that $\sg_\scrA$
and $\sg_\scrB$ are both invertible. Then $\psi$ is also a skew-Hopf
pairing of the semi-opposite Hopf algebras
$(\scrA,\mu_\scrA,\eta_\scrA,\Dt_\scrA^{\rm
op},\ep_\scrA,\sg_\scrA^{-1})$ and $(\scrB,\mu_\scrB^{\rm
op},\eta_\scrB,\Dt_\scrB,\ep_\scrB,\sg_\scrB^{-1})$ (resp.
$(\scrA,\mu_\scrA^{\rm
op},\eta_\scrA,\Dt_\scrA,\ep_\scrA,\sg_\scrA^{-1})$ and
$(\scrB,\mu_\scrB,\eta_\scrB,\Dt_\scrB^{\rm
op},\ep_\scrB,\sg_\scrB^{-1})$).
\end{Lem}

We end this section with the construction of the Drinfeld double
associated with the Ringel--Hall algebras $\Hallpi$ and $\Hallmi$
introduced in \S1.5.

First, we need a skew-Hopf pairing. Applying Lemma \ref{twist Hopf
structure} to \cite[Prop.~5.3]{X97} yields the following result. For
completeness, we sketch a proof. We introduce some notation which is
used in the proof. For each $\al=\sum_{i\in I}a_ii\in\mbz I$, write
$\tau\al=\sum_{i\in I}a_{i-1}i$. In particular, for each
$A\in\afThnp$, we have $\tau\bfd(A)=\bfd(\tau(A))$. Then for
$\al,\beta\in\mbz I$,
$$\lan\al,\beta\ran=(\al-\tau\al)\centerdot\beta=-\lan\beta,\tau\al\ran\;\;
\text{and}\;\;\ti K_\al=K_{\al-\tau\al}.$$\index{$\ti K_\al$, $\ti
K_i$}

\begin{Prop}\label{bilinear form} The $\mbq(\up)$-bilinear form
$\psi:\Hallpi\times \Hallmi\ra \mbq(\up)$ defined by
\begin{equation}\label{SHP}
\psi(u_A^+ K_{\al},K_{\bt}u_B^-)=\up^{\al\centerdot\bt-
\lan\bfd(A),\bfd(A)+\al\ran+2\fkd(A)}\fka_A^{-1}\dt_{A,B}
\end{equation}
 where $\al,\bt\in\mbz I$ and $A,B\in\afThnp$, is a skew-Hopf pairing.
Here $(\al,\bt)=\lan\al,\bt\ran +\lan\bt,\al\ran$ is the symmetric
Euler form.
\end{Prop}

\begin{proof} Condition (HP1) is obvious. We now check condition (HP2). Without loss of
generality, we take $a=u_A^+K_\al$, $b=K_\beta u_B^-$ and $b'=K_\ga
u_C^-$ for $\al,\bt,\ga\in\mbz I$ and $A,B,C\in\afThnp$. Then
$$\aligned
&\psi(u_A^+K_\al,K_\beta u_B^-K_\ga u_C^-)=
\psi(u_A^+K_\al,v^{\lr{\bfd(B),\ga}}K_{\beta+\ga}u_B^-u_C^-)\\
=\,&\psi(u_A^+K_\al,\,v^{\lr{\bfd(B),\ga}}K_{\beta+\ga}\sum_{D}v^{\lr{\bfd(C),\bfd(B)}}
\varphi^D_{C,B}u_D^-)\\
=\,&v^{x_1}\varphi^A_{C,B}\fka_A^{-1},
\endaligned$$
 where $x_1=\lr{\bfd(B),\ga}+\lr{\bfd(C),\bfd(B)}+\al\centerdot(\beta+\ga)
 -\lr{\bfd(A),\bfd(A)+\al}+2\fkd(A)$.

 On the other hand,
$$\aligned
&\psi(\Dt(u_A^+K_\al),K_\beta u_B^-\otimes K_\ga u_C^-)\\
=\,&\psi(\sum_{B',C'}v^{\lr{\bfd(B'),\bfd(C')}}\frac{\fka_{B'}\fka_{C'}}{\fka_{A}}
\varphi^A_{B',C'}u_{C'}^+K_\al\otimes u^+_{B'}\ti K_{\bfd(C')}K_\al,K_\beta u_B^-\otimes K_\ga u_C^-)\\
=\,&\sum_{B',C'}v^{\lr{\bfd(B'),\bfd(C')}}\frac{\fka_{B'}\fka_{C'}}{\fka_{A}}
\varphi^A_{B',C'}\psi(u_{C'}^+K_\al,K_\beta u_B^-) \psi(u^+_{B'}\ti
K_{\bfd(C')}K_\al,K_\ga u_C^-)\\
=\,&v^{x_2}\varphi^A_{C,B}\fka_A^{-1},
\endaligned$$
 where $x_2=\lr{\bfd(C),\bfd(B)}+\al\centerdot
\beta -\lr{\bfd(B),\bfd(B)+\al}
+(\bfd(B)-\tau\bfd(B)+\al)\centerdot\ga
-\lr{\bfd(C),\bfd(C)+\bfd(B)-\tau\bfd(B)+\al}+2\fkd(A).$ Here we
have assumed $\bfd(A)=\bfd(B)+\bfd(C)$ since $\varphi^A_{C,B}=0$
otherwise. A direct calculation shows $x_1=x_2$. Hence, (HP2) holds.

Condition (HP3) can be checked similarly as (HP2). It remains to
check (HP4). Take $a=u_A^+K_\al$ and $b=K_\beta u_B^-$. We may
suppose $A\not=0\not=B$. Then
$$\aligned
&\psi(\sg(u_A^+K_\al),K_\beta u_B^-)\\
=\,&\sum_{m\geq 1}(-1)^m v^{y_1}\sum_{C_1,\ldots,C_m\in\afThnp_1}
\frac{\fka_{C_1}\cdots
\fka_{C_m}}{\fka_A\fka_B}\vi_{C_1,\ldots,C_m}^A
\vi_{C_m,\ldots,C_1}^B
\endaligned$$
and
$$\aligned
&\psi(u_A^+K_\al,\sg^{-1}(K_\beta u_B^-))\\
=\,&\sum_{m\geq 1}(-1)^m v^{y_2}\sum_{C_1,\ldots,C_m\in\afThnp_1}
\frac{\fka_{C_1}\cdots
\fka_{C_m}}{\fka_A\fka_B}\vi_{C_1,\ldots,C_m}^A
\vi_{C_m,\ldots,C_1}^B,
\endaligned$$
 where
 $$y_1=-\lr{\bfd(B),\al}+(-\al-\bfd(A)+\tau\bfd(A))\centerdot\beta
 -\lr{\bfd(B),\bfd(B)-\al-\bfd(A)+\tau\bfd(A)}$$ and
$$y_2=\lr{\bfd(A),\bfd(B)-\tau\bfd(B)-\beta}+\al\centerdot(\bfd(B)-\tau\bfd(B)-\beta)-\lr{\bfd(A),\bfd(A)+\al}.$$
 Clearly, if $\bfd(A)\not=\bfd(B)$, then $\vi_{C_1,\ldots,C_m}^A\vi_{C_m,\ldots,C_1}^B=0$.
Hence, we may suppose $\bfd(A)=\bfd(B)$. Then
$$y_1=-\al\centerdot\beta-\lr{\bfd(A),\beta}+\lr{\bfd(A),\bfd(A)}=y_2.$$
Therefore, $\psi(\sg(u_A^+K_\al),K_\beta
u_B^-)=\psi(u_A^+K_\al,\sg^{-1}(K_\beta u_B^-))$, that is, (HP4)
holds.
\end{proof}

Second, the Hopf algebras $\Hallpi$ and $\Hallmi$ together with the
skew-Hopf pairing $\psi$ gives rise to the Drinfeld
double\index{Drinfeld double}
$$\widehat\dHallr:=D(\Hallpi, \Hallmi).$$
 Since as $\mbq(\up)$-vector spaces, we have
$$\widehat\dHallr=D(\Hallpi, \Hallmi)\cong
\Hallpi\ot\Hallmi,$$
 we sometimes write the elements in $\widehat\dHallr$ as linear combinations of $a\ot b$
for $a\in \Hallpi$ and $b\in\Hallmi$. Moreover, it follows from
\eqref{nonnegative-tri-decom} and \eqref{nonpositive-tri-decom}
that there is a $\mbq(\up)$-vector space isomorphism
$$\widehat\dHallr \cong \bfHall\ot_{\mbq(\up)} \mbq(\up)[K_1^{\pm 1},\ldots,K_n^{\pm
n}]\ot_{\mbq(\up)}\mbq(\up)[K_1^{\pm 1},\ldots,K_n^{\pm
n}]\ot_{\mbq(\up)}\bfHall^{\rm op}.$$

Finally, we define the {\it reduced} Drinfeld double\index{Drinfeld
double! reduced $\sim$}
\begin{equation}\label{def-double-RH-alg}
\dHallr=\widehat\dHallr/{\scr I}
\end{equation}
 where $\scr I$ denotes the ideal generated by $1\otimes K_\al-K_\al\ot 1$ for
all $\al\in\mbz I$. By the construction, $\scr I$ is indeed a Hopf
ideal of $\widehat\dHallr$. Thus, $\dHallr$ is again a Hopf
algebra. We call $\dHallr$ the {\it double Ringel--Hall algebra} of the
cyclic quiver $\tri$. \index{Ringel--Hall algebra!double $\sim$, $\dHallr$}
\index{double Ringel--Hall algebra!$\dHallr$, $\sim$}\index{$\dHallr$, double Ringel--Hall algebra}

Let $\dHallr^+$ (resp., $\dHallr^-$) be the $\mbq(\up)$-subalgebra
of $\dHallr$ generated by $u_A^+$ (resp., $u^-_A$) for all
$A\in\afThnp$. Let $\dHallr^0$ be the $\mbq(\up)$-subalgebra of
$\dHallr$ generated by $K_\al$ for all $\al\in\mbz I$. Then
\begin{equation}\label{pm parts}
\dHallr^+\cong\bfHall,\;\dHallr^-\cong\bfHall^{\rm op},\;
\dHallr^0\cong \mbq(\up)[K_1^{\pm 1},\ldots,K_n^{\pm 1}].
\end{equation}
 Moreover, the multiplication map
$$\dHallr^+\ot\dHallr^0\ot\dHallr^-\lra \dHallr$$
 is an isomorphism of $\mbq(\up)$-vector spaces. Also, we have
$${\dHallr^{\geq
0}:=\dHallr^+\ot\dHallr^0\cong\Hallpi}\atop
{\dHallr^{\leq 0}:=\dHallr^0\ot\dHallr^-\cong\Hallmi}.$$
We will identify $\dHallr^{\geq0}$ and $\dHallr^{\leq 0}$ with $\Hallpi$ and
$\Hallmi$, respectively, in the sequel. In
particular, we may use the PBW type basis for $\bfHall$ to display a
PBW type basis for $\dHallr$:
$$\{\ti u_A^+ K_\al \ti u_B^-\mid A,B\in\afThnp,\al\in\mbz I\}.$$\index{PBW type basis}

\begin{Rem} \label{NonrootOfUnity2}
By specializing $\up$ to $z\in\mbc$ which is not a root of unity,
\eqref{poly-auto-group} implies that the skew-Hopf pairing $\psi$
given in \eqref{SHP} is well-defined over $\mbc$. Hence the
construction above works over $\mbc$ with $\Hallpi$ etc. replaced by
the corresponding specialization $\Hall_\mbc^{\geq0}$, etc. (see
Remark \ref{mathcal A}). Thus, we obtain double Ringel--Hall
algebras $\DC(n)$. In fact, if we use the ring $\sA$ defined in
Remark \ref{mathcal A}, a same reason shows that the skew-Hopf
paring in \eqref{SHP} is defined over $\sA$. Hence, we can form a
double Ringel--Hall algebra $\fD_\vtg(n)_\sA$. Then $\DC(n)\cong
\fD_\vtg(n)_\sA\ot\mbc$ and $\dHallr=\fD_\vtg(n)_\sA\ot\mbq(\up)$.
\end{Rem}

\section{Schiffmann--Hubery generators and presentation of $\dHallr$}
We now relate $\dHallr$ with the
quantum enveloping algebra of a generalized Kac--Moody algebra
based on \cite{Sch,Hub1}; see also \cite{HX,DX1}.

We first describe the structure of $\dHallr$. As in
\eqref{central elements}, define for each $m\geq 1$,
 $$c_m^\pm=(-1)^mv^{-2nm}\sum_{A}(-1)^{{\rm dim}\,{\rm End}(M(A))}\fka_A u^\pm_A \in\dHallr^\pm,$$
 where the sum is taken over all $A\in \afThnp$ such that $\bfd (A)=m\dt$ and
${\rm soc}\,M(A)$ is square-free. We also define $c_0^\pm=1$ by
convention. By Theorem \ref{Schiffmann-Hubery}, the elements
$c^+_m$ and $c^-_m$ are central in $\dHallr^+$ and $\dHallr^-$,
respectively. Following \cite{Hub2}, we set $\pi^\pm_0=0$ and
define recursively $\pi^\pm_m\in\dHallr^\pm$ by
\begin{equation}\label{recursive formula}
\pi^\pm_m=\begin{cases}\frac {\up^n}{\up-\up^{-1}}c_1^\pm,&\text{ if }m=1;\\
\frac {\up^{nm}}{\up-\up^{-1}}c_m^\pm-\frac1m\sum_{s=1}^{m-1}
sv^{(m-s)n}\pi^\pm_sc^\pm_{m-s},&\text{ if }m\geq2.\end{cases}
\end{equation}
Note that $\pi^\pm_m$ is the element $\pi_{n,m}$ defined in
\cite{Hub2}. Hence, $\pi_m^+$ and $\pi_m^-$ are also central in
$\dHallr^+$ and $\dHallr^-$, respectively. Moreover, if $\bfcomp^+$
and $\bfcomp^-$ denote the $\mbq(\up)$-subalgebras of $\dHallr$
generated by $u_i^+$ and $u_i^-$ ($i\in I$), respectively, then by
Theorem \ref{Schiffmann-Hubery}
\begin{equation}\label{SH1}
\dHallr^\pm=\bfcomp^\pm\otimes\mbq(\up)[\pi_1^\pm,\pi_2^\pm,\ldots].
\end{equation}

Furthermore, by \cite[Lem.~12]{Hub2},
$$\pi^\pm_m=v^{m(1-n)}\frac{[m]}{m}\sum_{l=1}^n u^\pm_{E^\vtg_{l,l+mn}}+x^\pm
=\frac{[m]}{m}\sum_{l=1}^n {\ti u}^\pm_{E^\vtg_{l,l+mn}}+x^\pm,$$
 where $x^\pm$ is a linear combination of certain $u^\pm_A$ such that
$\bfd(A)=m\dt$ and $M(A)$ are decomposable, and $\tilde u_A$ is defined in \eqref{tilde u_A}.  Moreover, $\pi^\pm_m$
are primitive, i.e.,
$$\Dt(\pi_m^\pm)=\pi_m^\pm\otimes 1+1\otimes \pi_m^\pm,$$
where $\Dt$ is the comultiplication on $\dHallr$ induced by Green's
comultiplication given in Proposition \ref{Green Xiao}(b) and
Corollary \ref{Green Xiao 2}(b$'$).

For each $m\geq 1$, let
\begin{equation}\label{expression sfx sfy}
\sfz_m^\pm=\frac{m}{[m]}\pi_m^\pm=\sum_{l=1}^n {\ti
u}^\pm_{E^\vtg_{l,l+mn}}+\frac{m}{[m]}x^\pm,
\end{equation}
and let $\sg$ be the antipode of $\dHallr$ induced by Xiao's
antipode as given in Proposition \ref{Green Xiao}(c) and Corollary
\ref{Green Xiao 2}(c$'$).

\begin{Lem} \label{central-dHall}
For each $m\geq 1$,  the elements $\sfz_m^+$ and $\sfz_m^-$ are central in $\dHallr^+$ and
$\dHallr^-$, respectively, and satisfy
\begin{itemize}
\item[(1)] $\Dt(\sfz_m^\pm)=\sfz_m^\pm\otimes 1+1\otimes \sfz_m^\pm,$
\item[(2)] $\sg(\sfz^\pm_m)=-\sfz_m^\pm$.
\end{itemize}
Moreover, for all $i\in I$ and $m,m'\geq 1$, we have
$$[\sfz_m^+,\sfz_{m'}^-]=0,\; [\sfz_m^+,u_i^-]=0,\;[u_i^+,\sfz_{m'}^-]=0,\;[\sfz_m^\pm,K_i]=0.$$
 In other words, $\sfz_m^\pm$ are central elements in $\dHallr$.
\end{Lem}

\begin{proof} Since $\dHallr$ is a Hopf algebra, we have
$\mu(\sg\otimes 1)\Dt(\sfz_m^\pm)=\eta\ep(\sfz_m^\pm)=0$ for each
$m\geq 1$. Thus,
$$0=\mu(\sg\otimes 1)(\sfz_m^\pm\otimes 1+1\otimes
\sfz_m^\pm)=\sg(\sfz_m^\pm)+\sfz_m^\pm,$$
 i.e., $\sg(\sfz_m^\pm)=-\sfz_m^\pm$. Since
$\Dt(\sfz_m^+)=\sfz_m^+\otimes 1+1\otimes \sfz_m^+$ and
$\Dt(\sfz_{m'}^-)=\sfz_{m'}^-\otimes 1+1\otimes \sfz_{m'}^-$,
 applying \eqref{D-double-rel} to $\sfz_m^+$ and $\sfz_{m'}^-$ gives
$$\aligned
&\psi(\sfz_m^+,\sfz_{m'}^-)+\sfz_m^+\psi(1,\sfz_{m'}^-)+\sfz_{m'}^-\psi(\sfz_m^+,1)
+\sfz_{m'}^-\sfz_m^+\psi(1,1)\\
=&\sfz_m^+\sfz_{m'}^-\psi(1,1)+\sfz_{m'}^-\psi(\sfz_m^+,1)+\sfz_m^+\psi(1,\sfz_{m'}^-)+\psi(\sfz_m^+,\sfz_{m'}^-).
\endaligned$$
 It follows from $\psi(1,1)=1$ that
 $\sfz_{m'}^-\sfz_m^+=\sfz_m^+\sfz_{m'}^-$, i.e., $[\sfz_m^+,\sfz_{m'}^-]=0$.

Similarly, we have
$[\sfz_m^+,u_i^-]=[u_i^+,\sfz_{m'}^-]=[\sfz_m^\pm,K_i]=0$. The last
assertion follows from \eqref{SH1}.
\end{proof}

By Theorem \ref{CompAlg-affine sl_n}, there are isomorphisms
$$\bfcomp^+\stackrel{\sim}{\lra} \bfU(\afsl)^+, u_i^+\lm E_i;\;\;
\bfcomp^-\stackrel{\sim}{\lra} \bfU(\afsl)^-, u_i^-\lm F_i.$$
 Here we have applied the anti-involution $\bfU(\afsl)^+\ra
 \bfU(\afsl)^-,\,E_i\mapsto F_i$.

By \eqref{SH1} (see also \cite[Th.~1]{Hub2}),
there are decompositions
\begin{equation}\label{Zvtgpm}
\dHallr^+= \bfcomp^+ \otimes_{\mbq(\up)}{\bf Z}_\vtg(n)^+\quad\text{and}\quad
\dHallr^-= \bfcomp^-
\otimes_{\mbq(\up)}{\bf Z}_\vtg(n)^-,
\end{equation}
 where ${\bf Z}_\vtg(n)^\pm:={\mbq}(v)[\sfz^\pm_1,\sfz^\pm_2,\ldots]$  is the polynomial algebra in
$\sfz_m^\pm$  for $m\geq 1$.

\begin{Rem} \label{NonrootOfUnity1}
If $z\in\mbc$ is not a root of unity and
$\Hall_{\mbc}:=\Hall\otimes\mbc$ is the $\mbc$-algebra obtained by
specializing $\up$ to $z$, then, by \eqref{poly-auto-group}, we may use \eqref{recursive formula}
 to recursively define the central elements $\sfz_m$ in $\Hall_{\mbc}$, and hence, elements $\sfz_m^\pm$
 in $\DC(n)^\pm$; see Remark \ref{NonrootOfUnity2}.
Thus,  a $\mbc$-basis similar to the one given in Corollary \ref{monomial basis with central
elements} can be constructed for $\DC(n)^\pm$. In particular,  we obtain decompositions
$$\DC(n)^\pm=\comp^\pm_\mbc\otimes\mbc[\sfz^\pm_1,\sfz^\pm_2,\ldots],$$
and hence, the $\mbc$-algebra $\DC(n)$ can be presented by
generators $E_i,\ F_i,\ K_i,\ K_i^{-1},$ $\sfz^+_s,\
\sfz^-_s,\,\;i\in I,\ s\in\mbz^+$ and relations {\rm (QGL1)--(QGL8)}
as given in the last statement of Theorem \ref{presentation
dHallAlg} below.\footnote{We will see in Corollary
\ref{NonrootOfUnity3} that $\DC(n)$ can be obtained as a
specialization from the $\sZ$-algebra $\fD_\vtg(n)$ by base change
$\sZ\to\mbc,\up\mapsto z$. Thus, we can use the notation
$\fD_\vtg(n)_\mbc$ for this algebra.}
\end{Rem}

Let $C=(c_{i,j})$ be the Cartan matrix of affine type $A_{n-1}$
defined in \eqref{CMforAn-1}. By identifying $\bfcomp^\pm$ with
$\bfU(\afsl)^\pm$, we can describe a presentation for $\dHallr$ as
follows.

\begin{Thm}\label{presentation dHallAlg} The double Ringel--Hall algebra $\dHallr$ of
the cyclic quiver $\tri$ is the $\mbq(\up)$-algebra generated by
$E_i,\ F_i,\  K_i,\ K_i^{-1},\ \sfz^+_s,\ \sfz^-_s,\,\;i\in I,\
s\in\mbz^+$ with relations ($i,j\in I$ and $s,t\in \mbz^+$):
\begin{itemize}
\item[(QGL1)] $K_{i}K_{j}=K_{j}K_{i},\ K_{i}K_{i}^{-1}=1$;

\item[(QGL2)] $K_{i}E_j=v^{\dt_{i,j}-\dt_{i,j+1}}E_jK_{i}$,
$K_{i}F_j=v^{-\dt_{i, j}+\dt_{ i,j+1}} F_jK_i$;

\item[(QGL3)] $E_iF_j-F_jE_i=\delta_{i,j}\frac
       {K_{i}K_{i+1}^{-1}-K_{i}^{-1}K_{i+1}}{\up-\up^{-1}}$;

\item[(QGL4)]
$\displaystyle\sum_{a+b=1-c_{i,j}}(-1)^a\leb{1-c_{i,j}\atop a}\rib
E_i^{a}E_jE_i^{b}=0$ for $i\not=j$;

\item[(QGL5)]
$\displaystyle\sum_{a+b=1-c_{i,j}}(-1)^a\leb{1-c_{i,j}\atop a}\rib
F_i^{a}F_jF_i^{b}=0$ for $i\not=j$;

\item[(QGL6)] $\sfz^+_s\sfz^+_t=\sfz^+_t\sfz^+_s$,$\sfz^-_s\sfz^-_t=\sfz^-_t\sfz^-_s$,  $\sfz^+_s\sfz^-_t=\sfz^-_t\sfz^+_s$;
 \item[(QGL7)] $K_i\sfz^+_s=\sfz^+_s K_i$, $K_i\sfz^-_s=\sfz^-_s K_i$;

 \item[(QGL8)] $E_i\sfz^+_s=\sfz^+_s E_i$, $E_i\sfz^-_s=\sfz^-_s E_i,$ $F_i\sfz^-_s=\sfz^-_s F_i$, and  $\sfz^+_s F_i=F_i\sfz^+_s$.
 \index{defining relations!({\rm QGL1})--({\rm QGL8})}

\end{itemize}

Replacing $\mbq(\up)$ by $\mbc$ and $\up$ by a non-root of unity $z$, a similar result holds for $\DC(n)$.
\end{Thm}

We will prove this result by a ``standardly" defined quantum
enveloping algebra associated with a Borcherds--Cartan matrix.
See Proposition \ref{QEAGKMAlg-DHall} below for the proof.

For each $m\in\mbz^+$, let $J_m=[1,m]$ and set
$J_\infty=[1,\infty)=\mbz^+$. We then set $\ti I_m=I\cup J_m$ for
$m\in\mbz^+\cup\{\infty\}$. Define a matrix $\ti C_m=(\ti
c_{i,j})_{i,j \in\ti I_m}$ by setting
$$\ti c_{i,j}
=\begin{cases} c_{i,j},\;\;&\text{if $i,j\in I$;}\\
0,&\text{otherwise.}\end{cases}$$
 The $\ti C_m$ is a Borcherds--Cartan matrix \index{$\ti C_m$, $\ti C_\infty$, Borcherds--Cartan matrices}
 \index{Cartan matrix!Borcherds--$\sim$, $\ti C_m$, $\ti C_\infty$} which defines a
generalized Kac--Moody algebra; see \cite{Bor}. Thus, with $\ti C_m$
we can associate a quantum enveloping algebra $\bfU(\ti C_m)$ as
follows; see \cite{Kang}.

\begin{Def}\label{quantumGKMAlg} Let $m\in\mbz^+\cup\{\infty\}$ and $\ti C_m$ be
defined as above. The quantum enveloping algebra  $\bfU(\ti C_m)$
associated with $\ti C_m$ is the $\mbq(\up)$-algebra presented by
generators
$$E_i,\ F_i,\  K_i,\ K_i^{-1},\ \sfx_s,\ \sfy_s,\ \sfk_s,\ \sfk_s^{-1},\,\;i\in I,\ s\in J_m,$$
and relations: for $i,j\in I$ and $s,t\in J_m$,
\begin{itemize}
\item[(R1)] $K_{i}K_{j}=K_{j}K_{i},\ K_{i}K_{i}^{-1}=1,\
\sfk_s\sfk_t=\sfk_t\sfk_s,\, \sfk_s\sfk_s^{-1}=1,\
K_i\sfk_s=\sfk_s K_i$;

\item[(R2)] $K_{i}E_j=v^{\dt_{i,j}-\dt_{i,j+1}}E_jK_{i},\
K_i\sfx_s=\sfx_s K_i,\ \sfk_s\sfx_t=\sfx_t\sfk_s,\ \sfk_s
E_i=E_i\sfk_s$;

\item[(R3)] $K_{i}F_j=v^{-\dt_{i, j}+\dt_{ i,j+1}} F_jK_i, \
K_i\sfy_s=\sfy_s K_i,\ \sfk_s\sfy_t=\sfy_t\sfk_s,\ \sfk_s
F_i=F_i\sfk_s$;

\item[(R4)] $E_iF_j-F_jE_i=\delta_{i,j}\frac
       {\ti K_{i}-\ti K_{i}^{-1}}{\up-\up^{-1}}$, $\sfx_s\sfy_t-\sfy_t\sfx_s
       =\delta_{s,t}\frac{\sfk_s-\sfk_s^{-1}}{\up-\up^{-1}},\
       E_i\sfy_s=\sfy_s E_i,\ \sfx_s F_i=F_i\sfx_s$, where $\ti K_i =K_{i}K_{i+1}^{-1}$;\index{$\ti K_\al$,
$\ti K_i$}

\item[(R5)]
$\displaystyle\sum_{a+b=1-c_{i,j}}(-1)^a\leb{1-c_{i,j}\atop a}\rib
E_i^{a}E_jE_i^{b}=0$ for $i\not=j$;

\item[(R6)]
$\displaystyle\sum_{a+b=1-c_{i,j}}(-1)^a\leb{1-c_{i,j}\atop a}\rib
F_i^{a}F_jF_i^{b}=0$ for $i\not=j$;

 \item[(R7)] $E_i\sfx_s=\sfx_s E_i$ and $\sfx_s\sfx_t=\sfx_t\sfx_s$;

 \item[(R8)] $F_i\sfy_s=\sfy_s F_i$ and $\sfy_s\sfy_t=\sfy_t\sfy_s$. \index{defining relations!({\rm R1})--({\rm R8})}

\end{itemize}
\end{Def}

Moreover, $\bfU(\ti C_m)$ is a Hopf algebra with
comultiplication $\Dt$, counit $\ep$, and antipode $\sg$ defined
by
\begin{eqnarray*}
&\Delta(E_i)=E_i\otimes\ti K_i+1\otimes
E_i,\quad\Delta(F_i)=F_i\otimes
1+\ti K_i^{-1}\otimes F_i,&\\
&\Delta(\sfx_s)=\sfx_s\otimes\sfk_s+1\otimes
\sfx_s,\quad\Delta(\sfy_s)=\sfy_s\otimes
1+\sfk_s^{-1}\otimes \sfy_s,&\\
&\Delta(K^{\pm 1}_i)=K^{\pm 1}_i\otimes K^{\pm 1}_i,\quad
\Delta(\sfk^{\pm 1}_s)=\sfk^{\pm 1}_s\otimes \sfk^{\pm 1}_s;&\\
&\varepsilon(E_i)=\varepsilon(\sfx_s)=0=\varepsilon(F_i)=\varepsilon(\sfy_s),
\quad \varepsilon(K_i)=\varepsilon(\sfk_s)=1;&\\
&\sg(E_i)=-E_i\ti K_i^{-1},\quad \sg(F_i)=-\ti K_iF_i,\quad
\sg(K^{\pm 1}_i)=K^{\mp 1}_i,&\\
&\sg(\sfx_s)=-\sfx_s \sfk_s^{-1},\quad \sg(\sfy_s)=-\sfk_s
\sfy_s,\quad \sg(\sfk_s^{\pm 1})=\sfk_s^{\mp 1},&
\end{eqnarray*}
 where $i\in I$ and $s\in J_m$.

\begin{Rem}The comultiplication here is opposite to the one given in
\cite[3.1.10]{Lu93}. Consequently, the antipode is the inverse of the antipode
given in \cite[3.3.1]{Lu93}. This change is necessary for its action
on tensor space, commuting with the Hecke algebra action; cf. \cite[\S14.6]{DDPW}.
\end{Rem}

Let $\ti \bfU^+=\bfU(\ti C_m)^+$ (resp., $\ti \bfU^-=\bfU(\ti
C_m)^-$, $\ti \bfU^0=\bfU(\ti C_m)^0$) be the $\mbq(\up)$-subalgebra
of $\bfU(\ti C_m)$ generated by $E_i$ (resp., $F_i$, $K_i^{\pm1}$)
for all $i\in\ti I$ and $\sfx_s$ (resp. $\sfy_s, \sfk_s$) for all
$s\in J_m$. Then by \cite[Th.~2.23]{Kang}, $\bfU(\ti C_m)$ admits a
triangular decomposition
$$\bfU(\ti C_m)=\ti\bfU^+\otimes \ti\bfU^0\otimes \ti\bfU^-.$$

 It is clear that the subalgebra of $\bfU(\ti C_m)$ generated by
$E_i,F_i,\ti K_i,\ti K_i^{-1}$ ($i\in I$) is the quantum
enveloping algebra $\bfU(\afsl)$ as defined in \S1.3.  Also, we
denote by $\bfU_\vtg(n)$ the subalgebra of $\bfU(\ti C_m)$
generated by $E_i,F_i, K_i, K_i^{-1}$ ($i\in I$). We will call
$\bfU_\vtg(n)$ the {\it extended quantum affine}
$\frak{sl}_n$.\footnote{This is called quantum affine
$\mathfrak{gl}_n$ by Lusztig \cite{Lu99}, which is an extension of
quantum affine $\frak{sl}_n$ by adding an extra generator to the
0-part. This extension is similar to the extension of quantum
$\frak{gl}_n$ from quantum $\frak{sl}_n$. In the literature,
quantum affine $\frak{sl}_n$ refers also to the quantum loop
algebra associated with $\frak{sl}_n$; see \S2.3 below.}
\index{quantum affine $\frak{sl}_n$!extended $\sim$,
$\bfU_\vtg(n)$}\index{extended quantum affine $\frak{sl}_n$}
\index{$\bfU_\vtg(n)$, extended quantum affine $\frak{sl}_n$}

Now Theorem \ref{presentation dHallAlg} follows immediately from the following result which
describes the structure of $\dHallr$. It is a
generalization of a result for the double Ringel--Hall algebra of
a finite dimension tame hereditary algebra in \cite{HX} to the
cyclic quiver case. Its proof is based on Theorem
\ref{Schiffmann-Hubery}.

\begin{Prop} \label{QEAGKMAlg-DHall}
There is a unique surjective Hopf algebra homomorphism
$\Phi:\bfU(\ti C_\infty)\ra\dHallr$ satisfying for each $i\in I$
and $s\in J_\infty$,
$$E_i\lm u_i^+,\ \sfx_s\lm \sfz_s^+,\ F_i\lm u_i^-,\ \sfy_s\lm
\sfz_s^-,\ K^{\pm1}_i\lm K_i^\pm,\ \sfk_s^\pm\lm 1.$$
 Moreover, ${\rm Ker}\,\Phi$ is the ideal of $\bfU(\ti C_\infty)$ generated
by $\sfk_s^{\pm1}-1$ for $s\in J_\infty$.
\end{Prop}

\begin{proof} The existence and uniqueness of $\Phi$ clearly follow from
Definition \ref{quantumGKMAlg}, Theorem \ref{CompAlg-affine sl_n},
Theorem \ref{Schiffmann-Hubery} and Lemma \ref{central-dHall}.
Further, $\Phi$ induces surjective algebra homomorphisms
$$\Phi^\pm:\ti\bfU^\pm=\bfU(\ti C_\infty)^\pm\lra \dHallr^\pm.$$
 By (R7), the multiplication map
$$\mu: \bfU(\afsl)^+\ot \mbq(\up)[\sfx_1,\sfx_2,\ldots]\lra \ti\bfU^+$$
 is surjective. Hence, the composite $\Phi^+\mu$ is surjective.
By Theorem \ref{Schiffmann-Hubery},
$$\dHallr^+=\bfcomp^+ \otimes_{\mbq(\up)}{\mbq}(v)[\sfz_1^+,\sfz_2^+,\ldots].$$
 Since $\Phi^+\mu(\bfU(\afsl)^+)=\bfcomp^+\cong
 \bfU(\afsl)^+$, it follows that $\Phi^+\mu$ is an
isomorphism. Hence, $\Phi^+$ is an isomorphism. Similarly, $\Phi^-$
is an isomorphism, too.

Let $\mathcal K$ be the ideal of $\bfU(\ti C_\infty)$ generated by
$\sfk_s^{\pm1}-1$ for $s\in J_\infty$. It is clear that ${\mathcal
K}\subseteq {\rm Ker}\,\Phi$. The triangular decomposition $\bfU(\ti
C_\infty)=\ti\bfU^+\otimes \ti\bfU^0\otimes\ti\bfU^-$ implies that
$$\bfU(\ti C_\infty)/{\mathcal K}=\bigl(\ti\bfU^+\otimes \mbq(\up)[K_1^{\pm1},\ldots,K_n^{\pm1}]
\otimes \ti\bfU^-+{\mathcal K}\bigr)/{\mathcal K}.$$
 Since
$$\dHallr=\dHallr^+\otimes\dHallr^0\otimes\dHallr^-=\dHallr^+\otimes
\mbq(\up)[K_1^{\pm1},\ldots,K_n^{\pm1}]\otimes\dHallr^-,$$
 we conclude that ${\rm Ker}\,\Phi=\mathcal K$, as required.

Finally, it is straightforward to check that $\mathcal K$ is a Hopf ideal.
\end{proof}

\begin{Coro}\label{coalgebra structure} The double Ringel--Hall algebra $\dHallr$ is a Hopf algebra with
comultiplication $\Dt$, counit $\ep$, and antipode $\sg$ defined
by
\begin{eqnarray*}
&\Delta(E_i)=E_i\otimes\ti K_i+1\otimes
E_i,\quad\Delta(F_i)=F_i\otimes
1+\ti K_i^{-1}\otimes F_i,&\\
&\Delta(K^{\pm 1}_i)=K^{\pm 1}_i\otimes K^{\pm 1}_i,\quad
\Delta(\sfz_s^\pm)=\sfz_s^\pm\otimes1+1\otimes
\sfz_s^\pm;&\\
&\varepsilon(E_i)=\varepsilon(F_i)=0=\varepsilon(\sfz_s^\pm),
\quad \varepsilon(K_i)=1;&\\
&\sg(E_i)=-E_i\ti K_i^{-1},\quad \sg(F_i)=-\ti K_iF_i,\quad
\sg(K^{\pm 1}_i)=K^{\mp 1}_i,&\\
&\sg(\sfz_s^\pm)=-\sfz_s^\pm,&
\end{eqnarray*}
 where $i\in I$ and $s\in J_\infty$.

\end{Coro}

\begin{Rems} \label{subalgebra quantum gl} (1) For notational simplicity, we sometimes continue to use
$u_i^\pm$ as generators of $\dHallr$. By Theorem \ref{presentation dHallAlg}, we see that there is a $\mbq(\up)$-algebra involution
$\varsigma$ of $\boldsymbol{\fD}_\vtg(n)$ satisfying
$$K_i^{\pm1}\lm K_i^{\mp1},\,u_i^\pm\lm u_i^\mp,\,\sfz_s^\pm\lm
\sfz_s^\mp,$$
 for $i\in I$ and $s\geq 1$. Thus, $\boldsymbol{\fD}_\vtg(n)$ also admits a
 decomposition
$$\boldsymbol{\fD}_\vtg(n)=\boldsymbol{\fD}_\vtg(n)^-\ot\boldsymbol{\fD}_\vtg(n)^0\ot\boldsymbol{\fD}_\vtg(n)^+.$$
Moreover, two other types of generators for Ringel--Hall algebras
discussed in \S1.4 result in the corresponding generators for double
Ringel--Hall algebras. Thus, we may speak of semisimple generators,
etc. See more discussion in \S2.4. \index{semisimple
generators!$\sim$ for $\dHallr$}

(2) The subalgebra of $\dHallr$ generated by $E_i,F_i,
K_i,K_i^{-1}$ ($i\in I$) is isomorphic to $\bfU_\vtg(n)$. We will
always identify these two algebras and thus, view $\bfU_\vtg(n)$
as a subalgebra of $\dHallr$. In particular, we have
$$\bfU_\vtg(n)=\boldsymbol{\fC}_\vtg(n)^-\ot\boldsymbol{\fD}_\vtg(n)^0\ot\boldsymbol{\fC}_\vtg(n)^+.$$
Moreover, if ${\bf Z}_\vtg(n)={\mbq}(v)[\sfz^+_m,\sfz^-_m]_{m\ge1}$
denotes the {\it central subalgebra} of $\dHallr$ generated by
$\ldots,\sfz_2^-,\sfz_1^-,\sfz_1^+,\sfz_2^+,\ldots$,\index{central
subalgebra of $\dHallr$}\index{${\bf Z}_\vtg(n)$, central
subalgebra} then (1) together with \eqref{Zvtgpm} gives
$$\dHallr\cong \bfU_\vtg(n)\ot {\bf Z}_\vtg(n).$$

(3) The subalgebra $'\dHallr$\index{double Ringel--Hall algebra!$'\dHallr$, $\sim$ with reduced 0-part}
\index{$^\prime\dHallr$, $\dHallr$ with reduced 0-part}
 of $\dHallr$ generated by $u_A^+,u_A^-,\ti K_i,\ti K_i^{-1}$ ($A\in\afThnp,i\in I$) is isomorphic to
the reduced Drinfeld double of $'\!\Hallpi$ and $'\!\Hallmi$; see Remark \ref{remark for Hallpi}(2).
This subalgebra will be considered in \S5.5 regarding a compatibility condition on the transfer maps.

(4) For each $m\geq 1$, $\bfU(\ti C_m)$ can be naturally viewed
as a subalgebra of $\bfU(\ti C_\infty)$. If we let
$\dHallr^{(m)}$ be the subalgebra of $\dHallr$ generated by
$E_i,F_i,K_i,K_i^{-1}, \sfz^+_s,\sfz^-_s$ for $i\in I$ and $s\in
J_m$, then the homomorphism $\Phi$ in Theorem
\ref{QEAGKMAlg-DHall} induces a surjective Hopf algebra
homomorphism $\bfU(\ti C_m)\ra\dHallr^{(m)}$. Applying
Proposition \ref{generators-RH-alg} shows that $\dHallr^{(m)}$ is
also generated by $E_i,F_i,K_i,K_i^{-1}, u^+_{s\dt},u^-_{s\dt}$
for $i\in I$ and $s\in J_m$.
\end{Rems}

We end this section with a brief discussion on the double
Hall algebras $\DC(n)$ given in Remark \ref{NonrootOfUnity2} in terms of specialization.


Consider the $\sZ$-subalgebra $\ti U=\ti U_\sZ$ of $\bfU(\ti
C_\infty)$ generated by $K_i^{\pm1}$, $\big[{K_i;0\atop t}\big]$,
$\sfk_s^{\pm 1}$, $\sfx_s$, $\sfy_s$ together with the divided
powers $E^{(m)}_i=E_i^m/[m]^!$ and $F^{(m)}_i=F_i^m/[m]^!$ for $i\in
I$ and $s,m\geq 1$; see for example \cite{KS,JKK}. Further, we set
$$\ti U^+=\ti \bfU^+\cap \ti U,\;\ti U^-=\ti\bfU^-\cap \ti U,\; \ti U^0=\ti\bfU^0\cap\ti U.$$
 Then $\ti U^+$ (resp., $\ti U^-$) is the $\sZ$-subalgebra of $\bfU(\ti C_\infty)$
generated by $\sfx_s$ and $E_i^{(m)}$ (resp., $\sfy_s$ and
$F_i^{(m)}$) for $i\in I$ and $s,m\geq 1$. Consequently, we obtain
$$\ti U^+=U(\afsl)^+\ot\sZ[\sfx_1,\sfx_2,\ldots]\;\;\text{and}\;\;
\ti U^-=U(\afsl)^-\ot\sZ[\sfy_1,\sfy_2,\ldots],$$
 where $U(\afsl)^+$ and $U(\afsl)^-$ are the $\sZ$-subalgebras
of $\bfU(\afsl)$ generated by the divided powers $E_i^{(m)}$ and
$F_i^{(m)}$, respectively. This implies particularly that both $\ti
U^+$ and $\ti U^-$ are free $\sZ$-modules.

\begin{Rem}\label{freeness of U0} We do not know whether $\ti U^0$ is free over $\sZ$. In case
$\ti U^0$ is free, it would be interesting to find a $\sZ$-basis for
it.
\end{Rem}

Similarly, we define the {\it integral form} $\fD_\vtg(n)$ of $\dHallr$
to be the $\sZ$-subalgebra generated by $K_i^{\pm1}$,
$\big[{K_i;0\atop t}\big]$, $\sfz_s^+$, $\sfz^-_s$,
$(u_i^+)^{(m)}$ and $(u_i^-)^{(m)}$ for $i\in I$ and $s,t,m\geq
1$. It is easy to see that the surjective $\mbq(\up)$-algebra
homomorphism $\Phi:\bfU(\ti C_\infty)\ra \dHallr$ (see Proposition
\ref{QEAGKMAlg-DHall}) induces a surjective $\sZ$-algebra
homomorphism $\Phi:\ti U\ra\fD_\vtg(n)$. We also set
$$\fD_\vtg(n)^\epsilon=\dHallr^\epsilon\cap\fD_\vtg(n)\;\;\text{for $\epsilon\in\{+,-,0\}$}.$$
\index{$\fD_\vtg(n)$, integral form of $\dHallr$!$\fD_\vtg(n)^\pm$, $\pm$-part of $\fD_\vtg(n)$}
\index{$\fD_\vtg(n)$, integral form of $\dHallr$!$\fD_\vtg(n)^0$, $0$-part of $\fD_\vtg(n)$}
 Then $\fD_\vtg(n)^+$ (resp., $\fD_\vtg(n)^-$) is isomorphic to  $\ti U^+$ (resp.,
 $\ti U^-$). Hence,
$$\fD_\vtg(n)^+=\comp^+\ot_\sZ\sZ[\sfz_1^+,\sfz_2^+,\ldots]\;\;\text{and}\;\;
\fD_\vtg(n)^-=\comp^-\ot_\sZ\sZ[\sfz_1^-,\sfz_2^-,\ldots],$$
 where $\comp^+\cong U(\afsl)^+$ (resp., $\comp^-\cong U(\afsl)^-$)
is the $\sZ$-subalgebra of $\dHallr$ generated by the
$(u_i^+)^{(m)}$ (resp., $(u_i^-)^{(m)}$). Moreover, by
\cite[Th.~14.20]{DDPW}, $\fD_\vtg(n)^0$ is a free $\sZ$-module with
a basis
$$\biggl\{\ti K_1^{a_1}\cdots \ti K_{n-1}^{a_{n-1}}\leb{\ti K_1;0\atop t_1}\rib
\cdots\leb{\ti K_{n-1};0\atop t_{n-1}}\rib K_n^{a_n}\leb{K_n;0\atop
t_n}\rib\,\bigg|\, a_i=0,1, t_i\in\mbn\;\forall i\in I\biggr\}.$$
  Hence, the multiplication map
$$\fD_\vtg(n)^+\ot_\sZ\fD_\vtg(n)^0\ot_\sZ\fD_\vtg(n)^-\overset\sim\lra \fD_\vtg(n)$$
is a $\sZ$-module isomorphism; see \cite[Cor.~6.50]{DDPW}. In
particular, $\fD_\vtg(n)$ is a free $\sZ$-module.

The $\sZ$-algebra $\fD_\vtg(n)$ gives rise to a $\sZ$-form $U_\vtg(n)$ of the extended quantum affine
$\mathfrak{sl}_n$, $\bfU_\vtg(n)$:
\begin{equation}\label{integral form of U(n)}
U_\vtg(n)=\fD_\vtg(n)\cap \bfU_\vtg(n)=\comp^+\fD_\vtg(n)^0\comp^-.
\end{equation}

\begin{Coro}\label{NonrootOfUnity3}
If  $z\in\mbc$ is a complex number which is not a root of 1, then the
specialization $\fD_\vtg(n)_\mbc=\fD_\vtg(n)\ot\mbc$ at $\up=z$ is isomorphic to the
$\mbc$-algebra $\DC(n)$.
\end{Coro}

\begin{proof} If we assign to each
$u_A^+$ (resp., $u_A^-$, $K_i$) degree $\fkd(A)=\dim M(A)$ (resp.,
$-\fkd(A)$, $0$), then $\dHallr$ admits a $\mbz$-grading
$\dHallr=\oplus_{m\in\mbz}\dHallr_m$, and $\fD_\vtg(n)$ inherits
a $\mbz$-grading $\fD_\vtg(n)=
\oplus_{m\in\mbz}\fD_\vtg(n)_{m}$, where
$\fD_\vtg(n)_{m}=\fD_\vtg(n)\cap \dHallr_m.$
By the presentation of $\fD_{\vtg,\mbc}(n)$ as given in Remark \ref{NonrootOfUnity2}, it is clear that
 there is a graded algebra homomorphism
$\fD_{\vtg,\mbc}(n)\to \fD_\vtg(n)_\mbc$. This map is injective
since it is so on every triangular component. Now, a dimensional
comparison on homogeneous components shows it is an isomorphism.
\end{proof}

\begin{Rem}\label{Lform2}  Let $\Hall^+$ be the $\sZ$-submodule of
$\dHallr$ spanned by $u_A^+$ for $A\in\afThnp$.\index{$\ti{\fD}_\vtg(n)$, candidate of Lusztig form!$\Hall^+$,
$+$-part of $\ti{\fD}_\vtg(n)$} 
 Then $\Hall^+$ is a
$\sZ$-algebra which is isomorphic to $\Hall$. We point out that in
general, $\fD_\vtg(n)^+$ does not coincide with $\Hall^+$. For
example, let $n=2$. An easy calculation shows that
$$\sfz_1^+=u_1^+u_2^++u_2^+u_1^+-(\up+\up^{-1})u^+_\dt.$$
Since $\{u_1^+u_2^+,u_2^+u_1^+,\sfz_1^+\}$ is a $\sZ$-basis for
the homogeneous space $\fD_\vtg(2)^+_\dt$, we conclude that
$u^+_\dt\not\in \fD_\vtg(2)^+$. If $n=3$, then
$$\aligned
\sfz_1^+=&u^+_1u^+_2u^+_3+u^+_2u^+_3u^+_1+u_3^+u_1^+u_2^+\\
&-(\up+\up^{-1})(u^+_1u^+_3u^+_2+u^+_2u^+_1u^+_3+u_3^+u_2^+u_1^+)+(\up^2+1+\up^{-2})u^+_\dt.
\endaligned$$
 Analogously, we have $u^+_\dt\not\in\fD_\vtg(3)^+$.

It seems difficult to prove directly that all central elements
$\sfz_m^+$ lie in $\Hall^+$. However, we will provide a proof of
this fact in Corollary \ref{central-elem-integral} via the integral
quantum Schur algebras. Hence, $\fD_\vtg(n)^+$ is indeed a
subalgebra of $\Hall^+$.
\end{Rem}

\section{The quantum loop algebra $\bfU(\afgl)$}

In this section, we give an application of the presentation for
$\dHallr$ given in the previous section. More precisely, we will use
Hubery's work \cite{Hub2} to extend Beck's algebra embedding of the
quantum affine $\mathfrak{sl}_n$ into the quantum loop algebra
$\bfU(\afgl)$ defined by Drinfeld's new presentation \cite{Dr88} to
an explicit isomorphism from $\dHallr$ to $\bfU(\afgl)$.\index{loop
algebra of $\frak{gl}_n$}

As discussed in \S1.1, consider the loop algebra
$$\afgl={\frak{gl}_n}\ot\mbq[t,t^{-1}],$$
which is generated by $E_{i,i+1}\otimes t^s$,
$E_{i+1,i}\otimes t^s$, $E_{j,j}\otimes t^s$ for $s\in\mbz$, $1\leq
i\leq n-1$ and $1\leq j\leq n$. We have the following quantum
enveloping algebra associated with $\afgl$; see \cite{Dr88} (and
also \cite{FM,Hub2}).

\begin{Def}\label{QLA}
(1) The {\it quantum loop algebra} $\bfU(\afgl)$ \index{quantum loop
algebra!$\sim$ of $\mathfrak {gl}_n$, $\bfU(\afgl)$} (or {\it
quantum affine $\mathfrak {gl}_n$})\index{quantum affine $\mathfrak
{gl}_n$} is the $\mbq(\up)$-algebra generated by $\ttx^\pm_{i,s}$
($1\leq i<n$, $s\in\mbz$), $\ttk_i^{\pm1}$ and $\ttg_{i,t}$ ($1\leq
i\leq n$, $t\in\mbz\backslash\{0\}$) with the following relations:
\begin{itemize}
 \item[(QLA1)] $\ttk_i\ttk_i^{-1}=1=\ttk_i^{-1}\ttk_i,\,\;[\ttk_i,\ttk_j]=0$;
 \item[(QLA2)]
 $\ttk_i\ttx^\pm_{j,s}=v^{\pm(\dt_{i,j}-\dt_{i,j+1})}\ttx^\pm_{j,s}\ttk_i,\;
               [\ttk_i,\ttg_{j,s}]=0$;
 \item[(QLA3)] $[\ttg_{i,s},\ttx^\pm_{j,t}]
               =\begin{cases}0\;\;&\text{if $i\not=j,\,j+1$},\\
                  \pm v^{-is}\frac{[s]}{s}\ttx^\pm_{i,s+t},\;\;\;&\text{if $i=j$},\\
                  \mp v^{(1-i)s}\frac{[s]}{s}\ttx_{i-1,s+t}^\pm,\;\;\;&\text{if $i=j+1$;}
                \end{cases}$
 \item[(QLA4)] $[\ttg_{i,s},\ttg_{j,t}]=0$;
 \item[(QLA5)] $[\ttx_{i,s}^+,\ttx_{j,t}^-]=\dt_{i,j}\frac{\phi^+_{i,s+t}-\phi^-_{i,s+t}}{v-v^{-1}}$;
 \item[(QLA6)] $\ttx^\pm_{i,s}\ttx^\pm_{j,t}=\ttx^\pm_{j,t}\ttx^\pm_{i,s}$, for $|i-j|>1$, and
 $[\ttx_{i,s+1}^\pm,\ttx^\pm_{j,t}]_{v^{\pm c_{ij}}}
               =-[\ttx_{j,t+1}^\pm,\ttx^\pm_{i,s}]_{v^{\pm c_{ij}}}$;
 \item[(QLA7)] $[\ttx_{i,s}^\pm,[\ttx^\pm_{j,t},\ttx^\pm_{i,p}]_v]_v
               =-[\ttx_{i,p}^\pm,[\ttx^\pm_{j,t},\ttx^\pm_{i,s}]_v]_v\;$ for
               $|i-j|=1$, \index{defining relations!({\rm QLA1})--({\rm QLA7})}
\end{itemize}
 where $[x,y]_a=xy-ayx$, and $\phi_{i,s}^\pm$ are defined via the
 generating functions in indeterminate $u$ by
$$\Phi_i^\pm(u):={\ti\ttk}_i^{\pm 1}
\exp(\pm(v-v^{-1})\sum_{m\geq 1}\tth_{i,\pm m}u^{\pm m})=\sum_{s\geq
0}\phi_{i,\pm s}^\pm u^{\pm s}$$ with $\ti\ttk_i=\ttk_i/\ttk_{i+1}$
($\ttk_{n+1}=\ttk_1$) and
$\tth_{i,m}=v^{(i-1)m}\ttg_{i,m}-v^{(i+1)m}\ttg_{i+1,m}\,(1\leq
i<n).$

(2) Let $\afUglC$ be the quantum loop algebra defined by the same
generators and relations (QLA1)--(QLA7) with $\mbq(\up)$ replaced by
$\mbc$ and $\up$ by a non-root of unity $z\in \mbc$.
\end{Def}

 Observe that, if we put $\Theta_i^\pm(u)={\ttk}_i^{\pm 1}\exp(\pm(v-v^{-1})\sum_{m\geq 1}\ttg_{i,\pm m}u^{\pm m})$, then
$$\Phi_i^\pm(u)=\frac{\Theta_i^\pm(u\up^{i-1})}{\Theta_{i+1}^\pm(u\up^{i+1})}.$$
Also, $\phi_{i,-s}^+=\phi_{i,s}^-=0$ for all $s\geq1$. Hence, (QLA5) becomes
$$[\ttx_{i,s}^+,\ttx_{i,t}^-]=\begin{cases}\frac{\ti\ttk_i-\ti\ttk^-_{i}}{v-v^{-1}},&\text{ if }s+t=0;\\
\frac{\phi^+_{i,s+t}}{v-v^{-1}},&\text{ if } s+t>0;\\
\frac{-\phi^-_{i,s+t}}{v-v^{-1}},&\text{ if }s+t<0.\\
\end{cases}$$

By \cite{Be} (see also \cite[Th.~2.2]{Jing}, \cite[\S3.1]{FM}, and
\cite[Th.~3]{Hub2}), there is a monomorphism of Hopf algebras
$\sE_{\rm B}:\bfU(\afsl)\ra \bfU(\afgl)$ such that
$$\ti K_i\lm \ti\ttk_i,\,E_i\lm \ttx_{i,0}^+,\,F_i\lm \ttx_{i,0}^-\;(1\leq i<n),\,
E_n\lm\ep_n^+,\,F_n\lm\ep_n^-,$$
 where
 $$\aligned
 &\ep_n^+=(-1)^n[\ttx_{n-1,0}^-,\cdots,[\ttx_{3,0}^-,[\ttx_{2,0}^-,\ttx_{1,1}^-]_{v^{-1}}
 ]_{v^{-1}}\cdots]_{v^{-1}}\ti\ttk_n\;\;\text{and}\\
 &\ep_n^-=(-1)^n\ti\ttk_n^{-1}[\cdots[[\ttx_{1,-1}^+,\ttx_{2,0}^+]_v,\ttx_{3,0}^+]_v,
 \cdots,\ttx_{n-1,0}^+]_v.
\endaligned$$
The image $\text{Im}(\sE_{\rm B})$, known as quantum loop algebra of
$\mathfrak{sl}_n$,\index{quantum loop algebra!$\sim$ of $\mathfrak
{sl}_n$, $\bfU(\afsl)$} is the $\mbq(\up)$-subalgebra of
$\bfU(\afgl)$ generated by $\ti\ttk_i^{\pm1}$, $\ttx_{i,s}^\pm$, and
$\tth_{i,t}$ for $1\leq i< n$, $s, t\in\mbz$ with $t\not=0$.
Moreover, it can be proved (see, e.g., \cite[\S\S2.1\&3.1]{FM}) that
the defining relations of this subalgebra are the relations
(QLA5)--(QLA7) together with the relations:

\vspace{.3cm}
\begin{itemize}
\item[(GLA8)] $\ti\ttk_i\ti\ttk_j=\ti\ttk_j\ti\ttk_i$, $\ti\ttk_1\ti\ttk_2\cdots\ti\ttk_n=1$, $[\ti\ttk_i,\tth_{j,s}]=0,\;[\tth_{i,s},\tth_{j,t}]=0$;
\item[(GLA9)] $\ti\ttk_i\ttx_{j,s}^\pm=\up^{\pm c_{i,j}}\ttx_{j,s}^\pm\ti\ttk_i$,
$[\tth_{i,s},\ttx_{j,t}^\pm]=\pm\frac{[s\,
c_{i,j}]}{s}\ttx_{j,s+t}^\pm$.
\end{itemize}
 where $C=(c_{i,j})$ is the generalized Cartan matrix of type $\ti
 A_{n-1}$ as given in \eqref{CMforAn-1}. This is the Drinfeld's new presentation for $\bfU(\afsl)$.
Define $\afUslC$ similarly.

\vspace{.3cm}
We now describe the main result of \cite{Hub2}.
Set for each $s\in\mbz\backslash\{0\}$,
\begin{equation}\label{defn of theta_s}
\th_s=v^{ns}(\ttg_{1,s}+\cdots+\ttg_{n,s}).
\end{equation}
 It can be directly checked by the definition that the $\th_s$ are central in
$\bfU(\afgl)$.

\begin{Rem}\label{FM-central-elements} In \cite[\S4.2]{FM}, the
authors introduced central elements $\ttg_s$ ($s\in\mbz$) in
$\bfU(\afgl)$ defined by the generating function
$$\sum_{s\geq 0}\ttg_{\pm s}u^{\pm s}=\prod_{i=1}^n{\rm exp}\big(\mp\sum_{m>0}
\frac{\ttg_{i,\pm m}}{[m]}u^{\pm m}\big),$$
 that is,
$$\sum_{s\geq 0}\ttg_{\pm s}u^{\pm s}={\rm exp}\big(\mp\sum_{m>0}
\frac{\up^{\mp nm}}{[m]}\th_{\pm m}u^{\pm m}\big).$$
 Therefore, for each $s\geq 1$,
$$\aligned
\,& \mbq(\up)[\th_{\pm 1},\ldots,\th_{\pm s}]=\mbq(\up)[\ttg_{\pm
1},\ldots,\ttg_{\pm s}]\;\;\text{and}\\
& \th_{\pm (s+1)}\equiv \mp \up^{\pm n(s+1)}[s+1]\ttg_{\pm
(s+1)}\;\;{\rm mod}\;\mbq(\up)[\ttg_{\pm 1},\ldots,\ttg_{\pm s}].
\endaligned$$
\end{Rem}

By $\bfU(\afgl)^{\geq 0}$ we denote the subalgebra of $\bfU(\afgl)$
generated by $\ttk_i^{\pm1}$, $\ttg_{i,s}$ ($1\leq i\leq n,s\geq
1$), $\ttx^+_{i,0}$ and $\ttx^\pm_{i,s}$ ($1\leq i<n, s\geq 1$).
Dually, let $\bfU(\afgl)^{\leq 0}$ be the subalgebra of
$\bfU(\afgl)$ generated by $\ttk_i^{\pm1}$, $\ttg_{i,s}$ ($1\leq
i\leq n,s\leq -1$), $\ttx^-_{i,0}$ and $\ttx^\pm_{i,s}$ ($1\leq i<n,
s\leq -1$).

\begin{Prop}[\rm Main Theorem of \cite{Hub2}] \label{Hubery Theorm}
There is an isomorphism of Hopf algebras $\sE_{\rm H}^{\geq
0}:\dHallr^{\geq 0}\ra \bfU(\afgl)^{\geq 0}$ taking
$$K_i^{\pm1}\lm\ttk_i^{\pm1},\;u_i^+\lm \ttx^+_{i,0}\;(1\leq
i<n),\;u_n^+\lm\ep_n^+,\,\sfz^+_s\lm -\frac{s}{[s]}\th_s\;(s\geq
1).$$
 Moreover, the restriction of $\sE_{\rm H}^{\geq 0}$ and the Beck's
embedding $\sE_{\rm B}$ to $\bfcomp^+=\bfU(\afsl)^+$ coincide.
\end{Prop}

Dually, there is an isomorphism of Hopf algebras $\sE_{\rm H}^{\leq
0}:\dHallr^{\leq 0}\lra \bfU(\afgl)^{\leq 0}$ taking
$$K_i^{\pm1}\lm\ttk_i^{\pm1},\;u_i^-\lm \ttx^-_{i,0}\;(1\leq
i<n),\;u_n^-\lm\ep_n^-,\,\sfz^-_s\lm
-\frac{s}{[s]}\th_{-s}\;(s\geq 1),$$
 and the restrictions of $\sE_{\rm H}^{\leq 0}$  to $\bfcomp^-=\bfU(\afsl)^-$ is Beck's
embedding $\sE_{\rm B}$.

In view of the isomorphisms $\sE_{\rm H}^{\geq 0}$ and $\sE_{\rm H}^{\leq 0}$,
we see that $\bfU(\afgl)$ is generated by $\ttx^\pm_{i,0}$ ($1\leq
i<n$), $\ep_n^\pm$, $-\frac{s}{[s]}\th_{\pm s}$ ($s\geq 1$), and
$\ttk_i^{\pm 1}$ ($1\leq i\leq n$). It is easy, using Beck's embedding, to check that these
generators satisfy the relations similar to (QGL1)--(QGL8) in
Theorem \ref{presentation dHallAlg}. Consequently, we obtain the
following result.

\begin{Coro}\label{epi-doubleRH-QGL} There is a surjective $\mbq(\up)$-algebra
homomorphism $\sE_{\rm H}:\dHallr\ra\bfU(\afgl)$ taking
$$K_i^{\pm1}\lm\ttk_i^{\pm1},\;u_i^\pm\lm \ttx^\pm_{i,0}\;(1\leq
i<n),\;u_n^\pm\lm\ep_n^\pm,\,\sfz^\pm_s\lm -\frac{s}{[s]}\th_{\pm
s}\;(s\geq 1).$$
\end{Coro}

Now we define the following elements in $\dHallr$:
$$\aligned
 &\wh\ttk_i^{\pm 1}=K_i^{\pm1}=(\sE_{\rm H}^{\geq 0})^{-1}(\ttk_i^{\pm1})
  =(\sE_{\rm H}^{\leq 0})^{-1}(\ttk_i^{\pm1}),\\
 &\wh\ttx^+_{i,0}=u_i^+=(\sE_{\rm H}^{\geq0})^{-1}(\ttx^+_{i,0}),\;
  \wh\ttx^-_{i,0}=u_i^-=(\sE_{\rm H}^{\leq0})^{-1}(\ttx^-_{i,0}),\\
 &\wh\ttx^\pm_{i,s}=(\sE_{\rm H}^{\geq0})^{-1}(\ttx^\pm_{i,s}),\;
  \wh\ttg_{i,s}=(\sE_{\rm H}^{\geq0})^{-1}(\ttg_{i,s})\;\;(s\geq 1),\\
 &\wh\ttx^\pm_{i,s}=(\sE_{\rm H}^{\leq0})^{-1}(\ttx^\pm_{i,s}),\;
  \wh\ttg_{i,s}=(\sE_{\rm H}^{\leq0})^{-1}(\ttg_{i,s})\;\;(s\leq -1).
\endaligned$$
 Then the set
$$X:=\{\wh\ttk_i^{\pm 1},\wh\ttg_{i,s}\mid 1\leq i\leq
n,0\not=s\in\mbz\}\cup \{\wh\ttx^\pm_{i,s}\mid 1\leq
i<n,s\in\mbz\}$$
 is a generating set for $\dHallr$. Since all $\ttx^\pm_{i,s}$
belong to $\bfU(\afsl)$, it follows that
$\sE_{\rm H}(\wh\ttx^\pm_{i,s})=\sE_{\rm B}(\ttx^\pm_{i,s})$. Thus, the
$\wh\ttx^\pm_{i,s}$ satisfy the relations (QLA5)--(QLA7).
Furthermore, set for $1\leq i<n$ and $s\in\mbz$,
$$\wh\tth_{i,s}=v^{(i-1)s}\wh\ttg_{i,s}-v^{(i+1)s}\wh\ttg_{i+1,s}.$$
Then $\tth_{i,s}=\sE_{\rm H}(\wh\tth_{i,s})\in\bfU(\afsl)$. Hence,
$\wh\tth_{i,s}$ together with $\wh\ttk_i^{\pm1}$ and
$\wh\ttx^\pm_{i,s}$ satisfy the relations (QLA8)--(QLA9). On the other hand, the elements
$$\wh\th_s:=v^{ns}(\wh\ttg_{1,s}+\cdots+\wh\ttg_{n,s})
=\begin{cases} (\sE^{\geq
0}_H)^{-1}(\th_s)=-\frac{[s]}{s}\sfz^+_s,\;\;&\text{if $s\geq 1$};\\
(\sE^{\leq
0}_H)^{-1}(\th_s)=\frac{[-s]}{s}\sfz^-_{-s},\;\;&\text{if $s\leq
-1$.}\end{cases}$$
 are central in $\dHallr$. We finally get that the $\wh\ttg_{i,s}$
satisfy the relations (QLA2)--(QLA4). Therefore, the generators in
$X$ satisfy all the relations (QLA1)--(QLA7). We conclude that
there is a surjective $\mbq(\up)$-algebra homomorphism
$\sF:\bfU(\afgl)\mapsto \dHallr$ taking
$$\wh\ttk_i^{\pm1}\lm\ttk_i^{\pm1},\;\wh\ttg_{i,s}\lm\ttg_{i,s},\;
\wh\ttx^\pm_{i,s}\lm \ttx^\pm_{i,s}.$$
 Obviously, both the composites $\sE_{\rm H}\sF$ and $\sF\sE_{\rm H}$ are the identity
maps. This gives the following result.

\begin{Thm} \label{iso afgln dHallr} The algebra homomorphism
$\sE_{\rm H}:\dHallr\to\bfU(\afgl)$ defined in Corollary \ref{epi-doubleRH-QGL} is an isomorphism. In
particular, Theorem \ref{presentation dHallAlg} gives another
presentation for $\bfU(\afgl)$.
\end{Thm}\index{$\sE_{\rm H}$, isomorphism $\dHallr\overset\sim\to\bfU(\afgl)$}

Let $\bfU(\afgl)^+$ (resp., $\bfU(\afgl)^-$) be the subalgebra of
$\bfU(\afgl)$ generated by $\ttx^+_{i,0},\,\ep_n^+,\,\th_s$ (resp.,
$\ttx^-_{i,0},\,\ep_n^-,\,\th_{-s}$) for $1\leq i<n, \,s\geq 1$.
Also, let $\bfU(\afgl)^0$ be the subalgebra of $\bfU(\afgl)$
generated by the $\ttk_i^{\pm1}$. The triangular decomposition of
$\dHallr$ induces that of $\bfU(\afgl)$.

\begin{Coro} The multiplication map
$$\bfU(\afgl)^+\ot\bfU(\afgl)^0\ot \bfU(\afgl)^-\lra
\bfU(\afgl)$$
 is a $\mbq(\up)$-space isomorphism.
\end{Coro}

\begin{Rems}\label{iso afgln dHallr over C} (1) In \cite{Hub2}
Hubery deals with the extended Ringel--Hall algebra
$\sH_v^{\geq}(\triangle')$ of the cyclic quiver $\triangle'=\triangle'(n)$ which has the
opposite orientation of $\tri$. By \cite[Th.~1]{Hub2}, there is a
decomposition
$$\sH_v^{\geq}(\triangle')\cong \bfU(\afsl)^+\otimes_{\mbq(\up)}{\mbq}(v)[c_1,c_2,\ldots]\otimes_{\mbq(\up)}
\mbq(\up)[K_1^{\pm1},\ldots,K_n^{\pm1}].$$
 Hence, $\sH_v^{\geq}(\triangle')$ is isomorphic to $\Hallpi$ defined in \S1.5.

(2) The proof above can be easily modified to construct a
$\mbc$-algebra homomorphism $\sE_{{\rm H},\mbc}:\DC(n)\to\afUglC$,
where the algebras are defined over $\mbc$ with respect to a
non-root of unity $z\in\mbc$.\index{$\sE_{\rm H,\mbc}$, isomorphism $\DC(n)\overset\sim\to\afUglC$}

(3) With the above isomorphism, the double Ringel--Hall algebras
$\dHallr$, $\DC(n)$ will also be called a {\it quantum affine
$\frak{gl}_n$}. Moreover, the notation $\dHallr$, $\DC(n)$ will not
be changed to $\bfU(\afgl)$, $\afUglC$ throughout the paper except
the last chapter in order to emphasize the approach used in this
paper.

(4) There is another geometric realization for the $+$-part of the
quantum loop algebra $\bfU(\afgl)$ in terms of the Hall algebra of a
Serre subcategory of the category of coherent sheaves over a
weighted projective line; see \cite[4.3, 5.2]{Sch2} and \cite{Hub2}.
\end{Rems}

\section{Semisimple generators and commutator formulas}

By \cite[Th.~5.2]{DDX}, the Ringel--Hall algebra $\Hall$ can be generated
by semisimple modules over $\sZ$. Thus, semisimple generators would be
crucial to the study of integral forms of $\dHallr$.
In this section we derive commutator formulas between semisimple
generators of $\dHallr$ (cf. \cite[Prop.~5.5]{X97}).

Recall from \S1.2
that for the module $M(A)$ associated with $A\in \afThnp$, we write $\bfd(A)=\bfdim M(A)$ for the
dimension vector of $M(A)$.

Define a subset $\afThnp^{ss}$ of $\afThnp$ by setting
$$\afThnp^{ss}=\{A=(a_{i,j})\in\afThnp\mid a_{i,j}=0\;\text{for all
$j\not=i+1$}\}.$$
 In other words, $A\in\afThnp^{ss}\Longleftrightarrow M(A)\;\text{is semisimple.}$
Then, by Proposition \ref{generators-RH-alg}, $\dHallr^+$ (resp.,
$\dHallr^-$) is generated by $u_A^+$ (resp., $u^-_A$) for all $A\in
\afThnp^{ss}$. We sometimes identify $\afThnp^{ss}$ with $\afmbnn$ via the map
$\afmbnn\to\afThnp^{ss}$ sending $\la$ to $A=A_\la$ with $\la_i=a_{i,i+1}$ for all $i\in\mbz$.

\begin{Lem}\label{reduce lemma for presentation of double Hall algebra}
Let $X^+:=\spann\{u_A^+\mid A\in\afThnp^{ss}\}$ and
$X^-:=\spann\{u_A^-\mid A\in\afThnp^{ss}\}$. Then
$${\dHallr}\cong \Hallpi*\Hallmi/\sJ,$$
 where $\sJ$ is the ideal of the free product $\Hallpi*\Hallmi$ generated by
\begin{itemize}
\item[(1)]
$\sum(b_2*a_2)\psi(a_1,b_1)-\sum(a_1*b_1)\psi(a_2,b_2)$ for all $a\in
X^+,b\in X^-,$

\item[(2)]  $ K_{\al}*1-1* K_{\al}$ for all $\al\in\mbz I,$
\end{itemize}
where $\Dt(a)=\sum a_1\ot a_2$ and $\Dt(b)=\sum b_1\ot b_2$.
\end{Lem}

\begin{proof} For $a\in\Hallpi$ and $b\in \Hallmi$, we write
$$L(a,b)=\sum(b_2*a_2)\psi(a_1,b_1)\;\;\text{and}\;\;
R(a,b)=\sum(a_1*b_1)\psi(a_2,b_2).$$
 Define
$$\aligned
&X^{\geq 0}:=\spann\{u_A^+K_{\al}\mid\al\in\mbz I,
A\in\afThnp^{ss}\}\;\;\text{and}\\
&X^{\leq 0}:=\spann\{K_{\al}u_A^-\mid\al\in\mbz I,
A\in\afThnp^{ss}\}.\endaligned$$
 Then $X^{\geq 0}$ (resp., $X^{\leq 0}$) generates
$\Hallpi$ (resp., $\Hallmi$) and satisfies
$$\Dt(X^{\geq 0})\subseteq X^{\geq 0}\ot X^{\geq 0}\;\;\;({\rm resp.},
\Dt(X^{\leq 0})\subseteq X^{\leq 0}\ot X^{\leq 0}).$$
 Thus, by Lemma \ref{reduce lemma},
$$\dHallr=\Hallpi*\Hallmi/\widehat\sI,$$
 where $\widehat\sI$ is the ideal of $\Hallpi*\Hallmi$
generated by
\begin{itemize}
\item[(1$'$)]
$\sum(b_2*a_2)\psi(a_1,b_1)-\sum(a_1*b_1)\psi(a_2,b_2)\;\;(a\in
X^{\geq 0},b\in X^{\leq 0}),$

\item[(2$'$)]  $K_{\al}*1-1* K_{\al}\;\;(\al\in\mbz I).$
\end{itemize}
 Clearly, $\sJ\subseteq \widehat\sI$. It remains to show
the reverse inclusion $\widehat\sI\subseteq\sJ$. To show this, it
suffices to prove that for $u_A^+K_{\al}\in X^+$ and $K_{\bt}
u_B^-\in X^-$,
$$L(u_A^+K_{\al},K_{\bt}u_B^-)\equiv R(u_A^+K_{\al},K_{\bt} u_B^-)\;{\rm mod}\;\sJ.$$
 By the definition of comultiplications in $\Hallpi$ and $\Hallmi$,
$$\aligned
\Dt(u_A^+)&=\sum_{A_1,A_2\in\afThnp}f_{A_1,A_2}^A
u_{A_2}^+\ot u_{A_1}^+\ti K_{\bfd(A_2)}\;\;\text{and}\\
\Dt(u_B^-)& =\sum_{B_1,B_2\in\afThnp}g_{B_1,B_2}^B\ti
K_{-\bfd(B_1)} u_{B_2}^-\ot u_{B_1}^-,
\endaligned$$
 where $f_{A_1,A_2}^A=\up^{\lr{\bfd(A_1),\bfd(A_2)}}\frac{\fka_{A_1}\fka_{A_2}}{\fka_A}
\vi^A_{A_1,A_2}$ and
$g_{B_1,B_2}^B=\up^{-\lr{\bfd(B_2),\bfd(B_1)}}\frac{\fka_{B_1}\fka_{B_2}}{\fka_B}
\vi^B_{B_1,B_2}$. By the
 definition of $\sJ$,
\begin{equation}\label{L=R}
L(u_A^+,u_B^-)\equiv R(u_A^+,u_B^-)\;{\rm mod}\;\sJ,
\end{equation}
where
$$\aligned
L(u_A^+,u_B^-)&=\sum_{A_1,A_2,B_1,B_2}
f_{A_1A_2}^Ag_{B_1B_2}^B\bigl(u_{B_1}^-\ast(u_{A_1}^+\ti
K_{\bfd(A_2)})\bigr)\psi(u_{A_2}^+,\ti K_{-\bfd(B_1)}
u_{B_2}^-),\\
R(u_A^+,u_B^-)&=\sum_{A_1,A_2,B_1,B_2}
f_{A_1A_2}^Ag_{B_1B_2}^B\bigl(u_{A_2}^+\ast(\ti K_{-\bfd(B_1)}
u_{B_2}^-)\bigr)\psi(u_{A_1}^+\ti K_{\bfd(A_2)},u_{B_1}^-).
\endaligned$$
This together with Proposition \ref{bilinear form} implies that
\begin{equation*}
\begin{split}
&L(u_A^+K_{\al},K_{\bt}u_B^-)\\
&=\sum_{A_1,A_2,B_1,B_2} f_{A_1A_2}^Ag_{B_1B_2}^B\bigl((K_\bt
u_{B_1}^-)\ast(u_{A_1}^+\ti
K_{\bfd(A_2)}K_\al)\bigr)\psi(u_{A_2}^+K_\al,K_\bt \ti
K_{-\bfd(B_1)} u_{B_2}^-)\\
&\equiv\up^aK_{\al+\bt}\sum_{A_1,A_2,B_1,B_2}
f_{A_1A_2}^Ag_{B_1B_2}^B\bigl(u_{B_1}^-\ast(u_{A_1}^+\ti
K_{\bfd(A_2)})\bigr)\psi(u_{A_2}^+,\ti K_{-\bfd(B_1)}
u_{B_2}^-)\,\,(\text{by }\ref{reduce lemma for presentation of double Hall algebra}(2))\\
&\equiv\up^aK_{\al+\bt}\sum_{A_1,A_2,B_1,B_2}
f_{A_1A_2}^Ag_{B_1B_2}^B\bigl(u_{A_2}^+\ast(\ti K_{-\bfd(B_1)}
u_{B_2}^-)\bigr)\psi(u_{A_1}^+\ti K_{\bfd(A_2)},u_{B_1}^-)\,\,(\text{by }\eqref{L=R})\\
&\equiv\sum_{A_1,A_2,B_1,B_2}f_{A_1A_2}^Ag_{B_1B_2}^B
\bigl((u_{A_2}^+K_\al)\ast(K_\bt\ti K_{-\bfd(B_1)}
u_{B_2}^-)\bigr)\psi(u_{A_1}^+\ti K_{\bfd(A_2)}K_\al,K_\bt u_{B_1}^-)\\
&=R(u_A^+K_{\al},K_{\bt} u_B^-)\;{\rm mod}\;\sJ,
\end{split}
\end{equation*}
where $a=\al\centerdot\bt+\lr{\bfd(A_1)+\bfd(A_2),\al}=\al\centerdot\bt+\lr{\bfd(A),\al}$, as desired. \end{proof}

Recall the order relation $\leq$ defined in \eqref{order on afmbzn}.
\begin{Lem}\label{ProdOf2semi}
 For $\al=(\al_i),\bt=(\bt_i)\in\afmbnn$, let $\ga=\ga(\al,\bt)=(\ga_i)\in\afmbnn$ be
 defined by $\ga_i=\text{\rm min}\{\al_i,\bt_{i+1}\}$. For each $\la\le \ga$, define
$C_\la\in\afThnp$ by
$$M(C_\la)=\bigoplus_{i\in I}((\al_i+\bt_i-\la_i-\la_{i-1})S_i\oplus \la_iS_i[2]).$$
Then
$$u_\al u_\bt=\up^{\sum_i\al_i(\bt_i-\bt_{i+1})}\sum_{\la\le\ga}\prod_{i\in I}\bigg[\!\!\bigg[{{\al_i+\bt_i-\la_i-\la_{i-1}}
\atop {\bt_i-\la_{i-1}}}\bigg]\!\!\bigg]u_{C_\la}.$$
\end{Lem}

\begin{proof} Clearly, each $M(C_\la)$ with $\la\le \ga$ is an extension of
$M(A_\bt)$ by $M(A_\al)$, where $A_\al=\sum_{i=1}^n \al_i
E_{i,i+1}^\vtg$ and $A_\bt=\sum_{i=1}^n \bt_i E_{i,i+1}^\vtg$.
Conversely, each extension of $M(A_\bt)$ by $M(A_\al)$ is
isomorphic to $M(C_\la)$ for some $\la\le \ga$. Hence,
$$u_\al u_\bt=\up^{\lr{\al,\bt}}\sum_{\la\le\ga}\varphi^{C_\la}_{A_\al,A_\bt}u_{A_\la}
=\up^{\sum_i\al_i(\bt_i-\bt_{i+1})}\sum_{\la\le\ga}\varphi^{A_\la}_{A_\al,A_\bt}u_{C_\la}.$$
 The lemma then follows from the fact that
 $$\varphi^{C_\la}_{A_\al,A_\bt}=\prod_{i\in I}\bigg[\!\!\bigg[{{\al_i+\bt_i-\la_i-\la_{i-1}}
\atop {\bt_i-\la_{i-1}}}\bigg]\!\!\bigg].$$
\end{proof}

\begin{Thm} \label{alt-presentation-dHall}
The algebra $\dHallr$ has generators
$u_A^+$, $K_{\nu}$, $u_A^-$ $(A\in\afThnp^{ss},\nu\in\mbz I)$ which satisfy
the following relations: for $\nu,\nu'\in\mbz I$, $A,B\in\afThnp^{ss}$,
\begin{itemize}
\item[(1)] $K_0=u_0^+=u_0^-=1,\quad K_\nu K_{\nu'}=K_{\nu+\nu'}$;
\item[(2)] $K_\nu u_A^+=\nup^{\lr{\bfd(A),\nu}}u_A^+K_\nu,\;\;
  u_A^-K_{\nu}=\up^{\lr{\bfd(A),\nu}}K_{\nu} u_A^-$;
\item[(3)]
$u_A^+u_B^+=\up^{\sum_i\al_i(\bt_i-\bt_{i+1})}\sum_{\la\le\ga}\prod_{i\in I}\big[\!\!\big[{{\al_i+\bt_i-\la_i-\la_{i-1}}
\atop {\bt_i-\la_{i-1}}}\big]\!\!\big]u_{M(C_\la)},$ if $A=A_\al,B=B_\bt$ and $\ga=\ga(\al,\bt)$;

\item[(4)]
$u_A^-u_B^-=\up^{\sum_i\bt_i(\al_i-\al_{i+1})}\sum_{\la\le\ga'}\prod_{i\in I}\big[\!\!\big[{{\al_i+\bt_i-\la_i-\la_{i-1}}
\atop {\al_i-\la_{i-1}}}\big]\!\!\big]u_{M(C_\la)},$ if $A=A_\al,B=B_\bt$ and $\ga'=\ga(\bt,\al)$;

\item[(5)] {\rm commutator relations}:  for all $A,B\in\afThnp^{ss}$,
\begin{equation*}\label{commutator-formula}
\aligned
 \up^{\lan \bfd(B),\bfd(B)\ran}&\sum_{A_1,B_1}\vi_{A,B}^{A_1,B_1}
 \up^{\lan \bfd(B_1),\bfd(A)+\bfd(B)-\bfd(B_1)\ran}\ti K_{\bfd(B)-\bfd(B_1)}u_{B_1}^-u_{A_1}^+\\
 &=\up^{\lan\bfd(B),\bfd(A)\ran}\sum_{A_1,B_1}
 \ti{\vi_{A,B}^{A_1,B_1}}\up^{\lan \bfd(B)-\bfd(B_1),\bfd(A_1)\ran+\lan \bfd(B),\bfd(B_1)\ran}
 \ti K_{\bfd(B_1)-\bfd(B)}u_{A_1}^+u_{B_1}^-,\endaligned
 \end{equation*}
\end{itemize}
where
\begin{equation}\label{phA1B1AB}\begin{split}
\vi_{A,B}^{A_1,B_1}&=\frac{\fka_{A_1}\fka_{B_1}}{\fka_A\fka_B}\sum_{A_2\in\afThnp}
\up^{2\fkd(A_2)}\fka_{A_2}\vi_{A_1,A_2}^{A}\vi_{B_1,A_2}^B,\\
\ti{\vi_{A,B}^{A_1,B_1}}&=\frac{\fka_{A_1}\fka_{B_1}}{\fka_A\fka_B}
\sum_{A_2\in\afThnp}\up^{2\fkd(A_2)}\fka_{A_2}\vi_{A_2,A_1}^{A}\vi_{A_2,B_1}^B.\\
\end{split}
\end{equation}

\end{Thm}

\begin{proof} Relations (1) and (2) follow from definition, and (3) and (4) follow
from Lemma \ref{ProdOf2semi}. We now prove (5). As in the proof of Lemma \ref{reduce lemma for
presentation of double Hall algebra}, for $A,B\in\afThnp$,
\begin{equation*}
\begin{split}
L(u_A^+,&u_B^-)\\
&=\sum_{A_1,A_2,B_1,B_2}\up^{\lan \bfd(A_1),\bfd(A_2)\ran}
\frac{\fka_{_{A_1}}\fka_{_{A_2}}}{\fka_{_A}}\vi_{A_1,A_2}^A\up^{-\lan\bfd(B_2),\bfd(B_1)\ran}
\frac{\fka_{_{B_1}}\fka_{_{B_2}}}{\fka_{_B}}\vi_{B_1,B_2}^B\\
&\qquad\qquad\qquad\times\bigl(u_{B_1}^-\ast(u_{A_1}^+\ti
K_{\bfd(A_2)})\bigr)\psi(u_{A_2}^+,\ti K_{-\bfd(B_1)} u_{B_2}^-)\\
&=\sum_{A_1,A_2,B_1,B_2} \up^{\lan \bfd(A_1),\bfd(A_2)\ran}
\frac{\fka_{_{A_1}}\fka_{_{A_2}}}{\fka_{_A}}\vi_{A_1,A_2}^A\up^{-\lan\bfd(B_2),\bfd(B_1)\ran}
\frac{\fka_{_{B_1}}\fka_{_{B_2}}}{\fka_{_B}}\vi_{B_1,B_2}^B\\
&\qquad\qquad\qquad\times\bigl(u_{B_1}^-\ast(u_{A_1}^+\ti
K_{\bfd(A_2)})\bigr)\up^{-\lr{\bfd(A_2),\bfd(A_2)}+2\fkd(A_2)}\frac{1}{\fka_{A_2}}\delta_{A_2,B_2}\\
&\equiv\sum_{A_1,B_1,A_2}\up^{-\lan
\bfd(A_2),\bfd(A)\ran+\lr{\bfd(B_1),\bfd(A_2)}+2\fkd(A_2)}
\frac{\fka_{_{A_1}}\fka_{_{B_1}}\fka_{_{A_2}}}{\fka_{_A}\fka_{_B}}\vi_{A_1,A_2}^A\vi_{B_1,A_2}^B\\
&\qquad\qquad\qquad\times(\ti K_{\bfd(A_2)}u_{B_1}^-)\ast u_{A_1}^+\\
&\equiv\up^{-\lan\bfd(B),\bfd(A)\ran}\sum_{A_1,B_1}\up^{\lan
\bfd(B_1),\bfd(A)+\bfd(B)-\bfd(B_1)\ran}\vi_{A,B}^{A_1,B_1}\ti
K_{\bfd(B)-\bfd(B_1)}u_{B_1}^-*u_{A_1}^+\;{\rm mod}\;\sJ,
\end{split}
\end{equation*}
\begin{equation*}
\begin{split}
& R(u_A^+,u_B^-)\\
&=\sum_{A_1,A_2,B_1,B_2}\up^{\lan \bfd(A_1),\bfd(A_2)\ran}
\frac{\fka_{_{A_1}}\fka_{_{A_2}}}{\fka_{_A}}\vi_{A_1,A_2}^A\up^{-\lan
\bfd(B_2),\bfd(B_1)\ran}
\frac{\fka_{_{B_1}}\fka_{_{B_2}}}{\fka_{_{B}}}\vi_{B_1,B_2}^B\\
&\qquad\qquad\qquad\times\bigl(u_{A_2}^+\ast(\ti K_{-\bfd(B_1)}
u_{B_2}^-)\bigr)\psi(u_{A_1}^+\ti K_{\bfd(A_2)},u_{B_1}^-)\\
&\equiv\sum_{A_1,A_2,B_2}
\up^{\lan\bfd(A_1),\bfd(A_2)\ran-\lan\bfd(B),\bfd(A_1)\ran+2\fkd(A_1)}
\frac{\fka_{_{A_1}}\fka_{_{A_2}}\fka_{_{B_2}}}{\fka_{_A}\fka_{_{B}}}
\vi_{A_1,A_2}^A\vi_{A_1,B_2}^B\ti K_{-\bfd(A_1)}u_{A_2}^+*u_{B_2}^-\\
&\equiv\up^{-\lan\bfd(B),\bfd(B)\ran}\sum_{A_1,B_1}\up^{\lan\bfd(B)-\bfd(B_1),
\bfd(A_1)\ran+\lan\bfd(B),\bfd(B_1)\ran}\ti{\vi_{A,B}^{A_1,B_1}} \ti
K_{-\bfd(B)+\bfd(B_1)}u_{A_1}^+*u_{B_1}^-\;{\rm mod}\;\sJ,
\end{split}
\end{equation*}
by interchanging the running indexes $A_1$ and $A_2$, $B_1$ and $B_2$.
This proves (5). 
\end{proof}

\begin{Rem} The commutator relations in (5) hold for all
$A,B\in\afThnp$.
\end{Rem}

Theorem \ref{alt-presentation-dHall} does not give a presentation
for $\dHallr$ since the modules $M(C_\la)$ are not necessarily
semisimple. It would be natural to raise the following question.

\begin{Prob}\label{semisimple presentation}
Find the ``quantum Serre relations" associated
with semisimple generators to replace relations
\ref{alt-presentation-dHall}(3)--(4), and prove that, the
relations given in \ref{alt-presentation-dHall} are defining
relations for $\dHallr$. In this way, we obtain a presentation
with semisimple generators for $\dHallr$.
\end{Prob}

The relations in (5) above are usually called the {\it commutator
relations}. \index{commutator
relation} We now derive a finer version of the commutator
relations for semisimple generators. The next lemma follows
directly from the definition of comultiplication.

\begin{Lem}\label{Dt2}
For $A\in\afThnp$, we have
\[
\begin{split}
\Dt^{(2)}(u_A^+)&=\sum_{A^{(1)}, A^{(2)},A^{(3)}}\up^{\sum_{i>j}\lan
\bfd(A^{(i)}), \bfd(A^{(j)})\ran}
\vi^A_{A^{(3)},A^{(2)},A^{(1)}}\frac{\fka_{A^{(1)}}\fka_{A^{(2)}}\fka_{A^{(3)}}}{\fka_A}\\
&\qquad\qquad \times u^+_{A^{(1)}}\ot u^+_{A^{(2)}}\ti
K_{\bfd(A^{(1)})}
\ot u_{A^{(3)}}^+\ti K_{\bfd(A^{(1)})+\bfd(A^{(2)})},\\
\Dt^{(2)}(u_A^-)&=\sum_{{A^{(1)},A^{(2)},A^{(3)}}}
\up^{-\sum_{i<j}\lan\bfd(A^{(i)}),\bfd(A^{(j)})\ran}
\vi^A_{A^{(3)}, A^{(2)}, A^{(1)}}\frac{\fka_{A^{(1)}}\fka_{A^{(2)}}\fka_{A^{(3)}}}{\fka_A}\\
&\qquad\qquad\times\ti K_{-(\bfd(A^{(2)})+\bfd(A^{(3)}))}
u^-_{A^{(1)}} \ot\ti K_{-\bfd(A^{(3)})} u^-_{A^{(2)}}\ot
u_{A^{(3)}}^-.
\end{split}
\]
\end{Lem}

\begin{Prop}\label{general commute formula} Let $X,Y\in\afThnp$.
Then in $\dHallr$,
\[
\begin{split}
&u_Y^-u_X^+-u_X^+u_Y^- \\
=&\sum_{A,B,B'\in\afThnp \atop A\not=0}
\up^{\ell_1}\vi_{A,B}^X\vi_{A,B'}^Y\frac{\fka_A\fka_B\fka_{B'}}{\fka_X\fka_Y}
\ti K_{-\bfd(A)}u_B^+u_{B'}^-\\
&\quad+\sum_{A,B,B',C,C'\in\afThnp\atop
C\not=0\not=C'}\up^{\ell_2}\vi^{X}_{A,B,C}\vi^{Y}_{A,B',C'}\frac{\fka_A\fka_B\fka_{B'}}{\fka_X\fka_Y}\biggl(\sum_{m\geq
1\atop X_1,\ldots,X_m\in\afThnp_1}(-1)^m\\
&\qquad\up^{2\sum_{i<j}\lan\bfd(X_i),\bfd(X_j)\ran}
\fka_{X_1}\cdots \fka_{X_m}\vi^C_{X_1,\ldots,
X_m}\vi^{C'}_{X_1,\ldots, X_m}\biggr) \ti
K_{\bfd(C)-\bfd(A)}u_B^+u_{B'}^-.
\end{split}
\]
 where $\ell_1=\lan \bfd(A),\bfd(B)\ran-
 \lan\bfd(Y),\bfd(A)\ran+2\fkd(A)$ and
$$ \ell_2=
\lan\bfd(A),\bfd(B)\ran+\lan\bfd(Y),\bfd(C)-\bfd(A)\ran-\lan
\bfd(C),2\bfd(C)+\bfd(B)\ran +2\fkd(A)+2\fkd(C).$$
\end{Prop}

\begin{proof} By \cite[Lem.~3.2.2(iii)]{Jo95}, for $x\in\Hallpi$ and
$y\in\Hallmi$, we have in $\dHallr$,
$$yx=\sum\psi(x_1,\sg(y_1))(x_2y_2)\psi(x_3,y_3),$$
 where $\Dt^{(2)}(x)=\sum x_1\ot x_2\ot x_3$ and $\Dt^{(2)}(y)=\sum y_1\ot
 y_2\ot y_3$. This together with Lemma \ref{Dt2} gives the
 required equality.
\end{proof}

The following result is a direct consequence of the above
proposition together with the fact that for $\bt=(\bt_i),
\bt^{(1)}=(\bt^{(1)}_i),\ldots,\bt^{(m)}=(\bt^{(m)}_i)\in\afmbnn$,
\begin{equation}\label{multiGauss}
\vi_{\bt^{(1)},\ldots,\bt^{(m)}}^\bt=\prod_{i=1}^n\dleb{\bt_i\atop\bt^{(1)}_i,\ldots,
\bt_i^{(m)}}\drib=:\dleb{\bt\atop\bt^{(1)},\ldots,
\bt^{(m)}}\drib,
\end{equation}
where $\dleb{\bt_i\atop\bt^{(1)}_i,\ldots,
\bt_i^{(m)}}\drib=\frac{[\![\bt_i]\!]^!}{[\![\bt^{(1)}_i]\!]^!\ldots
[\![\bt_i^{(m)}]\!]^!}$.

\begin{Coro}\label{comm formula for semisimple}
For $\la,\mu\in\afmbnn$, we have
$$u_\mu^-u_\la^+-u_\la^+ u_\mu^-=
\sum_{\al\not=0,\,\al\in\afmbnn\atop\al\leq\la,\,\al\leq\mu}\sum_{0\leq\ga\leq\al}
x_{\al,\ga}\ti K_{2\ga-\al} u_{\la-\al}^+u_{\mu-\al}^-,$$
 where
$$\aligned
x_{\al,\ga}=&\up^{\lr{\al,\la-\al}+\lr{\mu,2\ga-\al}+2\lr{\ga,\al-\ga-\la}+2\sg(\al)}\\
 &\times \dleb{\la\atop\al-\ga, \la-\al,\ga}\drib\cdot\dleb{\mu\atop\al-\ga,
 \mu-\al,\ga}\drib\frac{\fka_{\al-\ga}\fka_{\la-\al} \fka_{\mu-\al}}{\fka_\la \fka_\mu}\\
 & \times\sum_{m\geq 1,\ga^{(i)}\not=0\,\forall
i\atop\ga^{(1)}+\cdots+\ga^{(m)}=\ga}(-1)^m\up^{2\sum_{i<j}\lan\ga^{(i)},\ga^{(j)}\ran}
\fka_{\ga^{(1)}}\cdots
\fka_{\ga^{(m)}}\dleb\ga\atop\ga^{(1)},\ldots,\ga^{(m)}\drib^2.
\endaligned$$
\end{Coro}

This is the {\it commutator formula} for semisimple generators. \index{commutator
formula}

\chapter{Affine quantum Schur algebras and the Schur--Weyl reciprocity}

Like the quantum Schur algebra, the affine quantum Schur algebra has
several equivalent definitions. We first present the geometric
definition, given by Ginzburg--Vasserot and Lusztig, which uses cyclic
flags and convolution product. We then discuss the two Hecke algebra
definitions given by R. Green and by Varagnolo--Vasserot. The former
uses $q$-permutation modules, while the latter uses tensor
spaces. Both versions are related by the Bernstein presentation for
Hecke algebras of affine type $A$.

In \S3.4, we review the
construction of BLM type bases for affine quantum Schur algebras and the multiplication formulas
between simple generators and BLM basis elements (Theorem \ref{multiplication formulas in affine q-Schur algebra})
developed by the last two authors \cite{DF09}. Through the central
element presentation for $\dHallr$ as given in Theorem \ref{presentation dHallAlg}, we introduce a
$\dHallr$-$\afbfHr$-bimodule structure on the tensor space in \S 3.5. This gives
a homomorphism $\xi_r$ from $\dHallr$ to $\afbfSr$. We then
prove in \S3.6 that the restriction of this bimodule action coincides
with the $\bfHall^{\text{op}}$-$\afbfHr$-bimodule structure defined by
Varagnolo--Vasserot in \cite{VV99}. Thus, we obtain an explicit description of
the map $\xi_r$ (Theorem \ref{xirl}).

In \S3.7, we develop a certain triangular relation (Proposition
\ref{triangular formula in A(bfj,r)}) among the structure constants
relative to the BLM basis elements. With this relation, we display
an integral PBW type basis and, hence, a triangular decomposition
for an affine quantum Schur algebra (Theorem \ref{PBW basis of
affine q-Schur algebras}). Using the triangular decomposition, we
easily establish the surjectivity of the homomorphism $\xi_r$ from
$\dHallr$ to $\afbfSr$ in \S3.8 (Theorem
\ref{surjective-dHall-aff}).

There are several important applications of this result which will be discussed in the next three chapters.  As a first application,
we end this chapter by establishing certain polynomial identities (Corollary \ref{polyidentity})
arising from the commutator formulas for semisimple generators
discussed in \S2.4.

\section{Cyclic flags: the geometric definition}

In this section we recall the geometric definition of affine
quantum Schur algebras introduced by
Ginzburg--Vasserot \cite{GV} and Lusztig \cite{Lu99}.

Let $\field$ be a field and fix an $\field[\ep,\ep^{-1}]$-free
module $V$ of rank $r\geq 1$, where $\ep$ is an indeterminate. A
lattice in $V$ is, by definition, a free $\field[\ep]$-submodule
$L$ of $V$ satisfying
$V=L\otimes_{\field[\ep]}\field[\ep,\ep^{-1}]$. For two lattices
$L',L$ of $V$, $L+L'$ is again a lattice. If, in addition,
$L'\subseteq L$, $L/L'$ is a finite generated torsion
 $\field[\ep]$-module. Thus, as an $\field$-vector space, $L/L'$ is finite dimensional.

Let $\afFn={\mathscr F}_{\vtg,n}$ be the set of all {\it cyclic
flags} \index{cyclic flag}\index{$\afFn={\mathscr F}_{\vtg,n}$, set
of cyclic flags of period $n$} $\bfL=(L_i)_{i\in\mbz}$ of lattices
of {\it period $n$}, where each $L_i$ is a lattice in $V$ such that
$L_{i-1}\han L_i$ and $L_{i-n}=\ep L_i$ for all $i\in\mbz$. The
group $G$ of automorphisms of the $\field[\ep,\ep^{-1}]$-module $V$
acts on $\afFn$ by $g\cdot\bfL=(g(L_i))_{i\in\mbz}$ for $g\in G$ and
$\bfL\in\afFn$. Thus, the map
$$\phi:\afFn\lra \afLanr,\quad\bfL\longmapsto (\dim_\field L_i/L_{i-1})_{i\in\mbz}$$
induces a bijection between the set
$\{\afFn_{,\la}\}_{\la\in\afLanr}$ of $G$-orbits in $\afFn$ and
$\afLanr$.\index{$\afFn_{,\la}$, $G$-orbit in $\afFn$}

Similarly, let $\afBr={\mathscr B}_{\vtg,r}$ be the set of all
complete cyclic flags $\bfL=(L_i)_{i\in\mbz}$ of lattices, where
each $L_i$ is a lattice in $V$ such that $L_{i-1}\han L_i$,
$L_{i-r}=\ep L_i$ and $\dim_\field(L_i/L_{i-1})=1$, for all
$i\in\mbz$.

The group $G$ acts on $\afFn\times\afFn$, $\afFn\times\afBr$ and $\afBr\times\afBr$ by
$g\cdot(\bfL,\bfL')=(g\cdot\bfL,g\cdot\bfL')$. For $\bfL=(L_i)_{i\in\mbz}$
and $\bfL'=(L_i')_{i\in\mbz}\in\afFn$, let
$X_{i,j}:=X_{i,j}(\bfL,\bfL')=L_{i-1}+L_i\cap L_j'$. By
lexicographically ordering the indices $i,j$, we obtain a
filtration $(X_{i,j})$ of lattices of $V$. For $i,j\in\mbz$, let
$a_{i,j}=\dim_\field(X_{i,j}/X_{i,j-1})=\dim_\field\frac{L_i\cap L_j'}{L_{i-1}\cap L_j'+L_i\cap L_{j-1}'}$. By \cite[1.5]{Lu99}
there is a bijection between the set of $G$-orbits in
$\afFn\times\afFn$ and the set $\afThnr$ by sending $(\bfL,\bfL')$
to $A=(a_{i,j})_{ij\in\mbz}$. Let $\sO_A\han\afFn\times\afFn$ be
the $G$-orbit corresponding to the matrix $A\in\afThnr$. By
\cite[1.7]{Lu99}, for $\bfL,\bfL'\in\afFn$,
\begin{equation}\label{transpose matrix}
(\bfL,\bfL')\in\sO_A\Longleftrightarrow (\bfL',\bfL)\in\sO_{\tA},
\end{equation}
 where $\tA$ is the transpose matrix of $A$. \index{$\tA$, transpose of $A$}

Similarly, putting $\og=(\ldots,1,1,\ldots)\in\La_\vtg(r,r)$ and
\begin{equation}\label{matrixafSr1}
\afThrr_\og=\{A\in\afThrr\mid\ro(A)=\co(A)=\og\},
\end{equation}
 the $G$-orbits $\sO_A$ on $\afBr\times\afBr$ are indexed
by the matrices $A\in\afThrr_\og$, while the $G$-orbits $\sO_A$ on
$\afFn\times\afBr$ are indexed by the set\footnote{The set is
denoted by $\ominus_\vtg(n,r)$ in \cite{DF09}.}
$$\Thnrr=\{A\in\Thnr\mid  \ro(A)\in\afLanr, \co(A)=\og\},$$
where, like $\afThn$,
\begin{equation}\label{ThnOVERr}
\Thnr=\{(a_{i,j})_{i,j\in\mbz}\mid a_{i,j}\in\mbn,
a_{i,j}=a_{i-n,j-r} \forall
i,j\in\mbz,\sum_{i=1}^n\sum_{j\in\mbz}a_{i,j}\in\mbn\}.
\end{equation}
Clearly, with this notation,
\begin{equation}\label{matrixafSr2}
   \afThrr_\og=\{A\in\Thrr\mid\ro(A)=\co(A)=\og\}.
   \end{equation}

Assume now that $\field=\field_q$ is the finite field of $q$
elements and write $\afFn(q)$ for $\afFn$ and $\afBr(q)$ for
$\afBr$, etc.
By regarding $\mbc\afFn(q)$ (resp., $\mbc\afBr(q)$) as a permutation
$G$-module, the endomorphism algebra $\sS_{\vtg,q}:=\End_{\mbc
G}(\mbc\afFn(q))^{\text{op}}$ (resp., $\sH_{\vtg,q}:=\End_{\mbc
G}(\mbc\afBr(q))^{\text{op}}$)\index{affine quantum Schur
algebra!$\sim$ at a prime power $q$, $\sS_{\vtg,q}$}\index{affine
Hecke algebra!$\sim$ at a prime power $q$, $\sH_{\vtg,q}$} has a
basis $\{e_{A,q}\}_{A\in\afThnr}$ (resp.,
$\{e_{A,q}\}_{A\in\afThrr_\og}$) with the following multiplication:
\begin{equation}\label{eee}
e_{A,q}e_{{A',q}}=
\begin{cases}
\sum_{A''\in\Theta}\afg_{A,A',A'';q}e_{{A'',q}}&
\text{ if $\co(A)=\ro(A')$}\\
0&\text{otherwise},
\end{cases}\end{equation}
where   $\Theta=\afThnr$ (resp., $\Theta=\Theta_\vtg(r,r)_\og$) and
\begin{equation}\label{afg}
\afg_{A,A',A'';q}=|\{\bfL'\in\afFn: (\bfL,\bfL')
\in\sO_{A},(\bfL',\bfL'')\in\sO_{A'}\}|
\end{equation}
for any fixed $(\bfL,\bfL'')\in\sO_{A''}$.

By \cite[1.8]{Lu99}, there exists a polynomial
$p_{A,A',A''}\in\sZ$\index{$\sZ=\mbz[\up,\up^{-1}]$, Laurent
polynomial ring in indeterminate $\up$} in $\up^2$ such that for
each finite field $\field$ with $q$ elements,
$\afg_{A,A',A'';q}=p_{A,A',A''}|_{\up^2=q}$. Thus, we have the
following definition; see \cite[1.11]{Lu99}.

\begin{Def}\label{generic affine quantum Schur algebra}
The {\it (generic) affine quantum Schur algebra} $\afSr$
\index{affine quantum Schur algebra!$\sim$ over $\sZ$, $\afSr$}
(resp., {\it affine Hecke algebra} $\afHr$) \index{affine Hecke
algebra!$\sim$ over $\sZ$, $\afHr$} is the free $\sZ$-module with
basis $\{e_{A}\mid A\in\afThnr\}$ (resp., $\{e_{A}\mid
A\in\afThrr_\og\}$), and multiplication defined by
\begin{equation}\label{e-basis multn}e_{A}e_{{A'}}=
\begin{cases}
\sum_{A''\in\Theta}p_{A,A',A''}e_{{A''}},&
\text{if $\co(A)=\ro(A')$;}\\
0,&\text{otherwise.}
\end{cases}\end{equation}
Both $\afSr$ and $\afHr$ are associative algebras over $\sZ$ with an
anti-automorphism $e_A\mapsto e_{\tA}$ (see \eqref{taur} below for a
modified version).
\end{Def}

Alternatively, we can interpret affine quantum Schur algebras in
terms of convolution algebras defined by $G$-invariant functions and convolution product. Again,
assume that $\field$ is the finite field of $q$ elements, and for
notational simplicity, let
$$\scrY=\afFn(q),\qquad\scrX=\afBr(q)\quad(\text{and }G=G(q)).$$
\index{$\scrY$, set of cyclic flags ${\mathscr
F}_{\vtg,n}(q)$}\index{$\scrX$, set of complete cyclic flags
${\mathscr B}_{\vtg,r}(q)$} Define $\sSq$, $\sTq$ and $\sHq$ to be
the $\mbc$-span of the characteristic functions $\chi_\sO$ of the
$G$-orbits $\sO$ on $\scrY\times\scrY$, $\scrY\times\scrX$ and
$\scrX\times\scrX$, respectively. With the {\it convolution product}
\index{convolution product}
\begin{equation}\label{S convolution product}
(\chi_\sO*\chi_{\sO'})(\bfL,\bfL'')=\sum_{\bfL'\in\scF'}
\chi_{\sO}(\bfL,\bfL')\chi_{\sO'}(\bfL',\bfL''),
\end{equation}
where $\sO\subset\scF\times\scF'$ and $\sO'\subset\scF'\times\scF''$
for various selections of $\scF,\scF'$ and $\scF''$, we obtain
{\it convolution algebras}
$\sSq$ and $\sHq$, and $\sSq$-$\sHq$-bimodule $\sTq$. It is clear
that $\sS_{\vtg,q}\cong\sSq$ \index{affine quantum Schur
algebra!$\sim$ at a prime power $q$, $\sSq$} and
$\sH_{\vtg,q}\cong\sHq$,\index{affine Hecke algebra!$\sim$ at a
prime power $q$, $\sHq$} and specializing $\up$ to $\sqrt{q}$ gives
an isomorphism
\begin{equation}\label{iso-Schur-Function}
\afSr_\mbc\lra\sSq\,\,(\text{resp., }
\afHr\ot_{\sZ}\mbc\lra\sHq)\end{equation} sending $e_A\otimes 1\lm
\chi_A$,
 where $\chi_A$ denotes the characteristic function of the orbit $\sO_A$. In the sequel, we shall identify
$\afSr_\mbc$ with $\sSq$.

Via convolution product, $\mbc_{G}(\scrY\times \scrX)$ becomes a
$\mbc_{G}(\scrY\times \scrY)$-$\mbc_{G}(\scrX\times
\scrX)$-bimodule. Thus, if we denote by
$\sT_\vtg(n,r)$ the {\it generic} form of $\sTq$, then $\sT_\vtg(n,r)$ becomes an $\afSr$-$\afHr$-bimodule
with a $\sZ$-basis $\{e_A\mid A\in\Thnrr\}$.\index{tensor space!$\sT_\vtg(n,r)$}

\begin{Rem}\label{implicit use bcp}
It is clear from the definition that this isomorphism continues to hold
if $\mbc$ is replaced by the ring $R=\mbz[\sqrt{q},\sqrt{q}^{-1}]$.
In fact, we will frequently use the isomorphism
$\afSr_R\cong
R_G(\scrY\times\scrY)$ to derive formulas in $\afSr$ by doing computations in
$R_G(\scrY\times\scrY)$; see, e.g., \S\S3.6--9 below.
\end{Rem}

Observe that, for $N\geq n$, $\afFn_{,n}$ is naturally a subset of
$\afFn_{,N}$, since every $\bfL=(L_i)\in\afFn_{,n}$ can be regarded
as $\ti\bfL=(\ti L_i)\in\afFn_{,N}$, where, for all $a\in\mbz$, $\ti
L_{i+aN}=L_{i+an}$ if $1\le i\le n$, and $\ti L_{i+aN}=L_{n+an}$ if
$n\leq i\leq N$. Thus, if $N=\text{\rm max}\{n,r\}$, then
$\afFn_{,n}\times\afFn_{,n}$, $\afFn_{,n}\times\afBr_{,r}$ and
$\afBr_{,r}\times\afBr_{,r}$ can always be regarded as $G$-stable
subsets of $\afFn_{,N}\times\afFn_{,N}$, and the $G$-orbit $\sO_A$
containing $(\bfL,\bfL')$ is the $G$-orbit $\sO_{\ti A}$ containing
$(\ti\bfL,\ti\bfL')$, where $A=(a_{i,j})$ and $\ti A=(\ti a_{i,j})$
are related by, for all $m\in\mbz$,
\begin{equation}\label{AtoAtilde}
\ti a_{k,l+mN}=\begin{cases} a_{k,l+mn}, &\text{ if }1\leq k,l\leq n;\\
0, &\text{ if either }n< k\leq N\text{ or }n< l\leq N.\end{cases}
\end{equation}

\begin{Lem} \label{tsp} Let $N=\text{\rm max}\{n,r\}$. By sending $e_A$ to $e_{\ti A}$, both
$\afSr$ and $\afHr$ can be identified as (centralizer) subalgebras
of $\sS_\vtg(N,r)$, and $\sT_\vtg(n,r)$ as a {\rm sub}bimodule of
the $\afSr$-$\afHr$-bimodule $\sS_\vtg(N,r)$.
\end{Lem}

\begin{proof}Define $\og\in \La_\vtg(N,r)$ by setting
\begin{equation}\label{111}
\og=\begin{cases}(\ldots,1^r,1^r,\ldots),&\text{ if }n\leq r\\
(\ldots,1^r,0^{n-r},1^r,0^{n-r},\ldots),&\text{ if }n> r.\\
\end{cases}
\end{equation}
For $\la\in\afLanr$, let
$\diag(\la)=(\dt_{i,j}\la_i)_{i,j\in\mbz}\in\afThnr$. If we embed
$\afLanr$ into $\La_\vtg(N,r)$ via the map $\mu\mapsto \ti\mu$,
where
$$\ti\mu=(\ldots,\mu_1,\ldots,\mu_n,0^{N-n},\mu_{n+1},\ldots,\mu_{2n},0^{N-n},\ldots),$$
 and put $e=\sum_{\mu\in\afLanr}e_{\diag(\ti\mu)}$ and
$e_\og=e_{\diag(\og)}$, then $\afSr\cong e\sS_\vtg(N,r)e$ and
$\afHr\cong e_{\og}\sS_\vtg(N,r)e_{\og}$, and $\sT_\vtg(n,r)\cong e\sS_\vtg(N,r)e_{\og}$ as
$\afSr$-$\afHr$-bimodules.
Here all three isomorphisms send $e_A$ to $e_{\ti A}$.
\end{proof}

For $A\in\afThnr$, let
\begin{equation}\label{sqAdA}
[A]=\up^{-d_A}e_{A},\quad\text{ where } \quad
d_{A}=\sum_{1\leq i\leq n\atop i\geq k,j<l}a_{i,j}a_{k,l}.
\end{equation}
(See \cite[4.1(b),4.3]{Lu99} for a geometric meaning of $d_A$.)\index{$d_A$, as in $[A]:=\up^{-d_A}e_A$}
 Then for $\la\in\afLanr$ and $A\in\afThnr$, we have
\begin{equation}
\label{product [diag(la)][A] in affine q-Schur algebra}
\begin{aligned}
\ [\diag(\la)]\cdot[A]=
\begin{cases}[A], & \text{if}\ \lambda=\ro(A);\\
0, & \text{otherwise;}
\end{cases}
\end{aligned} \quad \text{and}\
\begin{aligned}
\ [A][\diag(\la)]=
\begin{cases}[A], & \text{if}\ \lambda=\co(A);\\
0, & \text{otherwise.}
\end{cases}
\end{aligned}
\end{equation}
 Moreover, the $\sZ$-linear map
\begin{equation}\label{taur}
\tau_r:\afSr\lra\afSr,\;\;[A]\lm [\tA]
\end{equation} is an algebra
anti-involution.

We end this section with a close look at the basis $\{[A]\mid A\in\Thnrr\}$ for
$\sT_\vtg(n,r)$ from its specialization $\mbc_G(\scrY\times\scrX)$.
Let
\begin{equation}\label{afInr}
\afInr=\{\bfi=(i_k)_{k\in\mbz}\mid i_k\in\mbz,i_{k+r}=i_k+n\text{
for all $k\in\mbz$}\}.\end{equation} We may identify the elements of
$\afInr$ with functions $\bfi\,:\,\mbz\to\mbz$ satisfying \linebreak
$\bfi(s+r)=\bfi(s)+n$ for all $s\in\mbz$. (Note that a periodic
function $\bfi$ can be identified with $(i_1,\ldots,i_r)\in\mbz^r$,
where $i_s=\bfi(s)$ for all $s$. For more details, see \eqref{afInr4} below.)
Clearly, there is a bijection
\begin{equation}\label{afInr1}
\afInr\lra\Thnrr,\quad \bfi\longmapsto A^\bfi
\end{equation}
where $A^\bfi=(a_{k,l})$ with $a_{k,l}=\delta_{k,i_l}$. Thus, the
orbits of the diagonal action of $G$ on $\scrY\times
\scrX$ are labeled by the elements
of $\afInr$, and $\sO_\bfi:=\sO_{A^\bfi}$ is the orbit of the pair
$(\bfL_\bfi,\bfL_\emptyset)$, where the $i$th lattices of $\bfL_\bfi,\bfL_\emptyset$
are defined by
\begin{equation}\label{lattice Li}
\bfL_{\bfi,i}=\bigoplus_{\bfi(j)\leq i}\field\, v_j
\quad\text{and}\quad \bfL_{\emptyset,i}=\bigoplus_{j\leq i}\field\,
v_j.
\end{equation}
 Here $v_1,\ldots,v_r$ is a fixed
$\field[\ep,\ep^{-1}]$-basis of $V$ and $v_{i+rs}=\ep^{-s}\, v_i$
for all $s\in\mbz$. This is because the difference set
$$\{j\mid i_j\le k,j\le l\}\backslash
(\{j\mid i_j\le k-1,j\le l\}\cup\{j\mid i_j\le k,j\le l-1\})=\begin{cases}\{l\}, &\text{ if }k=i_l;\\
\emptyset, &\text{ otherwise.}\end{cases}$$

Let $d_\bfi=d_{A^\bfi}$ (see \eqref{sqAdA}). If we set, for each $\bfL\in \scrY$,
$$\scrX_{\bfi,\bfL}=\{\bfL'\in \scrX\mid(\bfL,\bfL')\in \sO_{\bfi}\},$$
then, by \cite[Lem.~4.3]{Lu99}, $d_\bfi$ is the dimension of $\scrX_{\bfi,\bfL}$ (in case where $\field$ is an
algebraically closed field). More precisely, a
direct calculation gives the following result.

\begin{Lem}\label{di} For each
$\bfi=(i_j)\in\afInr$, let
$$\text{\rm Inv}(\bfi)=\{(s,t)\in\mbz^2\mid
1\leq s\leq r,\, s<t,\,i_s\geq i_t\}.$$ Then $d_\bfi=|\text{\rm
Inv}(\bfi)|$.
\end{Lem}

\begin{proof} 
Applying \cite[Lem.~4.3]{Lu99} gives
\begin{equation*}
\begin{split}
d_\bfi&=d_{A^\bfi}=\sum_{1\leq t\leq r\atop s\geq k,\,t<l}
\dt_{s,i_t}\dt_{k,i_l}=\sum_{1\leq t\leq r}
\sum_{s\in\mbz}\dt_{s,i_t}(\sum_{s\geq
k,\,t<l}\dt_{k,i_l})=\sum_{1\leq t\leq r}(\sum_{i_t\geq
k,\,t<l}\dt_{k,i_l})\\
&=|\text{Inv}(\bfi)|,
\end{split}
\end{equation*}
as required.
\end{proof}
With the identification of \eqref{iso-Schur-Function}, we have
$$[A^\bfi]=q^{-\frac 12d_\bfi}\chi_\bfi,$$
where $\chi_\bfi$ is the characteristic
function of the orbit $\sO_\bfi$.

\section{Affine Hecke algebras of type $A$: the algebraic definition}
We now follow \cite{Gr99,VV99} to interpret affine quantum  Schur
algebras as endomorphism algebras of certain tensor spaces over the
affine Hecke algebras associated with the extended affine Weyl
groups of type $A$.

Let $\afsygr$ be the {\it affine symmetric group}\index{$\afsygr$,
affine symmetric group} consisting of all permutations
$w:\mbz\ra\mbz$ such that $w(i+r)=w(i)+r$ for $i\in\mbz$. This is
the subset of all bijections in $I_\vtg(r,r)$ which is defined in
\eqref{afInr}. Hence, $\text{Inv}(w)$ is well-defined. More
precisely, for any $a\in\mbz$, if we set
$$\text{Inv}(w,a)=\{(s,t)\in\mbz^2\mid
a+1\leq s\leq a+r,\, s<t,\,w(s)>w(t)\},$$
then $\text{Inv}(w)=\text{Inv}(w,0)$ and $|\text{Inv}(w)|=|\text{Inv}(w,a)|$ for all $a\in\mbz$.

There are several useful subgroups of $\affSr$. The subgroup $W$ of
$\afsygr$ consisting of $w\in\afsygr$ with
$\sum_{i=1}^rw(i)=\sum_{i=1}^ri$ is the affine Weyl group of type
$A$\index{affine Weyl group of type $A$, $W$} with generators $s_i$
($1\leq i\leq r$) defined by setting
$$s_i(j)=\begin{cases}j, &\text{ for $j\not\equiv i,i+1\mnmod r$};\\
 j-1,&\text{ for
$j\equiv i+1\mnmod r$};\\
j+1,&\text{ for $j\equiv i\mnmod r$.}
\end{cases}$$
If $S=\{s_i\}_{1\le i\le r}$, then $(W,S)$ is a Coxeter system. For
convenience of writing consecutive products of the form
$s_is_{i+1}s_{i+2}\cdots$ or  $s_is_{i-1}s_{i-2}\cdots$, we set
$s_{i+kr}=s_i$ for all $k\in\mbz$.

Observe that the cyclic subgroup $\langle\rho\rangle$ of $\affSr$ generated by the permutation $\rho$
of $\mbz$ sending $j$ to $j+1$, for all $j$, is in the complement of $W$. Observe also that
$s_{j+1}\rho=\rho s_j$ for all $j\in\mbz$.

The subgroup $A$ of $\afsygr$ consisting of permutations $y$ of
$\mbz$ satisfying $y(i)\equiv i(\text{mod}\,r)$ is isomorphic to
$\mbz^r$ via the map $y\mapsto
\la=(\la_1,\la_2,\ldots,\la_r)\in\mbz^r$ where $y(i)=\la_ir+i$ for
all $1\leq i\leq r$. We will identify $A$ with $\mbz^r$ in the
sequel. In particular, $A$ is generated by
$\bse_i=(0,\ldots,0,\underset {(i)}1,0,\ldots,0)$ for $1\leq i\leq
r$. Moreover, the subgroup of $W$ generated by $s_1,\ldots,s_{r-1}$
is isomorphic to the symmetric group $\fS_r$.

Recall the matrix set $\afThrr_\og$ defined in \eqref{matrixafSr1} which is clearly a group with matrix multiplication.

The following result is well-known; see, e.g., \cite[Prop.~1.1.3 \& 1.1.5]{Gr99} for the last isomorphism.

\begin{Prop} \label{affine symmetric groups}Maintain the notation above. There are group isomorphisms:
$$\affSr\cong\afThrr_\og\cong \fS_r\ltimes \mbz^r\cong\langle\rho\rangle\ltimes W.$$
\end{Prop}

\begin{proof} The first isomorphism is the restriction of the bijection defined in \eqref{afInr1} for $n=r$; see also \eqref{matrixafSr2}.
In particular, every $w\in\afsygr$ is sent to the matrix
$A_w=(a_{k,l})$ with $a_{k,l}=\delta_{k,w(l)}$. The second is seen
easily since every $w\in\afsygr$ can be written as $x\la$, where
$x\in\fSr$ and $\la\in \mbz^r$ are defined uniquely by $w(i)=\la_i
r+x(i)$ for all $1\le i, x(i)\le r$. If $\bse_k\in\affSr$
($k\in[1,r]$) denotes the permutation $\bse_k(i)=i$ for $i\neq k$,
$i\in[1,r]$ and $\bse_k(k)=r+k$, then $\bse_k=\rho s_{r+k-2}\cdots
s_{k+1}s_k$, proving the last isomorphism.
\end{proof}

 Let $\ell$ be the
length function of the Coxeter system $(W,S)$. \index{length
function} Then $\ell(w)$ for $w\in W$ is the length $m$ such that
$w=s_{i_1}s_{i_2}\cdots s_{i_m}$ is a reduced expression. By
\cite{IM},
$$\ell(w)=\sum_{1\le i<j\le r}\bigg|\bigg[\frac{w(j)-w(i)}r\bigg]\bigg|.$$ Extend the length function to
$\affSr$ be setting $\ell(\rho^iw')=\ell(w')$ for all $i\in\mbz$ and $w'\in W$. Since
Inv$(\rho^i)=\emptyset$, induction on $\ell(w)$ shows that
\begin{equation}\label{inversion=length}
\ell(w)=\text{Inv}(w),\qquad\text{ for all }w\in\affSr.
\end{equation}

The group $\affSr$ acts
on the set $\afInr$ given in \eqref{afInr} by place permutation:
\begin{equation}\label{place permutation}
\bfi w=(i_{w(k)})_{k\in\mbz},\quad\text{
for $\bfi\in \afInr$ and $w\in\affSr$.}
\end{equation}
 Clearly, every $\affSr$-orbit has a unique representative
in the {\it fundamental set}
\begin{equation}\label{afInr0}\aligned
\afJnr&=\{\bfi\in \afInr\mid 1\leq i_1\leq
i_2\leq\cdots\leq i_r\leq n\}\\
&=\{\bfi_\la^\vtg\mid \la\in\afLanr\},\endaligned
\end{equation}
where $\bfi_\la^\vtg=(i_s)_{s\in\mbz}\in \afJnr$ if and only if  $i_s=m$ for all $s\in
R_m^\la$ and $m\in\mbz$, or equivalently,
$\la_j=|\{k\in\mbz\mid i_k=j\}|$ for all $j\in\mbz$. Written in full,
$\bfi_\la^\vtg$ is the sequence
$$(\ldots,\underbrace{0,\ldots,0}_{\la_n},\underbrace{1,\ldots,1}_{\la_1},\underbrace{2,\ldots,2}_{\la_2},
\ldots,\underbrace{n,\ldots,n}_{\la_n},\underbrace{1+n,\ldots,1+n}_{\la_1},
\underbrace{2+n,\ldots,2+n}_{\la_2},\ldots).$$
Observe that every number of the form $i+kn$, for $1\leq i\leq n$, $k\in\mbz$, determines a
constant subsequence $(i_j)_{j\in R_{i+kn}^{\la}}=(i+kn,\ldots,i+kn)$ of $\bfi=\bfi_\la^\vtg$ (of length $\la_i$) indexed by the set
\begin{equation}\label{set R}
R_{i+kn}^{\la}=\{\la_{k,i-1}+1,\la_{k,i-1}+2,\ldots,\la_{k,i-1}+\la_i=\la_{k,i}\},
\end{equation}
where $\la_{k,i-1}=kr+\sum_{1\leq t\leq i-1}\la_t$. These sets form
a partition $\bigcup_{j\in\mbz}R_{j}^{\la}$ of
$\mbz$.\index{$R_{j}^{\la},j\in\mbz$, partition of $\mbz$} Note that
the fundamental set $I_\vtg(n,r)_0$ is obtained by shifting the one
used in \cite{VV99} by $n$.

For $\la,\mu\in\afLanr$, let $\fS_\la:=\fS_{(\la_1,\ldots,\la_n)}$
be the corresponding standard Young subgroup of $\fS_r$
\index{$\fS_\la$, Young subgroup associated with $\la$} (and hence,
of $\affSr$), and let
$$\afmsD_\la=\{d\mid d\in\affSr,\ell(wd)=\ell(w)+\ell(d)\text{ for
$w\in\fS_\la$}\}.$$
By \eqref{inversion=length}, one sees easily that
\begin{equation}\label{minimal coset representative}
\aligned
d^{-1}\in\afmsD_\la
&\iff d(\la_{k,i-1}+1)<d(\la_{k,i-1}+2)<\cdots<d(\la_{k,i-1}+\la_i),\,\forall 1\leq i\leq n,k\in\mbz\\
&\iff d(\la_{0,i-1}+1)<d(\la_{0,i-1}+2)<\cdots<d(\la_{0,i-1}+\la_i),\,\forall 1\leq i\leq n.\endaligned
\end{equation}
The
following result is the affine version of a well-known result for symmetric groups (see,~e.g.,~\cite[Th.4.15,\,(9.1.4)]{DDPW}.
It can be deduced from \cite[7.4]{VV99}. For a proof, see
\cite[9.2]{DF09}).

\begin{Lem}\label{the map jmath} Let $\afmsD_{\la,\mu}=\afmsD_{\la}\cap{\afmsD_{\mu}}^{-1}$.
There is a bijective map
\begin{equation}\label{Sjmath}
\jmath_\vtg:\{(\la, w,\mu)\mid
w\in\afmsD_{\la,\mu},\la,\mu\in\afLanr\}\lra\afThnr
 \end{equation} \index{$\jmath_\vtg$, bijection between double cosets and $\afThnr$}
 sending $(\la, w,\mu)$ to $A=(a_{k,l})$, where, if $\bfi_\la^\vtg=(i_a)_{a\in\mbz}$ and $\bfi_\mu^\vtg=(j_a)_{a\in\mbz}$, then
\begin{equation}\label{Ajmath}
a_{k,l}=|\{t\in\mbz\mid i_{w(t)}=k,j_t=l\}|=|R_k^\la\cap wR_l^\mu|
\end{equation} for all $k,l\in\mbz$.
In particular, by certain appropriate embedding, restriction gives
two bijections
\begin{equation}\label{Tjmath}
\jmath_\vtg:\{(\la, w,\omega)\mid
w\in\afmsD_{\la,\mu},\la\in\afLanr\}\lra\Thnrr\end{equation} and
$\jmath_\vtg:\{(\og, w,\omega)\mid w\in\affSr\}\ra\afThrr_\og.$
\end{Lem}

We remark that, for $w'\in\fS_\la w\fS_\mu$, the equality
$|R_k^\la\cap w'R_l^\mu|=|R_k^\la\cap wR_l^\mu|$ holds. Hence, the
matrix $A$ is completely determined by $\la,\mu$ and the double
coset $\fS_\la w\fS_\mu$, and is independent of the selection of the
representative. Moreover, if $\jmath_\vtg(\la,d,\mu)=A$, then
$\jmath_\vtg(\mu,d^{-1},\la)=\tA$, the transpose of $A$.


\begin{Coro}\label{double coset}
For $\la,\mu\in\afLanr$ and $d\in\msD_{\la,\mu}^\vtg$ with
$\jmath_\vtg(\la,d,\mu)=A\in\afThnr$, let $\nu^{(i)}$ be the composition of $\la_i$ obtained by
 removing all zeros from row $i$ of $A$. Then $\frak S_\la
\cap d\frak S_\mu d^{-1}=\frak S_\nu$, where
$\nu=(\nu^{(1)},\ldots,\nu^{(n)})$.
\end{Coro}
\begin{proof}
Let $\bfi^\vtg_\mu=(j_s)_{s\in\mbz}\in\afInr$. Then $j_s=k$ for all $s\in
R_k^\mu$ and $k\in\mbz$. Thus,  $l\in R_i^\la$ and $d^{-1}(l)\in R_j^\mu\iff l\in R_i^\la\cap dR_j^\mu$.  For $1\leq i\leq n$, if
$$I_i:=(\underbrace{k,\ldots,k}_{a_{i,k}})_{k\in\mbz}
=(\ldots,\underbrace{1,\ldots,1}_{a_{i,1}},\underbrace{2,\ldots,2}_{a_{i,2}},
\ldots,\underbrace{n,\ldots,n}_{a_{i,n}},\ldots)\in \mbz^{\la_i},$$
then, with the notation used in \eqref{set R}, \eqref{minimal coset representative} together with Lemma \ref{the map jmath}
implies $I_i=(j_{d^{-1}(\la_{0,i-1}+1)},\ldots,j_{d^{-1}(\la_{0,i-1}+\la_i)})$. Hence,
$(j_{d^{-1}(1)},\ldots,j_{d^{-1}(r)})=(I_1,\ldots,I_n)$. Since $\frak
S_\mu=\text{Stab}_{\affSr}(I_\mu)$, it follows that $\fS_\la\cap d\frak S_\mu
d^{-1}=\fS_\la\cap \text{Stab}_{\affSr}(I_\mu d^{-1})=\text{Stab}_{\fS_\la}(I_\mu d^{-1})=\frak S_\nu$.
\end{proof}

\begin{Coro}\label{di=length}There is a bijective map
$$\jmath_\vtg^*:\afInr\lra \{(\la,d,\og)\mid d\in\afmsD_\la,\la\in\afLanr\},\qquad\bfi\longmapsto (\la,d,\og),$$
where $\bfi=\bfi_\la^\vtg d$. Moreover, if $d^+$ denotes a
representative $\fS_\la d$ with maximal length, then
$d_\bfi=|\text{\rm Inv}(\bfi)|=\ell(d^+)$.
\end{Coro}

\begin{proof} Clearly, $\jmath_\vtg^*$ is the composition of the bijection given in \eqref{afInr1} and the inverse
of $\jmath_\vtg$ given in \eqref{Tjmath}. We now prove the last statement.

Let $\sL=\{(s,t)\in\mbz^2\mid
1\leq s\leq r,\; s<t\}$ and $\bfj=\bfi_\la^\vtg$. Then
$$\text{\rm Inv}(\bfi)=\{(s,t)\in\sL\mid
j_{d(s)}\geq j_{d(t)}\}=X_1\cup X_2 \quad(\text{a disjoint
union}),$$  where $X_1=\{(s,t)\in\sL\mid j_{d(s)}> j_{d(t)}\}$ and
$X_2=\{(s,t)\in\sL\mid j_{d(s)}= j_{d(t)}\}$. Since $(s,t)\in X_2$
if and only if $d(s),d(t)\in R_k^\la$ for some $k\in\mbz$,
\eqref{minimal coset representative} forces $d(s)<d(t)$. Hence, if
$\affSr$ acts on $\mbz^2$ diagonally, then
$$\aligned
X_2&=\{(s,t)\in\sL\mid d(s)<d(t), d(s),d(t)\in R_k^\la\text{ for some }k\in\mbz\}\\
&=d^{-1}\{(s,t)\in\mbz^2\mid 1\leq d^{-1}(s)\leq r, s<t, s,t\in R_k^\la\text{ for some }k\in\mbz\}.\\
\endaligned$$
Thus, $|X_2|$ is the length of the longest element $w_{0,\la}$ in $\frak S_\la$ by Lemma \ref{w0la}.
Since, for $(s,t)\in\sL\backslash X_2$,
$j_{d(s)}>j_{d(t)}\iff d(s)>d(t)$, applying \eqref{minimal coset representative} again yields
$$\text{\rm Inv}(d)=\{(s,t)\in\sL\mid
d(s)>d(t)\}=\{(s,t)\in\sL\backslash X_2\mid
d(s)>d(t)\}=X_1.$$
Consequently, $d_\bfi=|X_1|+|X_2|=\ell(d)+\ell(w_{0,\la})=\ell(d^+)$, as required.
\end{proof}

We record the following generalization (to affine symmetric groups) of a well-known result for Coxeter groups.

\begin{Lem}\label{decomposition of double coset}
Let $\la,\mu\in\afLanr$ and $d\in\afmsD_{\la,\mu}$. Then
$d^{-1}\fS_\la d\cap\fS_\mu$ is a standard Young subgroup of
$\fS_\mu$. Moreover, each element $w\in\fS_\la d\fS_\mu$ can be
written uniquely as a product $w=w_1dw_2$ with $w_1\in\fS_\la$ and
$w_2\in\afmsD_\nu\cap\fS_\mu$, where $\nu\in\afLanr$ is defined by
$\fS_\nu=d^{-1}\fS_\la d\cap\fS_\mu$, and the equality
$\ell(w)=\ell(w_1)+\ell(d)+\ell(w_2)$ holds.
\end{Lem}

Following \cite{Jones94}, the extended affine Hecke algebra
$\sH(\affSr)$ over $\sZ$ \index{affine Hecke algebra!$\sim$ of
$\affSr$, $\sH(\affSr)$} is defined to be the algebra generated by
$T_{s_i}$ ($1\leq i\leq r$), $T_{\rho}^{\pm 1}$ with the following
relations:
\begin{equation*}
\aligned
&T_{s_i}^2=(\up^2-1)T_{s_i}+\up^2,\\
&T_{s_i}T_{s_j}=T_{s_j}T_{s_i}\quad\text{($i-j\not\equiv\pm 1\nmod r$)},\\
&T_{s_i}T_{s_{j}}T_{s_i}=T_{s_{j}}T_{s_i}T_{s_{j}}\quad\text{($i-j\equiv\pm 1\nmod r$ and $r\geq 3$)},\\
&T_{\rho}T_{\rho}^{-1}=T_{\rho}^{-1}T_{\rho}=1\\
&T_{\rho}T_{s_i}=T_{s_{i+1}}T_{\rho},
\endaligned
\end{equation*}
where $T_{s_{r+1}}=T_{s_1}$. This algebra has a $\sZ$-basis $\{T_w\}_{w\in\affSr}$, where
$T_w=T_{s_{i_1}}\cdots T_{s_{i_m}}$ if $w=s_{i_1}\cdots s_{i_m}$ is reduced.

The following result is well-known due to Iwahori--Matsumoto \cite{IM}. Recall the algebra
$\afHr$ defined in \S3.1 and the isomorphism given in \eqref{iso-Schur-Function}.
\begin{Lem}\label{Hecke-Alg-geom-construction}
There is a $\sZ$-algebra isomorphism $\sH(\affSr)\cong\afHr$ whose specialization of $\up$ to $\sqrt{q}$ gives a $\mbc$-algebra isomorphism
$$\sH(\affSr)\otimes \mbc\cong\sHq.$$
\end{Lem}
Thus, {\it we will identify $\sH(\affSr)$ with $\afHr$ in the sequel.}

Let $\Hr=\sH(\fS_r)$ be the subalgebra of $\afHr$ generated by $T_{s_i}$
($1\leq i<r$). Then $\Hr$ is the Hecke algebra of the symmetric
group $\fSr$. We finally set
$$\afbfHr=\afHr\ot_{\sZ}\mbq(\up)\;\;\text{and}\;\;
\bsHr=\Hr\ot_{\sZ}\mbq(\up).$$\index{affine Hecke algebra! $\sim$
over $\mbq(\up)$, $\afbfHr$}

For each $\la\in\afLanr$, let
$x_\la=\sum_{w\in\fS_\la}T_w\in\Hr$ and define
$$\sS^{_\sH}_\vtg(n,r):=\End_{\afHr}\biggl(\bop_{\nu\in\La_\vtg(n,r)}x_\nu\afHr\biggr).$$
\index{affine quantum Schur algebra!the Hecke algebra definition,
$\sS^{_\sH}_\vtg(n,r)$} For $\la,\mu\in\afLanr$ and
$d\in\afmsD_{\la,\mu}$, define
$\phi_{\la,\mu}^d\in\sS^{_\sH}_\vtg(n,r)$ as follows:
\begin{equation}\label{def of standard basis}
\phi_{\la,\mu}^d(x_\nu h)=\dt_{\mu\nu}\sum_{w\in\fS_\la
d\fS_\mu}T_wh
\end{equation}
where $\nu\in\afLanr$ and $h\in\afHr$. Then the set
$\{\phi_{\la,\mu}^d\}$ forms a basis for $\sS^{_\sH}_\vtg(n,r)$.

\begin{Rems}\label{base change for afSr}
(1) We point out that, as a natural generalization of the $q$-Schur
algebra given in \cite{DJ89,DJ91}, the endomorphism algebra
$\sS^{_\sH}_\vtg(n,r)$ is called the {\it affine $q$-Schur algebra}
and the basis $\{\phi_{\la,\mu}^d\}$ was constructed by R. Green
\cite{Gr99}.

(2)
Let $R$ be a commutative ring with 1 which is a $\sZ$-algebra.
Then, by base change to $R$, a similar basis can be defined for
$$\sS^{_\sH}_\vtg(n,r;R)=\End_{\afHr_R}\biggl(\bop_{\la\in\afLanr}x_\la\afHr_R\biggr).$$ As a result of this, the
endomorphism algebra $\sS^{_\sH}_\vtg(n,r)$ satisfies the base change property:
$\sS^{_\sH}_\vtg(n,r;R)\cong\sS^{_\sH}_\vtg(n,r)_R.$ This property has already been mentioned for $R=\mbz[\sqrt q,\sqrt{q}^{-1}]$ in Remark \ref{implicit use bcp}.
\end{Rems}
 Combining the base change property and \cite[7.4]{VV99} gives the following result which extends
the isomorphism given in Lemma
\ref{Hecke-Alg-geom-construction} to affine quantum Schur
algebras.

\begin{Prop} \label{Green defn=geom defn}
The bijection $\jmath_\vtg$ given in Lemma \ref{the map jmath} induces a $\sZ$-algebra isomorphism
$${\mathfrak h}:\afSr\stackrel{\sim}{\lra}\sS^{_\sH}_\vtg(n,r),\;
e_A\lm \phi_{\la,\mu}^d$$ for all $A\in\afThnr$ with
$A=\jmath_\vtg(\la,d,\mu)$, where $\la,\mu\in\afLanr$ and
$d\in\afmsD_{\la,\mu}$. Moreover, regarding
$\bop_{\la\in\afLanr}x_\la\afHr$ as an $\afSr$-module via $\mathfrak
h$, we obtain an $\afSr$-$\afHr$-bimodule isomorphism
$$ev:\sT_\vtg(n,r)\overset\sim\lra\bop_{\la\in\afLanr}x_\la\afHr,\,\,\, e_A\longmapsto x_\la T_d,$$
for all $A\in\Thnrr$ with $A=\jmath_\vtg(\la,d,\og)$.
\end{Prop}

Note that if we regard $\sT_\vtg(n,r)$ as a subset of
$\sS_\vtg(N,r)$ as in Lemma \ref{tsp}, the bimodule isomorphism is
simply the evaluation map.

Recall that, by removing the supscript $^\vtg$, the notation
$\mathscr D_{\la,\mu}$ denotes the shortest
$(\fS_\la,\fS_\mu)$-coset representatives in $\fS_r$. If we identify
$\afLanr$ with $\La(n,r)$ via \eqref{flat2}, we obtain the
following.

\begin{Coro}\label{QSA}
The subspace spanned by all $\phi_{\la,\mu}^d$ with
$\la,\mu\in\afLanr$ and $d\in\mathscr D_{\la,\mu}$ is a subalgebra
which is isomorphic to the quantum Schur algebra $\sS(n,r)$.
\end{Coro}

Using the evaluation isomorphism, we now describe an explicit action of $\afHr$ on $\sT_\vtg(n,r)$.
First, for $\la\in\afLanr$, $d\in\afmsD_\la$ and $1\leq k\leq r$,
\begin{equation}\label{action1-VV}
x_\la T_d \cdot T_{s_k}=
\begin{cases}
\up^2x_\la T_d,&\text{if $ds_k\not\in\afmsD_\la\, ($then $\ell(ds_k)>\ell(d))$};\\
x_\la T_{ds_k},&\text{if $\ell(d{s_k})>\ell(d)$ and
$ds_k\in\afmsD_\la$};\\
\up^2x_\la T_{ds_k}+(\up^2-1)x_\la T_d,&\text{if
$\ell(d{s_k})<\ell(d) ($then $ds_k\in\afmsD_\la)$}.
\end{cases}
\end{equation}
Second, by Corollary \ref{di=length}, we obtain
$ev([A^\bfi])=\up^{-\ell(d^+)}x_\la T_d$ if
$\jmath_\vtg^*(\bfi)=(\la,d,\og)$, where $d^+$ is a representative
of $\fS_\la d$ with maximal length. For $w\in\affSr$, let $\ti
T_w=v^{-\ell(w)}T_w$. Thus, for $\bfj=\bfi_\la^\vtg$ and $d$ as
above, \eqref{action1-VV} becomes
\begin{equation*}
[A^{\bfj d}]\ti T_{s_k}=
\begin{cases}
\up[A^{\bfj d}],&\text{if $ds_k\not\in\afmsD_\la\, ($then $\ell(ds_k)>\ell(d))$};\\
[A^{\bfj ds_k}],&\text{if $\ell(d{s_k})>\ell(d)$ and
$ds_k\in\afmsD_\la$};\\
[A^{\bfj ds_k}]+(\up-\up^{-1})[A^{\bfj d}],&\text{if
$\ell(d{s_k})<\ell(d) ($then $ds_k\in\afmsD_\la)$}.
\end{cases}
\end{equation*}
This together with Lemma \ref{di} gives the first part of the
following (and part (2) is clear from definition).

\begin{Prop} \label{action2-VV}
Let $\bfi\in\afInr$.
\begin{itemize}
\item[(1)] For any $1\leq k\leq r$, we have
\begin{equation*}
[A^{\bfi}]\ti T_{s_k}=
\begin{cases}
\up[A^{\bfi}],&\text{if $i_k=i_{k+1}$};\\
[A^{\bfi s_k}],&\text{if $i_k<i_{k+1}$};\\
[A^{\bfi s_k}]+(\up-\up^{-1})[A^{\bfi }],& \text{if $i_k>i_{k+1}$}.
\end{cases}
\end{equation*}
\item[(2)] $[A^{\bfi}]T_\rho=[A^{\bfi\rho}]$, where $\rho\in\fS_{\vtg,r}$ is the permutation  sending $i$ to $i+1$ for all $i\in\mbz$.
\end{itemize}
\end{Prop}

\section{The tensor space interpretation}

We now interpret the right $\afHr$-module $\sT_\vtg(n,r)$ in
terms of the tensor space, following \cite{VV99}.

The Hecke algebra $\sH(\affSr)$ admits the so-called {\it Bernstein
presentation}\index{Bernstein presentation} which consists of
generators
$$T_i:=T_{s_i},\quad X_j:=\ti T_{\bse_1+\cdots+\bse_{j-1}}\ti T_{\bse_1+\cdots+\bse_{j}}^{-1}(\text{$i=1,\ldots,r-1$, $j=1,\ldots,r$}),$$
 and relations
$$\aligned
 & (T_i+1)(T_i-\up^2)=0,\\
 & T_iT_{i+1}T_i=T_{i+1}T_iT_{i+1},\;\;T_iT_j=T_jT_i\;(|i-j|>1),\\
 & X_iX_i^{-1}=1=X_i^{-1}X_i,\;\; X_iX_j=X_jX_i,\\
 & T_iX_iT_i=\up^2 X_{i+1},\;\;X_jT_i=T_iX_j\;(j\not=i,i+1).
\endaligned$$

Note that, for any dominant $\la=(\la_i)\in\mbz^r$ (meaning $\la_1\geq \cdots\geq\la_r$),
 $$X_\la:=X_1^{\la_1}\cdots X_r^{\la_r}=\ti T_\la^{-1}.$$
In particular, it takes $T_\rho=\ti T_\rho$ to
$X_1^{-1}\ti T_1^{-1}\cdots\ti T_{r-1}^{-1}$ since $\bse_1=\rho s_{r-1}\cdots s_2s_1$ and $X_1^{-1}=\ti T_{\bse_1}$.

By definition, we have, for
each $1\leq i\leq r-1$,
\begin{equation}\label{formula-in-HeckeAlg}
\aligned
& T_iX_{i+1}^{-1}=X_i^{-1}T_i+(1-\up^2)X_i^{-1},\\
& X_{i+1}^{-1}T_i=T_i X_i^{-1}+(1-\up^2)X_i^{-1},\\
& T_i^{-1}X_i^{-1}=X_{i+1}^{-1} T_i^{-1}+(1-v^{-2})X_{i+1}^{-1},\\
& X_i^{-1}T_i^{-1}=T_i^{-1} X_{i+1}^{-1}+(1-v^{-2})X_{i+1}^{-1}.
\endaligned
\end{equation}
So $X_iT_i=T_iX_{i+1}+(1-\up^2)X_{i+1}$ and $T_iX_i=X_{i+1}T_i+(1-\up^2)X_{i+1}$, etc.

Let $\Og$ be the free $\sZ$-module with basis $\{\og_i\mid
i\in\mbz\}$.\index{$\Og$, free  $\sZ$-module with basis $\{\og_i\mid
i\in\mbz\}$} Consider the $r$-fold {\it tensor space} $\Og^{\ot
r}$\index{tensor space!$\sim$ over $\sZ$, $\Og^{\ot r}$} and, for
each $\bfi=(i_1,\ldots,i_r)\in\mbz^r$, write
$$\og_\bfi=\og_{i_1}\ot\og_{i_2}\ot\cdots\ot \og_{i_r}=\og_{i_1}\og_{i_2}\cdots \og_{i_r}\in\Og^{\ot r}.$$
We now follow \cite{VV99} to define a right $\afHr$-module structure
on $\Og^{\ot r}$ and to establish an $\afHr$-module isomorphism from
$\sT_\vtg(n,r)$ to $\Og^{\ot r}$.


Recall the set $I_\vtg(n,r)$ defined in \eqref{afInr} and the action
\eqref{place permutation} of $\fS_{\vtg,r}$ on $I_\vtg(n,r)$. If we
identify $\afInr$ with $\mbz^r$ by the following bijection
\begin{equation}\label{afInr4}
I_\vtg(n,r)\lra \mbz^r,\;\; \bfi\longmapsto (i_1,\ldots,i_r),
\end{equation}
then the action of $\fS_{\vtg,r}$ on $I_\vtg(n,r)$ induces an action on $\mbz^r$.
Also, the usual action of the place permutation of $\fS_r$ on $I(n,r)$, where
$$I(n,r)=\{(i_1,\ldots,i_r)\in\mbz^r\mid 1\leq i_k\leq
n\,\forall k\},$$
is the restriction to $\fS_r$ of the action of $\fS_{\vtg,r}$ on $\mbz^r$ (restricted to $I(n,r)$).
We often identify $I(n,r)$ as a subset of $I_\vtg(n,r)$, or $I_\vtg(n,r)_0$ as a subset of $I(n,r)$, depending on the context.

 By the Bernstein presentation for $\afHr$, Varagnolo--Vasserot
extended in \cite{VV99} the action of $\sH(r)$ on the
finite tensor space $\Og_n^{\ot r}$, where
$\Og_n=\text{span}\{\og_1,\ldots,\og_n\}$, given in \cite{Ji} (see
also \cite{Du}) to an action on $\Og^{\ot r}$ via the place permutation
above. In other words, $\Og^{\ot r}$
admits a right $\afHr$-module structure defined by
\begin{equation}\label{afH action}
\begin{cases}
\og_{\bf i}\cdot X_t^{-1}=\og_{\bfi{\bse}_t}
=\og_{i_1}\cdots\og_{i_{t-1}}\og_{i_t+n}\og_{i_{t+1}}\cdots\og_{i_r},\qquad \text{ for all }\bfi\in \mbz^r\\
{\og_{\bf i}\cdot T_k=\left\{\begin{array}{ll} \up^2\og_{\bf
i},\;\;&\text{if $i_k=i_{k+1}$;}\\
v\og_{\bfi s_k},\;\;&\text{if $i_k<i_{k+1}$;}\qquad\text{ for all }\bfi\in I(n,r),\\
v\og_{\bfi s_k}+(\up^2-1)\og_{\bf i},\;\;&\text{if
$i_{k+1}<i_k$,}
\end{array}\right.}
\end{cases}
\end{equation}
where $1\leq k\leq r-1$ and $1\le t\le r$.

 In general, for an arbitrary ${\bf i}=(i_1,\ldots,i_r)\in\mbz^r$,
there exist integers $m_1,\ldots,m_r$ such that $\og_{\bf i}\cdot
X_1^{m_1}\cdots X_r^{m_r}=\og_{\bf j}$ with ${\bf
j}=(j_1,\ldots,j_r)\in I(n,r)$. For example, if all $m_i\geq 0$,
then we define recursively by using \eqref{formula-in-HeckeAlg}
$$\aligned
&\qquad \og_{\bf i}\cdot T_k  =(\og_{\bf j}\cdot (X_1^{-m_1}\cdots
X_r^{-m_r}))\cdot T_k =\og_{\bf j}\cdot (X_1^{-m_1}\cdots
X_r^{-m_r}T_k)\\
&=(\og_{\bf j}\cdot T_k)\cdot (X_1^{-m_1}\cdots X_{k-1}^{-m_{k-1}}
X_{k+1}^{-m_k}X_k^{-m_{k+1}}X_{k+2}^{-m_{k+2}}\cdots
X_r^{-m_r}))\\
&\quad -\sum_{s=1}^{m_{k+1}}(\up^2-1)\og_{\bf j}\cdot
 (X_1^{-m_1}\cdots X_k^{-m_k}X_{k+1}^{-m_{k+1}+s}X_k^{-s}X_{k+2}^{-m_{k+2}}\cdots X_r^{-m_r}))\\
&\quad +\sum_{s=0}^{m_k-1}(\up^2-1)\og_{\bf j}\cdot
 (X_1^{-m_1}\cdots X_{k-1}^{-m_{k-1}}X_k^{-m_k+s}X_{k+1}^{-s}X_k^{-m_{k+1}}X_{k+2}^{-m_{k+2}}\cdots
X_r^{-m_r})).
\endaligned$$

Varagnolo--Vasserot have further established in
\cite[Lem.~8.3]{VV99} an $\afHr$-module isomorphism between
$\sT_\vtg(n,r)$ and $\Og^{\ot r}$. This result justifies why the set
$I_\vtg(n,r)_0$ defined in \eqref{afInr0} is called a fundamental
set. Recall from \eqref{afInr1} the matrix $A^\bfi$ defined for
every $\bfi\in I_\vtg(n,r)$.


\begin{Prop}\label{bimodule-isom} There is a {\rm unique} $\afHr$-module isomorphism
$$g:\sT_\vtg(n,r)\lra\Og^{\ot r}\, \text{ such that }\,[A^\bfi]\longmapsto\og_\bfi\text{ for all $\bfi\in I_\vtg(n,r)_0$},$$
which induces a $\sZ$-algebra isomorphism
$$\mathfrak t:\afSr\overset\sim\lra\sS^{\frak t}_\vtg(n,r):=\End_{\afHr}(\Og^{\ot r}).$$
 \index{affine quantum Schur algebra!the tensor space definition, $\sS^{\frak t}_\vtg(n,r)$}
 In particular, $g$ induces an
$\afSr$-$\afHr$-bimodule isomorphism. Moreover, specializing $\up$
to $\sqrt{q}$ yields a
$\mbc_G(\scrY\times\scrY)$-$\mbc_G(\scrX\times\scrX)$-bimodule
isomorphism over $\mbc$ from $\mbc_G(\scrY\times \scrX)$ to
$\Og^{\ot r}_\mbc:=\Og^{\ot r}\ot\mbc$ sending $[A^\bfi]=q^{-\frac 12 d_\bfi}\chi_\bfi$ to
$\og_\bfi$, $\forall \bfi\in
I_\vtg(n,r)_0$.
\end{Prop}

\begin{proof} The first assertion follows from
\cite[Lem.~9.5]{DF09}. By regarding $\Og^{\ot r}$ as an
$\afSr$-module via $\mathfrak t$, $g$ induces an
$\afSr$-$\afHr$-bimodule isomorphism. The last assertion follows
from the isomorphism \eqref{iso-Schur-Function} and the definition
that $\sT_\vtg(n,r)$ is the generic form of $\mbc_G(\scrY\times
\scrX)$.
\end{proof}

\begin{Rem}\label{two bases for tsp} As seen above, since the action of $\ti T_{s_i}$ for $1\leq i\leq r-1$ on the basis elements
$\og_\bfi$ for $\bfi\in I(n,r)$ follows the same rules as the action
on $[A^\bfi]$, it follows that $g[A^\bfi]=\og_\bfi$ for all $\bfi\in
I(n,r)$. However, the actions of $\ti T_{s_r}$ are different. Hence,
if we identify $\sT_\vtg(n,r)$ with $\Og^{\ot r}$ under $g$, then
$\{[A^\bfi]\}_{\bfi\in\afInr}$ and $\{[\og_\bfi]\}_{\bfi\in\afInr}$
form two different bases with the subset
$\{[A^\bfi]=\og_\bfi\}_{\bfi\in I(n,r)}$ in common.
\end{Rem}

We now identify $\mbc_G(\scrY\times \scrX)$ with $\Og^{\ot
r}_\mbc$. Consequently, specializing $\up$ to $\sqrt{q}$ gives
isomorphisms
$$\End_{\afHr_\mbc}(\Og^{\ot
r}_\mbc)\cong\End_{\afHr_\mbc}(\mbc_G(\scrY\times
\scrX))\cong\mbc_G(\scrY\times \scrY)\cong\afSr_\mbc.$$
 These algebras will be identified in the sequel.

Also, let $\bdOg=\Og\ot_{\sZ}\mbq(\up)$, i.e., $\bdOg$ is a
$\mbq(\up)$-vector space with basis $\{\og_i\mid i\in\mbz\}$. Then
the right action of $\afHr$ on $\Og^{\ot r}$ extends to a right
action of $\afbfHr=\afHr\ot_\sZ\mbq(\up)$ on $\bdOg^{\ot r}$.
Hence, we have the $\mbq(\up)$-algebra isomorphism
$$\afbfSr:=\afSr\ot_\sZ\mbq(\up)\cong \End_{\afbfHr}(\bdOg^{\ot
r}).$$\index{affine quantum Schur algebra! $\sim$ over $\mbq(\up)$,
$\afbfSr$}
 We will identify $\afSr$ and $\afbfSr$ with $\End_{\afHr}(\Og^{\ot r})$
and $\End_{\afbfHr}(\bdOg^{\ot r})$, respectively.

\section{BLM bases and multiplication formulas}

We now follow \cite{BLM} (cf. \cite{VV99}) to define BLM bases for the affine quantum Schur algebra
$\afSr$ as discussed in \cite{DF09}. Let
$$\afThnpm=\{A\in\afThn\mid a_{i,i}
=0\text{ for all $i$}\}.$$
 For $A\in\afThnpm$ and $\bfj\in\afmbzn$, define $A(\bfj,r)\in
\afSr$ by
\begin{equation}\label{def-A(j,r)}
A(\bfj,r)=\begin{cases}
\sum_{\la\in\La_\vtg(n,r-\sg(A))}\up^{\la\centerdot\bfj}[A+\diag(\la)],&\text{ if }\sg(A)\leq r;\\
0,&\text{ otherwise,}\end{cases}
\end{equation}\index{$A(\bfj,r)$,
BLM basis element}
 where $\la\cdot\bfj=\sum_{1\leq i\leq
n}\la_ij_i$. For the convenience of later use, we extend the
definition to matrices in $M_{n,\,\vtg}(\mbz)$ by setting
$A(\bfj,r)=0$ if some off-diagonal entries of $A$ are negative.

The elements $A(\bfj,r)$ are the affine version of the elements
defined in \cite[5.2]{BLM} and have been defined in terms of
characteristic functions in \cite[7.6]{VV99}.

 The following result is the affine analogue of \cite[6.6(2)]{DFW};
see \cite[Prop.~4.1]{DF09}.

\begin{Prop}\label{BLMbasis}
For a fixed $1\leq i_0\leq n$, the set
$$\sB_{\vtg,i_0,r}:=\{A(\bfj,r)\mid A\in\afThnpm,\bfj\in\afmbnn,j_{i_0}=0,\sg(\bfj)+\sg(A)\leq r\}$$
forms a $\mbq(\up)$-basis for $\afbfSr$.
In particular, the set
$$\sB_{\vtg,r}:=\{A(\bfj,r)\mid A\in\afThnpm,\bfj\in\afmbnn,\sg(A)\leq r\}$$ is a
spanning set for $\afbfSr$.
\end{Prop}
We call $\sB_{\vtg,i_0,r}$ a {\it BLM basis} of $\afbfSr$ and call
$\sB_{\vtg,r}$ the BLM spanning set. \index{BLM basis} \index{BLM
spanning set} As in the finite case, one would expect that there is
a basis $\sB_{\vtg}$ for quantum affine $\mathfrak {gl}_n$ such that
whose image in $\afbfSr$ is $\sB_{\vtg,r}$ for every $r\ge0$. See
\S5.4 for a conjecture.

The affine analogue of the multiplication
formulas given in \cite[5.3]{BLM} has also been established in \cite{DF09}.
As seen in \cite[Th.~4.2]{DF09} these formulas are crucial to a modified approach to the realization problem for quantum affine
$\frak{sl}_n$. We shall also see in Chapter 5 that they are useful in finding a presentation
for affine quantum Schur algebras of degree $(r,r)$.
 Let
$\afal_i=\afbse_i-\afbse_{i+1},\afbt_i=-\afbse_i-\afbse_{i+1}\in\afmbzn$.

\begin{Thm}\label{multiplication formulas in affine q-Schur algebra}
Assume $1\leq h\leq n$. For $i\in\mbz \,\, \bfj,\bfj'\in\afmbzn$
and $A\in\afThnpm$, if we put
$f(i)=f(i,A)=\sum_{j\geq i}a_{h,j}-\sum_{j>i}a_{h+1,j}$ and
$f'(i)=f'(i,A)=\sum_{j< i}a_{h,j}-\sum_{j\leq i}a_{h+1,j}$, then
the following identities hold in $\afbfSr$ for all $r\geq0$:
\begin{equation}\label{formula1 (1) in the affine q-Schur algebra}
\begin{split}
0(\bfj,r)A(\bfj',r)&=\up^{\bfj\centerdot\ro(A)}A(\bfj+\bfj',r),\\
A(\bfj',r)0(\bfj,r)&=\up^{\bfj\centerdot\co(A)}A(\bfj+\bfj',r),
\end{split}
\end{equation}
where $0$ stands for the zero matrix,

\begin{equation}\label{formula1 (2) in the affine q-Schur algebra}
\begin{split}
 \afE_{h,h+1}(\bfl,r)&A(\bfj,r)=\sum_{i<h;a_{h+1,i}\geq 1}\up^{f(i)}\ol{\left[\!\!\left[{a_{h,i}+1
\atop 1}\right]\!\!\right]}(A+\afE_{h,i}-\afE_{h+1,i})(\bfj+\afal_h,r)\\
&+\sum_{i>h+1;a_{h+1,i}\geq
1}\up^{f(i)}\ol{\left[\!\!\left[{a_{h,i}+1\atop 1
}\right]\!\!\right]}(A+\afE_{h,i}-\afE_{h+1,i})(\bfj,r)\\
&+\up^{f(h)-j_h-1}\frac{(A-\afE_{h+1,h})(\bfj+\afal_h,r)-(A-\afE_{h+1,h})(\bfj+\afbt_h,r)}{1-\up^{-2}}\\
&+\up^{f(h+1)+j_{h+1}}\ol{\left[\!\!\left[{a_{h,h+1}+1\atop 1
}\right]\!\!\right]}(A+\afE_{h,h+1})(\bfj,r),
\end{split}
\end{equation}
\begin{equation}\label{formula1 (3) in the affine q-Schur algebra}
\begin{split}
 \afE_{h+1,h}(\bfl,r)&A(\bfj,r)=\sum_{i<h;a_{h,i}\geq
 1}\up^{f'(i)}\ol{\left[\!\!\left[{a_{h+1,i}+1\atop 1
}\right]\!\!\right]}(A-\afE_{h,i}+\afE_{h+1,i})(\bfj,r)\\
&+\sum_{i>h+1;a_{h,i}\geq
1}\up^{f'(i)}\ol{\left[\!\!\left[{a_{h+1,i}+1\atop 1
}\right]\!\!\right]}(A-\afE_{h,i}+\afE_{h+1,i})(\bfj-\afal_h,r)\\
&+\up^{f'(h+1)-j_{h+1}-1}\frac{(A-\afE_{h,h+1})
(\bfj-\afal_h,r)-(A-\afE_{h,h+1})(\bfj+\afbt_h,r)}{1-\up^{-2}}\\
&+\up^{f'(h)+j_{h}}\ol{\left[\!\!\left[{a_{h+1,h}+1\atop 1
}\right]\!\!\right]}(A+\afE_{h+1,h})(\bfj,r).
\end{split}
\end{equation}
\end{Thm}

According to Proposition \ref{indecomposable basis}, the double
Ringel--Hall algebra $\dHallr$ has generators corresponding to
simple modules and homogeneous indecomposable modules. It would be
natural to raise the following question. In fact, we will see in
\S5.4 that the solution to this problem is the key to prove the
realization conjecture \ref{realization conjecture}; cf. Problem
\ref{Prob-realization}.

\begin{Prob}\label{Problem for MF}
Find multiplication formulas for $\afE_{h,h+sn}(\bfl,r)A(\bfj,r)$ for all  $s\neq0$ in $\mbz$.
\end{Prob}

\section{The $\dHallr$-$\afbfHr$-bimodule structure on tensor spaces}

As in the previous section, let $\bdOg$ be the $\mbq(\up)$-vector
space with basis $\{\og_s\mid s\in\mbz\}$. Let $\bfU(\ti C_\infty)$
be the quantum enveloping algebra associated with the
Borcherds--Cartan matrix $\ti C_\infty$; see Definition
\ref{quantumGKMAlg}. Our idea is to use the {\it natural} module
structure on $\bfOg$ for the quantum enveloping algebra $\bfU(\ti
C_\infty)$ associated with the Borcherds--Cartan matrix $\ti
C_\infty$. This induces a left $\bfU(\ti C_\infty)$-module structure
on the tensor space $\bdOg^{\otimes r}$ which commutes with the
right action of the affine Hecke algebra $\afbfHr$, and hence, an
algebra homomorphism $\ti\xi_r:\bfU(\ti C_\infty)\ra \afbfSr$. We
then show that this homomorphism factors through $\Phi:\bfU(\ti
C_\infty)\ra \dHallr$. In this way, we obtain a
$\dHallr$-$\afbfHr$-bimodule structure on $\bdOg^{\otimes r}$, and
will then partially establish the affine analogue of the quantum
Schur--Weyl duality in \S3.7.

 For $i\in I$ and $t\in
\mbz^+$, we define the actions of $K_i^{\pm 1}$, $E_i$, $F_i$,
$\sfk_t^{\pm1}$, $\sfx_t$, and $\sfy_t$ on $\bdOg$ by
\begin{equation} \label{QGKMAlg-action}
\begin{array}{rll}
&K_i^{\pm 1}\cdot \og_s=\up^{\pm\dt_{i,\bar s}}\og_s,\; &\sfk_t^{\pm1}\cdot \og_s=\og_s, \\
&E_i\cdot \og_s=\dt_{i+1,\bar s}\og_{s-1},\;&\sfx_t\cdot\og_s=\og_{s-tn},\\
&F_i\cdot \og_s=\dt_{i,\bar
s}\og_{s+1},\;&\sfy_t\cdot\og_s=\og_{s+tn}.
\end{array}
\end{equation}

\begin{Lem} With the action defined as above, $\bdOg$ becomes a left
$\bfU(\ti C_\infty)$-module.
\end{Lem}

\begin{proof}
For $i\in I$ and $t\in \mbz^+$, we denote by $\kappa_i^{\pm 1},
\phi^+_i$, $\phi^-_i$, $\chi_t^{\pm1}$, $\psi_t^+$, and $\psi_t^-$
the $\mbq(\up)$-linear transformations of $\bdOg$ induced by the
actions of $K_i^{\pm 1}$, $E_i$, $F_i$, $\sfk_t^{\pm1}$, $\sfx_t$,
and $\sfy_t$ on $\bdOg$ defined above, respectively. The assertion
follows from the fact that $\kappa_i^{\pm 1}, \phi^+_i$, $\phi^-_i$,
$\chi_t^{\pm1}$, $\psi_t^+$, and $\psi_t^-$ satisfy the relations
similar to (R1)--(R8) in Definition \ref{quantumGKMAlg}. This is
because those relations involving $\chi_t^{\pm1}$, $\psi_t^+$, and
$\psi_t^-$ are clear, while the others (involving only $K_i^{\pm
1}$, $E_i$, $F_i$) follow from the natural representation of the
extended quantum affine $\mathfrak{sl}_n$ on $\bfOg$ (see, e.g.,
\cite[(9.5.1)]{DF09}).
\end{proof}



Since all $\sfk_t^\pm$ act identically on $\bdOg$, it follows from
Proposition \ref{QEAGKMAlg-DHall} that the action of $\bfU(\ti
C_\infty)$ on $\bdOg$ induces an action of $\dHallr$ on $\bdOg$,
i.e., for each $x\in\dHallr$ and $s\in\mbz$, we have
$$x\cdot \og_s=y\cdot \og_s,\;\;\text{where $y\in\Phi^{-1}(x)$.}$$
 Let $\xi_1:\dHallr\ra \End_{\afbfHr}(\bdOg)$ be the algebra homomorphism
defined by this action.

As shown in \S2.1, there is a PBW type basis for $\dHallr$. The
action of these basis elements on the basis $\{\og_s \mid
s\in\mbz\}$ of $\bdOg$ can be described as follows. \index{PBW type
basis}

\begin{Prop} \label{compareson with VV's action} For each
$0\not=A\in\afThnp$ and $s\in\mbz$, we have
$$\ti u_A^+ \cdot\og_s=\sum_{t<s}\dt_{A,E_{t,s}^\vtg}\og_t\;\;\text{and}\;\;
\ti u_A^- \cdot\og_s=\sum_{s<t}\dt_{A,E_{s,t}^\vtg} \og_t,$$
 where, as in \eqref{tilde u_A}, $\ti u_A^\pm=\up^{d_A'}u_A^\pm$ with $d'_A=\dim \End(M(A))-\dim M(A)$.
In particular, if $A\not=0$ and $M(A)$ is decomposable, then
$u_A^+\cdot \og_s=0=u_A^-\cdot \og_s$.
\end{Prop}

\begin{proof} Define a $\mbq(\up)$-linear map $\phi=\phi^+:\dHallr^+\ra
\End_{\mbq(\up)}(\bdOg), \ti u^+_A\mapsto \phi_A$ by setting
$\phi_0={\rm id}$ and for $A\not=0$,
$$\phi_A: \bdOg\lra\bdOg,\;\og_s\lm
\sum_{t<s}\dt_{A,E_{t,s}^\vtg}\og_t.$$
 Define $\phi^-:\dHallr^-\ra \End_{\mbq(\up)}(\bdOg)$ in a similar
 way.

We only prove the first equality. The second one can be proved
analogously. We first show that $\phi$ is an algebra homomorphism.
For $i<j$ in $\mbz$, set $M^{i,j}=M(E_{i,j}^\vtg)=S_i[j-i]$ as in
\S1.2 and write $j-i=an+b$ for $a\geq 0$ and $0\leq b<n$. Then it is
easy to check that
\begin{equation}\label{End-Alg-dim}
\dim\End(M^{i,j})=\left\{\begin{array}{ll}
                 a,\;\;& \text{if $b=0$;}\\
                 a+1, &\text{if $b\not=0$.}
                 \end{array}\right.\end{equation}
We claim that for $l<s<t$,
\begin{equation}\label{End-dim-eqality}
\dim\End(M^{l,t}) =\dim\End(M^{l,s}) +\dim\End(M^{s,t})+\lr{\bfdim
M^{l,s},\bfdim M^{s,t}}
\end{equation}
(cf. proof of \cite[Lem.~8.2]{DDX}). Indeed, write
$$s-l=a_1n+b_1\;\,\text{and}\;\,t-s=a_2n+b_2$$
for $a_1,a_2\geq 0$ and $0\leq b_1,b_2<n$. Then
$t-l=(a_1+a_2)n+b_1+b_2$. If $b_1=0$ or $b_2=0$, the equality
follows from \eqref{End-Alg-dim} since $\lr{\bfdim M^{l,s},\bfdim
M^{s,t}}=0$ by \eqref{Euler form 2}. Suppose now $b_1\not=0$ and
$b_2\not=0$. Then
$$\dim\End(M^{l,s}) +\dim\End(M^{s,t})=a_1+a_2+2\;\;
\text{and}$$
$$\dim\End_{k\Dt}(M^{l,t})=\left\{\begin{array}{ll}
                 a_1+a_2+1,\;\;& \text{if $b_1+b_2\leq n$;}\\
                 a_1+a_2+2, &\text{if $b_1+b_2>n$.}
                 \end{array}\right.$$
Let $\bfd=(d_i)=\bfdim M^{l,s}$. Then
$$\lr{\bfdim M^{l,s},\bfdim
M^{s,t}}=d_{\overline{t-1}}-d_{\overline{s-1}}=\left\{\begin{array}{ll}
                 -1,\;\;& \text{if $b_1+b_2\leq n$;}\\
                 0, &\text{if $b_1+b_2>n$.} \end{array}\right.$$
Therefore, the equality \eqref{End-dim-eqality} holds.

For $A,B\in\afThnp$, we have
$$\ti u^+_A\ti u_B^+=v^{d_A'+d_B'}u_A^+u_B^+=\up^{\lr{\bfdim
M(A),\bfdim M(B)}+d_A'+d_B'}\sum_C \up^{-d_C'}\varphi^C_{A,B}\ti
u^+_C.$$
 Thus, to prove that $\phi$ is an algebra homomorphism, it suffices
to show that for $A,B\in\afThnp$,
$$\phi_A\phi_B=\up^{\lr{\bfdim
M(A),\bfdim M(B)}+d_A'+d_B'}\sum_C
\up^{-d_C'}\varphi^C_{A,B}\phi_C.$$

By definition, we have for $s\in\mbz$,
$$\phi_A\phi_B(\og_s)=0=\up^{\lr{\bfdim M(A),\bfdim M(B)}+d_A'+d_B'}\sum_C
\up^{-d_C'} \varphi^C_{A,B}\phi_C(\og_s)$$
 unless $B=E_{s,t}^\vtg$ and $A=E_{l,s}^\vtg$ for some $l<s<t$.
Now we suppose that $B=E_{s,t}^\vtg$ and $A=E_{l,s}^\vtg$ for some
$l<s<t$. Then
$$\phi_A\phi_B(\og_s)=\og_l\;\;\text{and}$$
$$\sum_C \up^{-d_C'}\varphi^C_{A,B}\phi_C(\og_s)=\up^{-d_{D}'}\og_l,$$
 where $D=E_{l,t}^\vtg$ (and $\varphi^D_{A,B}=1$). From \eqref{End-dim-eqality} it follows that
$$\phi_A\phi_B(\og_s)=\og_l=\up^{\lr{\bfdim
M(A),\bfdim M(B)}+d_A'+d_B'}\sum_C
\up^{-d_C'}\varphi^C_{A,B}\phi_C(\og_s).$$
 Hence, $\phi$ is an algebra homomorphism.

To prove the proposition, it remains to show that $\phi$ coincides
with the restriction of $\xi_1$ to $\dHallr^+$. Since $\dHallr^+$ is
generated by $u_i^+$ and $\sfz^+_t$ for $i\in I$ and $t\in\mbz^+$,
we need to check that
$$\phi(u_i^+)=\xi_1(u_i^+)\;\;\text{and}\;\;
\phi(\sfz^+_t)=\xi_1(\sfz^+_t).$$
 The equality $\phi(u_i^+)=\xi_1(u_i^+)$ is trivial. For each
 $s\in\mbz$, we have by definition
$$\xi_1(\sfz_t^+)(\og_s)=\sfz_t^+\cdot\og_s=\og_{s-tn}.$$
On the other hand, by \eqref{expression sfx sfy},
$$\sfz_t^+=\sum_{l=1}^n {\ti
u}^+_{E^\vtg_{l,l+tn}}+\frac{t}{[t]}x^+,$$
 where $x^+$ is a linear combination of certain $u^+_A$ with
$M(A)$ decomposable. Hence, $\phi(x^+)=0$. Thus,
$$\phi(\sfz^+_t)(\og_s)=\sum_{l=1}^n \phi({\ti
u}^+_{E^\vtg_{l,l+tn}})(\og_s)=\sum_{l=1}^n\sum_{b<s}\dt_{E^\vtg_{l,l+tn},E_{b,s}^\vtg}\og_b=\og_{s-tn}.$$
 (Note that $E_{a,b}^\vtg=E_{a+n,b+n}^\vtg$ for $a,b\in\mbz$.)
 Hence, $\phi(\sfz^+_t)=\xi_1(\sfz^+_t)$. This completes the proof.
\end{proof}

\begin{Rem}\label{xi1} The above proposition implies that the action of
$\dHallr^-$ on $\bdOg$ induced from $\phi^-$ coincides with the
action given in \cite[Lem.~8.1]{VV99} which is defined geometrically
through the algebra homomorphism
$\zeta_1^-:\dHallr^-\to\boldsymbol\sS_\vtg(n,1)$; see
\eqref{zetarpm} below.
\end{Rem}

Now, for each $r\geq 1$, the Hopf algebra structure of $\bfU(\ti
C_\infty)$ induces a left $\bfU(\ti C_\infty)$-module structure on
the tensor space $\bdOg^{\otimes r}$ which has a $\mbq(\up)$-basis
$\{\og_\bfi\mid\bfi\in\mbz^r\}$.\index{tensor space!$\sim$ over $\mbq(\up)$, $\bdOg^{\otimes r}$} Since $\sfx_t$ and $\sfy_t$ are
primitive elements, we have by \eqref{QGKMAlg-action} that for each
$t\geq 1$ and $\og_\bfi=\og_{i_1}\ot\cdots\ot\og_{i_r}\in \bdOg^{\ot
r}$,
\begin{equation}\label{action-of-x_t-y_t}
\aligned &\sfx_t\cdot\og_\bfi=\sum_{s=1}^r
\og_{i_1}\ot\cdots\ot
\og_{i_{s-1}}\ot\og_{i_s-tn}\ot\og_{i_{s+1}}\ot\cdots\ot\og_{i_r}\;\;\text{and}\\
&\sfy_t\cdot\og_\bfi=\sum_{s=1}^r \og_{i_1}\ot\cdots\ot
\og_{i_{s-1}}\ot\og_{i_s+tn}\ot\og_{i_{s+1}}\ot\cdots\ot\og_{i_r}.
\endaligned\end{equation}
 Recall from \S3.3 that $\afbfHr$ has a right action on
$\bdOg^{\otimes r}$. The following result says that the left action
of $\bfU(\ti C_\infty)$ and the right action of $\afbfHr$ on
$\bdOg^{\otimes r}$ commute.

\begin{Prop} \label{actions tensor space commute} For each $r\geq 1$, the actions
of $\bfU(\ti C_\infty)$ and $\afbfHr$ on $\bdOg^{\otimes r}$
commute. In other words, the tensor space $\bdOg^{\otimes r}$ is a
$\bfU(\ti C_\infty)$-$\afbfHr$-bimodule.
\end{Prop}

\begin{proof} Since $\bfU(\ti C_\infty)$ is generated by the set
${\scr U}:=\{K_i^{\pm1},\sfk_t^{\pm1}, E_i, F_i, \sfx_t, \sfy_t\mid
i\in I, t\in\mbz^+\}$ and $\afbfHr$ is generated by the set ${\scr
H}:=\{T_k, X_s\mid 1\leq k<r, 1\leq s\leq r\}$, it suffices to show
that
\begin{equation}\label{action-lambda-x}
u\cdot \bigl(\og_{\bfi}\cdot h\bigr) =\bigl(u\cdot
\og_{\bfi}\bigr)\cdot h
\end{equation}
 for all $u\in\scr U$, $h\in\scr H$, and $\bfi\in\mbz^r$. It is easy
to check from the definition that \eqref{action-lambda-x} holds for
$u=K_i^{\pm1}$ or $\sfk_t^{\pm1}$ and arbitrary $h\in\scr H$ (resp.
for $h=X_s$ and arbitrary $u\in\scr U$).

Furthermore, by \eqref{action-of-x_t-y_t}, for $t\geq 1$ and
$\bfi\in\mbz^r$,
$$\sfx_t\cdot\og_\bfi =
\og_\bfi\cdot\bigl( \sum_{1\leq s\leq
r}X_s^t\bigr)\;\,\text{and}\;\, \sfy_t\cdot\og_\bfi =\og_\bfi\cdot
\bigl(\sum_{1\leq s\leq r}X_s^{-t}\bigr).$$
 This implies that for any $h\in\scr H$,
$$\aligned
\,&(\sfx_t\cdot\og_\bfi)\cdot h = \bigg(\og_\bfi\cdot\big(
\sum_{1\leq s\leq r}X_s^{t}\big)\bigg)\cdot h=(\og_\bfi\cdot
h)\cdot\big(\sum_{1\leq s\leq r}X_s^{t}\big)=
\sfx_t\cdot(\og_\bfi\cdot h),\\
&(\sfy_t\cdot\og_\bfi)\cdot h = \bigg(\og_\bfi\cdot\big( \sum_{1\leq
s\leq r}X_s^{-t}\big)\bigg)\cdot h=(\og_\bfi\cdot
h)\cdot\big(\sum_{1\leq s\leq r}X_s^{-t}\big)=
\sfy_t\cdot(\og_\bfi\cdot h)
\endaligned$$
since $\sum_{1\leq s\leq r}X_s^{\pm t}$ are central elements in
$\boldsymbol{\sH}_\vtg(r)$.


Consequently, it remains to prove that \eqref{action-lambda-x} holds
for $\bfi\in\mbz^r$, $u=E_i, F_i$ ($i\in I$), and $h=T_k$ ($1\leq
k<r$).

Clearly, the subalgebra of $\bfU(\ti C_\infty)$ generated by
$K_i^{\pm 1}$, $E_j$ and $F_j$ ($i\in I$ and $j\in
I\backslash\{n\}$) is isomorphic to the quantum enveloping algebra
$\bfU_\up(\frak{gl}_n)$ of $\frak{gl}_n$, and the subalgebra of
$\afbfHr$ generated by $T_k$ ($1\leq k<r$) is isomorphic to the
Hecke algebra $\boldsymbol{\sH}(r)$ of ${\frak S}_r$. Thus, by
applying the result on quantum Schur algebras (see, e.g.,
\cite[Lem.~14.23]{DDPW}), we obtain that \eqref{action-lambda-x}
holds for $u=E_i, F_i$ ($i\in I\backslash\{n\}$), $h=T_k$ ($1\leq
k<r$), and $\bfi\in I(n,r)=\{(i_1,\ldots,i_r)\in\mbz^r\mid 1\leq
i_s\leq n\;\forall s\}$.

Suppose now $\bfi=(i_1,\ldots,i_r)\in I(n,r)$. Then, by definition,
$$E_n\cdot\og_\bfi=\sum_{1\leq j\leq r }\dt_{i_j,1}\up^{g(\bfi,j)}\og_{\bfi-\bse_j}$$
where $g(\bfi,j)=|\{s\mid j<s\leq r,\;i_s=n\}| -|\{s\mid j<s\leq
r,\;i_s=1\}|$ and $\bse_j=(\delta_{s,j})_{1\leq s\leq r}\in\mbz^r$.
In the following we show case by case that $(E_n\cdot\og_\bfi)\cdot
T_k=E_n\cdot(\og_\bfi\cdot T_k)$ for $1\leq k <r$.

\medskip

\noindent{\bf Case $i_k=i_{k+1}$}. In this case,
\begin{equation}\label{ik=ik+1}
(E_n\cdot\og_\bfi)\cdot T_k=\up^2\sum_{1\leq j\leq r\atop
j\not=k,k+1}\dt_{i_j,1}\up^{g(\bfi,j)}\og_{\bfi-\bse_j}
+\dt_{i_k,1}(\up^{g(\bfi,k)}\og_{\bfi-\bse_k}\cdot T_k
+\up^{g(\bfi,k+1)}\og_{\bfi-\bse_{k+1}}\cdot T_{k}).
\end{equation}
If $i_k=i_{k+1}=1$, then $g(\bfi,k)=g(\bfi,k+1)-1$. Hence,
\begin{equation*}
\begin{split}
&\qquad\dt_{i_k,1}(\up^{g(\bfi,k)}\og_{\bfi-\bse_k}\cdot T_k
+\up^{g(\bfi,k+1)}\og_{\bfi-\bse_{k+1}}\cdot T_{k})\\
&=\dt_{i_k,1}\big(\up^{g(\bfi,k)}\og_{\bfi+(n-1)\bse_k}\cdot(X_kT_k)
+\up^{g(\bfi,k+1)}\og_{\bfi+(n-1)\bse_{k+1}}\cdot
(X_{k+1}T_k)\big)\\
&= \dt_{i_k,1}\big(\up^{g(\bfi,k)+2}
\og_{\bfi+(n-1)\bse_k}\cdot(T_k^{-1}X_{k+1})
+\up^{g(\bfi,k+1)}\og_{\bfi+(n-1)\bse_{k+1}}\cdot(T_kX_k+(\up^2-1)X_{k+1})\big)\\
&= \dt_{i_k,1}\big(\up^{g(\bfi,k)+1} \og_{\bfi-\bse_{k+1}}
+\up^{g(\bfi,k+1)}(\up\og_{\bfi
-\bse_{k}}+(\up^2-1)\og_{\bfi-\bse_{k+1}})\big)\\
&=\dt_{i_k,1}\up^2(\up^{g(\bfi,k)}\og_{\bfi-\bse_k}+
\up^{g(\bfi,k+1)}\og_{\bfi-\bse_{k+1}}).
\end{split}
\end{equation*}
Applying \eqref{ik=ik+1} gives that
$$(E_n\cdot\og_\bfi)\cdot T_k=\up^2\sum_{1\leq j\leq
r}\dt_{i_j,1}\up^{g(\bfi,j)}\og_{\bfi-\bse_j}=\up^2E_n\cdot \og_\bfi
=E_n\cdot(\og_\bfi\cdot T_k).$$

\noindent{\bf Case $i_k<i_{k+1}$}. In particular, $i_{k+1}\not=1$.
By definition, we have
\begin{equation*}
\begin{split}
\,&\qquad(E_n\cdot \og_\bfi)\cdot T_k=
\sum_{1\leq j\leq r,\;i_j=1}\up^{g(\bfi,j)}\og_{\bfi-\bse_j}\cdot T_k\\
&=
\sum_{1\leq j\leq r,\;i_j=1}\up^{g(\bfi,j)}\og_{\bfi+(n-1)\bse_j}\cdot(X_jT_k)\\
&= \sum_{1\leq j\leq r,\;i_j=1\atop
j\not=k,k+1}\up^{g(\bfi,j)}\og_{\bfi+(n-1)\bse_j}\cdot(T_kX_j)
+\dt_{i_k,1}\up^{g(\bfi,k)}\og_{\bfi+(n-1)\bse_k}\cdot(\up^2
T_k^{-1}X_{k+1})\\
&=\sum_{1\leq j\leq r,\;i_j=1\atop
j\not=k,k+1}\up^{g(\bfi,j)+1}\og_{\bfi s_k-\bse_j}
+\dt_{i_k,1}\up^{g(\bfi,k)}\up^{1-\dt_{i_{k+1},n}}\og_{\bfi
s_k-\bse_{k+1}}
\end{split}
\end{equation*}
Since $g(\bfi s_k,k+1)=g(\bfi,k)-\dt_{i_{k+1},n}$ and $g(\bfi
s_k,j)=g(\bfi,j)$ for $j\not=k,k+1$, we get that
$$\aligned
(E_n\cdot \og_\bfi)\cdot T_k &=\up\sum_{1\leq j\leq r,\;i_j=1\atop
j\not=k,k+1}\up^{g(\bfi s_k,j)}\og_{\bfi s_k-\bse_j}
+\up\dt_{i_k,1}\up^{g(\bfi s_k,k+1)}\og_{\bfi s_k-\bse_{k+1}}\\
&=\up E_n\cdot \og_{\bfi s_k}=E_n\cdot(\og_\bfi\cdot T_k).
\endaligned$$

\noindent{\bf Case $i_k>i_{k+1}$.} In this case, $i_k\not=1$. It
follows from the definition that
\begin{equation*}
\begin{split}
\,&\;\;\;(E_n\cdot \og_\bfi)\cdot T_k =
\sum_{1\leq j\leq r,\;i_j=1}\up^{g(\bfi,j)}\og_{\bfi+(n-1)\bse_j}\cdot(X_jT_k)\\
&= \sum_{1\leq j\leq r,\;i_j=1\atop
j\not=k,k+1}\up^{g(\bfi,j)}\og_{\bfi+(n-1)\bse_j}\cdot(T_kX_j)
+\dt_{i_{k+1},1}\up^{g(\bfi,k+1)}\og_{\bfi+(n-1)
\bse_{k+1}}\cdot(T_kX_k+(\up^2-1)X_{k+1})\\
&=\sum_{1\leq j\leq r,\;i_j=1\atop j\not=k,k+1}\up^{g(\bfi,j)}(\up\og_{\bfi s_k-\bse_j}+(\up^2-1)\og_{\bfi-\bse_j})\\
&\qquad\qquad\qquad
+\dt_{i_{k+1},1}\up^{g(\bfi,k+1)}(\up^{1+\dt_{i_k,n}}\og_{\bfi s_k-\bse_k}+(\up^2-1)\og_{\bfi-\bse_{k+1}})\\
&= \up\bigg(\sum_{1\leq j\leq r,\;i_j=1\atop
j\not=k,k+1}\up^{g(\bfi,j)} \og_{\bfi
s_k-\bse_j}+\dt_{i_{k+1},1}\up^{g(\bfi,k+1)+\dt_{i_k,n}}
\og_{\bfi s_k-\bse_k}\bigg)\\
&\qquad\qquad\qquad
+(\up^2-1)\sum_{1\leq j\leq r,\;i_j=1}\up^{g(\bfi,j)}\og_{\bfi-\bse_j}.\\
\end{split}
\end{equation*}
Since $g(\bfi s_k,k)=g(\bfi,k+1)-\dt_{i_{k},n}$  and $g(\bfi
s_k,j)=g(\bfi,j)$ for $j\not=k,k+1$, we have
\begin{equation*}
\begin{split}
(E_n\cdot\og_\bfi)\cdot T_k&=
\up\sum_{1\leq j\leq r\atop(\bfi s_k)_j=1}\up^{g(\bfi s_k,j)}\og_{\bfi s_k-\bse_j}+(\up^2-1)E_n\cdot\og_\bfi\\
&=\up E_n\cdot \og_{\bfi s_k}+(\up^2-1)E_n\cdot \og_\bfi =E_n\cdot
(\og_\bfi\cdot T_k).
\end{split}
\end{equation*}
Similarly,  we have for $\bfi\in I(n,r)$ and $1\leq k<r$,
$$(F_n\cdot \og_\bfi)\cdot T_k=F_n\cdot(\og_\bfi\cdot T_k).$$
 We conclude that \eqref{action-lambda-x} holds
for $\bfi\in I(n,r)$, $u=E_i, F_i$, and $h=T_k$, where $i\in I$ and
$1\leq k<r$.

In general, for an arbitrary $\bfj\in\mbz^r$, write
$\og_\bfj=\og_\bfi\cdot( X_1^{t_1}\cdots X_r^{t_r})$ with $\bfi\in
I(n,r)$ for some $t_1,\ldots, t_r\in\mbz$. Let $\sX$ be the
subalgebra of $\afbfHr$ generated by $X_1^\pm,\ldots, X_r^\pm$. Then
$X_1^{t_1}\cdots X_r^{t_r}T_k=T_k x+y$ for some $x,y\in\sX$. By the
above discussion, we infer that for $u=E_i$ or $F_i$ ($i\in I$),
\begin{equation*}
\begin{split}
(u\cdot \og_\bfj)\cdot T_k&
=\big((u\cdot \og_\bfi)\cdot (X_1^{t_1}\cdots X_r^{t_r})\big)\cdot T_k
=(u\cdot\og_\bfi)\cdot(T_k x+y)\\
&=\big(u\cdot(\og_\bfi\cdot T_k)\big)\cdot x+u\cdot(\og_\bfi\cdot y)
=u\cdot\big((\og_\bfi\cdot T_k)\cdot x\big)+u\cdot(\og_\bfi\cdot y)\\
&=u\cdot\big(\og_\bfi\cdot( T_k x+y)\big)=u\cdot(\og_\bfj\cdot T_k).
\end{split}
\end{equation*}
The proof is completed.
\end{proof}

For each $r\geq 1$, the $\bfU(\ti C_\infty)$-$\afbfHr$-bimodule
structure on $\bdOg^{\otimes r}$ induces a $\mbq(\up)$-algebra
homomorphism
$$\ti\xi_r:\bfU(\ti C_\infty)\lra \End_{\afbfHr}(\bdOg^{\otimes
r})=\afbfSr,$$ which factors through the surjective Hopf algebra
homomorphism $\Phi:\bfU(\ti C_\infty)\ra\dHallr$. Hence, we obtain a
$\mbq(\up)$-algebra homomorphism
\begin{equation}\label{xir}
\xi_r:\dHallr\lra \End_{\afbfHr}(\bdOg^{\otimes r})=\afbfSr.
\end{equation}\index{$\xi_r$, epimorphism $\dHallr\to\afbfSr$}
In other words, the following diagram is commutative
 \begin{center}
\begin{pspicture}(2,-0.1)(4,2.2)
\psset{xunit=.8cm,yunit=.7cm} \uput[u](1,2){$\bfU(\ti C_\infty)$}
\uput[d](1,0.8){$\dHallr$}
\uput[r](4.5,1.4){$\End_{\afbfHr}(\bdOg^{\otimes
r})=\afbfSr$}\uput[l](1.1,1.4){$\Phi$} \psline{->}(1.1,2)(1.1,0.75)
\psline{->}(1.95,2.45)(4.3,1.5) \psline{->}(1.95,0.35)(4.3,1.2)
\uput[u](3.2,1.9){$\ti\xi_r$} \uput[d](3.2,0.9){$\xi_r$}
\end{pspicture}
\end{center}
 Since $\Phi(\sfx_t)=\sfz_t^+$ and $\Phi(\sfy_t)=\sfz_t^-$ for all $t\geq 1$,
 it follows from \eqref{action-of-x_t-y_t} that for each $t\geq 1$ and
$\og_\bfi=\og_{i_1}\ot\cdots\ot\og_{i_r}\in \bdOg^{\ot r}$,
\begin{equation}
\label{action central elts} \aligned
&\sfz_t^+\cdot\og_\bfi=\sum_{s=1}^r \og_{i_1}\ot\cdots\ot
\og_{i_{s-1}}\ot\og_{i_s-tn}\ot\og_{i_{s+1}}\ot\cdots\ot\og_{i_r}\;\;\text{and}\\
&\sfz_t^-\cdot\og_\bfi=\sum_{s=1}^r \og_{i_1}\ot\cdots\ot
\og_{i_{s-1}}\ot\og_{i_s+tn}\ot\og_{i_{s+1}}\ot\cdots\ot\og_{i_r}.
\endaligned
\end{equation}

\begin{Rem}  By Remark \ref{xi1}, the PBW type basis action for the Hall
algebra $\dHallr^-$ described in Proposition \ref{compareson with
VV's action} induces via comultiplication an action on the tensor
space $\bdOg^{\otimes r}$. Varagnolo and Vasserot have outlined a
proof of the fact that this Hall algebra action commutes with the
action of $\afbfHr$. Proposition \ref{compareson with VV's action}
and Remark \ref{xi1} show that this Hall algebra action coincides
with the restriction of the double Hall algebra action above, while
the proof in Proposition \ref{actions tensor space commute} is a
complete and more natural proof via the natural representation of
the quantum enveloping algebra $\bfU(\ti C_\infty)$. See \S3.6 for
an explicit description of the map $\xi_r$ and a further comparison
with work of Varagnolo and Vasserot.
\end{Rem}

We now describe the action of semisimple generators of $\dHallr$ on
$\bdOg^{\otimes r}$. Recall that for each
$\bfa=(a_i)\in\mbn^n=\mbn_\vtg^n$,
$$S_\bfa=\oplus_{i\in
I}a_iS_i,\;\;u_\bfa^\pm=u^\pm_{[S_\bfa]}\;\;\text{and}\;\; {\wt
u^\pm_{[S_\bfa]}}=\up^{\sum_{1\leq i\leq n}
a_i(a_i-1)}u^\pm_{[S_\bfa]}.$$
 The following result is a direct consequence of the definition of the
comultiplications in Proposition \ref{Green Xiao}(b) and Corollary
\ref{Green Xiao 2}(b$'$). See \cite[8.3]{VV99} for the second
formula.

\begin{Prop} \label{comult-form-ss} For $\bfa\in\mbn^n$, we have
\begin{equation*} \begin{split}
 &\Dt^{(r-1)}(\wt u_\bfa^+)
 =\sum_{\bfa=\bfa^{(1)}+\cdots+\bfa^{(r)}}
 \up^{\sum_{s>t}\lan\bfa^{(s)},\bfa^{(t)}\ran}\times\\
 &\hspace{4cm}\wt u_{\bfa^{(1)}}^+\ot\wt u_{\bfa^{(2)}}^+\ti K_{\bfa^{(1)}}\ot\cdots
 \ot\wt
 u_{\bfa^{(r)}}^+\ti K_{\bfa^{(1)}+\bfa^{(2)}+\cdots+\bfa^{(r-1)}},\\
 &\Dt^{(r-1)}(\wt u_\bfa^-)
 =\sum_{\bfa=\bfa^{(1)}+\cdots+\bfa^{(r)}}
 \up^{\sum_{s>t}\lan\bfa^{(s)},\bfa^{(t)}\ran}\times\\
 &\hspace{4cm}\wt u_{\bfa^{(1)}}^-\ti K_{-(\bfa^{(2)}+\cdots+\bfa^{(r)})}
 \ot\wt u_{\bfa^{(2)}}^-\ti K_{-(\bfa^{(3)}+\cdots+\bfa^{(r)})}
 \ot\cdots \ot\wt
 u_{\bfa^{(r)}}^-.
\end{split}
\end{equation*}
\end{Prop}

This proposition together with Proposition \ref{compareson with VV's
action} gives the following corollary.

\begin{Coro} \label{action-ss} For $\bfa=(a_i)\in\afmbnn$ and ${\bf
i}=(i_1,\ldots,i_r)\in\mbz^r$, we have \begin{equation*}
\begin{split} \wt u_\bfa^+\cdot\og_{\bf i}
 &=\sum_{m_i\in\{0,1\}\,\forall i\atop
 \bfa=m_1\bfe_{i_1-1}+\cdots+m_r\bfe_{i_r-1}}
 \up^{\sum_{s>t}m_t(m_s-1)\lan\bfe_{i_s},\bfe_{i_t}\ran}
 \og_{i_1-m_1}\ot\cdots\ot\og_{i_r-m_r},\\
 \wt u_\bfa^-\cdot\og_{\bf i}
 &=\sum_{m_i\in\{0,1\}\,\forall i\atop
 \bfa=m_1\bfe_{i_1}+\cdots+m_r\bfe_{i_r}}
 \up^{\sum_{s>t}m_s(m_t-1)\lan\bfe_{i_s},\bfe_{i_t}\ran}
 \og_{i_1+m_1}\ot\cdots\ot\og_{i_r+m_r}.
 \end{split}
\end{equation*}
In particular, if $\sg(\bfa)=\sum_{i\in I}a_i>r$, then $\wt
u_\bfa^+\cdot\og_{\bf i}=0=\wt u_\bfa^-\cdot\og_{\bf i}$.
\end{Coro}

\begin{Coro}\label{n>r}
If $n>r$, then $\xi_r(\dHallr)=\xi_r(\bfU_\vtg(n))$.
\end{Coro}

\begin{proof} By Proposition \ref{generators-RH-alg}, $\dHallr$ is generated
by $u^+_i, u_i^-, K_i^{\pm 1}, u_{s\dt}^+, u_{s\dt}^-$ ($i\in I$,
$s\in\mbz^+$). If $n>r$, then $\dim S_{s\dt}=sn>r$. By Corollary
\ref{action-ss},
$$u_{s\dt}^+\cdot \og_\bfi=0=u_{s\dt}^-\cdot \og_\bfi,\,\, \text{ for all $\og_\bfi\in\bdOg^{\ot r}$}.$$
Hence, $\xi_r(u_{s\dt}^+)=0=\xi_r(u_{s\dt}^-)$ for all $s\geq1$, and
$\xi_r(\dHallr)$ is generated by the images of $u^+_i, u_i^-,
K_i^{\pm 1}$. This proves the equality.
\end{proof}

\begin{Rem}\label{NonrootOfUnity4}
We remark that, if $z\in\mbc$ is not a root of unity, and $\DC(n)$ is the specialized double Hall algebra
considered in Corollary \ref{NonrootOfUnity3}, there is a $\DC(n)$-action on the complex
space $\Og_\mbc^{\ot r}$ which gives rise to an algebra homomorphism
\begin{equation}\label{xi_{r,z}}
\xi_{r,\mbc}:\DC(n)\lra\afSr_{\mbc}.
\end{equation}\index{$\zrC$, epimorphism $\DC(n)\to\afSrC$}
\end{Rem}

\section{A comparison with the Varagnolo--Vasserot action}

In \cite{VV99}, Varagnolo--Vasserot defined a Hall algebra action on
the tensor space via the action on $\bfOg$ (Remark \ref{xi1}) and
the comultiplication $\Delta$ and proved in \cite[8.3]{VV99} that
this action agrees with the affine quantum Schur algebra action via
the algebra homomorphism $\zeta_r^-$ described in Proposition
\ref{afzrpm}(1) below. We will define the algebra homomorphism
$\zeta_r^+$ in \ref{afzrpm}(2) opposite to $\zeta_r^-$ and show that
$\xi_r$ in \eqref{xir} is their extension. In other words, $\xi_r$
coincides with $\zeta_r^+$ and $\zeta_r^-$ upon restriction (Theorem
\ref{xirl}).

We first follow \cite{VV99} to define an algebra homomorphism from
the Ringel--Hall algebra $\Hall$ of the cyclic quiver $\tri$ to
the affine quantum Schur algebra $\afSr$. This definition relies on
an important relation between cyclic flags and representations
of cyclic quivers.

Recall from \eqref{iso-Schur-Function} the geometric
characterization of affine quantum Schur algebras and the flag
varieties $\scrY=\afFn(q)$, $\scrX=\afBr(q)$ (and $G=G(q)$) over the
finite field $\field=\field_q$. Let
$\bfL=(L_i)_{i\in\mbz},\bfL'=(L_i')_{i\in\mbz}\in \scrY$ satisfy
$\bfL'\han\bfL$ i.e., $L_i'\subseteq L_i$ for all $i\in \mbz$. By
\cite[\S9]{GV} and \cite[5.1]{Lu99}, we can view $\bfL/\bfL'$ as a
nilpotent representation $V=(V_i,f_i)$ of $\tri$ over $\field$
such that $V_i=L_i/L_i'$ and $f_i$ is induced by the inclusion
$L_i\subseteq L_{i+1}$ for each $i\in I$. Here we identify
$L_{n+1}/L_{n+1}'$ with $L_1/L_1'$ via the multiplication by $\ep$.
Further, for $\bfL\han\bfL'$, we define the integers
\begin{equation}\label{a and c}
\aligned
a(\bfL,\bfL')&=\sum_{1\leq i\leq n}\dim_\field(L_i'/L_i)(\dim_\field(L_{i+1}/L_i)-\dim_\field(L_i'/L_i))\;\;\text{and}\\
c(\bfL',\bfL)&=\sum_{1\leq i\leq n}\dim_\field(L_{i+1}'/L_{i+1})
(\dim_\field(L_{i+1}/L_i)-\dim_\field(L_i'/L_i)).
\endaligned
\end{equation}

Recall also from \eqref{Hall algebra defined using function} the geometric characterization
$\HallL$ of the Ringel--Hall algebra $\Hall$. Thus, by
specializing $v$ to $q^{\frac{1}{2}}$, there
is a $\mbc$-algebra isomorphism
$$\Hall\ot_{\sZ}\mbc\lra \HallL,\;\ti u_A\lm
\lan\ttO_{A}\ran=q^{-\frac{1}{2}\dim\ttO_{A}}\chi_{_{\ttO_{A}}},$$
where $\ttO_{A}$ is the $G_V$-orbit in the representation variety
$E_V$ corresponding to the isoclass of $M(A)$.

By identifying $\afSr\ot_{\sZ}\mbc$ with $\mbc_G(\scrY\times \scrY)$, $\Hall\ot_{\sZ}\mbc$
with $\HallL$ (and hence, $\Hall^{\rm op}\ot_{\sZ}\mbc$ with
$\HallL^{\rm op}$), and recalling the elements $A(\bfj,r)\in\afSr$ for each $A\in\afThnpm$ and
$\bfj\in\afmbzn$ defined in \eqref{def-A(j,r)}, \cite[Prop.~7.6]{VV99} can now be stated
as the first part of the following (cf. footnote 2 of Chapter 1).

\begin{Prop} \label{afzrpm}
{\rm(1)} There is a $\sZ$-algebra homomorphism
$$\afzrm:\Hall^{\rm op}\lra\afSr,\;\;\ti u_A\lm \tA(\bfl,r)\;\;
\text{for all $A\in\afThnp$}$$
 such that the induced map
$\afzrm\ot {\rm id}_\mbc:\HallL^{\rm
op}\lra\mbc_G(\scrY\times \scrY)$
 is given by
\begin{equation}\label{equation1 in A(bfj,r) acts on (bfL,bfL')}
(\afzrm\ot {\rm id}_\mbc)(f)(\bfL,\bfL')=
\begin{cases}
q^{-\frac{1}{2}a(\bfL,\bfL')}f(\bfL'/\bfL), &
\text{if $\bfL\han\bfL'$;}\\
0, &\text{otherwise.}
\end{cases}
\end{equation}
{\rm(2)} Dually, there is a $\sZ$-algebra homomorphism
$$\afzrp:\Hall\lra\afSr,\;\;\ti u_A \lm A(\bfl,r)\;\;\text{for all
$A\in\afThnp$}$$
with the induced map
$\afzrp\ot {\rm id}_\mbc:\HallL\lra\mbc_G(\scrY\times \scrY)$
given by
\begin{equation}\label{equation3 in A(bfj,r) acts on (bfL,bfL')}
(\afzrp\ot {\rm id}_\mbc)(f)(\bfL,\bfL')=
\begin{cases}
q^{-\frac{1}{2}c(\bfL,\bfL')}f(\bfL/\bfL'), &\text{if $\bfL'\han\bfL$;}\\
0,&\text{otherwise.}
\end{cases}
\end{equation}
\end{Prop}

\begin{proof} Statement (1) is given in \cite[7.6]{VV99}. We only need to prove (2).
We first observe from the proof of \cite[Lem.~1.11]{Lu99} that
$$d_A-d_{\tA}=\frac{1}{2}\sum_{1\leq i\leq n}\bigg(\sum_{j
\in\mbz}a_{ij}\bigg)^2-\frac{1}{2}\sum_{1\leq j\leq
n}\bigg(\sum_{i\in\mbz}a_{ij}\bigg)^2.$$ Let $\zeta_r^+$ be the
composition of the algebra homomorphisms
$$\Hall\stackrel{(\zeta_r^-)^{\rm op}}{\lra}\afSr^{\rm op}\stackrel{\tau_r}{\lra}
\afSr,$$
where $\tau_r$ is the anti-involution on $\afSr$ given in \eqref{taur}. Thus, $\tau_r$ induces a map $\tau_r\ot
{\rm id}:\mbc_G(\scrY\times\scrY)\ra\mbc_G(\scrY\times\scrY)$. Applying this to the characteristic function $\chi_A$
of the orbit $\sO_A$ in $\scrY\times\scrY$ and noting
\eqref{transpose matrix} yields
$$(\tau_r\ot
{\rm id})(\chi_A)(\bfL,\bfL')=q^{\frac{1}{2}(d_A-d_{\tA})}
\chi_{\tA}(\bfL,\bfL')=q^{\frac{1}{2}(d_A-d_{\tA})}
\chi_A(\bfL',\bfL)$$ which is nonzero if and only if
$(\bfL',\bfL)\in\sO_A$. This implies
$a_{i,j}=\dim_\field\frac{L_i'\cap L_j}{L_{i-1}'\cap L_j+L_i'\cap
L_{j-1}}$. Hence, by the proof of \cite[1.5(a)]{Lu99},
$\dim_\field(L_i'/L_{i-1}')=\sum_{j \in\mbz}a_{i,j}$ and
$\dim_\field(L_j/L_{j-1})=\sum_{i \in\mbz}a_{i,j}$. Putting
$$b(\bfL,\bfL')=\frac{1}{2}(\sum_{1\leq i\leq n}
 ((\dim_\field(L_i'/L_{i-1}'))^2-\sum_{1\leq i\leq n}
 (\dim_\field(L_i/L_{i-1}))^2)),$$
we obtain $(\tau_r\ot {\rm id})(\chi_A)(\bfL,\bfL')=
q^{\frac{1}{2}b(\bfL,\bfL')}\chi_{_A}(\bfL',\bfL)$. Hence,
$(\tau_r\ot {\rm
id})(g)(\bfL,\bfL')=q^{\frac{1}{2}b\bigl(\bfL,\bfL')}g(\bfL',\bfL)$
for $g\in\mbc_G(\scrY\times\scrY\bigr)$. Taking $g=(\afzrm\ot {\rm
id})(f)$ with $f\in\HallL$, and applying \eqref{equation1 in
A(bfj,r) acts on (bfL,bfL')} gives
\begin{equation}\label{equation2 in A(bfj,r) acts on (bfL,bfL')}
(\tau_r\ot {\rm id})((\afzrm)^{\rm op}\ot {\rm id})(f)(\bfL,\bfL')=
\begin{cases}
q^{\frac{1}{2}(b(\bfL,\bfL')-a(\bfL',\bfL))}f(\bfL/\bfL'),&
\text{if $\bfL'\han\bfL$};\\
0,& \text{otherwise}.
\end{cases}
\end{equation}
For $\bfL'\han\bfL$, we have
\begin{equation*}
\begin{split}
&\quad \sum_{1\leq i\leq n}\dim_\field(L_{i+1}/L_i)^2 \\&
=\sum_{1\leq i\leq n}(\dim_\field(L_{i+1}/L_{i+1}')+
\dim_\field(L_{i+1}'/L_i')-\dim_\field(L_i/L_i'))^2\\
&=2\sum_{1\leq i\leq n}(\dim_\field(L_i/L_i'))^2+\sum_{1\leq i\leq
n}(\dim_\field(L_{i+1}'/L_i'))^2\\
&\quad +2\sum_{1\leq i\leq n}\dim_\field(L_{i+1}/L_{i+1}')\dim_\field(L_{i+1}'/L_i')\\
&\quad-2\sum_{1\leq i\leq
n}\dim_\field(L_{i+1}/L_{i+1}')\dim_\field(L_i/L_i')-2\sum_{1\leq
i\leq n}\dim_\field(L_{i+1}'/L_i')\dim_\field(L_i/L_i').
\end{split}
\end{equation*}
Hence, $b(\bfL,\bfL')-a(\bfL',\bfL)=-c(\bfL,\bfL')$ and \eqref{equation3 in A(bfj,r) acts on (bfL,bfL')} follows now from
\eqref{equation2 in A(bfj,r) acts on (bfL,bfL')}.
\end{proof}

For notational simplicity, we write $A(\bfj,r)$ for $A(\bfj,r)\ot 1$
in $\afSr\ot_{\sZ}\mbc$. By taking $f$ to be
$q^{-\frac{1}{2}\dim\sO_{A}}\chi_{_{\ttO_{A}}}$ in \eqref{equation1
in A(bfj,r) acts on (bfL,bfL')} and \eqref{equation3 in A(bfj,r)
acts on (bfL,bfL')}, we obtain the following.

\begin{Coro}\label{A(bfj,r) acts on (bfL,bfL')}
For $A\in\afThnp$ and $(\bfL,\bfL')\in\mathscr Y\times\mathscr Y$,
we have
\begin{equation*}
A(\bfl,r)(\bfL,\bfL')=
\begin{cases}
q^{-\frac{1}{2}(c(\bfL,\bfL')+\dim\ttO_A)},
&\qquad\text{if $\bfL'\han\bfL$ and $\bfL/\bfL'\in\ttO_A$};\\
0,&\qquad\text{otherwise}
\end{cases}\\
\end{equation*}
and
\begin{equation*}
\tA(\bfl,r)(\bfL,\bfL')=
\begin{cases}
q^{-\frac{1}{2}(a(\bfL,\bfL')+\dim\ttO_A)},&\qquad\text{if $\bfL\han\bfL'$ and $\bfL'/\bfL\in\ttO_A$};\\
0,&\qquad\text{otherwise.}
\end{cases}
\end{equation*}
\end{Coro}

We now identify $\bfHall$ as $\dHallr^+$ and $\bfHall^{\rm op}$ as $\dHallr^-$ via \eqref{pm parts}. Then $\zeta_r^\pm$ induce $\mbq(\up)$-algebra homomorphisms
\begin{equation}\label{zetarpm}\zeta_r^\pm:\dHallr^\pm\lra\afbfSr.
\end{equation}

\begin{Thm}\label{xirl} For every $r\geq0$, the map $\xi_r:\dHallr\lra\afbfSr$ defined in \eqref{xir} is the (unique) algebra homomorphism satisfying
$$\xi_r(K_1^{j_1}\cdots K_n^{j_n})=0(\bfj,r),\;\xi_r(\ti u_A^-)=\zeta_r^-(\ti u_A^-),\;\;\text{and}\;\;
\xi_r(\ti u_A^+)=\zeta_r^+(\ti u_A^+),$$
for all $\bfj=(j_1,\ldots,j_n)\in \mbz^n$ and $A\in \afThnp$.
In particular,
we have $\xi_r|_{\dHallr^\pm}=\zeta_r^\pm$.\end{Thm}

\begin{proof} Since $\dHallr$ is generated by $K_i^{\pm1}$, $1\le i\le n$, together with semisimple generators
$u_{A_\la}^\pm$, $\la\in\afmbnn$, where
$A_\la=\sum_{i=1}^n\la_iE_{i,i+1}^\vtg$. By Proposition
\ref{afzrpm}, it suffices to prove
$$(1)\,\,\xi_r(K_i)=0(\bfe_i,r),\;\;(2)\;\;\xi_r(\ti u_{A_\la}^-)=(\tA_\la)(\bfl,r),\;\;\text{and}\;\;(3)\;\;
\xi_r(\ti u_{A_\la}^+)=A_\la(\bfl,r).$$
To prove them, it suffices by Propositions \ref{bimodule-isom} and \ref{actions tensor space commute} to compare the actions of
both sides on $\og_\bfi=[A^\bfi]$ for all $\bfi\in\afJnr$.

Suppose $\bfi=\bfi_\mu^\vtg$. By \eqref{QGKMAlg-action}, $K_i\cdot\og_\bfi=\up^{\mu_i}\og_\bfi$.
Since $\ro(A^\bfi)=\mu$, \eqref{product [diag(la)][A] in affine q-Schur algebra} implies $0(\bfe_i,r)[A^\bfi]=\up^{\mu_i}[A^\bfi]$.
Hence, $0(\bfe_i,r)\cdot \og_\bfi=K_i\cdot\og_\bfi$,
proving (1). The proof of (2) is given in \cite[8.3]{VV99}. We now prove (3) for completeness.
So we need to show that
\begin{equation}\label{eq(3)}
A_\la(\bfl,r)\cdot\og_\bfi=\ti
u_{\la}^+\cdot\og_\bfi\;\;\text { for all $\bfi\in \afJnr$ and
$\la\in\afmbnn$.}
\end{equation}
The equality is trivial if $\sg(\la)>r$ as both sides are 0. We now
assume $\sg(\la)\le r$ and prove the equality by writing both sides
as a linear combination of the basis $\{[A^\bfi]\}_{\bfi\in\afInr}$;
cf. Remark \ref{two bases for tsp}.

By Proposition \ref{bimodule-isom}, the left hand side of \eqref{eq(3)} becomes
$A_\la(\bfl,r)\cdot\og_\bfi=A_\la(\bfl,r)[A^\bfi].$
We now compute this by regarding $A_\la(\bfl,r)[A^\bfi]$ as the convolution product $q^{-\frac{1}{2}d_\bfi}A_\la(\bfl,r)*\chi_\bfi$;
see Remark \ref{implicit use bcp} and compare \cite[8.3]{VV99}.
By the definition \eqref{S convolution product}, for $\bfj\in \afInr$ and
$\la\in\afmbnn$,
$$(A_\la(\bfl,r)*\chi_\bfi)(\bfL_\bfj,\bfL_\emptyset)
=\sum_{\bfL\in
\scrY}(A_\la(\bfl,r))(\bfL_\bfj,\bfL)\chi_\bfi(\bfL,\bfL_\emptyset)
=\sum_{\bfL\in \scrY\atop (\bfL,\bfL_\emptyset)\in
\sO_\bfi}(A_\la(\bfl,r))(\bfL_\bfj,\bfL),$$ where $\bfL_\bfj$ is defined in
\eqref{lattice Li}, and
$\sO_\bfi=\sO_{A^\bfi}$ is the orbit containing
$(\bfL_\bfi,\bfL_\emptyset)$ (see also \eqref{lattice Li} for the definition of $\bfL_\emptyset$). For $\bfL=(L_i)\in \scrY$ and
$(\bfL,\bfL_\emptyset)\in \sO_\bfi$,  there is $g\in G$ such that
$(\bfL,\bfL_\emptyset)=g(\bfL_\bfi,\bfL_\emptyset)$. In other words,
$\bfL=g\bfL_\bfi$ and $\bfL_\emptyset=g\bfL_\emptyset$. The fact
$\bfi\in \afJnr$ implies that for each $t\in\mbz$, there exists
$l_t\in\mbz$ such that $\bfL_{\bfi,t}=\bfL_{\emptyset,l_t}$. (More
precisely, if $\bfi=\bfi_\mu^\vtg$ and $t=s+kn$ with $1\le s\le n$,
then $l_t=\mu_1+\cdots+\mu_s+kr$.)
 Thus,
$L_t=g(\bfL_{\bfi,t})=g(\bfL_{\emptyset,l_t})
=\bfL_{\emptyset,l_t}=\bfL_{\bfi,t}$ for any $t\in\mbz$. This
implies that $\bfL=\bfL_\bfi$. Hence, by Corollary \ref{A(bfj,r) acts on (bfL,bfL')},
$$(A_\la(\bfl,r)*\chi_\bfi)(\bfL_\bfj,\bfL_\emptyset)=
(A_\la(\bfl,r))(\bfL_\bfj,\bfL_\bfi)=q^{-\frac{1}{2}
(c(\bfj,\bfi)+\dim\ttO_{A_\la})},$$
 if and only if $\bfL_\bfi\subseteq \bfL_\bfj, \bfL_\bfj/\bfL_\bfi\cong M(A_\la)$,
 where $c(\bfj,\bfi)=c(\bfL_\bfj,\bfL_\bfi)$.
 The latter is equivalent by definition to the conditions
 $$\{k\in\mbz\mid i_k\le t\}\subseteq \{k\in\mbz\mid j_k\le t\}\subseteq \{k\in\mbz\mid i_k\le t+1\}, \text{ and } \dim_\field \bfL_{\bfj,t}/\bfL_{\bfi,t}=\la_t,$$
for all $t\in\mbz$. Hence, 
\begin{equation*}
\begin{split}
&\bfL_\bfi\subseteq \bfL_\bfj, \bfL_\bfj/\bfL_\bfi\cong M(A_\la) \\ \Longleftrightarrow&
i_t-1\leq j_t\leq i_t\ \text{and} \
\la_t=|\bfj^{-1}(t)\cap\bfi^{-1}(t+1)|\ \text{for all}\ t\in\mbz\\
\Longleftrightarrow& \bfj=\bfi-\bfm \ \text{and}\
\la=m_1\afbse_{i_1-1}+\cdots+m_r\afbse_{i_r-1} \text{ for some
$m_s\in\{0,1\}$,}
\end{split}
\end{equation*}
since, for all $s\in\mbz$,
$$m_s=1\iff j_s=i_s-1 \iff s\in \cup_{t\in\mbz}(\bfj^{-1}(t)\cap\bfi^{-1}(t+1)).$$
Here, $\bfm\in\mbz^r_\vtg$ is uniquely determined by $(m_1,m_2,\ldots,m_r)\in\mbz^r$
and $\bse_i^\vtg\in\mbz^r_\vtg$ corresponds to $\bse_i=(\dt_{k,i})_{1\leq k\leq r}$
under \eqref{flat2}. (Note that $\bfi-\bfm\in\afInr$ if $\bfi\in\afInr$.)
Since $\dim\ttO_{A_\la}=0$ (see, e.g., \cite[(1.6.2)]{DDPW}), we
obtain, for $\bfi\in \afJnr$ and $\la\in\afmbnn$,
\begin{equation*}\label{ac}
A_\la(\bfl,r)[A^\bfi]=q^{-\frac{1}{2}d_\bfi}A_\la(\bfl,r)*\chi_\bfi=\sum_{\bfm\in\mathfrak M}q^{\frac{1}{2}
(-c(\bfi-\bfm,\bfi)-d_{\bfi})}\chi_{\bfi-\bfm}
\end{equation*}
where $\mathfrak M=\{\bfm\in\mbz^r_\vtg\mid m_i\in\{0,1\},
\la=m_1\afbse_{i_1-1}+\cdots+m_r\afbse_{i_r-1}\}$.
Hence,
\begin{equation}\label{ac}
A_\la(\bfl,r)\cdot\og_\bfi=\sum_{\bfm\in\mathfrak M}\up^{-c(\bfi-\bfm,\bfi)-d_{\bfi}+d_{\bfi-\bfm}}[A^{\bfi-\bfm}]
\end{equation}

We now calculate the right hand side of \eqref{eq(3)}. By Corollary \ref{action-ss},
$$\ti u_\la^+\cdot\og_\bfi =\ti
u_\la^+\cdot(\og_{i_1}\ot\cdots\ot\og_{i_r})=\sum_{\bfm\in\mathfrak M}
 \up^{-c'(\bfi,\bfm)}
 \og_{i_1-m_1}\ot\cdots\ot\og_{i_r-m_r}.$$
where
\begin{equation}\label{eq c'}
c'(\bfi,\bfm):=\sum_{t<s}m_t(1-m_s)\lan\afbse_{i_s},
\afbse_{i_t}\ran=\!\sum_{1\leq t<s\leq r}m_t(1-m_s)\dt_{i_t,i_s}-d(\bfi,\bfm),\\
\end{equation}
with $d(\bfi,\bfm)=\sum_{1\leq t<s\leq r}m_t(1-m_s)\dt_{\ol{i_t},\ol{i_s+1}}$.
To make a comparison with \eqref{ac}, we need to write
$\og_{\bfi-\bfm}:=\og_{i_1-m_1}\ot\cdots\ot\og_{i_r-m_r}$ as a linear combination of the basis $\{[A^\bfj]\}_{\bfj\in\afInr}$.
For $\bfi\in\afJnr$, order the set
$$\{t\in\mbz\mid 1\leq t\leq
r,\,i_t=1,\,m_t=1\}=\{t_1,t_2,\ldots,t_a\}$$ by
$t_1<t_2<\cdots<t_a$. Then, by definition, $
\og_{\bfi-\bfm}X_{t_1}^{-1}X_{t_2}^{-1}\cdots
X_{t_a}^{-1}=\og_{\bfj}, $ where
$\bfj=\bfi-\bfm+n(\bse_{t_1}^\vtg+\bse_{t_2}^\vtg+\cdots+\bse_{t_a}^\vtg)\in
I(n,r)$. By Proposition \ref{bimodule-isom} (see also Remark
\ref{two bases for tsp}), $\og_\bfj=[A^{\bfj}]$. Thus,
$$\og_{\bfi-\bfm}=\og_{\bfj}X_{t_a}X_{t_{a-1}}\cdots
X_{t_1}=[A^{\bfj}]X_{t_a}X_{t_{a-1}}\cdots X_{t_1}$$
Since $\ti T_{\rho^{-1}}=\ti T_{r-1}\cdots\ti T_2\ti T_1 X_1$, it follows that
$
X_k=\ti T_{k}^{-1}\cdots \ti T_{r-2}^{-1}\ti T_{r-1}^{-1}\ti T_{\rho^{-1}} \ti
T_{1}\ti T_{2}\cdots \ti T_{k-1},
$ for $1\leq k\leq r$.
Now, by Proposition \ref{action2-VV} and noting $\ti T_{t_a}^{-1}=\ti T_{t_a}-(\up-\up^{-1})$,
$$[A^{\bfj}]\ti T_{t_a}^{-1}=\begin{cases} \up^{-1}[A^{\bfj s_{t_a}}],&\text{ if }n=j_{t_a}=j_{t_a+1};\\
[A^{\bfj s_{t_a}}],&\text{ if }n=j_{t_a}>j_{t_a+1}.\\
\end{cases}$$
Thus, a repeatedly application of Proposition \ref{action2-VV} gives
$$[A^{\bfj}]\ti T_{t_a}^{-1}\cdots\ti T_{r-2}^{-1}\ti T_{r-1}^{-1}\ti T_{\rho^{-1}}=\up^{-b} [A^{\bfj s_{t_a}\cdots s_{r-1}\rho^{-1}}],$$
where $b=|\{s\in\mbz\mid t_a< s\leq
r,\,i_s=n,\,m_s=0\}|$. Since $1\leq i_1\leq \cdots\leq i_r\leq n$, it follows that
$$b=|\{s\in\mbz\mid 1\leq s\leq
r,\,i_s=n,\,m_s=0\}|.$$ If $\bfj'=\bfj s_{t_a}\cdots s_{r-1},
\bfj''=\bfj s_{t_a}\cdots s_{r-1}\rho^{-1}$, then $j'_r=n$ and so
$j''_1=j'_0=0<j_k''$ for all $2\leq k\leq r$. The last inequality is
seen from the fact that $\{j_k''\}_{2\leq k\leq r}=\{j_k'\}_{1\leq
k\leq r-1}$ contains only positive integers. Applying Proposition
\ref{action2-VV} again yields
$$[A^{\bfj''}]\ti T_1\cdots\ti
T_{t_a-1}=[A^{\bfj''s_1\cdots s_{t_a-1}}]=[A^{\bfj-n\bse_{t_a}}].$$
Hence, $[A^{\bfj}]X_{t_a}=\up^{-b}[A^{\bfj-n\bse_{t_a}}]$.
Continuing this argument, we obtain eventually $\og_{\bfi-\bfm}=
\up^{-ab}[A^{\bfi-\bfm}]$, where
$$\aligned
ab&=|\{(s,t)\in\mbz^2\mid 1\leq s,t\leq
r,\,i_s=n,\,m_s=0,\,i_t=1,\,m_t=1\}|\\
&=\sum_{1\le s,t\le r}m_t(1-m_s)\dt_{i_s,n}\dt_{i_t,1}=\sum_{1\le t<s\le r}m_t(1-m_s)\dt_{i_s,n}\dt_{i_t,1}=d(\bfi,\bfm),
\endaligned$$
since if $1\leq t< s\leq r$ and $\ol{i_t}=\ol{i_s+1}$ then $i_s=n$ and $i_t=1$.
Therefore, we obtain
$$\ti u_\la^+\cdot\og_\bfi =\sum_{\bfm\in\mathfrak M}
 \up^{-c'(\bfi,\bfm)-d(\bfi,\bfm)}
[A^{\bfi-\bfm}].$$
Comparing this with \eqref{ac}, it remains to prove that
$$d(\bfi,\bfm)+c'(\bfi,\bfm)=c(\bfi-\bfm,\bfi)+d_{\bfi}-d_{\bfi-\bfm}.$$
The number $c(\bfi-\bfm,\bfi)=c(\bfL_{\bfi-\bfm},\bfL_\bfi)$ is defined in
\eqref{a and c}. Since
$$\dim \bfL_{\bfi,i+1}/\bfL_{\bfi,i}=|\bfi^{-1}(i+1)|=\sum_{1\leq s\leq
r}\dt_{\ol{i_s},\ol{i+1}}\quad\text{and}\quad\la_i=\sum_{1\leq s\leq
r}m_s\dt_{\ol{i_s-1},\ol{i}}\quad \text{for}\ i\in\mbz,$$
it follows that
\begin{equation*}\label{eq1}
\aligned
c(\bfi-\bfm,\bfi)&=\sum_{1\leq i\leq
n}\la_{i+1}
 (|\bfi^{-1}(i+1)|-\la_i)=\sum_{1\leq s,t\leq r}m_t(1-m_s)
 \dt_{\ol{i_t-1},\ol{i_s}}\\
&=\sum_{1\leq s<t\leq r}m_t(1-m_s)
 \dt_{\ol{i_t-1},\ol{i_s}}+\sum_{1\leq t<s\leq r}m_t(1-m_s)
 \dt_{\ol{i_t-1},\ol{i_s}}.\\
 &=\sum_{1\leq s<t\leq r}m_t(1-m_s)\dt_{{i_t-1},{i_s}}(1-\dt_{i_s,n})+
 d(\bfi,\bfm).
 \endaligned
 \end{equation*}
Hence,
$$c(\bfi-\bfm,\bfi)-c'(\bfi,\bfm)\!=\!2d(\bfi,\bfm)
+\!\!\!\sum_{1\leq s<t\leq
r}m_t(1-m_s)\dt_{{i_t-1},{i_s}}(1-\dt_{i_s,n}) -\!\!\!\sum_{1\leq
t<s\leq r}m_t(1-m_s)\dt_{i_s,i_t}.$$ On the other hand, for any
$\bfi\in\afInr$, $d_{\bfi}=|\text{Inv}(\bfi)|$ by Lemma \ref{di}.
Observe that the set
$$\sI:=\{(s,t)\in\sL\mid i_s\ge i_t+1\}\cup\{(s,t)\in\sL\mid i_s=i_t,m_s=m_t \text{ or }m_s=0, m_t=1\}$$
is the intersection $\text{Inv}(\bfi-\bfm)\cap\text{Inv}(\bfi)$,
where $\sL=\{(s,t)\in\mbz^2\mid 1\le s\le r, s<t\}$. It turns out that
$\text{Inv}(\bfi-\bfm)\backslash\sI=\{(s,t)\in\sL\mid i_s+1=i_t, m_s=0,m_t=1\}$ and $\text{Inv}(\bfi)\backslash\sI=\{(s,t)\in\sL\mid i_s=i_t, m_s=1,m_t=0\}$.
Hence,
$$d_{\bfi-\bfm}-d_{\bfi}=\sum_{1\leq s\leq r\atop t\in\mbz,\, s<t}
m_t(1-m_s)\dt_{i_t,i_s+1}-\sum_{1\leq s\leq r\atop
t\in\mbz,\,s<t}m_s(1-m_t)\dt_{i_t,i_s}.$$ Since $m_{a+r}=m_a$,
$i_{a+r}=i_a+n$, for all $a\in\mbz$, and $\bfi\in \afJnr$, it
follows that $i_s\neq i_t$, for all $1\leq s\leq r, t>r$,  and
$i_t=n+1\iff i_{t'}=1$ for some $1\leq t'\leq r$ with $t=t'+r$.
Thus, the above sums can be rewritten as
\begin{equation*}\label{eq2}
\begin{split}
d_{\bfi-\bfm}-d_{\bfi}&= \sum_{1\leq s<t\leq
r}m_t(1-m_s)\delta_{i_t,i_s+1}(1-\delta_{i_s,n})\\
&\quad+\sum_{1\leq t<s\leq r}m_t(1-m_s)\delta_{i_t,1}\delta_{i_s,n}-
\sum_{1\leq s<t\leq r}m_s(1-m_t)\delta_{i_t,i_s}\\
&=c(\bfi-\bfm,\bfi)-c'(\bfi,\bfm)-d(\bfi,\bfm),
 \end{split}
\end{equation*}
as required.
\end{proof}

\section{Triangular decompositions of affine quantum Schur algebras}

In this section we study the ``triangular parts" of $\afSr$. We first
 show that $\afSr$ admits a triangular relation for certain structure
 constants relative to the BLM basis. This relation allows us to produce
 an integral basis from which we obtain a triangular decomposition.
 We will give an application of this decomposition in the next section
 by proving that the algebra homomorphism $\xi_r:\dHallr\ra \afbfSr$
defined in \S3.4 is surjective.

Keep the notation in the
previous sections. We will continue to use convolution product to derive properties or formulas via the isomorphism
mentioned in Remark \ref{implicit use bcp}.
 Thus, when the convolution product $*$ is used, we automatically
mean the affine quantum Schur algebra is the algebra $\afSr_R$ with $R=\mbz[\sqrt{q},\sqrt{q}^{-1}]$
obtained by specializing $\up$ to $\sqrt{q}$,
for a prime power $q$, and
is identified with the convolution algebra $R_G(\scrY\times\scrY)$.
The following results are taken from \cite{Lu99}.

\begin{Lem} \label{triangular matrix}
Let $A=(a_{i,j})\in\afThnr$ and let $(\bfL,\bfL')\in\sO_A$ where $\bfL=(L_i)_{i\in\mbz}$ and $\bfL'=(L_i')_{i\in\mbz}$.
\begin{itemize}
\item[(1)] $A$ is upper triangular if and only if $L_i'\han L_i$ for
all $i$.
\item[(2)] $A$ is lower triangular if and only if $L_i\han L_i'$ for all
$i$.
\item[(3)] $\dim(L_i/L_{i-1})=\sum_{k\in\mbz}a_{i,k},\quad\ \ \  \dim(L_i'/L_{i-1}')=\sum_{k\in\mbz}a_{k,i}$.
\item[(4)] $\dim\left(\frac{L_i}{L_i\cap L_{j-1}'}\right)=\sum_{s\leq i,t\geq
j}a_{s,t},\quad \dim\left(\frac{L_j'}{L_{i-1}\cap
L_{j}'}\right)=\sum_{s\geq i,t\leq j}a_{s,t}.$
\end{itemize}
\end{Lem}

For $A\in\afMnz$, let
$$\sg_{i,j}(A)=\begin{cases}
\sum\limits_{s\leq i,t\geq j}a_{s,t}, \;\; &\text{if $i<j$;}\\
\sum\limits_{s\geq i,t\leq j}a_{s,t},  &\text{if $i>j$}.
\end{cases}$$
For any fixed $x_0\in\mbz$ and $i<j$, it is easy to see that there are two bijective maps
\begin{equation*}
\begin{split}
\{(b,s,t)\in\mbz^3\mid &s-bn\leq i<j\leq t-bn,s\in[x_0+1,x_0+n]\}
\longrightarrow \{(s,t)\in\mbz^2\mid s\leq i,t\geq j\}\\
&\hspace{60mm}(b,s,t)\lm (s-bn,t-bn)\\
\{(b,s,t)\in\mbz^3\mid &t-bn\leq i<j\leq s-bn,s\in[x_0+1,x_0+n]\}
\longrightarrow \{(s,t)\in\mbz^2\mid s\geq j,t\leq i\}\\
 &\hspace{60mm}(b,s,t)\lm (s-bn,t-bn).
\end{split}
\end{equation*}
Thus, we obtain an alternative interpretation of $\sg_{i,j}(A)$:
$$\sg_{i,j}(A)=\begin{cases}
\sum\limits_{x_0+1\leq s\leq x_0+n\atop s<t}a_{s,t}|\{
b\in\mbz\mid s-bn\leq i<j\leq t-bn\}|,\;\; &\text{if $i<j$;}\\
\sum\limits_{x_0+1\leq s\leq x_0+n\atop s>t}a_{s,t}|\{ b\in\mbz\mid
t-bn\leq j<i\leq s-bn\}|, &\text{if $i>j$}.
\end{cases}$$
In particular, $\sg_{i,j}(A)<\infty$.

Further, for $A,B\in\afMnz$, define (see \cite[\S6]{DF09})
\begin{equation}\label{order preceq}
B\pr A \Longleftrightarrow\sg_{i,j}(B)\leq\sg_{i,j}(A)\;\;\text{for all $i\not=j$}.
\end{equation}\index{$\pr$, order on $\afMnz$ or $\afThn$}
\noindent We put $B\p A$ if $B\pr A$ and, for some pair $(i,j)$ with
$i\not=j$, $\sg_{i,j}(B)<\sg_{i,j}(A)$. It is shown in
\cite[Th.~6.2]{DF09} that if $A,B\in\afThnp$ satisfy
$\bfd(A)=\bfd(B)$, then
$$B\leq_\deg A\Longleftrightarrow B\pr A.$$\index{degeneration order}
In other words, the ordering $\pr$ is an extension of the
degeneration ordering defined in \eqref{partial order leq on the set
afThnp}.

For $A\in\afThn$, write
$A=A^++A^0+A^-=A^\pm+A^0,$
 where $A^+\in\afThnp$, $A^-\in\afThnm:=\{\tB\mid B\in\afThnp\}$, $A^\pm=A^++A^-$, and $A^0$ is a diagonal matrix.

\begin{Lem}\label{lemma2 for triangular decomposition of affine q-Schur algebra}
Let $A\in\afThnpm$ with $\sg(A)\leq r$.
\begin{itemize}
\item[(1)] For $\mu\in\afLa(n,r-\sg(A^+))$, $\nu\in\afLa(n,r-\sg(A^-))$, if $$([A^++\diag(\mu)]*[A^-+\diag(\nu)])(\bfL,\bfL'')\not=0,$$
where $(\bfL,\bfL'')\in\sO_B$
for some $B\in\afThn$, then $B\pr A.$

\item[(2)] For $\la\in\afLa(n,r-\sg(A))$, if $(\bfL,\bfL'')\in\sO_{A+\diag(\la)}$, then
$$\bin_{\mu\in\afLa(n,r-\sg(A^+))\atop\nu\in
\afLa(n,r-\sg(A^-))}\{\bfL'\in\afFn\mid (\bfL,\bfL')
\in\sO_{A^++\diag(\mu)},(\bfL',\bfL'')\in\sO_{A^-+\diag(\nu)}\}=\{\bfL\cap\bfL''\}.$$
\end{itemize}
\end{Lem}

\begin{proof}
(1)  Since $([A^++\diag(\mu)]*[A^-+\diag(\nu)])(\bfL,\bfL'')\not=0$, there exists
$\bfL'\in\afFn$ such that $(\bfL,\bfL')\in\sO_{A^++\diag(\la)}$ and
$(\bfL',\bfL'')\in\sO_{A^-+\diag(\mu)}$. Also,
$(\bfL,\bfL'')\in\sO_B$. Hence, by Lemma \ref{triangular matrix}(1) and (2),
$\bfL'\subseteq\bfL$ and $\bfL'\subseteq\bfL''$, and by
Lemma \ref{triangular matrix}(4),
\begin{equation*}
\sg_{i,j}(A)=
\begin{cases}
\dim\left(\frac{L_i}{L_i\cap L_{j-1}'}\right),&\quad\text{if $i<j$};\\
\dim\left(\frac{L_j''}{L_{i-1}'\cap L_j''}\right),&\quad\text{if
$i>j$,}
\end{cases}\qquad
\sg_{i,j}(B)=
\begin{cases}
\dim\left(\frac{L_i}{L_i\cap L_{j-1}''}\right),&\quad\text{if $i<j$};\\
\dim\left(\frac{L_j''}{L_{i-1}\cap L_j''}\right),&\quad\text{if
$i>j$.}
\end{cases}
\end{equation*}
Therefore, $\bfL'\han\bfL\cap\bfL''$ and
\begin{equation*}
\sg_{i,j}(A)-\sg_{i,j}(B)=
\begin{cases}
\dim\left(\frac{L_i\cap L_{j-1}''}{L_i\cap L_{j-1}'}\right)\geq 0,&\quad\text{if $i<j$};\\
\dim\left(\frac{L_{i-1}\cap L_j''}{L_{i-1}'\cap L_j''}\right)\geq 0,
&\quad\text{if $i>j$}.
\end{cases}
\end{equation*}
Consequently, $B\pr A$.

(2) If $\bfL'\in\afFn$ satisfies
$$(\bfL,\bfL')\in\sO_{A^++\diag(\mu)},(\bfL',\bfL'')\in\sO_{A^-+\diag(\nu)}$$
for some $\mu\in\afLa(n,r-\sg(A^+)$ and
$\nu\in\afLa(n,r-\sg(A^-)$, then $\bfL'\han\bfL\cap\bfL''$ as seen above.
Thus, for all $i\in\mbz$,
\begin{equation*}
\begin{split}
\dim\left(\frac{L_i\cap L_i''}{L_i'}\right)&=\dim(L_i/L_i')-\dim\left(\frac{L_i}{L_i\cap L_i''}\right)\\
&=\dim\left(\frac{L_i}{L_i\cap L_i'}\right)-\dim\left(\frac{L_i}{L_i\cap L_i''}\right)\\
&=\sum_{s\leq i<t}a_{s,t}-\sum_{s\leq i<t}a_{s,t}=0,
\end{split}
\end{equation*}
by Lemma \ref{triangular matrix}(4) again. Hence, $\bfL'=\bfL\cap\bfL''$, proving the assertion.
\end{proof}

\begin{Prop}\label{triangular formula in A(bfj,r)}
Let $A\in\afThnpm$. Then the following {\sf triangular relation}\index{triangular relation}
relative to $\p$ holds:
$$
\aligned
A^+(\bfl,r)A^-(\bfl,r)&=A(\bfl,r)+\sum_{C\in\afThnr\atop C\p A} f_{C,A}[C]\quad(\text{in $\afSr$})\\
&=A(\bfl,r)+\sum_{B\in\afThnpm\atop B\p A,\bfj\in\afmbzn}g_{B,\bfj,A;r}B(\bfj,r)\quad(\text{in $\afbfSr$}),
\endaligned$$
where $f_{C,A}\in\sZ$, $g_{B,\bfj,A;r}\in\mbq(\up)$.
\end{Prop}

\begin{proof}  Since the elements $B(\bfj,r)$ span $\afbfSr$ by Proposition \ref{BLMbasis}, the second equality
follows from the first one. We now prove the first equality.

Let $r^\pm=\sg(A)$, $r^+=\sg(A^+)$ and $r^-=\sg(A^-)$. There is nothing to prove if $r^\pm>r$.
Assume now $r^\pm\leq r$. By \eqref{e-basis multn},
$$\aligned
A^+(\bfl,r)A^-(\bfl,r)&=\sum_{\mu\in\La(n,r-r^+)}\sum_{\nu\in\La(n,r-r^-)}[A^++\diag(\mu)][A^-+\diag(\nu)]\\
&=\sum_{C\in\afThnr}\sum_{\mu,\nu}\up^{-d_{A^++\diag(\mu)}-d_{A^-+\diag(\nu)}}p_{A^++\diag(\mu),A^-+\diag(\nu),C}\up^{d_C}[C]\\
&=\sum_{C\in\afThnr}f_{C,A}[C],
\endaligned$$
where
$$f_{C,A}=\sum_{\mu\in\afLa(n,r-r^+)\atop\nu\in
\afLa(n,r-r^-)}\up^{d_C-d_{A^++\diag(\mu)}-d_{A^-+\diag(\nu)}}
p_{A^++\diag(\mu),A^-+\diag(\nu),C}.$$
If $f_{C,A}\not=0$, then  $p_{A^++\diag(\mu),A^-+\diag(\nu),C}\not=0$ for some $\mu,\nu$ as above. Thus, by definition, there is a finite filed $\field$ of
$q$ elements such that
$$p_{A^++\diag(\mu),A^-+\diag(\nu),C}|_{\up^2=q}=\afg_{A^++\diag(\mu),A^-+\diag(\nu),C;q}
\not=0,$$
where $\afg_{A^++\diag(\mu),A^-+\diag(\nu),C;q}$ is defined in \eqref{afg}.
By Lemma
\ref{lemma2 for triangular decomposition of affine q-Schur
algebra}(1), we conclude $C\pr A$. We need to prove that
$f_{C,A}=1$ if $C=A+\diag(\la)$ for any $\la\in\afLa(n,r-r^\pm)$. First,
Lemma \ref{lemma2 for triangular decomposition of affine q-Schur
algebra}(2) implies that there exist unique $\mu,\nu$ such that
$\afg_{A^++\diag(\mu),A^-+\diag(\nu),A+\diag(\la);q}=1$.  Thus,
\begin{equation}\label{Lem-ddd}
f_{A+\diag(\la),A}=\up^{d_{A+\diag(\la)}-d_{A^++\diag(\mu)}-d_{A^-+\diag(\nu)}}.
\end{equation}
Second, since $d_{A+\diag(\la)}=d_{A^++\diag(\mu)}+d_{A^-+\diag(\nu)}$ by a direct computation
(see Lemma \ref{ddd}),
 we obtain $f_{A+\diag(\la),A}=1.$
Finally, since $\pr$ is a partial ordering on $\afThnpm$ by
\cite[Lem.~6.1]{DF09}, we conclude that
$$A^+(\bfl,r)A^-(\bfl,r)=A(\bfl,r)+g,$$
 where $g$ is a $\sZ$-linear combination of $[C]$ with $C\in\afThnr$
and $C\p A$.
\end{proof}

We now use the triangular relation to establish a triangular
decomposition for the $\sZ$-algebra $\afSr$.

Consider the following $\sZ$-submodules of $\afSr$
$$\aligned
&\afSrp=\spann_\sZ\{A(\bfl,r)\mid A\in\afThnp\},\\
&\afSrm=\spann_\sZ\{A(\bfl,r)\mid A\in\afThnm\},\;\;\text{ and}\\
&\afSrz=\spann_\sZ\{[\diag(\la)]\mid \la\in\afLanr\}.
\endaligned$$\index{affine quantum Schur algebra!the $\pm,0$-part, $\sS_\vtg(n,r)^\pm$, $\afSrz$}
As homomorphic images of Ringel--Hall algebras (see Proposition
\ref{afzrpm}), both $\afSrp$ and $\afSrm$ are $\sZ$-subalgebras of
$\afSr$. It can be directly checked that $\afSrz$ is also a
$\sZ$-subalgebra of $\afSr$. Moreover, $\afSrz$ is isomorphic to the
zero part of the (non-affine) quantum Schur algebra. We will call
the subalgebras $\afSrm$, $\afSrz$ and $\afSrp$ the {\it negative,
zero} and {\it positive parts} of $\afSr$, respectively. The next
result shows that all three subalgebras are free as $\sZ$-modules.
Recall the notation $\leb{X;a\atop t}\rib$ introduced in
\eqref{Gauss poly2}.

\begin{Prop}\label{afSrz} Let $\bffkk_i=0(\bfe_i,r)$, $1\leq i\leq n$.

$(1)$ The set $\{A(\bfl,r)\mid A\in\afThnp,\sg(A)\leq r\}$ (resp.,
$\{A(\bfl,r)\mid A\in\afThnm,\sg(A)\leq r\}$) forms a $\sZ$-basis
for $\afSrp$ (resp., $\afSrm$).

$(2)$ For $\la\in\afLanr$,
$$[\diag(\la)]=\bfone_\la:=\leb{\bffkk_1;0\atop\la_1}\rib
\leb{\bffkk_2;0\atop\la_2}\rib\cdots\leb{\bffkk_n;0\atop\la_n}\rib$$\index{$\bfone_\la$,
idempotent $[\diag(\la)]$} In particular, the set
$\{\bfone_\la\mid\la\in\afLanr\}$ forms a $\sZ$-basis of $\afSrz$.
\end{Prop}
\begin{proof}
Assertion (1) follows from the definition of $A(\bfl,r)$. To see
(2), since $\bffkk_i[\diag(\mu)]=\sum_{\nu\in\afLanr}\up^{\nu_i}
[\diag(\nu)][\diag(\mu)]=\up^{\mu_i}[\diag(\mu)],$ by \eqref{product
[diag(la)][A] in affine q-Schur algebra}, it follows that
$\leb{\bffkk_i;0\atop t}\rib[\diag(\mu)]=\leb{\mu_i\atop
t}\rib[\diag(\mu)]$. Hence, for $\la\in\afLanr$,
$$\bfone_\la=\sum_{\mu\in\afLanr}\bfone_\la[\diag(\mu)]
=\sum_{\mu\in\afLanr}\leb{\mu_1\atop\la_1}\rib
\cdots\leb{\mu_n\atop\la_n}\rib[\diag(\mu)]=[\diag(\la)],$$ as
desired.\footnote{In the literature, $\bfone_\la$ is denoted by
$\ttk_\la$ or $1_\la$. We modified the notation in order to
introduce its preimage ${\frak L}_\la$ in $\dHallr$; see
\eqref{kk-lambda}.}
\end{proof}

As in \S2.2, let $\dHallr^\pm$ (resp., $\Hall^\pm$) be the
$\mbq(\up)$-submodules (resp., $\sZ$-submodules) of $\dHallr$
spanned by $u_A^\pm$ for $A\in\afThnp$. They are respectively the
$\mbq(\up)$-subalgebras and $\sZ$-subalgebras of $\dHallr$. The
above proposition together with the results in \S3.6 gives the
following result; see Remark \ref{Lform2}.

\begin{Coro} \label{central-elem-integral}
For each $s\geq 1$, we have $\sfz_s^+\in\Hall^+$ and
$\sfz_s^-\in\Hall^-$.
\end{Coro}

\begin{proof} We only prove $\sfz_s^+\in\Hall^+$. The proof for
$\sfz_s^-\in\Hall^-$ is similar.

By Proposition \ref{afzrpm} and Theorem \ref{xirl}, the restriction
of $\xi_r:\dHallr\ra\afbfSr$ gives an algebra homomorphism
$$\zeta_r^+:\dHallr^+\lra \afbfSr= \End_{\afbfHr}(\bfOg^{\ot r})$$
taking $\ti u_A^+\mapsto A(\bfl,r)$ for $A\in\afThnp$. Write
$$\sfz_s^+=\sum_{A\in\afThnp}f_A\ti u_A^+$$
 where all but finitely many $f_A\in\mbq(\up)$ are zero. By \eqref{action central
 elts}, we have $\sfz_s^+(\Og^{\ot r})\han \Og^{\ot r}$. Hence,
$\zeta_r^+(\sfz_s^+)\in\End_{\afHr}(\Og^{\ot r})=\sS_\vtg(n,r)$; see
Proposition \ref{bimodule-isom}. In other words,
$$\zeta_r^+(\sfz_s^+)=\sum_{A\in\afThnp}f_AA(\bfl,r)=\sum_{A\in\afThnp,\,\sg(A)\leq
r}f_AA(\bfl,r)\in\afSrp.$$
 By Proposition \ref{afSrz}(1), $\{A(\bfl,r)\mid A\in\afThnp,\sg(A)\leq r\}$ is a $\sZ$-basis
for $\afSrp$. Hence, $f_A\in\sZ$ for all $A\in\afThnp$ with
$\sg(A)\leq r$. Since $r$ can be chosen to be an arbitrary positive
integer, it follows that $f_A\in\sZ$ for all $A\in\afThnp$. We
conclude that $\sfz_s^+ \in\Hall^+$.
\end{proof}

We now patch the three bases to obtain a basis for $\afSr$. For $A\in\afThn$, define
(cf. \cite{BLM})
$$|\!|A|\!|=\sum_{1\leq i\leq n\atop
i<j}\frac{(j-i)(j-i+1)}{2}a_{i,j}+\sum_{1\leq i\leq n\atop
i>j}\frac{(i-j)(i-j+1)}{2}a_{i,j}.$$

\begin{Lem}\label{||A||}
For $A\in\afThn$, the equality
$$|\!|A|\!|=\sum_{1\leq i\leq n\atop
i<j}\sg_{i,j}(A)+\sum_{1\leq i\leq n\atop i>j}\sg_{i,j}(A)$$
 holds. In particular, if $A,B\in\afThn$ satisfy $A\p B$, then $|\!|A|\!|<|\!|B|\!|$.
\end{Lem}

\begin{proof}
By definition, we have
\begin{equation*}
\begin{split}
\sum_{1\leq i\leq n\atop i<j}\sg_{i,j}(A)&=\sum_{1\leq i\leq n\atop
s\leq i<j\leq t}a_{s,t}=\sum_{1\leq s\leq n\atop
s<t}a_{s,t}\mid\{(i,j)\mid s\leq i<j\leq t\}|\\
&=\sum_{1\leq s\leq n\atop s<t}\frac{(t-s)(t-s+1)}{2}a_{s,t},
\end{split}
\end{equation*}
\begin{equation*}
\begin{split}
 \sum_{1\leq i\leq n\atop i>j}\sg_{i,j}(A)&
 =\sum_{1\leq i\leq n\atop t\leq j<i\leq s}a_{s,t}
 =\sum_{1\leq s\leq n\atop s>t}a_{s,t}|\{(i,j)\mid t\leq j<i\leq s\}|\\
&=\sum_{1\leq s\leq n\atop s>t}\frac{(s-t)(s-t+1)}{2}a_{s,t}.\\
\end{split}
\end{equation*}
The assertion follows from the definition of $|\!|A|\!|$.
\end{proof}

For $A\in\afThn$ and $i\in\mbz$, define the ``hook sum''
$$\sg_{i}(A)=a_{i,i}+\sum_{j<i}(a_{i,j}+a_{j,i}).$$ It is easy to see
that $\sg(A)=\sum_{1\leq i\leq n}\sg_i(A)$. Let
\begin{equation}\label{sequence of hook sums}
\bfsg(A)=(\sg_i(A))_{i\in\mbz}\in\afLanr\quad\text{ and }\quad
\ttp_{_A}=A^+(\bfl,r)\bfone_{\bfsg(A)}A^-(\bfl,r)
\end{equation}
\index{$\bfsg(A)$, sequence of hook sums}\index{$\ttp_{_A}$, PBW
type basis element for $\afSr$} We now describe a PBW type basis for
$\afSr$.\index{PBW type basis for $\afSr$}

\begin{Thm}\label{PBW basis of affine q-Schur algebras}
Keep the notation introduced above. There exist $g_{_{B,A}}\in\sZ$
such that
$$\ttp_{_A}=[A]+\sum_{B\in\afThnr,B\p A}g_{_{B,A}}[B].$$
Moreover, the set $\sP_r:=\{\ttp_{_{A}}\mid A\in\afThnr\}$
 forms a $\sZ$-basis for $\afSr$. In particular, we obtain a (weak) triangular decomposition:\index{triangular decomposition of $\afSr$}
 $$\afSr=\afSrp\afSrz\afSrm.$$
\end{Thm}

\begin{proof} By the notational convention right above Lemma \ref{lemma2 for triangular decomposition of affine q-Schur algebra}, if $A\in\afThnr$, then
$\sg_{i}(A^\pm)$ is the $i$-th component of $\co(A^+)+\ro(A^-)$ and
$$\bfone_{\bfsg(A)}A^-(\bfl,r)=[\diag(\bfsg(A))]A^-(\bfl,r)=A^-(\bfl,r)[\diag
(\bfsg(A)+\co(A^-)-\ro(A^-))].$$ On the other hand,
$A^+(\bfl,r)A^-(\bfl,r)=A^{\pm}(\bfl,r)+g$, where $g$ is a
$\sZ$-linear combination of $[B]$ with $B\in\afThnr$ and $B\p A$, by
Proposition \ref{triangular formula in A(bfj,r)}. Thus,
\begin{equation*}
\begin{split}
\ttp_{_A}&=A^+(\bfl,r)\bfone_{\bfsg(A)} A^-(\bfl,r)\\
&=A^+(\bfl,r)A^-(\bfl,r)[\diag(\bfsg(A)+\co(A^-)-\ro(A^-))]\\
&=A^\pm(\bfl,r)[\diag(\bfsg(A)+\co(A^-)-\ro(A^-))]+g'\\
&=[A^\pm+\diag(\bfsg(A)-(\co(A^+)+\ro(A^-))]+g'\\
&=[A]+g'
\end{split}
\end{equation*}
where $g'$ is the $\sZ$-linear combination of $[B]$ with
$B\in\afThnr$ and $B\p A$. Thus, the set $\sP_r$ is linearly
independent. To see it spans, we can apply Lemma \ref{||A||} and an
induction on $|\!|A|\!|$ to show that $[A]$ is a $\sZ$-linear
combination of $\ttp_{_{B}}$ with $B\in\afThnr$. Hence, $\sP_r$
forms a $\sZ$-basis for $\afSr$. The last assertion follows from
Proposition \ref{afSrz}.
\end{proof}

\section{Affine quantum Schur--Weyl duality, I}

We now use the triangular decomposition given in Theorem \ref{PBW basis of affine q-Schur algebras}
to partially establish an affine analogue of the quantum Schur--Weyl reciprocity.

As in Remark \ref{subalgebra quantum gl}(4), for each $m\geq 0$,
let $\dHallr^{(m)}$ (resp., $\dHallr^{(0)}$) denote the subalgebra
of $\dHallr$ generated by $u^+_i, u_i^-, K_i^{\pm 1}, \sfz_s^+,
\sfz_s^-$ (resp., $u^+_i, u_i^-, K_i^{\pm 1}$) for $i\in I$, $1\leq
s\leq m$ (resp., for $i\in I$, $m=0$). Thus,
$\dHallr^{(0)}=\bfU_\vtg(n)$.

\begin{Thm} \label{surjective-dHall-aff}
Let $n, r$ be two positive integers with $n\geq 2$.

\begin{itemize}

\item[(1)] The algebra homomorphism $\xi_r:\dHallr\ra \afbfSr$ is
surjective. 
\index{$\xi_r$, epimorphism $\dHallr\to\afbfSr$}

\item[(2)] If we write $r=mn+m_0$ with $m\geq 0$ and $0\leq
m_0<n$, then $\xi_r$ induces a surjective algebra homomorphism
$\dHallr^{(m)}\ra \afbfSr$.
In particular, if $n>r$, then $\xi_r$
induces a surjective algebra homomorphism $\bfU_\vtg(n)\ra
\afbfSr$.
\end{itemize}
\end{Thm}

\begin{proof} (1) As in Remark \ref{Lform2}, let $\Hall^+$ (resp. $\Hall^-$) 
\index{$\ti{\fD}_\vtg(n)$, candidate of Lusztig form!$\Hall^+$,
$+$-part of $\ti{\fD}_\vtg(n)$} \index{$\ti{\fD}_\vtg(n)$, candidate of Lusztig form!$\Hall^-$,
$-$-part of $\ti{\fD}_\vtg(n)$} 
be the $\sZ$-subalgebra of
$\dHallr$ generated by $u_A^+$ (resp., $u_A^-$) for all
$A\in\afThnp$, $\fD_\vtg(n)^0$ the $\sZ$-subalgebra of $ \dHallr$
generated by $K_i^{\pm1}$ and $\leb{K_i;0\atop t}\rib$ for $i\in I$
and $t>0$, and set
\begin{equation}\label{integral form 2}
\ti{\fD}_\vtg(n)=\Hall^+\,\fD_\vtg(n)^0\,\Hall^-\cong
\Hall^+\otimes\fD_\vtg(n)^0\otimes\Hall^-.
\end{equation} \index{$\ti{\fD}_\vtg(n)$, candidate of Lusztig form}
\index{integral
form!$\ti{\fD}_\vtg(n)$, candidate of Lusztig form} By Theorem
\ref{xirl} and Proposition \ref{afSrz}, $\xi_r$ maps $\Hall^\ep$
onto $\sS_\vtg(n,r)^\ep$ for $\ep=+,-$ and $\fD_\vtg(n)^0$ onto
$\sS_\vtg(n,r)^0$. Hence, $\xi_r$ maps $\ti{\fD}_\vtg(n)$ onto
$\afSr$ by Theorem \ref{PBW basis of affine q-Schur algebras}.
Taking base change to $\mbq(\up)$ gives the required surjectivity.


(2) By Proposition \ref{generators-RH-alg}, $\dHallr$ is generated
by $u^+_i, u_i^-, K_i^{\pm 1}, u_{s\dt}^+, u_{s\dt}^-$ ($i\in I$,
$s\in\mbz^+$). For each $s\geq 1$, the semisimple module $S_{s\dt}$
has dimension $sn$. Thus, if $s>m\geq 1$, then $\dim S_{s\dt}>r$.
Moreover, in this case, we have by Corollary \ref{action-ss} that
for each $\og_\bfi\in\bdOg^{\ot r}$,
$$u_{s\dt}^+\cdot \og_\bfi=0=u_{s\dt}^-\cdot \og_\bfi.$$
 In other words, $\xi_r(u_{s\dt}^+)=0=\xi_r(u_{s\dt}^-)$ whenever
$s>m$ and $m\geq 1$. The assertion follows from the fact that
$\dHallr^{(m)}$ is generated by $u^+_i, u_i^-, K_i^{\pm 1},
u_{s\dt}^+, u_{s\dt}^-$ ($i\in I$, $0\leq s\leq m$); see Remark
\ref{subalgebra quantum gl}(4). The last assertion follows from
Corollary \ref{n>r}.
\end{proof}

Combining the above theorem with Corollary \ref{action-ss} yields
the following result which can also be derived from \cite[\S4.1,
Th~8.2]{Lu99}. \footnote{Lusztig constructed a canonical basis in
\S4.1 for $\afSr$ and proved that those canonical basis elements
labeled by aperiodic matrices form a basis for
$U_\vtg(n,r):=\xi_r(U_\vtg(n))$.}

\begin{Coro} \label{n>r-integral-surjectivity} Suppose $n>r$. Then $\xi_r:\dHallr\ra \afbfSr$ induces
a surjective $\sZ$-algebra homomorphism
$$\theta_r:U_\vtg(n)\lra\afSr,$$
 where $U_\vtg(n)$ is the $\sZ$-subalgebra of $\dHallr$ generated by
$K_i^{\pm1}$, $\big[{K_i;0\atop t}\big]$, $(u_i^+)^{(m)}$ and
$(u_i^-)^{(m)}$ for $i\in I$ and $t,m\geq 1$ (see \S2.2).
\end{Coro}

\begin{proof} By Proposition \ref{afSrz}, $\xi_r:\dHallr\ra \afbfSr$
induces surjective $\sZ$-algebra homomorphisms
$$\xi^+_{r,\sZ}:\Hall^+\lra \afSrp\;\;\text{and} \;\;\xi^-_{r,\sZ}:\Hall^-\lra
\afSrm.$$
 By \cite[Th.~5.2]{DDX}, $\Hall^+$ is generated by $(u_i^+)^{(m)}$
and $u_\bfa^+$ for $i\in I$, $m\geq 1$, and sincere $\bfa\in\mbn^n$.
Since $n>r$, it follows from Corollary \ref{action-ss} that
$\xi^+_{r,\sZ}(u^+_\bfa)=0$ for all sincere $\bfa\in\mbn^n$. Thus,
$\xi^+_{r,\sZ}$ gives rise to a surjective $\sZ$-algebra
homomorphism
$$\theta_r^+:\comp^+\lra \afSrp,$$
 where $\comp^+$ is the $\sZ$-subalgebra of $\dHallr$ generated by the
$(u_i^+)^{(m)}$. Similarly, we obtain a surjective $\sZ$-algebra
homomorphism
$$\theta_r^-:\comp^-\lra \afSrm,$$
where $\comp^-$ is the $\sZ$-subalgebra of $\dHallr$ generated by
the $(u_i^-)^{(m)}$. By \eqref{integral form of U(n)}, we have
$$U_\vtg(n)=\comp^+\,\fD_\vtg(n)^0\,\comp^-.$$
 The assertion then follows from the triangular decomposition
of $\afSr$ given in Theorem \ref{PBW basis of affine q-Schur
 algebras}.
\end{proof}

If $z\in\mbc$ is not a root of unity, and $\DC(n)$ is the double
Hall algebra over $\mbc$ with parameter $z$ considered in Remark
\ref{NonrootOfUnity2}, then we have an algebra homomorphism
$\xi_{r,\mbc}:\DC(n)\ra\afSr_{\mbc}$ as given in \eqref{xi_{r,z}}.
The proof of the theorem above gives the following.
\begin{Coro}\label{surjective-dHall-aff/C}The $\mbc$-algebra homomorphism
$$\xi_{r,\mbc}:\DC(n)\lra\afSr_{\mbc}$$
is surjective.
\end{Coro}

The above corollary together with \cite[Th.~8.1]{GP} shows the
following Schur--Weyl reciprocity in the affine quantum case.

\begin{Coro} \label{Schur--Weyl duality over field} Let $q$ be a prime power.
 By specializing $\up$ to $\sqrt{q}$, the
$\DC(n)$-$\afHr_{\mbc}$-bimodule $\Og_\mbc^{\ot r}$
induces algebra homomorphisms
$$\xi_{r,\mbc}:\DC(n)\lra\End_{\mbc}(\Og_\mbc^{\ot
r})\;\;\text{and}\;\; \xi_{r,\mbc}^\vee:\afHr_{\mbc}\lra\End_{\mbc}(\Og_{\mbc}^{\ot
r})$$
 such that
$${\rm Im}\,(\xi_{r,\mbc})=\End_{\afHr_{\mbc}}(\Og_\mbc^{\ot
r})=\afSr_{\mbc}\;\;\text{and}\;\; {\rm
Im}\,(\xi_{r,\mbc}^\vee)=\End_{\afSr_{\mbc}}(\Og_\mbc^{\ot r}).$$
\end{Coro}

\begin{Rems} \label{Schur-Weyl-D}
(1) As established in \S2.3,
$\dHallr\cong\bfU(\widehat{\frak{gl}}_n)$, the quantum loop algebra.
Hence, $\xi_r$ induces a surjective algebra homomorphism
$\bfU(\widehat{\frak{gl}}_n)\ra\afbfSr$. Similarly, $\xi_{r,\mbc}$
induces an algebra epimorphism $\afUglC\to\afSr_\mbc.$ Here
$\afUglC$ is the quantum loop algebra over $\mbc$ defined in
Definition \ref{QLA}, which is isomorphic to $\DC(n)$ by Remarks \ref{iso afgln dHallr over C}(2);
cf. Theorem \ref{iso afgln dHallr}.
It would be interesting to find explicit formulas for the action of
generators of $\bfU(\widehat{\frak{gl}}_n)$ on the tensor space
$\bdOg^{\ot r}$.

(2) In \cite[Th.~2]{V}, Vasserot has also constructed a surjective
map $\Psi_z$ from $\afUglC$ to the $K$-theoretic construction
$K^G(Z)_z$ of $\afSr_\mbc$. It would also be interesting to know if
$\Psi_z$ is equivalent to the epimorphism $\xi_{r,\mbc}$, namely, if
$ g_r\circ \Psi_z=f\circ\xi_{r,\mbc}$ under the isomorphisms
$f:\afUglC\overset\sim\to\DC(n)$ and
$g_r:K^G(Z)_z\overset\sim\to\afSr_\mbc$ (see \cite[(9.4)]{GV}).

(3) By the epimorphisms $\xi_r$ and $\zrC$, different types of generators for
double Ringel--Hall algebras (see Remark \ref{subalgebra quantum
gl}(1)) give rise to corresponding generators for affine quantum
Schur algebras. Thus, we may speak of semisimple generators,
homogeneous indecomposable generators, etc., for $\afbfSr$. See
\S\S5.4, 6.2. \index{semisimple generators!$\sim$ for $\afbfSr$}

\end{Rems}

We end this section with a few conjectures.

In the proof of the surjectivity of $\xi_r$ in Theorem
\ref{surjective-dHall-aff}(1), we proved that the restriction of
$\xi_r$ to the free $\sZ$-module $\ti \fD_\vtg(n)$ defined in
\eqref{integral form 2} maps onto the (integral) algebra $\afSr$.
Since the $\Hall$ is generated by certain semisimple generators (see
\cite[Th.~5.2(ii)]{DDX}), it is natural to expect that the
commutator formulas given in Corollary \ref{comm formula for
semisimple} hold in $\ti \fD_\vtg(n)$. Naturally, if the following
conjecture was true, we would call $\ti\fD_\vtg(n)$ an {\it integral
form of Lusztig type} for $\dHallr$.

\begin{Conj}\label{Lform1}
The $\sZ$-module $\ti{\fD}_\vtg(n)$ is a subalgebra of $\dHallr$.
\end{Conj}

We make some comparisons with the integral form
 $\fD_\vtg(n)$ for $\dHallr$ introduced at the end
of  \S2.2 and the restricted integral form discussed in \cite[\S7.2]{FM}.

\begin{Rems} (1)
Since the integral composition algebra $\comp^\pm$ is a
$\sZ$-subalgebra of the integral Ringel--Hall algebra $\Hall^\pm$ which also contains the central generators
 $\sfz^\pm_m$ by Corollary \ref{central-elem-integral}, it follows that
\begin{equation}\label{chain of integral forms}
\fD_\vtg(n)\subset\ti\fD_\vtg(n)\subset\dHallr.
\end{equation}
However, we will see in Remark \ref{Counterexample} that the
restriction to $\fD_\vtg(n)$ of the homomorphism $\xi_r$ in general
does not map onto the integral affine quantum Schur algebra $\afSr$.
Thus, $\fD_\vtg(n)\not=\ti\fD_\vtg(n)$, and we cannot use this
integral form to get the Schur--Weyl theory at the roots of unity.

(2) The restricted integral form $U^{\text{res}}_\up(\afgl)$ is the
$\mbc[\up,\up^{-1}]$-subalgebra of $\bfU(\afgl)$ generated by
divided powers $(\ttx^\pm_{i,s})^{(m)}$, $\ttk_i^\pm$,
$\leb{\ttk_i;0\atop t}\rib$, and $\frac{\ttg_{i,m}}{[m]}$ (see
\cite[\S7.2]{FM}). If we identify $\dHallr$ with $\bfU(\afgl)$ under
the isomorphism $\sE_{\rm H}$ given in Theorem \ref{iso afgln
dHallr}, \eqref{defn of theta_s} implies that
$\frac{\sfz_m^\pm}{m}=-\frac{\theta_{\pm m}}{[m]}\in
U^{\text{res}}_\up(\afgl)$ for all $m\geq1$. Now the integral action
of $\sfz^\pm_m$ on the tensor space \eqref{action central elts}
shows that $\frac{\sfz_m^\pm}{m}\not\in\ti\fD_\vtg(n)$. Hence, it
seems to us that the restricted integral form
$U^{\text{res}}_\up(\afgl)$ cannot be defined over
$\sZ=\mbz[\up,\up^{-1}]$. If it were defined over $\sZ$, it would
not be a subalgebra of $\ti\fD_\vtg(n)$.
\end{Rems}

The surjective homomorphism
$\xi_{r,\mbc}^\vee:\afHr_{\mbc}\ra\End_{\afSr_{\mbc}}(\Og_{\mbc}^{\ot
r})$ for $\up=\sqrt{q}$ was established by a geometric method. We do
not know if the surjectivity holds over $\mbq(\up)$. Since both
$\xi_r:\dHallr\to\afbfSr$ and $\xi_{r,\mbc}:\DC(n)\to\afSr_\mbc$ are
surjective, the following conjecture gives the affine Schur--Weyl reciprocity over $\mbq(\up)$ and $\mbc$ for a non-root of unity
specialization.

 \begin{Conj} The algebra homomorphisms
 $$\xi_r^\vee:\afbfHr\lra\End_{\afbfSr}(\bdOg^{\ot
r}),\qquad\xi_{r,\mbc}^\vee:\afHr_\mbc\lra\End_{\afSr_\mbc}(\Og_\mbc^{\ot
r})$$ are surjective, where base change to $\mbc$ is obtained by specializing $\up$ to a non-root of unity $z\in\mbc$.
\end{Conj}

With the truth of Conjecture \ref{Lform1}, specializing $\up$ to any element
 in a field $\field$ (of any characteristic) results
in a surjective homomorphism $\ti\fD_\vtg(n)\ot \field \to
\afSr\ot\field$. It is natural to further expect that the  affine
Schur--Weyl reciprocity holds at roots of unity.

\begin{Conj} The affine quantum Schur--Weyl reciprocity over any field $\field$ holds.
\end{Conj}


\section[Polynomial identities arising from commutator formulas]
{Polynomial identities arising from commutator formulas for semisimple generators}

In this last section we will give an application of our theory. We use the commutator formulas in Theorem
\ref{alt-presentation-dHall}(5) to derive certain polynomial
identities which seems to be interesting of its own.

For $\la\in\afmbnn$, set as in \eqref{semisimple module}
$$S_{\la}=\bop\limits_{i=1}^n\la_iS_i\;\;\text{and}\;\;
A_\la=\sum_{i=1}^n\la_iE^\vtg_{i,i+1}\in\afThnp.$$
 Then, $\fkd(A_\la)=\dim S_\la=\sg(\la)$ and $M(A_\la)=S_\la$. Furthermore,
for $\la,\al,\bt\in\afmbnn$, set
$\vi_{\al,\bt}^{\la}$ to be the Hall polynomial $\vi_{A_\al,A_\bt}^{A_\la}$ and
$$\dleb{\la\atop\al}\drib=\dleb{\la_1\atop\al_1}\drib\dleb{\la_2\atop\al_2}\drib
\cdots\dleb{\la_n\atop\al_n}\drib\quad(\text{cf.~}\eqref{multiGauss}).$$
Recall from \eqref{order on afmbzn} that, for
$\la=(\la_i),\mu=(\mu_i)\in\afmbnn$, $\la\leq\mu$ means
$\la_i\leq\mu_i$ for all $i$. Recall also from \eqref{tilde u_A} and
\eqref{poly-auto-group} the number $d_A'$ and the polynomial
$\fka_{A}$. For semisimple modules, we have the following easy
formulas.

\begin{Lem}\label{vi(al,ga)^la}
Let $\la,\al\in\afmbnn$ satisfy $\al\leq\la$. Then
$$\vi_{\al,\la-\al}^\la=\dleb{\la\atop\al}\drib,\;d_{A_\la}'=\sum_{1\leq
i\leq n}(\la_i^2-\la_i)\;\;
\text{and}\;\;\fka_\la:=\fka_{A_\la}=\prod\limits_{1\leq i\leq
n\atop 0\leq s\leq\la_i-1}(\up^{2\la_i}-\up^{2s}).$$
In particular, $\vi_{\la-\al,\al}^\la=\vi_{\al,\la-\al}^\la=\up^{2\sum_{i=1}^n\al_i(\al_i-\la_i)}\frac{\fka_\la}{\fka_\al\fka_{\la-\al}}.$

\end{Lem}

We will also use the abbreviation for the elements
\begin{equation*}\begin{split}
\vi_{A,B}^{A_1,B_1}&=\frac{\fka_{A_1}\fka_{B_1}}{\fka_A\fka_B}\sum_{A_2\in\afThnp}
v^{2\fkd(A_2)}\fka_{A_2}\vi_{A_1,A_2}^{A}\vi_{B_1,A_2}^B,\\
\ti{\vi_{A,B}^{A_1,B_1}}&=\frac{\fka_{A_1}\fka_{B_1}}{\fka_A\fka_B}
\sum_{A_2\in\afThnp}v^{2\fkd(A_2)}\fka_{A_2}\vi_{A_2,A_1}^{A}\vi_{A_2,B_1}^B.\\
\end{split}
\end{equation*}
defined in \eqref{phA1B1AB} by setting,
 for $\la,\mu,\al,\bt\in\afmbnn$,
$$\vi_{\la,\mu}^{\al,\bt}=\vi_{A_\la, A_\mu}^{A_\al,
A_\bt}\;\;\text{and}\;\;
\ti{\vi_{\la,\mu}^{\al,\bt}}=\ti{\vi_{A_\la ,A_\mu}^{A_\al,
A_\bt}}.$$

\begin{Prop}\label{vi(la,mu)^A,B}
For $\la,\mu\in\afmbnn$ and $A,B\in\afThnp$, if $\vi_{A_\la,A_\mu}^{A,B}$ or $\ti{\vi_{A_\la,A_\mu}^{A,B}}$ is nonzero, then there exist
$\al,\bt\in\afmbnn$ such that $A=A_\al$, $B=A_\bt$,
$\la-\al=\mu-\bt\geq \bf0$, and
$$\vi_{\la,\mu}^{\al,\bt}=\ti{\vi_{\la,\mu}^{\al,\bt}}=\frac{1}
{\fka_{\la-\al}}\up^{2\sg(\la-\al)+\sum_{1\leq i\leq
n}2(\al_i(\al_i-\la_i)+\bt_i(\bt_i-\mu_i))}.$$
\end{Prop}

\begin{proof} If either
$\vi_{A_\la,A_\mu}^{A,B}\neq0$ or $\ti{\vi_{A_\la,A_\mu}^{A,B}}\neq0$, then
$\vi_{A,C}^{A_\la}\vi_{B,C}^{A_\mu}\neq0$ or $\vi_{C,A}^{A_\la}\vi_{C,B}^{A_\mu}\neq0$
for some $C$. In either case, both $M(A)$ and $M(B)$, as submodules or quotient modules of a semisimple
module, are semisimple, and
$\la-\bfd(A)=\bfd(C)=\mu-\bfd(B)$.  Write $A=A_\al$, $B=A_\bt$ for some $\al,\bt\in\afmbnn$. Then
$\la-\al=\mu-\bt$ and so $\fka_{\la-\al}=\fka_{\mu-\bt}$. By the last assertion of Lemma \ref{vi(al,ga)^la}, one sees immediately that
\begin{equation*}
\begin{split}
\vi_{A_\la,A_\mu}^{A,B}=\ti{\vi_{A_\la,A_\mu}^{A,B}}&=\frac{\fka_\al \fka_\bt}{\fka_\la\fka_\mu}
\dleb{\la\atop\al}\drib\dleb{\mu\atop\bt}\drib\fka_{\la-\al}\up^{2\sg(\la-\al)}\\
&=\frac{1} {\fka_{\la-\al}}\up^{2\sg(\la-\al)+\sum_{1\leq i\leq
n}2(\al_i(\al_i-\la_i)+\bt_i(\bt_i-\mu_i))},
\end{split}
\end{equation*}
as required.\end{proof}

Recall the surjective homomorphism $\xi_r:\dHallr\ra \afbfSr$ as
explicitly described in Theorem \ref{xirl}. For $A,B\in\afThnp$, let
$X_{A,B}:=\xi_r(L_{A,B})$ and $Y_{A,B}:=\xi_r(R_{A,B})$, where
$L_{A,B}, R_{A,B}$ are defined in Theorem
\ref{alt-presentation-dHall}(5). Then $X_{A,B}=Y_{A,B}$. In fact,
since these commutator formulas continue to hold in $\DC(n)$ where
$\up$ is specialized to a non-root of unity in $\mbc$ (see Remark
\ref{NonrootOfUnity2}), $X_{A,B}=Y_{A,B}$ in $\afSr_\mbc$. In
particular,  for each prime power $q$, by specializing $\up$ to
$\sqrt{q}$, we will view both $X_{A,B}$ and $Y_{A,B}$ as elements in
the convolution algebra $\mbc_G(\scrY\times\scrY)\cong \afSr_\mbc$.
In this case, denote $X_{A,B}$ and $Y_{A,B}$ by $X_{A,B}^q$ and
$Y_{A,B}^q$, respectively. Thus, we have $X_{A,B}^q=Y_{A,B}^q$,
where
\begin{equation}\label{XABYAB}
\aligned
X_{A,B}^q=\sqrt{q}^{\lan
\bfd(B),\bfd(B)\ran}\sum_{A_1,B_1}\vi_{A,B}^{A_1,B_1}(q)
 &\sqrt{q}^{\lan
 \bfd(B_1),\bfd(A)+\bfd(B)-\bfd(B_1)\ran-d_{A_1}'-d_{B_1}'}\\
 &\;\times \ti\bffkk_{\bfd(B)-\bfd(B_1)}*({^tB_1})(\bfl,r)*A_1(\bfl,r),\\
Y_{A,B}^q =\sqrt{q}^{\lan\bfd(B),\bfd(A)\ran}\sum_{A_1,B_1}
 \ti{\vi_{A,B}^{A_1,B_1}}(q)&\sqrt{q}^{\lan \bfd(B)-\bfd(B_1),\bfd(A_1)\ran+\lan
 \bfd(B),\bfd(B_1)\ran-d_{A_1}'-d_{B_1}'}\\
&\; \times\ti \bffkk_{\bfd(B_1)-\bfd(B)}*A_1(\bfl,r)*({^t
B_1})(\bfl,r),
\endaligned
\end{equation}
with $\bffkk_i=\xi_r(K_i)=0(\bfe_i,r)$ for any $i\in I$ and
 $\ti\bffkk_{\bfa}=\prod_{i=1}^n\bffkk_i^{a_i}\bffkk_{i+1}^{-a_i}$
for any $\bfa=(a_i)\in \afmbnn$.

In the following we are going to use the equality
$X_{A,B}^q=Y_{A,B}^q$ to derive some interesting polynomial
identities. In the rest of this section, we fix the finite field
$\field$ with $q$ elements. As in \S3.1, let $\scrY=\afFn(q)$ be the set of
all cyclic flags $\bfL=(L_i)_{i\in\mbz}$ of lattices in a fixed
$\field[\ep,\ep^{-1}]$-free module $V$ of rank $r\geq 1$.

\begin{Lem}\label{tiafK}
For $\bfa\in\afmbnn$ and $(\bfL,\bfL')\in\scrY\times\scrY$,
\begin{equation*}
\ti\bffkk_{\bfa}(\bfL,\bfL')=\begin{cases}q^{\frac{1}{2}\lan\bfa,\la\ran},
&\quad\text{if $\bfL=\bfL'$ and $(\bfL,\bfL)\in\sO_{\diag(\la)}$ for $\la\in\afLanr$;}\\
0,&\quad\text{otherwise.}
\end{cases}
\end{equation*}
\end{Lem}
\begin{proof}
Let $\la\in\afLanr$. By Lemma \ref{triangular matrix}(1) and (2),
$(\bfL,\bfL')\in\sO_{\diag(\la)}$ if and only if $\bfL=\bfL'$ and
$\la_i=\dim(L_i/L_{i-1})$ for all $i$. Thus, if
$\ti\bffkk_{\bfa}(\bfL,\bfL')\not=0$, then $\bfL=\bfL'$. Now we
assume $\bfL=\bfL'$ and $(\bfL,\bfL)\in\sO_{\diag(\la)}$. Since
$\ti\bffkk_{\bfa}=\sum_{\mu\in\La_\vtg(n,r)}q^{\frac{1}{2}\sum_{1\leq
i\leq n}(a_i-a_{i-1})\mu_i}[\diag(\mu)]$ and
$[\diag(\mu)]=\chi_{\sO_{\diag(\mu)}}$, it follows that
$\ti\bffkk_{\bfa}(\bfL,\bfL)=q^{\frac{1}{2}\sum_{1\leq i\leq
n}(a_i-a_{i-1})\la_i}=q^{\frac{1}{2}\lan\bfa,\la\ran}.$
\end{proof}

For integers $N,t$ with $t\geq 0$, let $\dleb{N\atop
t}\drib_q=\prod\limits_{1\leq i\leq
t}\frac{q^{N-i+1}-1}{q^{i}-1}$. For $\al,\bt\in\mbn_\vtg^n$ and
$(\bfL,\bfL'')\in\scrY\times\scrY$, consider the subsets of
$\scrY$:
\begin{equation*}
\begin{split}
X(\al,\bt,\bfL,\bfL'')&:=
\{\bfL'\in\scrY\mid\bfL,\bfL''\han\bfL',\bfL'/\bfL
\cong S_{\bt},\bfL'/\bfL''\cong S_{\al}\},\\
Y(\al,\bt,\bfL,\bfL'')&:=\{\bfL'\in\scrY\mid
\bfL'\han\bfL,\bfL'',\bfL/\bfL'\cong S_{\al}, \bfL''/\bfL'\cong
S_{\bt}\}.
\end{split}
\end{equation*}
Here, $\bfL,\bfL''\han\bfL'$ and $\bfL'\han\bfL,\bfL''$ are the short form
of $\bfL\han\bfL',\bfL''\han\bfL'$ and $\bfL'\han\bfL,\bfL'\han\bfL''$, respectively.
\begin{Lem}\label{X(al,bt,bfL,bfL'')}
For $\al,\bt\in\scrY$ and $(\bfL,\bfL'')\in\scrY\times\scrY$, if
$X(\al,\bt,\bfL,\bfL'')\neq\emptyset$ or $Y(\al,\bt,\bfL,\bfL'')\neq\emptyset$,
then
$L_i+L_i''\han L_{i+1}\cap L_{i+1}''$ and
$\bt_i-\dim((L_i+L_i'')/L_i)=\al_i-\dim((L_i+L_i'')/L_i'')\geq 0$
for all $i\in\mbz$. Moreover, in this case,
\begin{equation*}
\begin{split}
|X(\al,\bt,\bfL,\bfL'')|&=
\prod_{1\leq i\leq n}\dleb{\dim(L_{i+1}\cap L_{i+1}''/(L_i+L_i''))\atop \al_i-\dim((L_i+L_i'')/L_i'')}\drib_q\text{ and }\\
|Y(\al,\bt,\bfL,\bfL'')|&= \prod_{1\leq i\leq n}\dleb{\dim(L_{i}\cap
L_{i}''/(L_{i-1}+L_{i-1}''))\atop
\al_i-\dim((L_i+L_i'')/L_i'')}\drib_q.
\end{split}
\end{equation*}
\end{Lem}
\begin{proof} If, as representations of $\tri$, both $\bfL'/\bfL=(L_i'/L_i,f_i)$ and $\bfL'/\bfL''=(L_i'/L_i'',f_i'')$ are semisimple, then
the linear maps $f_i,f_i''$ induced from inclusion $L_i'\subseteq L_{i+1}'$ are zero maps. This forces $L_i'\subseteq L_{i+1}$ and
$L_i'\subseteq L_{i+1}''$ for all $i$. Hence,
\begin{equation*}
\begin{split}
X(\al,\bt,\bfL,\bfL'')\neq\emptyset&\iff\exists\,\bfL'\in X(\al,\bt,\bfL,\bfL'')\\
&\Longleftrightarrow
\begin{cases}
L_t+L_t''\han L_t'\han L_{t+1}\cap L_{t+1}'', & \text{for $t\in\mbz$;}\\
\dim(L_t'/L_t)=\bt_t,\dim(L_t'/L_t'')=\al_t, &\text{for
$t\in\mbz$}
\end{cases}\\
&\Longleftrightarrow
\begin{cases}
L_t+L_t''\han L_t'\han L_{t+1}\cap L_{t+1}'', & \text{for $t\in\mbz$;}\\
\dim(L_t'/(L_t+L_t''))=\bt_t-\dim((L_t+L_t'')/L_t),& \text{for $t\in\mbz$;}\\
\dim(L_t'/(L_t+L_t''))=\al_t-\dim((L_t+L_t'')/L_t''),&\text{for
$t\in\mbz$}
\end{cases}
\end{split}
\end{equation*}
and
\begin{equation*}
\begin{split}
Y(\al,\bt,\bfL,\bfL'')\neq\emptyset&\Longleftrightarrow
\exists\,\bfL'\in Y(\al,\bt,\bfL,\bfL'')
\\
&\Longleftrightarrow
\begin{cases}
L_t'\han L_t\han L_{t+1}',\ \dim(L_t/L_t')=\al_t, & \text{for $t\in\mbz$;}\\
L_t'\han L_t''\han L_{t+1}',\ \dim(L_t''/L_t')=\bt_t, &\text{for
$t\in\mbz$}
\end{cases}\\
&\Longleftrightarrow
\begin{cases}
L_{t-1}\han L_t'\han L_{t},\ \dim(L_t/L_t')=\al_t, & \text{for $t\in\mbz$;}\\
L_{t-1}''\han L_t'\han L_{t}'',\ \dim(L_t''/L_t')=\bt_t,
&\text{for $t\in\mbz$}
\end{cases}\\
&\Longleftrightarrow
\begin{cases}
L_{t-1}+L_{t-1}''\han L_t'\han L_{t}\cap L_{t}'', & \text{for $t\in\mbz$;}\\
\dim(L_t\cap L_t''/L_t')=\al_t-\dim((L_t+L_t'')/L_t''),& \text{for $t\in\mbz$;}\\
\dim(L_t\cap L_t''/L_t')=\bt_t-\dim((L_t+L_t'')/L_t),&\text{for
$t\in\mbz$}.
\end{cases}
\end{split}
\end{equation*}
The rest of the proof is clear.
\end{proof}

For any
$\la=(\la_i)_{i\in\mbz},\mu=(\mu_i)_{i\in\mbz}\in\mbn_\vtg^n$,
define polynomials in $\up^2$ over $\mbz$:
\begin{equation*}
\begin{split}
P_{\la,\mu}(\up^2)&=\sum_{0\leq\nu\leq\lambda\atop\nu\in\mathbb
N_\vartriangle^n} \up^{2\bigl(\sum_{1\leq i\leq
n}\frac{\nu_i^2-\nu_i}{2}+(\la_i-\nu_i)(\mu_i-\nu_{i-1})\bigr)}
\prod_{i=1}^n(\up^2-1)^{\nu_i}[\![\nu_i]\!]^!\ggp{\la_i}{\nu_i}\ggp{\mu_{i+1}}{\nu_i},\\
P'_{\la,\mu}(\up^2)&=\sum_{0\leq\nu\leq\lambda\atop\nu\in\mathbb
N_\vartriangle^n}\up^{2\bigl(\sum_{1\leq i\leq
n}\frac{\nu_i^2-\nu_i}{2}+(\la_i-\nu_i)(\mu_{i+1}-\nu_{i+1})\bigr)}
\prod_{i=1}^n(\up^2-1)^{\nu_i}[\![\nu_i]\!]^!\ggp{\la_i}{\nu_i}\ggp{\mu_i}{\nu_i}.
\end{split}
\end{equation*}

We now prove that these polynomials occur naturally in the
coefficients of $X_{A,B}, Y_{A,B}$ for $A,B\in\afThnp^{ss}$ when
they are written as a linear combination of $e_C$, $C\in\afThnr$.

\begin{Thm}\label{lem of Poly}
For $\la,\mu\in\afmbnn$, $x,y\in\La_\vtg(n,r)$ and
$(\bfL,\bfL'')\in\afFn_{,x}(q)\times\afFn_{,y}(q)$, let
$\tila:=\la-\gamma$ and $\timu:=\mu-\dt$, where
 $$\gamma=(\dim(L_i+L_i'')/L_i'')_{i\in\mbz},\;\dt=(\dim(L_{i}+L_i'')/L_i)_{i\in\mbz}.$$

If $X_{\la,\mu}^q(\bfL,\bfL'')\neq0$
or $Y_{\la,\mu}^q(\bfL,\bfL'')\neq0$, then $L_i+L_i''\han
L_{i+1}\cap L_{i+1}''$ for all $i\in\mbz$ and $\tila=\timu\geq 0$. Moreover,
putting $z_i=\dim(L_i\cap L_i''/(L_{i-1}+L_{i-1}'')))$ for all
$i\in\mbz$, we have in this case:
$$\aligned
X^q_{\la,\mu}(\bfL,\bfL''):=X^q_{A_\la,A_\mu}(\bfL,\bfL'')&=q^{\frac{1}{2}(f_{\tila}+\sum_{1\leq
i\leq n}(\la_i+\mu_i+\tila_i-\tila_i^2))}\prod\limits_{1\leq i\leq
n\atop 1\leq s\leq
\tila_i}\frac{1}{q^s-1}\cdot P_{\tila,\bsz}(q)\\
Y^q_{\la,\mu}(\bfL,\bfL''):=Y^q_{A_\la,A_\mu}(\bfL,\bfL'')&=q^{\frac{1}{2}(f_{\tila}+\sum_{1\leq
i\leq n}(\la_i+\mu_i+\tila_i-\tila_i^2))}\prod\limits_{1\leq i\leq
n\atop 1\leq s\leq
\tila_i}\frac{1}{q^s-1}\cdot P'_{\tila,\bsz}(q)\\
\endaligned
$$
where $f_{\tila}=\sum_{1\leq i\leq
n}\bigl[\la_i(\la_{i-1}-\la_i-y_i)
-x_{i+1}(\dt_i+\tila_i)\bigr]+\lan\mu,\mu\ran+\lan\dt+\tila,\la\ran.$
\end{Thm}

\begin{proof} We need to compute the value of \eqref{XABYAB} at $(\bfL,\bfL'')$ for $A=A_\la,B=A_\mu$ (hence, $A_1=A_\al$ and $B_1=A_\bt$).
By Corollary \ref{A(bfj,r) acts on (bfL,bfL')} and Lemma \ref{tiafK}, and noting $\dim\ttO_{A_\nu}=0$ for $\nu\in\afmbnn$,
$$
\aligned
(\ti\bffkk_{\mu-\bt}*{\tA_\bt}(\bfl,r)*A_\al(\bfl,r))(\bfL,\bfL'')&=\sum_{\bfL'\in\scrY}
(\ti\bffkk_{\mu-\bt}*{\tA_\bt}(\bfl,r))(\bfL,\bfL')A_\al(\bfl,r)(\bfL',\bfL'')\\
&=\sum_{\bfL'\in\scrY,\bfL,\bfL''\han\bfL'\atop
\bfL'/\bfL\cong S_\bt,\bfL'/\bfL''\cong S_{\al}}q^{\frac{1}{2}(\lan\mu-\bt,x\ran-a(\bfL,\bfL')-c(\bfL',\bfL''))},
\endaligned$$
where $a(\bfL,\bfL'),c(\bfL',\bfL'')$ are defined in \eqref{a and c}.
Thus, if $X_{\la,\mu}^q(\bfL,\bfL'')\neq0$
or $Y_{\la,\mu}^q(\bfL,\bfL'')\neq0$, then some $\vi_{A_\la,A_\mu}^{A,B}=\ti{\vi_{A_\la,A_\mu}^{A,B}}\neq0$.
So applying the first assertion of Proposition \ref{vi(la,mu)^A,B} yields
\begin{equation*}
\begin{split}
X^q_{\la,\mu}(\bfL,\bfL'')=
\sum_{\al,\bt\in\afmbnn\atop\la-\al=\mu-\bt\geq
0}\vi_{\la,\mu}^{\al,\bt}(q)
&q^{\frac{1}{2}(\lan\mu,\mu\ran+\lan\bt,\mu-\bt\ran+
\lan\bt,\la\ran-d_{A_\bt}'-d_{A_\al}')}\\
&\qquad \times\sum_{\bfL'\in\scrY,\bfL,\bfL''\han\bfL'\atop
\bfL'/\bfL\cong S_\bt,\bfL'/\bfL''\cong S_{\al}}q^{\frac{1}{2}(\lan\mu-\bt,x\ran-a(\bfL,\bfL')-c(\bfL',\bfL''))},\\
\end{split}\end{equation*}
Similarly,
\begin{equation*}
\begin{split}
Y^q_{\la,\mu}(\bfL,\bfL'')=\sum_{\al,\bt\in\afmbnn\atop\la-\al=\mu-\bt\geq
0}\vi_{\la,\mu}^{\al,\bt}(q)&q^{\frac{1}{2}(\lan\mu,\la\ran+\lan\mu-\bt,\al\ran+
\lan\mu,\bt\ran-d_{A_\bt}'-d_{A_\al}')}\\
&\qquad \times\sum_{\bfL'\in\scrY,\bfL'\han\bfL,\bfL''\atop
\bfL/\bfL'\cong S_\al,\bfL''/\bfL'\cong S_{\bt}}q^{\frac{1}{2}(\lan\bt-\mu,x\ran-a(\bfL',\bfL'')-c(\bfL,\bfL'))}.\\
\end{split}
\end{equation*}
If $\bfL'\in\scrY$ satisfies $\bfL,\bfL''\han\bfL'$,
$\bfL'/\bfL\cong S_\bt$ and $\bfL'/\bfL''\cong S_\al$, then
$$a(\bfL,\bfL')+c(\bfL',\bfL'')=\sum_{1\leq i\leq
n}\bt_i(x_{i+1}-\bt_i)+\sum_{1\leq i\leq
n}\al_{i+1}(y_{i+1}-\al_i).$$ Likewise, if $\bfL'\in\scrY$ satisfies
$\bfL'\subseteq\bfL,\bfL''$, $\bfL/\bfL'\cong S_\al$ and
$\bfL''/\bfL'\cong S_\bt$, then $L_i\subseteq L_{i+1}'$ and
$L_i''\subseteq L_{i+1}'$ for all $i$. Thus,
$(L_{i+1}'/L_i')/(L_i/L_i')\cong L_{i+1}'/L_i$ and
$(L_{i+1}/L_i)/(L_{i+1}'/L_i)\cong L_{i+1}/L_{i+1}'$, and so
$\dim(L_{i+1}'/L_i')=x_{i+1}-\al_{i+1}+\al_i$. Similarly,
$\dim(L_{i+1}'/L_i')=y_{i+1}-\bt_{i+1}+\bt_i$. Hence,
\begin{equation*}
\begin{split}
&a(\bfL',\bfL'')+c(\bfL,\bfL')\\
=&\sum_{1\leq i\leq n}\bt_i(\dim(L_{i+1}'/L_i')-\bt_i)+\sum_{1\leq i\leq n}\al_{i+1}(\dim(L_{i+1}'/L_i')-\al_i)\\
=&\sum_{1\leq i\leq n}\bt_i(x_{i+1}-\al_{i+1}+\al_i-\bt_i)+\sum_{1\leq i\leq n}\al_{i+1}(y_{i+1}-\bt_{i+1}+\bt_i-\al_i)\\
=&\sum_{1\leq i\leq n}\bt_i(x_{i+1}-\bt_i)+\sum_{1\leq i\leq
n}\al_{i+1}(y_{i+1}-\al_i).
\end{split}
\end{equation*}
and consequently,
\begin{equation*}
\begin{split}
X^q_{\la,\mu}(\bfL,\bfL'')=\sum_{\al,\bt\in\afmbnn\atop\la-\al=\mu-\bt\geq
0}\vi_{\la,\mu}^{\al,\bt}(q)&q^{\frac{1}{2}
(-d_{A_\al}'-d_{A_\bt}'-\sum_{1\leq i\leq n}\bt_i(x_{i+1}-\bt_i)-\sum_{1\leq i\leq n}\al_{i+1}(y_{i+1}-\al_i))}\\
&\qquad\times q^{\frac{1}{2}(\lan\mu,\mu\ran+
\lan\bt,\mu-\bt\ran+\lan\bt,\la\ran+\lan\mu-\bt,x\ran)}|X(\al,\bt,\bfL,\bfL'')|,\\
Y^q_{\la,\mu}(\bfL,\bfL'')=\sum_{\al,\bt\in\afmbnn\atop\la-\al=\mu-\bt\geq
0}\vi_{\la,\mu}^{\al,\bt}(q)&q^{\frac{1}{2}
(-d_{A_\al}'-d_{A_\bt}'-\sum_{1\leq i\leq n}\bt_i(x_{i+1}-\bt_i)-\sum_{1\leq i\leq n}\al_{i+1}(y_{i+1}-\al_i))}\\
&\qquad\times q^{\frac{1}{2}(\lan\mu,\la\ran+
\lan\mu-\bt,\al\ran+\lan\mu,\bt\ran+\lan\bt-\mu,x\ran)}|Y(\al,\bt,\bfL,\bfL'')|.\\
\end{split}
\end{equation*}
Thus, $X^q_{\la,\mu}(\bfL,\bfL'')\not=0$ or
$Y^q_{\la,\mu}(\bfL,\bfL'')\not=0$ implies some $X(\al,\bt,\bfL,\bfL'')\neq\emptyset$ or
$Y(\al,\bt,\bfL,\bfL'')\neq\emptyset$, which implies, by Lemma
\ref{X(al,bt,bfL,bfL'')}, $L_i+L_i''\han
L_{i+1}\cap L_{i+1}''$ for all $i$ and $\bt-\dt=\al-\gamma\geq 0$. The latter together with
$\la-\al=\mu-\bt\geq 0$ implies $\tila-\timu=(\la-\mu)-(\gamma-\dt)=(\al-\bt)-(\gamma-\dt)=0$ and
$\tila=\la-\gamma\geq\al-\gamma\geq 0$. So we have proved the first assertion.

It remains to simplify $X^q_{\la,\mu}(\bfL,\bfL'')$ and
$Y^q_{\la,\mu}(\bfL,\bfL'')$ under the assumption that
$L_i+L_i''\han L_{i+1}\cap L_{i+1}''$ for all $i$ and
$\tila=\timu\geq 0$. First, putting $z_i=\dim(L_i\cap
L_i''/(L_{i-1}+L_{i-1}'')))$, Lemma \ref{X(al,bt,bfL,bfL'')} gives
$|X(\al,\bt,\bfL,\bfL'')|=\prod_{1\leq i\leq
n}\dleb{z_{i+1}\atop\al_i-\gamma_i}\drib_q$. Second, for
$\al,\bt\in\afmbnn$ satisfying $\la-\al=\mu-\bt\geq 0$, Lemma
\ref{vi(al,ga)^la} and Proposition \ref{vi(la,mu)^A,B} imply
\begin{equation*}
\begin{split}
&\quad\,\vi_{\la,\mu}^{\al,\bt}(q)q^{\frac{1}{2}
(-d_{A_\al}'-d_{A_\bt}'-\sum_{1\leq i\leq n}\bt_i(x_{i+1}-\bt_i)-\sum_{1\leq i\leq n}\al_{i+1}(y_{i+1}-\al_i))}\\
&=\frac{q^{\sum_{1\leq i\leq n}(\la_i-\al_i+\al_i(\al_i-\la_i)+\bt_i(\bt_i-\mu_i))}}{|\Aut(S_{\la-\al})|}
q^{\frac{1}{2}\sum_{1\leq i\leq n}[(\bt_i-\bt_i^2+\al_i-\al_i^2)-\bt_i(x_{i+1}-\bt_i)-\al_{i}(y_{i}-\al_{i-1})]}\\
&=q^{\frac{1}{2}\sum_{1\leq i\leq
n}(\la_i+\mu_i)}\frac{q^{\frac{1}{2}\sum_{1\leq i\leq
n}(\al_i(\al_i+\al_{i-1}-2\la_i-y_i)+\bt_i(2\bt_i-2\mu_i-x_{i+1}))}}{|\Aut(S_{\la-\al})|},
\end{split}
\end{equation*}since $2\la_i-\al_i+\bt_i=\la_i+\mu_i$ for all $i$.
Hence,
\begin{equation*}
\begin{split}
X^q_{\la,\mu}(\bfL,\bfL'')&=q^{\frac{1}{2}\sum_{1\leq i\leq n}(\la_i+\mu_i)}\sum_{\al,\bt\in\afmbnn\atop\la-\al=\mu-\bt\geq 0}\frac{q^{\frac{1}{2}\sum_{1\leq i\leq n}[\al_i(\al_i+\al_{i-1}-2\la_i-y_i)+\bt_i(2\bt_i-2\mu_i-x_{i+1})]}}{|\Aut(S_{\la-\al})|}\\
&\qquad\times q^{\frac{1}{2}(\lan\mu,\mu\ran+
\lan\bt,\mu-\bt\ran+\lan\bt,\la\ran+\lan\mu-\bt,x\ran)}\prod_{1\leq i\leq n}\dleb{z_{i+1}\atop\al_i-\gamma_i}\drib_q.\\
\end{split}
\end{equation*}
Here we implicitly assumed $\al\geq\ga$ (or equivalently, $\bt\geq\dt$).  Setting $\nu=\al-\gamma=\bt-\dt$
gives
$$X^q_{\la,\mu}(\bfL,\bfL'')=q^{\frac{1}{2}\sum_{1\leq i\leq
n}(\la_i+\mu_i)}\sum_{\nu\in\afmbnn\atop
0\leq\nu\leq\la-\gamma}\frac{q^{\frac{1}{2}f_\nu}}{|\Aut(S_{\la-\gamma-\nu})|}\prod_{1\leq
i\leq n}\dleb{z_{i+1}\atop\nu_i}\drib_q,\,\,\text{ where}$$
\begin{equation}\label{fnu}
 \begin{split}
f_{\nu}&=\sum_{1\leq i\leq
n}\bigl((\gamma_i+\nu_i)(\gamma_i+\nu_i+\gamma_{i-1}+\nu_{i-1}-2\la_i-y_i)+
(\dt_i+\nu_i)(2(\gamma_i+\nu_i-\la_i)-x_{i+1})\bigr)\\
&\quad+\lan\mu,\mu\ran+\lan\dt+\nu,\la-(\gamma+\nu)\ran+\lan\dt+\nu,\la\ran+\lan\la-(\gamma+\nu),x\ran
\end{split}
\end{equation}
for $\nu\in\afmbnn$ with $0\leq\nu\leq\tila$. Similarly,
\begin{equation*}
\begin{split}
Y^q_{\la,\mu}(\bfL,\bfL'')
&=q^{\frac{1}{2}\sum_{1\leq i\leq n}(\la_i+\mu_i)}\sum_{\al,\bt\in\afmbnn\atop\la-\al=\mu-\bt\geq 0}\frac{q^{\frac{1}{2}\sum_{1\leq i\leq n}(\al_i(\al_i+\al_{i-1}-2\la_i-y_i)+\bt_i(2\bt_i-2\mu_i-x_{i+1}))}}{|\Aut(S_{\la-\al})|}\\
&\qquad\times q^{\frac{1}{2}(\lan\mu,\la\ran+
\lan\mu-\bt,\al\ran+\lan\mu,\bt\ran+\lan\bt-\mu,x\ran)}\prod_{1\leq i\leq n}\dleb{z_{i}\atop\al_i-\gamma_i}\drib_q\\
&=q^{\frac{1}{2}\sum_{1\leq i\leq
n}(\la_i+\mu_i)}\sum_{\nu\in\afmbnn\atop
0\leq\nu\leq\la-\gamma}\frac{q^{\frac{1}{2}g_\nu}}{|\Aut(S_{\la-\gamma-\nu})|}
\prod_{1\leq i\leq n}\dleb{z_{i}\atop\nu_i}\drib_q,\,\,\text{ where }
\end{split}
\end{equation*}
\begin{equation}\label{gnu}
\begin{split}
g_{\nu}&=\sum_{1\leq i\leq n}
\bigl((\gamma_i+\nu_i)(\gamma_i+\nu_i+\gamma_{i-1}+\nu_{i-1}-2
\la_i-y_i)+(\dt_i+\nu_i)(2(\gamma_i+\nu_i-\la_i)-x_{i+1})\bigr)\\
&\qquad+\lan\mu,\la\ran+\lan\la-\gamma-\nu,
\gamma+\nu\ran+\lan\mu,\dt+\nu\ran+\lan\gamma+\nu-\la, x\ran.
\end{split}
\end{equation}
Simplifying $f_{\tila}-f_{\nu}$ and $f_{\tila}-g_{\nu}$ and substituting, we obtain by Lemma \ref{appendix},
\begin{equation*}
\begin{split}
\frac{X^q_{\la,\mu}(\bfL,\bfL'')}{q^{\frac{1}{2}\sum_{1\leq i\leq
n}(\la_i+\mu_i)+\frac{1}{2}f_{\tila}}}&=
\sum_{0\leq\nu\leq\tila\atop\nu\in\afmbnn}\frac{q^{\sum_{1\leq i\leq n}
(\tila_iz_i-\nu_i(z_i+\tila_i+\tila_{i+1}-\nu_i-\nu_{i+1}))}}
{|\Aut(S_{\tila-\nu})|}\prod_{1\leq i\leq n}\dleb{z_{i+1}\atop\nu_i}\drib_q,\\
&=\sum_{0\leq\nu\leq\tila\atop\nu\in\afmbnn}h_{\nu}q^{\sum_{1\leq i\leq n}
(\tila_iz_i-\nu_iz_i-\nu_i\tila_{i+1})}\prod_{1\leq i\leq n\atop 1\leq s\leq\nu_i}(q^{z_{i+1}-s+1}-1),\\
\end{split}
\end{equation*}
\begin{equation*}
\begin{split}
\frac{Y^q_{\la,\mu}(\bfL,,\bfL'')}{q^{\frac{1}{2}\sum_{1\leq i\leq
n}(\la_i+\mu_i)
+\frac{1}{2}f_{\tila}}}&=\sum_{0\leq\nu\leq\tila\atop\nu\in\afmbnn}
\frac{q^{\sum_{1\leq i\leq
n}(\tila_iz_{i+1}-\nu_i(z_{i+1}+\tila_i+\tila_{i-1}-\nu_i-\nu_{i+1}))}}
{|\Aut(S_{\tila-\nu})|}\prod_{1\leq i\leq n}\dleb{z_{i}\atop\nu_i}\drib_q\\
&=\sum_{0\leq\nu\leq\tila\atop\nu\in\afmbnn}h_{\nu}q^{\sum_{1\leq
i\leq
n}(\tila_iz_{i+1}-\nu_iz_{i+1}-\nu_i\tila_{i-1})}\prod_{1\leq
i\leq n\atop 1\leq s\leq\nu_i}(q^{z_i-s+1}-1),
\end{split}
\end{equation*}
where
$$h_{\nu}=\frac{q^{-\sum_{1\leq i\leq n}\nu_i(\tila_i-\nu_i-\nu_{i+1})}}{|\Aut(S_{\tila-\nu})|}
\prod_{1\leq i\leq n\atop 1\leq s\leq\nu_i}\frac{1}{q^s-1}.$$
Further simplification by Lemma \ref{vi(al,ga)^la}, we obtain
\begin{equation*}
\begin{split}
h_\nu&=\frac{q^{-\sum_{1\leq i\leq
n}\nu_i(\tila_i-\nu_i-\nu_{i+1})}}{q^{\sum_{1\leq i\leq
n}\frac{1}{2}(\tila_i-\nu_i-1)(\tila_i-\nu_i)}}
\prod_{1\leq i\leq n\atop 1\leq s\leq\tila_i-\nu_i}\frac{1}{q^s-1}\prod_{1\leq i\leq n\atop 1\leq s\leq\nu_i}\frac{1}{q^s-1}\\
&=q^{\sum_{1\leq i\leq n}\frac{1}{2}(\tila_i-\tila_i^2)}\prod_{1\leq
i\leq n\atop 1\leq s\leq\tila_i}\frac{1}{q^s-1}q^{\sum_{1\leq i\leq
n}(\frac{\nu_i^2-\nu_i}{2}+\nu_i\nu_{i+1})}\prod_{1\leq i\leq
n}\dleb{\tila_i\atop\nu_i}\drib_q.
\end{split}
\end{equation*}
Thus,
\begin{equation*}
\begin{split}
&X^q_{\la,\mu}(\bfL,\bfL'')=q^{\frac{1}{2}(f_{\tila}+\sum_{1\leq
i\leq n}(\la_i+\mu_i+\tila_i-\tila_i^2))}\prod\limits_{1\leq i\leq
n\atop 1\leq s\leq
\tila_i}\frac{1}{q^s-1}\cdot Q_{\tila,\bsz}\\
&Y^q_{\la,\mu}(\bfL,\bfL'')=q^{\frac{1}{2}(f_{\tila}+\sum_{1\leq
i\leq n}(\la_i+\mu_i+\tila_i-\tila_i^2))}\prod\limits_{1\leq i\leq
n\atop 1\leq s\leq
\tila_i}\frac{1}{q^s-1}\cdot Q'_{\tila,\bsz}\\
\end{split}
\end{equation*}
where, for $p_{\nu,\la}(q)=q^{\sum_{1\leq i\leq
n}(\frac{\nu_i^2-\nu_i}{2}+\nu_i\nu_{i+1})}
\prod_{1\leq i\leq
n}\left[\!\!\left[{\lambda_i\atop\nu_i}\right]\!\!\right]_q,$
\begin{equation*}
\begin{split}
Q_{\tila,\bsz}&=\sum_{0\leq\nu\leq\tila\atop\nu\in\mathbb
N_\vartriangle^n}p_{\nu,\tila}(q)\cdot q^{\sum_{1\leq i\leq
n}(\tila_iz_i-\nu_iz_i-\nu_i\tila_{i+1})}\prod_{1\leq i\leq
n\atop 1\leq s\leq \nu_i}(q^{z_{i+1}-s+1}-1)\\
Q_{\tila,\bsz}'&=\sum_{0\leq\nu\leq\tila\atop\nu\in\mathbb
N_\vartriangle^n}p_{\nu,\tila}(q)\cdot q^{\sum_{1\leq i\leq
n}(\tila_iz_{i+1}-\nu_iz_{i+1}-\nu_{i+1}\tila_{i})}\prod_{1\leq
i\leq n\atop 1\leq s\leq \nu_i}(q^{z_{i}-s+1}-1),
\end{split}
\end{equation*}
Finally, it is clear to see that $P_{\tila,\bsz}=Q_{\tila,\bsz}$ and $P'_{\tila,\bsz}=Q'_{\tila,\bsz}$.
\end{proof}

The fact that $\xi_r$ is an algebra homomorphism gives immediately
the following polynomial identity.\index{polynomial identity}

\begin{Coro} \label{polyidentity} For any
$\la=(\la_i)_{i\in\mbz},\mu=(\mu_i)_{i\in\mbz}\in\mbn_\vtg^n$, $P_{\la,\mu}(\up^2)=P'_{\la,\mu}(\up^2)$.
\end{Coro}

\begin{proof} If $\mu=\bf0$, then the equality holds trivially. Now
suppose $\mu\not=0$ and set $r=\sum_{i=1}^n \mu_i$. Let $\field$ be a
finite field with $q$ elements. By Theorems \ref{xirl} and \ref{alt-presentation-dHall}(5),
we have $X_{\la,\mu'}^q(\bfL,\bfL'')=Y_{\la,\mu'}^q(\bfL,\bfL)$ for all $\la,\mu'$ and $(\bfL,\bfL'')$.
Take $\bfL=(L_i)_{i\in\mbz}\in\scrY$ such that $\dim
L_i/L_{i-1}=\mu_i$ for all $i\in \mbz$, and $\bfL''=\bfL$ (thus, $\tila=\la=\mu=\mu'$). Applying Theorem \ref{lem of Poly} to the equality
$X_{\la,\la}^q(\bfL,\bfL)=Y_{\la,\la}^q(\bfL,\bfL)$ gives the equality $P_{\la,\mu}(q)=P'_{\la,\mu}(q)$.
\end{proof}

\begin{Rems} We point out that the polynomial identity $P_{\la,\mu}(\up^2)=P'_{\la,\mu}(\up^2)$
for all $\la,\mu$ is {\sf equivalent} to the fact that the algebra
homomorphisms $\zeta_r^\pm:\dHallr^\pm\to\afbfSr$ constructed by
Varagnolo and Vasserot \cite{VV99} (see \eqref{zetarpm}), which have
easy extensions $\zeta_r^{\geq0}:\dHallr^{\geq0}\to\afbfSr$ and
$\zeta_r^{\leq0}:\dHallr^{\leq0}\to\afbfSr$,
 can be extended to
an algebra homomorphism $\dHallr\to\afbfSr$.

In fact, there is an obvious linear extension $\xi_r'$ which is an
algebra homomorphism if and only if $\xi_r'$ preserves the
commutator relations in Theorem \ref{alt-presentation-dHall}(5) on
semisimple generators (see Lemma \ref{reduce lemma for presentation
of double Hall algebra}). This is the case by Theorem \ref{lem of
Poly} if $P_{\la,\mu}(q)=P'_{\la,\mu}(q)$ for every prime power $q$
and $\la,\mu$.
\end{Rems}

\begin{Prob} Give a direct (or combinatorial) proof for the polynomial identity $P_{\la,\mu}(\up^2)=P'_{\la,\mu}(\up^2)$.
\end{Prob}

\section{Appendix}
In this Appendix, we prove a few lemmas which have been used in the
previous sections. The first one reflects some affine phenomenon for
the length of  the longest element $w_{0,\la}$ of $\fS_\la$. This is
used in the proof of Corollary \ref{di=length}.

\begin{Lem}\label{w0la} For $\la\in\afLanr$ and $d\in\mathscr D_\la^\vtg$, let
$$\aligned
Y&=\{(s,t)\in\mbz^2\mid 1\leq d^{-1}(s)\leq r, s<t, s,t\in R_k^\la\text{ for some }k\in\mbz\}\\
Z&=\{(s,t)\in\mbz^2\mid 1\leq s\leq r, s<t, s,t\in R_k^\la\text{ for some }1\leq k\leq n\}.\\
\endaligned$$
Then $|Y|=|Z|=\ell(w_{0,\la})$.
\end{Lem}
\begin{proof}For $a\in\mbz$ let $Y_a=\{b\in\mbz\mid a<b,\;a,b\in R_k^{\la}\text{ for some }k\in\mbz\}.$
We first claim $|Y_{a_1}|=|Y_{a_2}|$ whenever $a_1=a_2\mod  r$.
Indeed, write $a_2=a_1+cr$ and assume $a_1\in R_k^\la$ (see \eqref{set R}). Then $a_2\in R_{k+cn}^\la$. By definition,
$$Y_{a_2}=\{b\in\mbz\mid a_2<b,\;b\in R_{k+cn}^\la\}
=\{b\in\mbz\mid a_1<b-cr,\;b-cr\in R_{k}^\la\}.$$
Hence, there is a bijection from $Y_{a_2}$ to $Y_{a_1}$ defined by sending $b$ to $b-cr$, proving the claim.

Since the remainders of $d(i)$ when divided by $r$ are all distinct and
$$Y=\bin_{1\leq s\leq r}\{(d(s),j)\mid j\in Y_{d(s)}\},$$
it follows from the claim that
\begin{equation*}
\begin{split}
|Y|&=\sum_{1\leq s\leq r}|Y_{d(s)}|=\sum_{1\leq s\leq r}|Y_{s}|\\
&=\sum_{1\leq s\leq r}|\{t\in\mbz\mid s<t,\;s,t\in R_k^{\la}\text{ for some }k\in\mbz\}|\\
&=|Z|=\ell(w_{0,\la}).
\end{split}
\end{equation*}
\end{proof}

The next lemma is used in the proof of Proposition \ref{triangular formula in A(bfj,r)}.
\begin{Lem}\label{ddd}
Keep the notations $A,A^+,A^-,\la,\mu,\nu$ used in \eqref{Lem-ddd}.
We have $$d_{A+\diag(\la)}=d_{A^++\diag(\mu)}+d_{A^-+\diag(\nu)}.$$
\end{Lem}
\begin{proof} Recall from Lemma \ref{lemma2 for triangular decomposition of affine q-Schur algebra}(2)
that $\mu,\nu$ are uniquely determined by the conditions
$$(\bfL,\bfL')\in\sO_{A^++\diag(\mu)}\;\;\text{ and}\;\;
(\bfL',\bfL'')\in\sO_{A^-+\diag(\nu)},$$
whenever $(\bfL,\bfL'')\in\sO_{A+\diag(\la)}$ and $\bfL'=\bfL\cap\bfL''$.
 By Lemma \ref{triangular matrix}(3), for $i\in\mbz$,
\begin{equation*}
\begin{split}
\la_i+\sum_{k,k\not=i}a_{i,k}&=\dim(L_i/L_{i-1})=\mu_i+\sum_{k,i<k}a_{i,k},\\
\la_i+\sum_{k,k\not=i}a_{k,i}&=\dim(L_i''/L_{i-1}'')=\nu_i+\sum_{k,k>i}a_{k,i}.
\end{split}
\end{equation*}
Thus, $\mu_i=\la_i+\sum_{k<i}a_{i,k}$ and
$\nu_i=\la_i+\sum_{k<i}a_{k,i}$ for all $i$.  Moreover, since for
such $\mu,\nu$,
$$d_{A+\diag(\la)}=\sum_{1\leq i\leq n\atop i\geq
k;j<l}a_{i,j}a_{k,l}+ \sum_{1\leq i\leq n\atop k\leq
i<l}\la_ia_{k,l}+\sum_{j<k\leq i}\la_ka_{i,j},$$
$$\aligned
d_{A^++\diag(\mu)}&=\sum_{\tiny\mbox{$\begin{array}{c}
1\leq i\leq n\\i\geq k;j<l\\i<j;k<l\end{array}$}}a_{i,j}a_{k,l}+\sum_{1\leq i\leq n\atop k\leq i<l}\mu_ia_{k,l},\\
&=\sum_{\tiny\mbox{$\begin{array}{c}1\leq i\leq n\\i\geq k;j<l\\i<j;k<l
\end{array}$}}a_{i,j}a_{k,l}+\sum_{1\leq i\leq n\atop k\leq i<l}\la_ia_{k,l}+
\sum_{\tiny\mbox{$\begin{array}{c}1\leq i\leq n\\k\leq
i<l\\i>j\end{array}$}}a_{i,j}a_{k,l},
\endaligned$$
$$\aligned
d_{A^-+\diag(\nu)}&=\sum_{\tiny\mbox{$\begin{array}{c}1
\leq i\leq n\\i\geq k;j<l\\i>j;k>l\end{array}$}}a_{i,j}a_{k,l}+\sum_{1\leq i\leq n\atop j<k\leq i}\nu_ka_{i,j}\\
&=\sum_{\tiny\mbox{$\begin{array}{c}1\leq i\leq n\\i\geq
k;j<l\\i>j;k>l\end{array}$}}a_{i,j}a_{k,l}+\sum_{1\leq i\leq n\atop
j<k\leq i}\la_ka_{i,j}+\sum_{\tiny\mbox{$\begin{array}{c}1\leq i\leq
n\\j<l\leq i\\k<l\end{array}$}}a_{i,j}a_{k,l},
\endaligned$$
it follows that
$$\aligned
&\quad d_{A+\diag(\la)}-d_{A^++\diag(\mu)}-d_{A^-+\diag(\nu)}\\
&=\sum_{1\leq i\leq n\atop i\geq
k;j<l}a_{i,j}a_{k,l}-\sum_{\tiny\mbox{$\begin{array}{c}1\leq i\leq
n\\i\geq k,j<l\\i<j,k<l\end{array}$}}a_{i,j}a_{k,l}-
\sum_{\tiny\mbox{$\begin{array}{c}1\leq i\leq n\\i\geq
k;j<l\\i>j;k>l\end{array}$}}a_{i,j}a_{k,l}\\
&\quad-\sum_{\tiny\mbox{$\begin{array}{c}1\leq i\leq n\\k\leq
i<l\\i>j\end{array}$}}a_{i,j}a_{k,l}-\sum_{\tiny\mbox{$\begin{array}{c}1\leq
i\leq n\\j<l\leq i\\k<l\end{array}$}}a_{i,j}a_{k,l}\\
&=\sum_{\tiny\mbox{$\begin{array}{c}1\leq i\leq n\\i\geq
k,j<l\\i>j,k<l
\end{array}$}}a_{i,j}a_{k,l}-
\sum_{\tiny\mbox{$\begin{array}{c}1\leq i\leq n\\k\leq i<l\\i>j
\end{array}$}}a_{i,j}a_{k,l}-\sum_{\tiny\mbox{$\begin{array}{c}1\leq i\leq n\\j<l\leq i\\k<l
\end{array}$}}a_{i,j}a_{k,l}\\
&=0,
\endaligned$$
as required.
\end{proof}

Let $f_\nu,g_\nu$ be the numbers defined as in \eqref{fnu},\eqref{gnu}. The following lemma tells a
certain relationship between these numbers and is used in the proof of
Theorem \ref{lem of Poly}.
\begin{Lem}\label{appendix}Maintain the notation in the proof of Theorem \ref{lem of Poly}.
For any ${\bf0}\leq\nu\leq\tila$, we have $f_\tila=g_\tila$ and
\begin{equation*}
\aligned
f_{\tila}-f_{\nu}&=-2\sum_{1\leq i\leq n}\tila_iz_i+\sum_{1\leq i\leq
n}2\nu_i(z_i+\tila_i+\tila_{i+1}-\nu_i-\nu_{i+1})\\
g_{{\tila}}-g_{\nu}&=-2\sum_{1\leq i\leq
n}\tila_iz_{i+1}+\sum_{1\leq i\leq
n}2\nu_i(z_{i+1}+\tila_i+\tila_{i-1}-\nu_i-\nu_{i+1}).\endaligned
\end{equation*}
\end{Lem}

\begin{proof}
By definition, for $0\leq\nu\leq\tila=\la-\gamma$,
\begin{equation*}
\begin{split}
&f_{\tila}-f_{\nu}\\
&=\sum_{1\leq i\leq n}(\tila_i+\gamma_i)(\tila_{i-1}+\gamma_{i-1}-(\gamma_i+\tila_i)-y_i)-
\sum_{1\leq i\leq n}(\dt_i+\timu_i)x_{i+1}+\lan\dt+\timu,\gamma+\tila\ran\\
&\quad-\sum_{1\leq i\leq n}(\gamma_i+\nu_i)(-\gamma_i+\nu_i+\gamma_{i-1}+\nu_{i-1}-2\tila_i-y_i)
-\sum_{1\leq i\leq n}(\dt_i+\nu_i)(2(\nu_i-\tila_i)-x_{i+1})\\
&\quad-\lan\dt+\nu,\tila-\nu\ran-\lan\dt+\nu,\gamma+\tila\ran-\lan\tila-\nu,x\ran\\
&=\sum_{1\leq i\leq n}(\tila_i\tila_{i-1}+\tila_i\gamma_{i-1}-\tila_i^2-\tila_iy_i+
\gamma_i\tila_{i-1}-x_{i+1}\timu_i-\gamma_i\nu_{i-1}-3\nu_i^2-\nu_i\gamma_{i-1}-\nu_i\nu_{i-1})\\
&\quad+\sum_{1\leq i\leq n}(4\nu_i\tila_i+\nu_iy_i-2\dt_i\nu_i+2\dt_i\tila_i+\nu_ix_{i+1})\\
&\quad+\lan\timu-\nu,\gamma+\tila\ran-\lan\dt+\nu,\tila-\nu\ran-\lan\tila-\nu,x\ran=f_{\nu}'+f_{\nu}'',
\end{split}
\end{equation*}
where
\begin{equation*}
\begin{split}
f_{\nu}'&=\sum_{1\leq i\leq n}(\tila_i\tila_{i-1}+\tila_i\gamma_{i-1}-\tila_i^2-\tila_iy_i
+\gamma_i\tila_{i-1}-x_{i+1}\timu_i+\dt_i\tila_i\\
&\qquad\qquad+\timu_i\tila_i+\timu_i\gamma_i-\timu_i\gamma_{i+1}-\timu_i\tila_{i+1})
+\sum_{1\leq i\leq n}(\dt_i\tila_{i+1}-x_i\tila_i+x_{i+1}\tila_i),\\
f_{\nu}''&=\sum_{1\leq i\leq n}(-\gamma_i\nu_i-\gamma_i\nu_{i-1}-2\nu_i^2-\nu_i\gamma_{i-1}
-\nu_i\nu_{i-1}+2\nu_i\tila_i+\nu_iy_i-\dt_i\nu_i+\nu_i\gamma_{i+1}+2\nu_i\tila_{i+1})\\
&\quad+\sum_{1\leq i\leq
n}(-\dt_i\nu_{i+1}-\nu_i\nu_{i+1}+x_i\nu_i).
\end{split}
\end{equation*}
 Since $\timu=\tila$ and $\gamma_{i-1}+\gamma_i-y_i-x_i+\dt_{i-1}+\dt_i=-2z_i$ for all $i\in\mbz$, we have
\begin{equation*}
\begin{split}
f_{\nu}'&=\sum_{1\leq i\leq n}\tila_i(\gamma_{i-1}-y_i+\dt_i+\gamma_i-\gamma_{i+1}-x_i)+\sum_{1\leq i\leq n}\gamma_i\tila_{i-1}+\sum_{1\leq i\leq n}\dt_i\tila_{i+1}\\
&=\sum_{1\leq i\leq n}\tila_i(\gamma_{i-1}+\gamma_i-y_i-x_i+\dt_{i-1}+\dt_i)\\
&=-2\sum_{1\leq i\leq n}\tila_iz_i,\\
\end{split}
\end{equation*}
\begin{equation*}
\begin{split}
f_{\nu}''&=\sum_{1\leq i\leq n}\nu_i(-(\gamma_{i-1}+\gamma_i-y_i-x_i+\dt_{i-1}+\dt_i)-2\nu_i-\nu_{i-1}+2\tila_i+2\tila_{i+1}-\nu_{i+1})\\
&=\sum_{1\leq i\leq
n}2\nu_i(z_i+\tila_i+\tila_{i+1}-\nu_i-\nu_{i+1}).
\end{split}
\end{equation*}
Consequently, for $0\leq\nu\leq\tila$, we obtain that
\begin{equation*}\label{f(tila)-f(ga)}
f_{\tila}-f_{\nu}=-2\sum_{1\leq i\leq n}\tila_iz_i+\sum_{1\leq i\leq
n}2\nu_i(z_i+\tila_i+\tila_{i+1}-\nu_i-\nu_{i+1}).
\end{equation*}
Since $\timu=\tila$, for $0\leq\nu\leq\tila$, we have
\begin{equation*}
\begin{split}
g_{\tila}-g_{\nu}&=\sum_{1\leq i\leq n}(\tila_i+\gamma_i)(\gamma_{i-1}+\tila_{i-1}-\gamma_i-\tila_i-y_i)
-\sum_{1\leq i\leq n}(\dt_i+\tila_i)x_{i+1}+\lan\dt+\tila,\dt+\tila\ran\\
&\quad-\sum_{1\leq i\leq n}(\gamma_i+\nu_i)(-\gamma_i+\gamma_{i-1}+\nu_i+\nu_{i-1}-2\tila_i-y_i)\\
&\quad -\sum_{1\leq i\leq n}(\dt_i+\nu_i)(2(\nu-\tila_i)-x_{i+1})-\lan\tila
-\nu,\gamma+\nu\ran-\lan\dt+\tila,\dt+\nu\ran-\lan\nu-\tila,x\ran\\
&=g_{\nu}'+g_{\nu}'',
\end{split}
\end{equation*}
where
\begin{equation*}
\begin{split}
g_{\nu}'&=\sum_{1\leq i\leq n}(\tila_i\gamma_{i-1}-\tila_i\gamma_i-\tila_iy_i+\gamma_i\tila_{i-1}
-2x_{i+1}\tila_i+3\dt_i\tila_i-\dt_i\tila_{i+1}+\tila_i\gamma_{i+1}+x_i\tila_i),\\
g_{\nu}''&=\sum_{1\leq i\leq n}(\gamma_i\nu_i-\gamma_i\nu_{i-1}-2\nu_i^2-\nu_i\gamma_{i-1}
-\nu_i\nu_{i-1}+2\nu_i\tila_i+\nu_iy_i-3\dt_i\nu_i+2\nu_ix_{i+1}+\dt_i\nu_{i+1}\\
&\quad+2\tila_i\nu_{i+1}-\nu_i\gamma_{i+1}-\nu_i\nu_{i+1}-x_i\nu_i).
\end{split}
\end{equation*}
It is easy to check that for all $i\in\mbz$,
$$\gamma_{i-1}-\gamma_i+2\gamma_{i+1}+x_i-y_i-2x_{i+1}+3\dt_i-\dt_{i-1}
=-2z_{i+1}.$$
 Therefore,
\begin{equation*}
\begin{split}
g_{\nu}'&=\sum_{1\leq i\leq n}\tila_i(\gamma_{i-1}-
\gamma_i+2\gamma_{i+1}+x_i-y_i-2x_{i+1}+3\dt_i-\dt_{i-1})
=-2\sum_{1\leq i\leq n}\tila_iz_{i+1},\\
g_{\nu}''&=\sum_{1\leq i\leq n}\bigl(-(\gamma_{i-1}-\gamma_i+2\gamma_{i+1}+x_i-y_i-2x_{i+1}+3\dt_i-\dt_{i-1})\\
&\qquad\qquad\qquad\qquad\qquad\qquad-2\nu_i-\nu_{i-1}-\nu_{i+1}+2\tila_i+2\tila_{i-1}\bigr)\\
&=\sum_{1\leq i\leq
n}2\nu_i(z_{i+1}+\tila_i+\tila_{i-1}-\nu_i-\nu_{i+1}).
\end{split}
\end{equation*}
Note that $g_{\tila}=f_{\tila}$. Thus, for $0\leq\nu\leq\tila$,
\begin{equation*}\label{f(tila)-g(ga)}
f_{\tila}-g_{\nu}=g_{{\tila}}-g_{\nu}=-2\sum_{1\leq i\leq
n}\tila_iz_{i+1}+\sum_{1\leq i\leq
n}2\nu_i(z_{i+1}+\tila_i+\tila_{i-1}-\nu_i-\nu_{i+1}).
\end{equation*}
\end{proof}

\chapter{Representations of affine quantum Schur algebras}

The quantum Schur--Weyl duality links representations of quantum
$\frak{gl}_n$ with those of Hecke algebras of symmetric groups. More
precisely, there are category equivalences between level $r$
representations of quantum $\frak{gl}_n$ and representations of the
Hecke algebra of $\fS_r$. This relationship provides in turn two
approaches to representations of quantum Schur algebras. First, one
uses the polynomial representation theory of quantum $\frak{gl}_n$
to determine representations of quantum Schur algebras. This is a
{\it downwards} approach. Second, representations of quantum Schur
algebras can also be determined by those of Hecke algebras. We call
this an {\it upwards} approach. In this chapter, we will investigate
the affine versions of the two approaches.

Classification of finite dimensional simple (or irreducible) polynomial representations of the
quantum loop algebra $\afUglC$ was completed by Frenkel--Mukhin
\cite{FM}, based on Chari--Pressley's classification of finite dimensional simple
representations of the quantum loop algebra $\afUslC$
\cite{CP91,CPbk}. On the other hand, Zelevinsky \cite{Zelevinsky}
and Rogawski \cite{Rogawski} classified simple representations
of the Hecke algebra $\sH_\vtg(r)_\mbc$. Moreover, Chari--Pressley
have also established a  category equivalence between the module
category $\sH_\vtg(r)_\mbc\hmod$ and a certain full subcategory of
$\afUslC\hmod$ when $n>r$ without using affine quantum Schur
algebras.

We will first establish in \S4.1 a category equivalence between the
categories $\afSr_\field\hMod$ and $\afHr_\field\hMod$ when $n\geq r$ (Theorem \ref{category equiv}).
As an immediate application of this equivalence, we prove that every simple
representation in $\afSr_\field\hMod$ is finite dimensional (Theorem \ref{fd simple modules}).
We will briefly review the classification theorems of Chari--Pressley and Frenkel--Mukhin in \S4.2 and
of Rogawski in \S4.3. We will then use the category equivalence to classify simple
representations of affine quantum Schur algebras via Rogawski's
classification of simple representations of $\sH_\vtg(r)_\mbc$ (Theorem \ref{weakly category equiv}).

From \S 4.4 onwards, we will investigate the downwards approach.
Using the surjective homomorphism $\zrC:\afUglC \to \afSrC$, every
simple $\afSrC$-module is inflated to a simple $\afUglC$-module. Our
question is how to identify those inflated simple modules. We use a
two-step approach. First, motivated by a result of Chari--Pressley,
we identify in the $n>r$ case the simple $\afSrC$-modules arising
from simple $\afHrC$-modules in terms of simple polynomial
representations of $\afUglC$ (Theorem \ref{n>r representation}).
Then, using this identification, we will classify all simple
$\afSr_\mbc$-modules through simple polynomial representations of
$\afUglC$ (Theorem \ref{representation}). Thus, all finite
dimensional simple polynomial representations of $\afUglC$ are
inflations of simple representations of affine quantum Schur
algebras. In this way, like the classical theory, affine quantum
Schur algebras play a bridging role between polynomial
representations of quantum affine $\frak{gl}_n$ and those of affine
Hecke algebras of $\fS_r$. We end the chapter by presenting a
classification of simple $U_\vtg(n,r)_\mbc$-modules (Theorem
\ref{simple U(n,r)-modules}).

Throughout the chapter, all representations are defined over a
ground field $\mathbb F$ with a fixed specialization $\sZ\to\mathbb
F$\index{$\sZ=\mbz[\up,\up^{-1}]$, Laurent polynomial ring in
indeterminate $\up$} by sending $\up$ to a non-root of unity $z\in
\mathbb F$. From \S4.2 onwards, $\mathbb F$ will be the complex
number field $\mbc$. The algebras $\afUslC$ and $\afUglC$ are
defined in \S2.3 with Drinfeld's new presentation. For any algebra
$\scr A$ considered in this chapter, the notation ${\scr A}\hmod$
represents the category of all finite dimensional left $\scr
A$-modules. We also use ${\scr A}\hMod$ to denote the category of
all left $\scr A$-modules.

\section{Affine quantum Schur--Weyl duality, II}

In this section we establish the affine version of the category
equivalences mentioned in the introduction above. As a byproduct, we
prove that every irreducible representation over the affine quantum
Schur algebra is finite dimensional.

Recall from \S2.2 (see footnote 2 there) the extended affine
$\frak{sl}_n$, $\bfU_\vtg(n)$, generated by $E_i,F_i,K_i^\pm$
subject to relations (QGL1)--(QGL5) given in Definition
\ref{presentation dHallAlg}. By dropping the subscripts $_\vtg$ from
$\bfU_\vtg(n)$, $\afbfSr$, and $\bsH_\vtg(r)$, we obtain the
notations $\bfU(n)$, $\bfSr$, and $\bsH(r)$ for quantum
$\frak{gl}_n$, quantum Schur algebras and Hecke algebras of type
$A$.
 The algebras $\bfU(n)$, $\bfSr$ and $\bsH(r)$ will be
naturally viewed as subalgebras of $\bfU_\vtg(n)$,  $\afbfSr$ and
$\bsH_\vtg(r)$, respectively. Moreover, we may also regard $\bfU(n)$
as a subalgebra of $\dHallr=\bfU(\afgl)$. Let $U_\vtg(n)$ be the
$\sZ$-form of $\bfU_\vtg(n)$ defined in \eqref{integral form of
U(n)} and let $U(n)$, $\sS(n,r)$ and $\sH(r)$ denote the
corresponding $\sZ$-free subalgebras of $U_\vtg(n)$, $\afSr$ and
$\afHr$.

The tensor space $\Og^{\otimes r}$ over $\sZ$ is a
$U_\vtg(n)$-$\afHr$-bimodule via the actions \eqref{QGKMAlg-action}
and \eqref{afH action}, where the affine $\afHr$-action is a natural
extension from the $\sH(r)$-action on $\Og_n^{\ot r}$, where $\Og_n$
is the $\sZ$-subspace of $\Og$ spanned by the elements $\og_i$
$(1\leq i\leq n)$; see \eqref{afH action}. Thus, $\Og_n^{\ot r}$ is
a $U(n)$-$\sH(r)$-subbimodule of $\Og^{\ot r}$ by restriction.

\begin{Lem}\label{Lem1 for category equivalent}
There is a $U(n)$-$\sH_\vtg(r)$-bimodule isomorphism
$$\Og_n^{\ot r}\ot_{\sH(r)}\sH_\vtg(r)\stackrel{\sim}{\lra}\Og^{\ot r},\;
x\ot h\longmapsto xh.$$
\end{Lem}

\begin{proof} Clearly, there is a $U(n)$-$\sH_\vtg(r)$-bimodule
homomorphism
\begin{equation*}
\vi:\Og_n^{\ot r}\ot_{\sH(r)}\sH_\vtg(r)\lra\Og^{\ot r},\; x\ot
h\longmapsto x h.
\end{equation*}
Since the set $\{T_wX_1^{a_1}\cdots X_r^{a_r}\mid
w\in\fS_r,\,a_i\in\mbz,\,1\leq i\leq r\}$ forms a $\sZ$-basis for
$\sH_\vtg(r)$, it follows that the set
$$\sX:=\{\og_\bfi\ot X_1^{a_1}\cdots X_r^{a_r}\mid\bfi\in I(n,r),\,a_i\in\mbz,\,1\leq i\leq r\}$$
forms a $\sZ$-basis for $\Og_n^{\ot r}\ot_{\sH(r)}\sH_\vtg(r)$.
Furthermore, by \eqref{afH action},
$$\vi(\og_\bfi\ot X_1^{a_1}\cdots X_r^{a_r})
=\og_\bfi X_1^{a_1}\cdots
X_r^{a_r}=\og_{\bfi-n(a_1,a_2,\ldots,a_r)},$$ for all
$(a_1,a_2,\ldots,a_r)\in\mbz^r$ and $\bfi\in I(n,r)$.
 Thus, the bijection \eqref{afInr4} from $\afInr$ to $\mbz^r=I(n,r)+n\mbz^r$ implies
that $\vi(\sX)=\{\og_{\bfi}\mid\bfi\in\afInr\}$ forms a $\sZ$-basis
for $\Og^{\ot r}$. Hence, $\vi$ is a $U(n)$-$\sH_\vtg(r)$-bimodule
isomorphism.
\end{proof}

For $N=\text{max}\{n,r\}$, let $\og\in\Lambda_\vtg(N,r)$ be defined
as in \eqref{111}. Thus, if $\og=(\og_i)$, then $\og_i=1$ for $1\leq
i\leq r$, and $\og_i=0$ if $n>r$ and $r< i\leq n$. Let
$e_\og=e_{\diag(\og)}.$

\begin{Coro}\label{Lem2 for category equivalent}
If $n\geq r$, then there is an $\sS(n,r)$-$\sH_\vtg(r)$-bimodule
isomorphism
$$\afSr e_\og \cong\sS(n,r)e_\og \ot_{\sH(r)} \sH_\vtg(r).$$
\end{Coro}

\begin{proof} By the proof of Lemma \ref{tsp}, there is an $\sS_\vtg(n,r)$-$\sH_\vtg(r)$-bimodule
isomorphism $\afSr e_\og\cong \sT_\vtg(n,r)$. By Proposition \ref{bimodule-isom},
this isomorphism extends to a bimodule isomorphism $\afSr
e_\og\cong\Og^{\otimes r}$. Now, the required result follows from
the lemma above.
\end{proof}

We are now ready to establish a category equivalence. {\it For the
rest of this section, we assume $\field$ is a (large enough) field which
is a $\sZ$-algebra such that the image $z\in \field$ of $\up$ is not
a root of unity.} Thus, both $\sH(r)_\field$ and
 $\sS(n,r)_\field$ are semisimple $\field$-algebras, and every finite dimensional $U(n)_\field$-module is completely reducible.
Let $\afSr_\field\hmod$ (resp., $\sH_\vtg(r)_\field\hmod$) be the
category of {\it finite dimensional} $\afSr_\field$-modules (resp.,
$\sH_\vtg(r)_\field$-modules). The following result generalizes the
category equivalence between $\sS(n,r)_\field\hmod$ and
$\sH(r)_\field\hmod$ $(n\geq r)$ to the affine case.

\begin{Thm}\label{category equiv}
Assume $n\geq r$. The categories $\afSr_\field\hMod$ and
$\sH_\vtg(r)_\field\hMod$ are equivalent. It also induces a category
equivalence between $\afSr_\field\hmod$ and
$\sH_\vtg(r)_\field\hmod$. Hence, in this case, the algebras
$\afSr_\field$ and $\afHr_\field$ are Morita
equivalent.\index{Morita equivalence}
\end{Thm}

\begin{proof}
The $\afSr_\field$-$\sH_\vtg(r)_\field$-bimodule $\Og_\field^{\ot
r}$ induces a functor
\begin{equation}\label{functor sfF}
\sfF:\sH_\vtg(r)_\field\hMod\lra\afSr_\field\hMod,\;L\longmapsto\Og_\field^{\ot
r}\ot_{\sH_\vtg(r)_\field}L.
\end{equation}
 Since $n\geq r$, there is a functor, the Schur functor,
\begin{equation*}
\sfG:\afSr_\field\hMod\lra\sH_\vtg(r)_\field\hMod,\; M\longmapsto
e_\og M.
\end{equation*}
Here we have identified $ e_\og \afSr_\field e_\og $ with
$\sH_\vtg(r)_\field$. We need to prove that there are natural
isomorphisms
 $\sfG\circ \sfF\cong \mathsf{id}_{\sH_\vtg(r)_\field\hMod}$, which is clear (see, e.g., \cite[(6.2d)]{Gr80}),
 and $\sfF\circ \sfG\cong \mathsf{id}_{\afSr_\field\hMod}$.

By the semisimplicity, for any $\sS(n,r)_\field$-module $M$, there
is a left $\sS(n,r)_\field$-module isomorphism
\begin{equation}\label{the iso f}
f:\sS(n,r)_\field e_\og  \ot_{\sH(r)_\field} e_\og M\cong M
\end{equation}
defined by $f(x\ot m)=xm$ for any $x\in\sS(n,r)_\field e_\og $ and
$m\in  e_\og M$. By Corollary \ref{Lem2 for category equivalent},
\eqref{the iso f} induces a left $\sS(n,r)_\field$-module
isomorphism
$$g:\afSr_\field e_\og
\ot_{\sH_\vtg(r)_\field} e_\og M\cong M$$
 satisfying $g(x\ot m)=xm$ for any $x\in\sS(n,r)_\field e_\og $
and $m\in  e_\og M$.

We now claim that $g$ is an $\afSr_\field$-module isomorphism.
Indeed, by identifying $\afSr_\field$ with $\sS^{_\sH}_\vtg(n,r)$
under the isomorphism given in Proposition \ref{Green defn=geom
defn} and considering the basis $\{\phi_{\la,\mu}^d\}$, we have
$e_\og=\phi_{\og,\og}^1$ and
$$\afSr_\field e_\og =\bigoplus_{\la\in\afLanr}\Hom_{\afHr_\field}(\afHr_\field,x_\la\afHr_\field)=\bop_{\la\in\afLanr}
\phi_{\la,\og}^1\sH_\vtg(r)_\field,$$ where
$\phi_{\la,\og}^1\in\sS(n,r)_\field$. Since $g(\phi_{\la,\og}^1\ot
m)=\phi_{\la,\og}^1 m,$  for all $\la\in\afLanr$ and $m\in e_\og M$,
it follows that for any $\la\in\afLanr$, $h\in\sH_\vtg(r)_\field$
and $m\in e_\og M$,
$$g(\phi_{\la,\og}^1h\ot m)=g(\phi_{\la,\og}^1\ot
hm)=\phi_{\la,\og}^1 (h m)=(\phi_{\la,\og}^1 h) m.$$ Hence, $g(x\ot
m)=xm$ for all $x\in\afSr_\field e_\og $ and $m\in  e_\og M$. Thus,
$g$ is an $\afSr_\field$-module isomorphism. Therefore, for any
$M\in\afSr_\field\hMod$, there is a natural isomorphism
$$\sfF\circ \sfG(M)\cong \afSr_\field e_\og
\ot_{\sH_\vtg(r)_\field} e_\og M\cong M,$$ proving $\sfF\circ
\sfG\cong \mathsf{id}_{\afSr_\field\hMod}$.

 The last assertion follows from the fact that, if $N$ is a finite dimensional $\afHr_\field$-module,
 then Lemma \ref{Lem1 for category equivalent} implies that $\sfF(N)=\Og_\field^{\ot
r}\ot_{\sH_\vtg(r)_\field}N$ is also finite dimensional.
\end{proof}

\begin{Rems}  (1) For $\field =\mbc$, a direct category equivalence from $\afHrC\hmod$ to
a full subcategory of $\afUglC\hmod$ has been established in \cite[Th.~6.8]{GRV} when $n\geq r$.
The construction there is geometric, using intersection cohomology complexes.

(2) See \cite{Vig} for a similar equivalence
in the context of representations of $p$-adic groups. See also
\cite{VV04} for a connection with representations of double affine Hecke
algebras of type $A$.
\end{Rems}

Following \cite[2.5]{CP}, a finite dimensional
$U(n)_\field$-module $M$ is said to be of {\it level} $r$ if every
irreducible component of $M$ is isomorphic to an irreducible
component of $\Og_{n,\field}^{\ot r}$. In other words, a
$U(n)_\field$-module $M$ has level $r$ if and only if it is an
$\sS(n,r)_\field$-module. We will generalize this definition to the
affine case; see Corollaries \ref{level r} and \ref{level r2} below.
Note that, since levels are defined for modules of quantum
$\frak{sl}_n$ in \cite[2.5]{CP}, the condition $n>r$ there is
necessary.

Recall that a $U_\vtg(n)_\field$-module $M$ is called of {\it type
1}, if it is a direct sum $M=\oplus_{\la\in\mbz^n}M_\la$ of its
weight spaces $M_\la$ which has the form
$$M_\la=\{x\in M\mid K_ix=z^{\la_i}x\}.$$
In other words, a $U_\vtg(n)_\field$-module $M$ is of type 1 if it
is of type 1 as a $U(n)_\field$-module.

Theorem \ref{category equiv} immediately implies the following
category equivalence due to Chari--Pressley \cite{CP}.  However, a
different functor is used in \cite[Th.~4.2]{CP}. We will make a
comparison of the two functors in the next section.

\begin{Coro}\label{ChariPressley}
If $n>r$, then the functor $\sfF$ induces a category equivalence
between the category of finite dimensional
$U_\vtg(n)_\field$-modules of type 1 which are of level $r$ when
restricted to $U(n)_\field$-modules and the category of finite
dimensional $\sH_\vtg(r)_\field$-modules.
\end{Coro}
\begin{proof} By Lemma \ref{S(n,r)=U(n,r)},
the condition $n>r$ implies $\afSr_\field=U_\vtg(n,r)_\field$ is a
homomorphic image of $U_\vtg(n)_\field$. Thus, level $r$
representations considered are precisely $\afSr_\field$-modules. Now
the result follows immediately from the theorem above.
\end{proof}


If $N\geq n$, then there is a natural injective map
$$\ti{\,\,}:\afThn\lra\Th_\vtg(N),\quad A=(a_{i,j})\longmapsto\ti
A=(\ti a_{i,j}),$$ where $\ti A$ is defined in \eqref{AtoAtilde}.
Similarly, there is an injective map
\begin{equation}\label{tilde-map}
\ti\ :\mbz_\vtg^n\lra\mbz_\vtg^N,\;\;\la\longmapsto\ti\la,
\end{equation}
 where $\ti\la_i=\la_i$ for $1\leq i\leq n $ and $\ti\la_i=0$ for
$n+1\leq i\leq N$. See the proof of Lemma \ref{tsp}.

These two maps induce naturally by Lemma \ref{tsp} an injective
algebra homomorphism (not sending 1 to 1)
$$\iota_r=\iota_{n,N,r}:\afSr\lra\sS_\vtg(N,r),\;\;[A]\longmapsto [\ti A]\;\;\text{for $A\in\afThnr$}.$$
In other words, we may identify $\afSr$ as the centralizer
subalgebra $e\sS_\vtg(N,r)e$ of $\sS_\vtg(N,r)$, where
$e=\sum_{\la\in\afLanr}\phi_{\ti\la,\ti\la}^1=\sum_{\la\in\afLanr}[\diag(\ti\la)]\in\sS_\vtg(N,r)$.

Now, the fact that every simple module of an affine Hecke algebra is
finite dimensional implies immediately that the same is true for
affine quantum Schur algebras.

\begin{Thm}\label{fd simple modules}
 Every simple $\afSr_\field$-module is finite dimensional.
\end{Thm}

\begin{proof} Let $L$ be a simple $\afSr_\field$-module.
By Corollary \ref{QSA}, we identify the quantum Schur algebra
$\sS(n,r)_\field$ as a subalgebra of $\afSr_\field$. Thus, for any
$0\neq x\in L$, $\sS(n,r)_\field x\subset L$ is a nonzero finite
dimensional $\sS(n,r)_\field$-module. Hence, it is a direct sum of
simple $\sS(n,r)_\field$-modules.

If $n\geq r$, then the Hecke algebra $\sH(r)_\field$ identifies the
centralizer subalgebra $e_\omega\sS(n,r)_\field e_\omega$, where
$e_\omega$ is the idempotent $\phi_{\omega,\omega}^1$. Thus,
$e_\omega\sS(n,r)_\field x$ is a nonzero $\sH(r)_\field$-module. We
conclude that $e_\omega L$ is a nonzero simple
$\sH_\vtg(r)_\field$-module. Since every simple
$\sH_\vtg(r)_\field$-module is finite dimensional, the simple
$\afSr_\field$-module $L\cong\Og_\field^{\ot
r}\ot_{\sH_\vtg(r)_\field}e_\og L$ is finite dimensional.

If $n<r$, then, for $N=r$, $\afSr_\field\cong e\sS_\vtg(r,r)_\field
e$, where $e=\sum_{\la\in\afLanr}\phi_{\ti\la,\ti\la}^1$. It follows
that each simple $\afSr_\field$-module is isomorphic to $eL$ for a
simple $\sS_\vtg(r,r)_\field$-module $L$; see, for example,
\cite[6.2(g)]{Gr80}. As shown above, all simple
$\sS_\vtg(r,r)_\field$-modules are finite dimensional. Hence, so are
all simple $\afSr_\field$-modules.
\end{proof}

\section{Chari--Pressley's category equivalence and classification}

We first prove that the functor $\sfF$ defined in \eqref{functor
sfF} coincides with the functor $\sF$ defined in \cite[Th.~4.2]{CP}.
Then, we describe the Chari--Pressley's classification of simple
$\afUslC$-modules and its generalization to $\afUglC$ for later use.

{\it From now on, we assume the ground field $\field=\mbc$, the
complex number field, which is a $\sZ$-module where $\up$ is mapped
to $z$ with $z^m\neq1$ for all $m\geq1$.}

Following \cite{CP}, let $E_\theta$ and $F_\theta$ be the operators
on $\Og_{n,\mbc}$ defined respectively by
\begin{equation}\label{Etheta}
E_\theta\og_i=\dt_{i,n}\og_1\;\;\text{and}\;\;
F_\theta\og_i=\dt_{i,1}\og_n.
\end{equation}
 Let $\ti K_\theta=\ti K_1\cdots\ti K_{n-1}=\ti K_0^{-1}=K_0^{-1}K_1$.\footnote{The $K_0$ and
 $E_0,F_0$ below should be regarded as $K_n, E_n, F_n$, respectively, if the index set
$I=\mbz/n\mbz$ is identified as $\{1,2,\ldots,n\}$.} For a left
$\sH_\vtg(r)_F$-module $M$, define
$$\sF(M)=\Og_{n,\mbc}^{\ot r}\ot_{\sH(r)_\mbc}M.$$
 Then $\sF(M)$ is equipped with the natural
action of $U(n)_\mbc$-module structure induced by that on
$\Og_{n,\mbc}^{\ot r}$. By \cite[Th.~4.2]{CP}, $\sF(M)$ becomes a
$U_\vtg(n)_\mbc$-module via the action
\begin{equation*}
\begin{split}
E_0(w\ot m)&=\sum_{j=1}^r(Y_j^+w)\ot(X_jm),\\
F_0(w\ot m)&=\sum_{j=1}^r(Y_j^-w)\ot(X_j^{-1}m),
\end{split}
\end{equation*}
where $w\in\Og_{n,\mbc}^{\ot r}$, $m\in M$ and the operators
$Y_j^\pm\in\End_{\mbc}(\Og_{n,\mbc}^{\ot r})$ $(1\leq j\leq r)$ are
defined by
\begin{equation*}
\begin{split}
Y_j^+&=1^{\ot j-1}\ot F_\theta\ot(\ti K_\theta^{-1})^{\ot r-j}=1^{\ot j-1}\ot F_\theta\ot(\ti K_0)^{\ot r-j},\\
Y_j^-&=(\ti K_\theta)^{\ot j-1}\ot E_\theta\ot 1^{\ot r-j}=(\ti
K_0^{-1})^{\ot j-1}\ot E_\theta\ot 1^{\ot r-j}.
\end{split}
\end{equation*}
Hence, we obtain a functor
$$\sF:\sH_\vtg(r)_\mbc\hmod\lra U_\vtg(n)_\mbc\hmod.$$
When $n>r$, Chari--Pressley used this functor to establish a
category equivalence between $\sH_\vtg(r)_\mbc\hmod$ and the full
subcategory of $U_\vtg(n)_\mbc\hmod$ consisting of finite
dimensional $U_\vtg(n)_\mbc$-modules which are of level $r$ when
restricted to $U(n)_\mbc$-modules; see Corollary
\ref{ChariPressley}.

On the other hand, the algebra homomorphism
$\xi_{r,\mbc}:U_\vtg(n)_\mbc\to\sS_\vtg(n,r)_\mbc$ (cf. Remark
\ref{NonrootOfUnity4}) gives an inflation functor
$$\Upsilon:\sS_\vtg(n,r)_\mbc\hmod\lra U_\vtg(n)_\mbc\hmod.$$
Thus, we obtain the functor
$$\Upsilon\sfF=\Upsilon\circ\sfF:\sH_\vtg(r)_\mbc\hmod\lra U_\vtg(n)_\mbc\hmod.$$

\begin{Prop}\label{Chari-Presslry} There is a natural isomorphism of functors $\vi:\sF\rightsquigarrow\Upsilon\sfF$.
In other words, for any $\sH_\vtg(r)_F$-modules $M,M'$ and
homomorphism $f:M\to M'$, there is a $U_\vtg(n)_\mbc$-module
isomorphism
$$\vi_M:\sF(M)\overset\sim\lra \Upsilon\sfF(M),\quad w\ot m\longmapsto w\ot_\vtg m,$$
 for all $w\in\Og_{n,\mbc}^{\ot r}$ and $m\in
M$ such that $\Upsilon\sfF(f)\vi_M=\vi_{M'}\sF(f)$. Here
$\otimes=\otimes_{\sH(r)_\mbc}$ and
$\otimes_\vtg=\otimes_{\sH_\vtg(r)_\mbc}$.
\end{Prop}

\begin{proof} We first recall from Remark \ref{NonrootOfUnity4} that the representation
$U_\vtg(n)_\mbc\to\End_\mbc(\Og_\mbc^{\ot r})$ factors through the
algebra homomorphism
$\xi_{r,\mbc}:U_\vtg(n)_\mbc\to\sS_\vtg(n,r)_\mbc$. Thus, by Lemma
\ref{Lem1 for category equivalent}, there is a $U(n)_\mbc$-module
isomorphism
\begin{equation*}
\begin{split}
\vi=\vi_M:\sF(M)&\lra \Upsilon\sfF(M) \\
w\ot m&\longmapsto w\ot_\vtg m\quad(\text{for
all}\;w\in\Og_{n,\mbc}^{\ot r},\ m\in M).
\end{split}
\end{equation*}
It remains to prove that $\vi(E_0(\og_\bfi\ot
m))=E_0(\og_\bfi\ot_\vtg m)$ and $\vi(F_0(\og_\bfi\ot
m))=F_0(\og_\bfi\ot_\vtg m)$ for all $\bfi\in I(n,r)$ and $m\in M$.

By the definition above, for $\bfi\in I(n,r)$ and $m\in M$,
\begin{equation}\label{eq1 for category equiv}
\begin{split}
\vi(E_0(\og_\bfi\ot m))&=\vi(\sum_{j=1}^r(Y_j^+\og_\bfi)\ot(X_jm))\\
&=\sum_{j=1}^rY_j^+\og_\bfi\ot_\vtg
X_jm=\sum_{j=1}^r(Y_j^+\og_\bfi) X_j\ot_\vtg m,\\
\end{split}
\end{equation}
where, by \eqref{Etheta} and \eqref{afH action},
\begin{equation}\label{eq2 for category equiv}
\begin{split}
(Y_j^+\og_\bfi) X_j&= \up^{\sum_{j+1\leq k\leq
r}(\dt_{i_k,n}-\dt_{i_k,1})}\dt_{i_j,1}(\og_{i_1}
 \cdots\og_{i_{j-1}}
 \og_n \og_{i_{j+1}} \cdots \og_{i_r})X_j\\
&=\up^{\sum_{j+1\leq k\leq
r}(\dt_{i_k,n}-\dt_{i_k,1})}\dt_{i_j,1}\og_{i_1} \cdots\og_{i_{j-1}}
\og_0 \og_{i_{j+1}} \cdots \og_{i_r}.
\end{split}
\end{equation}
On the other hand, $E_0$ acts on $\Og_\mbc^{\ot r}$ via the
comultiplication $\Dt$ as described in Corollary \ref{coalgebra
structure}. Since $$\Dt^{(r-1)}(E_0)=\sum_{j=1}^r1^{\ot j-1}\ot
E_0\ot\ti K_0^{\ot (r-j)},$$ this together with
\eqref{QGKMAlg-action} gives
\begin{equation*}
\begin{split}
E_0\og_\bfi&=\sum_{j=1}^r\up^{\sum_{j+1\leq k\leq
r}(\dt_{i_k,n}-\dt_{i_k,1})}\dt_{i_j,1}\og_{i_1} \cdots\og_{i_{j-1}}
\og_0 \og_{i_{j+1}} \cdots \og_{i_r}\\
&=\sum_{j=1}^r(Y_j^+\og_\bfi) X_j.
\end{split}
\end{equation*}
(Note that, for $1\leq i,j\leq n$, the values $\dt_{i,j}$ is
unchanged regardless viewing $i,j$ as elements in $\mbz$ or in
$I=\mbz/n\mbz$.)
 Hence, by \eqref{eq1 for category equiv},
$$E_0\vi(\og_\bfi\ot m)=E_0\og_\bfi\ot_\vtg m=\vi(E_0(\og_\bfi\ot m)).$$

Similarly, we can prove that for $\bfi\in I(n,r)$ and $m\in M$,
\begin{equation*}
\begin{split}
F_0\vi(\og_\bfi\ot m)&=\sum_{j=1}^r\up^{\sum_{1\leq k\leq
j-1}(\dt_{i_k,1}-\dt_{i_k,n})}\dt_{i_j,n}\og_{i_1}
\cdots\og_{i_{j-1}}
\og_{n+1} \og_{i_{j+1}} \cdots \og_{i_r}\ot_\vtg m\\
&=\vi(F_0(\og_\bfi\ot m)),
\end{split}
\end{equation*}
as required. The commutativity relation
$\Upsilon\sfF(f)\circ\vi_M=\vi_{M'}\circ\sF(f)$ is clear.
\end{proof}

We now recall another theorem of Chari--Pressley which classifies
finite dimensional simple $\afUslC$-modules and its generalization
by Frenkel--Mukhin to the classification of finite dimensional
irreducible polynomial representations of $\afUglC$.

For $1\leq j\leq n-1$ and $s\in\mbz$, define the elements $\ms
P_{j,s}\in\afUslC$ through the generating functions
\begin{equation*}
\begin{split}
& \ms P_j^\pm(u):=\exp\bigg(-\sum_{t\geq
1}\frac{1}{[t]_\ttv}\tth_{j,\pm t} (\ttv u)^{\pm
t}\bigg)=\sum_{s\geq 0}\ms P_{j,\pm s} u^{\pm
s}\in\afUslC[[u,u^{-1}]].
\end{split}
\end{equation*}
 Here we may view $\ms P_j^\pm(u)$ as formal power series (at 0 or $\infty$) with coefficients in $\afUslC$.
Note that these elements are related to the elements
$\phi_{j,s}^\pm$ used in the definition of $\afUslC$ (see Definition
\ref{QLA}(2)) through the formula
\begin{equation}\label{PhiP relation}
\Phi_j^\pm(u)=\ti\ttk_j^{\pm1}\frac{\ms P_j^\pm(\ttv^{-2}u)}{\ms
P_j^\pm(u)}.
\end{equation}
(For a proof, see, e.g., \cite[p.~291]{CP97}.)\footnote{If $\mathcal
P_j^\pm(u)$ denote the elements defined in line 2 above
\cite[(4.1)]{FM}, then $\ms P_j^\pm(u)=\mathcal P_j^\pm(uz)$. We
have corrected a typo by removing the $+$ sign in exp$(\mp\sum
...)$. If the $+$ is there, then the $-$ case of
$\Phi^\pm_i(u)=k_i^{\pm1}\mathcal P_j^\pm(uz^{-1})/\mathcal
P_j^\pm(uz)$ as given in \cite[(4.1)]{FM} is no longer true. Also,
\eqref{PhiP relation} shows that $+$ sign is unnecessary.}

For any polynomial $f(u)=\prod_{1\leq i\leq m}(1-a_iu)\in\mbc[u]$
with constant term $1$ and $a_i\in\mbc^*$, define $f^\pm(u)$ (so
that $f^+(u)=f(u)$) as follows:
\begin{equation}\label{f^pm(u)}
f^\pm(u)=\prod_{1\leq i\leq m}(1-a_i^{\pm1}u^{\pm1}).
\end{equation}
Note that $f^-(u)=(\prod_{i=1}^m(-a_iu)^{-1})f^+(u)$.

Let $V$ be a finite dimensional representation of $\afUslC$ of type
1. Then $V=\oplus_{\la\in\mbz^{n-1}}V_\la$ and, since all $\ms
P_{i,s}$ commute with the $\ti\ttk_j$,  each $V_\la$ is a direct sum
of generalized eigenspaces of the form
\begin{equation}\label{geigenspace}
V_{\la,\gamma}=\{x\in V_\la\mid (\ms
P_{i,s}-\gamma_{i,s})^px=0\text{ for some $p$}\, (1\leq i\leq
n-1,s\in\mbz)\},
\end{equation}
 where $\gamma=(\gamma_{i,s})$ with $\gamma_{i,s}\in\mbc$. If we put $\Gamma_j^\pm(u)=\sum_{s\geq 0}\gamma_{j,\pm s} u^{\pm
s}$, then by \cite[Prop.~1]{FR2}, there exist polynomials
$f_{i}(u)=\prod_{1\leq j\leq m_{i}}(1-a_{j,i}u)$ and
$g_{i}(u)=\prod_{1\leq j\leq n_{i}}(1-b_{j,i}u)$ in $\mbc[u]$, such
that
\begin{equation}\label{poly Gamma}
\Gamma_i^\pm(u)=\frac{f_{i}^\pm(u)}{g_{i}^\pm(u)}\quad\text{ and
}\quad \la_i=m_i-n_i.
\end{equation}

Following \cite[12.2.4]{CPbk}, a nonzero ($\mu$-weight) vector $w\in V$ is called a {\it pseudo-highest weight vector},
\index{pseudo-highest weight vector} if there exist some
$P_{j,s}\in\mbc$ such that
$$
\ttx_{j,s}^+w=0,\quad\ms P_{j,s} w=P_{j,s} w,\quad
\ti\ttk_jw=\ttv^{\mu_j}w
$$
for all $1\leq j\leq n-1$ and $s\in\mbz$. The module $V$ is called a
{\it pseudo-highest weight module}\footnote{There are simple highest weight integrable modules $\La_\la$
considered in \cite[6.2.3]{Lu93} (defined in \cite[3.5.6]{Lu93}) which, in general, are infinite dimensional.}
\index{pseudo-highest weight module} if $V=\afUslC w$ for some pseudo-highest weight
vector $w$.

For notational simplicity, the expressions $\ms P_{j,s} w=P_{j,s} w$
for all $s\in\mbz$ will be written by a single expression
$$\ms P_{j}^\pm(u) w=P_{j}^\pm(u) w\in V[[u,u^{-1}]],$$
where
\begin{equation}\label{Pis}
P_j^\pm(u)=\sum_{s\geq 0}P_{j,\pm s} u^{\pm s},
\end{equation}



 Let $\sP(n)$ be the set
of $(n-1)$-tuples of polynomials with constant terms $1$. For
$\bfP=(P_1(u),\ldots,P_{n-1}(u))\in\sP(n)$, define $P_{j,s}\in\mbc$,
for $1\leq j\leq n-1$ and $s\in\mbz$, as in $P_j^\pm(u)=\sum_{s\geq
0}P_{j,\pm s} u^{\pm s}$, where $P_j^\pm(u)$ is defined by
\eqref{f^pm(u)}.

Let $\Icp(\bfP)$ be the left ideal of $\afUslC$ generated by
$\ttx_{j,s}^+ ,\quad\ms P_{j,s}-P_{j,s},\quad
\ti\ttk_j-\ttv^{\mu_j}$, for $1\leq j\leq n-1$ and $s\in\mbz$, where
$\mu_j=\mathrm{deg}P_j(u)$, and define ``Verma module''
$$\Mcp(\bfP)=\afUslC/\Icp(\bfP).$$
Then $\Mcp(\bfP)$ has a unique simple quotient, denoted by
$\Lcp(\bfP)$. The polynomials $P_i(u)$ are called {\it Drinfeld
polynomials} associated with $\Lcp(\bfP)$. \index{Drinfeld
polynomials}

The following result is due to Chari--Pressley (see
\cite{CP91,CPbk} or \cite[pp.7-8]{CP95}).

\begin{Thm}
The modules $\Lcp(\bfP)$ with $\bfP\in\sP(n)$ are all nonisomorphic
finite dimensional simple $\afUslC$-modules of  type $1$.
\end{Thm}

Let $\afUglC$ be the algebra over $\mathbb C$ generated by
$\ttx^\pm_{i,s}$ ($1\leq i<n$, $s\in\mbz$), $\ttk_i^{\pm1}$ and
$\ttg_{i,t}$ ($1\leq i\leq n$, $t\in\mbz\backslash\{0\}$) with
relations similar to (QLA1)--(QLA7) as defined in \S2.3.
Chari--Pressley's classification is easily generalized to simple
$\afUglC$-modules as follows; see \cite{FM}.

For $1\leq i\leq n$ and $s\in\mbz$, define the elements $\ms
Q_{i,s}\in\afUglC$ through the generating functions
\begin{equation*}
\begin{split}
&\quad\qquad\ms Q_i^\pm(u):=\exp\bigg(-\sum_{t\geq
1}\frac{1}{[t]_\ttv}g_{i,\pm t} (\ttv u)^{\pm t}\bigg)=\sum_{s\geq
0}\ms Q_{i,\pm s} u^{\pm s}\in\afUglC[[u,u^{-1}]].
\end{split}
\end{equation*}
Since
$\tth_{i,m}=\ttv^{(i-1)s}\ttg_{i,m}-\ttv^{(i+1)s}\ttg_{i+1,m}$, it
follows from the definitions that
\begin{equation*}
\begin{split}
& \ms P_j^\pm(u)=\frac{\ms Q_j^\pm(u\ttv^{j-1})}{\ms
Q_{j+1}^\pm(u\ttv^{j+1})},
\end{split}
\end{equation*}
for $1\leq j\leq n-1$.

As in the case of $\afUslC$, if $V$ is a representation of $\afUglC$
then a nonzero ($\la$-weight) vector $w\in V$ is called a {\it pseudo-highest weight
vector},\index{pseudo-highest weight vector} if there exist
some $Q_{i,s}\in\mbc$ such that
\begin{equation}\label{HWvector}
\ttx_{j,s}^+w=0,\quad\ms Q_{i,s}w=Q_{i,s}w,\quad
\ttk_iw=\ttv^{\la_i}w
\end{equation}
for all $1\leq i\leq n$ and $1\leq j\leq n-1$ and $s\in\mbz$. The
module $V$ is called a {\it pseudo-highest weight module}\index{pseudo-highest weight module} if $V=\afUglC w$
for some pseudo-highest weight vector $w$. Associated to the sequence
$(Q_{i,s})_{s\in\mbz}$, defined two formal power series by
\begin{equation}\label{power series Q}
Q_i^\pm(u)=\sum_{s\geq 0}Q_{i,\pm s}u^{\pm s}.
\end{equation}
We also write the short form $\ms Q_i^\pm(u)w=Q_i^\pm(u)w$ for the
relations $\ms Q_{i,s}w=Q_{i,s}w\,(s\in\mbz)$.

A finite dimensional $\afUglC$-module $V$ is called a {\it
polynomial representation}\index{polynomial representation} if it is
of type 1 and, for every weight $\la=(\la_1,\ldots,\la_n)\in\mbz^n$
of $V$, the formal power series $\Gamma_i^\pm(u)$ associated to the
eigenvalues $(\gamma_{i,s})_{s\in\mbz}$ defining the generalized
eigenspaces $V_{\la,\gamma}$ as given in \eqref{geigenspace}, where
$\ms P_{i,s}$ is replaced by $\ms Q_{i,s}$ and $n-1$ by $n$, are
polynomials in $u^\pm$ of degree $\la_i$ so that the zeroes of the
functions $\Gamma_i^+(u)$ and $\Gamma_i^-(u)$ are the same.

Following \cite{FM}, an $n$-tuple of polynomials
$\bfQ=(Q_1(u),\ldots,Q_n(u))$ with constant terms $1$ is called {\it
dominant} if for $1\leq i\leq n-1$ the ratio
$Q_i(\ttv^{i-1}u)/Q_{i+1}(\ttv^{i+1}u)$ is a polynomial. Let
$\sQ(n)$ be the set of dominant $n$-tuples of polynomials.

For $\bfQ=(Q_1(u),\ldots,Q_{n}(u))\in\sQ(n)$, define
$Q_{i,s}\in\mbc$, for $1\leq i\leq n$ and $s\in\mbz$, by the
following formula
$$Q_i^\pm(u)=\sum_{s\geq 0}Q_{i,\pm s}u^{\pm s},$$
where $Q_i^\pm(u)$ is defined using \eqref{f^pm(u)}. Let $I(\bfQ)$
be the left ideal of $\afUglC$ generated by $\ttx_{j,s}^+ ,\quad\ms
Q_{i,s}-Q_{i,s},\quad \ttk_i-\ttv^{\la_i}$ for $1\leq j\leq n-1$,
$1\leq i\leq n$ and $s\in\mbz$, where $\la_i=\mathrm{deg}Q_i(u)$,
and define
$$M(\bfQ)=\afUglC/I(\bfQ).$$
Then $M(\bfQ)$ has a unique simple quotient, denoted by $L(\bfQ)$.
The polynomials $Q_i(u)$ are called {\it Drinfeld
polynomials}\index{Drinfeld polynomials} associated with $L(\bfQ)$.

\begin{Thm}[\cite{FM}]\label{classification of simple afUglC-modules}
The $\afUglC$-modules $L(\bfQ)$ with $\bfQ\in\sQ(n)$ are all
nonisomorphic finite dimensional simple polynomial representations
of $\afUglC$. Moreover,
$$L(\bfQ)|_{\afUslC}\cong \bar L(\bfP)$$ where
$\bfP=(P_1(u),\ldots,P_{n-1}(u))$ with
$P_i(u)=Q_i(\ttv^{i-1}u)/Q_{i+1}(\ttv^{i+1}u)$.
\end{Thm}

\section{Classification of simple $\sS_\vtg(n,r)_\mbc$-modules: the upwards approach}

We first recall the classification of irreducible representations of
$\afHrC$ or equivalently, simple $\afHrC$-modules.

Let $\mbc^*=\mbc\backslash\{0\}$. For
$\bfa=(a_1,\ldots,a_r)\in(\mbc^*)^r$, let $M_{\bfa}=\afHrC/J_\bfa$,
where $J_\bfa$ is the left ideal of $\afHrC$ generated by $X_j-a_j$
for $1\leq j\leq r$. Then $M_{\bfa}$ is an $\afHrC$-module of
dimension $r!$ and, regarded as an $\sH(r)_\mbc$-module by
restriction, $M_{\bfa}$ is isomorphic to the regular representation
of $\sH(r)_\mbc$.

Applying the functor $\sfF$ to $M_{\bfa}$ yields an $\afSrC$-module
$\OgC^{\ot r}\ot_{\afHrC}M_\bfa$ and hence, a $\DC(n)$-module
inflated by the homomorphism $\xi_{r,\mbc}$ given in
\eqref{xi_{r,z}}. The following result, which will be used in the
next section, tells how the central generators $\sfz_t^{\pm}$ of
$\DC(n)$ defined in \eqref{expression sfx sfy} act on this module.
\begin{Lem}\label{action1 central elts}
Let $\bfa=(a_1,\ldots,a_r)\in(\mbc^*)^r$ and $w\in\OgC^{\ot
r}\ot_{\afHrC}M_\bfa$. Then we have $\sfz_t^{\pm}w=\sum_{1\leq s\leq
r}a_s^{\pm t}w$ for $t\geq 1$.
\end{Lem}
\begin{proof}
We may assume $w=\og_\bfi\ot \bar h$ for some $\bfi\in\afInr$ and
$\bar h\in M_\bfa$. Since for each $t\geq 1$, $\sum_{1\leq s\leq
r}X_s^{\pm t}$ is a central element in $\afHrC$, it follows from
\eqref{action central elts} that
$$
\sfz_t^{\pm}w=\og_\bfi \sum_{1\leq s\leq r}X_s^{\pm t}\ot\bar h
=\og_\bfi \ot h \sum_{1\leq s\leq r}X_s^{\pm t}\bar 1=\sum_{1\leq
s\leq r}a_s^{\pm t}\og_\bfi\ot h\bar 1=\sum_{1\leq s\leq r}a_s^{\pm
t}w.
$$
\end{proof}

A {\it segment} $\sfs$ with center $a\in\mbc^*$ is by definition an
ordered sequence
$$\sfs=(a\ttv^{-k+1},a\ttv^{-k+3},\ldots,a\ttv^{k-1})\in(\mbc^*)^k.$$
 Here $k$ is called the length
of the segment, denoted by $|\sfs|$. If
$\bfs=\{\sfs_1,\ldots,\sfs_p\}$ is an unordered collection of
segments, define $\wp(\bfs)$ to be the partition associated with the
sequence $(|\sfs_1|,\ldots,|\sfs_p|)$. That is,
$\wp(\bfs)=(|\sfs_{i_1}|,\ldots,|\sfs_{i_p}|)$ with
$|\sfs_{i_1}|\geq\cdots\geq|\sfs_{i_p}|$, where
$|\sfs_{i_1}|,\ldots,|\sfs_{i_p}|$ is a permutation of
$|\sfs_1|,\ldots,|\sfs_p|$. We also call $|\bfs|:=|\wp(\bfs)|$ the
length of $\bfs$.

Let $\mathscr S_r$ be the set of unordered collections of segments
$\bfs$ with $|\bfs|=r$. Then $\mathscr
S_r=\cup_{\la\in\Lambda^+(r)}\mathscr S_{r,\la}$, where $\mathscr
S_{r,\la}=\{\bfs\in\mathscr S_r\mid\wp(\bfs)=\la\}$.

For $\bfs=\{\sfs_1,\ldots,\sfs_p\}\in\mathscr S_{r,\la}$, let
$$\bfa(\bfs)=(\sfs_1,\ldots,\sfs_p)\in(\mbc^*)^r$$
be the $r$-tuple obtained by juxtaposing the segments in $\bfs$.
Then the element
$$y_{\la}=\sum_{w\in\fS_\la}(-\ttv^2)^{-\ell(w)}T_w\in\afHrC$$
generates the submodule $\afHrC \bar y_\la$ of $M_{\bfa(\bfs)}$
which, as an $\sH(r)_\mbc$-module, is isomorphic to $\sH(r)_\mbc
y_\la$. However,
\begin{equation}\label{signed permutation module}
\sH(r)_\mbc y_\la\cong S_\la\oplus(\bigoplus_{\mu\vdash r,
\mu\unrhd\la}m_{\mu,\la}S_\mu),
\end{equation}
where $S_\nu$ is the left cell module defined by the
Kazhdan--Lusztig's C-basis \cite{KL79} associated with the left cell
containing $w_{0,\la}$.

Let $V_\bfs$ be the unique composition factor of the $\afHrC$-module
$\afHrC \bar y_\la$ such that the multiplicity of $S_\la$ in
$V_\bfs$ as an $\sH(r)_\mbc$-module is nonzero.

We now can state the following classification theorem due to
Zelevinsky \cite{Zelevinsky} and Rogawski \cite{Rogawski}. The
construction above follows \cite{Rogawski}.

\begin{Thm}\label{classification irr affine Hecke algebra}
Let $\text{\rm Irr}(\afHrC)$ be the set of isoclasses of all simple
$\afHrC$-modules. Then the correspondence $\bfs\mapsto V_\bfs$
defines a bijection from $\mathscr S_r$ to $\text{\rm Irr}(\afHrC)$.
\end{Thm}

We record the following general result.

\begin{Lem}\label{SH module}
Let $S$ and $H$ be two algebras over a field $\field$. If $V$ is an
$S$-$H$-bimodule,  $M$ is a left $H$-module, and $e\in S$ is an
idempotent element, then $eV$ is an $eSe$-$H$-bimodule and there is
an $eSe$-module isomorphism
$$
(eV)\ot_H M\cong e(V\ot_H M)\quad(ew\ot m\map e(w\ot m)).
$$
\end{Lem}
\begin{proof}
There is a natural map $\al:(eV)\ot_H M\ra
 V\ot_HM$ defined by sending $ew\ot m$ to $ew\ot m$. The map $\al$ induces a
surjective map $\bar\al:(eV)\ot_H M\ra e(V\ot_HM)$. On the other
hand, there is a surjective right $H$-module homomorphism from $V$
to $eV$ defined by sending $w$ to $ew$ for $w\in V$. This map
induces a natural surjective map $\bt:V\ot_H M\ra (eV)\ot_HM$
defined by sending $w\ot m$ to $ew\ot m$ for $w\in V$ and $m\in M$.
By restriction, we get a map $\bar\beta:e(V\ot_HM)\ra (eV)\ot_H M$.
Since $\bar\al\bar\bt=\id$ and $\bar\bt\bar\al=\id$, the assertion
follows.
\end{proof}

Let
$$\mathscr S_{r}^{(n)}=\{\bfs\in\mathscr S_r\mid \OgC^{\ot r}\ot_{\afHrC}V_\bfs\neq0\}.$$
Then $\mathscr S_{r}^{(n)}=\mathscr S_{r}$ for all $r\leq n$. We
have the following classification theorem.

\begin{Thm}\label{weakly category equiv}  The set
 $$\{\OgC^{\ot r}\ot_{\afHrC} V_\bfs\mid\bfs\in\mathscr S_{r}^{(n)}\}$$
is a complete set of nonisomorphic simple
$\sS_\vtg(n,r)_\mbc$-modules. In particular, if $n\geq r$, then the
set $\{\OgC^{\ot r}\ot_{\afHrC} V_\bfs\mid\bfs\in\mathscr S_r\}$ is
a complete set of nonisomorphic simple $\sS_\vtg(n,r)_\mbc$-modules.
\end{Thm}
\begin{proof}
If $n\geq r$, then the assertion follows from Theorem \ref{category
equiv}. Now we assume $n<r$. As in the proof of Theorem \ref{fd
simple modules}, we choose $N=r$ and regard $\afLanr$ as a subset of
$\afLarr$ via the map $\mu\mapsto\ti\mu$ given in \eqref{tilde-map}.
Then $\afSrC$ identifies a centralizer subalgebra
$e\sS_\vtg(r,r)_\mbc e$ of $\sS_\vtg(r,r)_\mbc$, where
 $e=\sum_{\la\in\afLanr}[\diag(\ti\la)]=\sum_{\la\in\afLanr}\phi_{\ti\la,\ti\la}^1$.

Recall from Proposition \ref{bimodule-isom} that the tensor space
$\Og^{\ot r}$ identifies $\sT_\vtg(n,r)$. Thus, $\sT_\vtg(r,r)_\mbc$
is the tensor space on which $\sS_\vtg(r,r)_\mbc$ acts. Hence, the
set
$$\{\sT_\vtg(r,r)_\mbc\ot_{\afHrC} V_\bfs\mid\bfs\in\mathscr S_r\}$$
forms a complete set of nonisomorphic simple
$\sS_\vtg(r,r)_\mbc$-modules. By \cite[6.2(g)]{Gr80} or a direct
argument, the set
$$\{e(\sT_\vtg(r,r)_\mbc\ot_{\afHrC} V_\bfs)\mid\bfs\in\mathscr S_r\}\backslash\{0\}$$
forms a complete set of nonisomorphic simple
$\sS_\vtg(n,r)_\mbc$-modules. Since
$$e\sT_\vtg(r,r)_\mbc\cong e\biggl(\bop_{\mu\in\afLa(r,r)}
x_\mu\afHrC\biggr)\cong\bop_{\mu\in\afLanr}x_{\ti
\mu}\afHrC\cong\bop_{\mu\in\afLanr}x_{ \mu}\afHrC\cong\OgC^{\ot
r},$$
 by Lemma \ref{SH module}, there is an $\afSrC$-module isomorphism
 $$
 e(\sT_\vtg(r,r)_\mbc\ot_{\afHrC}V_\bfs) \cong (e\sT_\vtg(r,r)_\mbc)\ot_{\afHrC}V_\bfs\cong \OgC^{\ot r}\ot_{\afHrC}V_\bfs,
$$
proving the $n<r$ case.
\end{proof}

We end this section with a discussion on how the index set $\mathscr
S_{r}^{(n)}$ of simple $\afSrC$-modules for $n<r$ is related to a
certain Branching Rule.

Since the parameter $z$ for the specialization
$\sZ\to\mbc,\up\mapsto z$ is not a root of unity, the Hecke algebra
$\sH(r)_\mbc$ is semisimple. Thus, as an $\sH(r)_\mbc$-module,
$V_\bfs$ is semisimple. By \eqref{signed permutation module}, we can
decompose the $\HrC$-module
\begin{equation}\label{Branching}
V_\bfs|_{\HrC}=\bop_{\mu(\bfs)\vdash
r,\mu(\bfs)\unrhd\wp(\bfs)}m_{\mu(\bfs)} S_{\mu(\bfs)}.
\end{equation}
Here, for $\la =\wp(\bfs)$, $m_\la=1$, and $m_{\mu(\bfs)}\leq
m_{\mu(\bfs),\la}$ (see \eqref{signed permutation module}).

\begin{Prop} Maintain the notation above. We have $\bfs\in\mathscr S_{r}^{(n)}$
if and only if $m_{\mu(\bfs)}\neq0$ for some partition
$\mu(\bfs)\in\La^+(n,r)$.
\end{Prop}
\begin{proof}
By Lemma \ref{Lem1 for category equivalent}, we have
$\sS(n,r)_\mbc$-module isomorphisms
$$\OgC^{\ot r}\ot_{\afHrC}V_\bfs\cong\OgnC^{\ot r}\ot_{\HrC}V_\bfs\cong
\bop_{\mu(\bfs)\vdash r,\mu(\bfs)\unrhd\wp(\bfs)}m_{\mu(\bfs)}
(\OgnC^{\ot r}\ot_{\HrC}S_{\mu(\bfs)}).$$ Since $\OgnC^{\ot
r}\ot_{\HrC}S_{\mu(\bfs)}\not=0$ if and only if
$\mu(\bfs)\in\La^+(n,r)$, it follows that $\OgC^{\ot
r}\ot_{\afHrC}S_{\mu(\bfs)}\not=0$ if and only if
$m_{\mu(\bfs)}\not=0$ for some partition $\mu(\bfs)\in\La^+(n,r)$.
\end{proof}

The upwards approach to the classification of simple
$\afSrC$-modules when $n<r$ depends on the description of the set
$\mathscr S_{r}^{(n)}$. From the proof above, one sees easily that
$$\bigcup_{\la\in\La^+(n,r)}\mathscr S_{r,\la}\subseteq \mathscr S_{r}^{(n)}.$$
The next example shows that $\mathscr S_{r}^{(n)}\neq \mathscr
S_{r}$ does occur.

\begin{Example} For any $a\in\mbc$, there is an algebra homomorphism
$ev_a:\afHr_\mbc\to\sH(r)_\mbc$, call the {\it evaluation map} (see
\cite[5.1]{CP}), such that
\begin{enumerate}
\item $ev_a(T_i)=T_i,\qquad1\leq i\leq r-1$;
\item $ev_a(X_j)=a\ttv^{-2(j-1)}T_{j-1}\cdots T_2T_1T_1T_2\cdots T_{j-1}, 1\leq j\leq r$.
\end{enumerate}
Thus, every simple $\sH(r)_\mbc$-module $S_\la$ is also a simple
$\afHrC$-module. Those $\bfs$ such that $V_\bfs$ is isomorphic to
$S_\la$ with $\la\not\in\La^+(n,r)$ are not in $\mathscr
S_{r}^{(n)}$.
\end{Example}


However, the proof above also shows that if we know explicitly the
decomposition in \eqref{Branching}, especially for those
$\bfs\in\mathscr S_r$ with $\wp(\bfs)$ having more than $n$ parts,
then we are able to determine $\mathscr S_{r}^{(n)}$. We will call
the description of the nonzero multiplicities in \eqref{Branching}
the {\it affine-to-finite Branching Rule} or simply the {\it affine
Branching Rule}.\index{affine Branching Rule}

\begin{Prob} \label{Prob-Branching-Rules} Describe the affine Branching Rule. In other words,
find a necessary and sufficient condition for the nonzero
multiplicities $m_{\mu(\bfs)}$ given in \eqref{Branching}.
\end{Prob}

We will use the downwards approach to complete the classification in
\S4.5. We will prove that each  simple $\afSrC$-module is an
irreducible polynomial representation of $\afUglC$ in Proposition
\ref{step1} and prove that each irreducible polynomial
representation of $\afUglC$ is a simple $\afSrC$-module for some $r$
in Proposition \ref{step2}. Combining Propositions \ref{step1} with
\ref{step2}, we can classify the simple modules for $\afSrC$.

\section{Identification of simple $\sS_\vtg(n,r)_\mbc$-modules: the $n>r$ case}


Recall the $\mathbb C$-algebra $\DC(n)$ defined in Remark
\ref{NonrootOfUnity2} and the $\mathbb C$-algebra isomorphism
$\sE_{\rm H,\mbc}:\DC(n)\to\afUglC$\index{$\sE_{\rm H,\mbc}$, isomorphism $\DC(n)\overset\sim\to\afUglC$}
 discussed in Theorem \ref{iso
afgln dHallr} and Remarks \ref{iso afgln dHallr over C}(2). By the
presentation given in Theorem \ref{presentation dHallAlg} it is easy
to see that there is an automorphism $g$ of $\DC(n)$ satisfying
$$\aligned
&g(K_i^{\pm1})=K_i^{\pm1},\;\qquad
g(u_i^\pm)=u_i^\pm,\;(1\leq i<n),\;\\
&g(u_n^\pm)=(-1)^nz^{\pm1}u_n^\pm,\,\quad g(\sfz^\pm_s)=-z^{\mp(n-1)s}\sfz^\pm_s\;(s\geq 1).\\
\endaligned$$
Combining the two gives the following.

\begin{Prop}\label{isom f}
 There is a $\mbc$-algebra isomorphism
 $$f=\sE_{\rm H,\mbc}\circ g:\DC(n)\lra
\afUglC$$ such that
$$\aligned
&K_i^{\pm1}\lm\ttk_i^{\pm1},\;\qquad
u_i^\pm\lm \ttx^\pm_{i,0}\;(1\leq i<n),\;\\
&u_n^\pm\lm(-1)^n\ttv^{\pm 1}\ep_n^\pm,\,\quad\sfz^\pm_s\lm
\ttv^{\mp(n-1)s}\frac{s}{[s]_\ttv}\th_{\pm
s}\;(s\geq 1).\\
\endaligned$$
\end{Prop}

Now the $\mbc$-algebra epimorphism
$\zrC:\DC(n)\to\afSrC$\index{$\zrC$, epimorphism $\DC(n)\to\afSrC$}
described in Corollary \ref{surjective-dHall-aff/C} (and Theorem
\ref{surjective-dHall-aff}) together with $f$ gives a $\mbc$-algebra
epimorphism \index{$\zrC'$, epimorphism $\afUglC\to\afSrC$}
\begin{equation}\label{zrCf}
\zrC':=\zrC\circ f^{-1}:\afUglC\lra\afSrC.
\end{equation}
Thus, every $\afSrC$-module will be inflated into a $\afUglC$-module
via this homomorphism. In particular, every simple $\afSrC$-module
given in Theorem \ref{weakly category equiv} is a simple
$\afUglC$-module. We now identify them for the $n>r$ case in terms
of the irreducible polynomial representations of $\afUglC$ described
in Theorem \ref{classification of simple afUglC-modules}.

Recall from \S4.2 that $\sQ(n)$ is the set of dominant $n$-tuples of
polynomials. For $r\geq 1$, let
$$\sQ(n)_r=\big\{\bfQ=(Q_1(u),\ldots,Q_n(u))\in\sQ(n)\mid r=\sum_{1\leq i\leq n}\mathrm{deg}\, Q_i(u)\big\}.$$

 For $\bfs=\{\sfs_1,\ldots,\sfs_p\}\in\ms S_r$ with
$$\sfs_i=(a_i\ttv^{-\mu_i+1},a_i\ttv^{-\mu_i+3},\ldots,a_i\ttv^{\mu_i-1})\in(\mbc^*)^{\mu_i},$$
let $Q_n(u)=1$ and for $1\leq i\leq n-1$, define
$$\aligned
P_{i}^\pm(u)&=\prod_{1\leq j\leq p\atop \mu_j=i}(1-a_j^{\pm 1}u^{\pm 1}),\\
Q_i^\pm(u)&=P_i^\pm(uz^{-i+1})P_{i+1}^\pm(uz^{-i+2})\cdots
P_{n-1}^\pm(uz^{n-2i}),\,\,\text{ and }\\
\bfQ_\bfs&=(Q_1(u),\ldots,Q_n(u)).
\endaligned$$
We now have the following identification theorem.

\begin{Thm}\label{n>r representation} Maintain the notation above and let $n>r$.
The map $\bfs\mapsto\bfQ_\bfs$ defines a bijection from $\mathscr
S_r$ to $\sQ(n)_r$, and induces $\afUglC$-module isomorphisms
$\OgC^{\ot r}\ot_{\afHrC}V_\bfs\cong L(\bfQ_\bfs)$ for all
$\bfs\in\ms S_r$. Hence, the set $\{L(\bfQ)\mid\bfQ\in\sQ(n)_r\}$
forms a complete set of nonisomorphic simple $\afSrC$-modules.
\end{Thm}

\begin{proof} By the algebra homomorphism $\xi_{r,\mbc}'$ given in \eqref{zrCf}, every $\afSrC$-module $M$
is regarded as a $\afUglC$-module.  Let $[M]$ denote the isoclass of
$M$. By Theorem \ref{weakly category equiv}, it suffices to prove
that
\begin{itemize}
\item[(1)] $\OgC^{\ot r}\ot_{\afHrC}V_\bfs\cong L(\bfQ_\bfs)$, and
 \item[(2)] $\{[L(\bfQ)]\mid\bfQ\in\sQ(n)_r\}=\{[\OgC^{\ot r}\ot_{\afHrC}V_\bfs]\mid\bfs\in\mathscr S_r\}.$
 \end{itemize}

We first prove (1). Let $\bfs=\{\sfs_1,\ldots,\sfs_p\}\in\ms S_r$ be
an unordered collection of segments with
$$\sfs_i=(a_i\ttv^{-\mu_i+1},a_i\ttv^{-\mu_i+3},\ldots,a_i\ttv^{\mu_i-1})\in(\mbc^*)^{\mu_i}.$$
Thus, $r=\sum_{1\leq i\leq p}\mu_i$. Let
$\bfa=\bfa(\bfs)=(\sfs_1,\ldots,\sfs_p)\in(\mbc^*)^r$ be the
sequence obtained by juxtaposing the segments in $\bfs$.  Then the
simple $\afSrC$-module $\OgC^{\ot r}\ot_{\afHrC}V_\bfs$ becomes a
simple $\afUglC$-module via \eqref{zrCf}. As a simple
$\afUslC$-module, this module is isomorphic to the Chari--Pressley
module $\sF(V_\bfs)$ by Proposition \ref{Chari-Presslry}. Applying
\cite[7.6]{CP} yields a $\afUslC$-module isomorphism $\OgC^{\ot
r}\ot_{\afHrC}V_\bfs\cong\bar L(\bfP)$, where
$\bfP=(P_1(u),\ldots,P_{n-1}(u))$ with
\begin{equation}\label{Drinfeld poly}
P_{i}^\pm(u)=\prod_{1\leq j\leq p\atop \mu_j=i}(1-a_j^{\pm 1}u^{\pm
1}),\quad 1\leq i\leq n-1.
\end{equation}
As a simple $\afUglC$-module, by \cite[Lem.~4.2]{FM}, if
$w_0\in\OgC^{\ot r}\ot_{\afHrC}V_\bfs$ is the pseudo-highest weight
vector of weight $\la=(\la_1,\ldots,\la_n)$, then $\la$ is a
partition of $r$, since $\la$ is also a highest weight as an
$\sS(n,r)_\mbc$-module, and there exist $Q_i^+(u)\in\mbc[[u]]$ and
$Q_i^-(u)\in\mbc[[u^{-1}]]$, $1\leq i\leq n$, such that
\begin{equation}\label{Q P}
\begin{split}
&\ms Q_i^\pm(u)w_0= Q_i^\pm (u)w_0,\quad \ms P_j^\pm(u)w_0=
P_j^\pm(u) w_0,\quad K_iw_0=\ttv^{\la_i}w_0\\
&\qquad\qquad\qquad
P_j^\pm(u)=\frac{Q_j^\pm(u\ttv^{j-1})}{Q_{j+1}^\pm(u\ttv^{j+1})},\,\,\text{
and }\,\, \text{deg}P_i(u)=\la_i-\la_{i+1}.
\end{split}
\end{equation}
We now prove that $Q_i^\pm (u)$ are polynomials of degree $\la_i$.

Since $f(\sfz^\pm_t)=\ttv^{\mp(n-1)t}\frac{t}{[t]}\th_{\pm t}$ for
all $t\geq 1$ as in Proposition \ref{isom f}, it follows from
\eqref{defn of theta_s} and Lemma \ref{action1 central elts} that
$$
\frac{t\ttv^{\pm t}}{[t]_\ttv}\sum_{1\leq i\leq n}g_{i,\pm t}w_0
=\sfz_t^\pm w_0=\sum_{1\leq i\leq p\atop 1\leq k\leq\mu_i}
(a_i\ttv^{-\mu_i+2k-1})^{\pm t}w_0.
$$
 Thus,
\begin{equation*}
\begin{split}
\prod_{1\leq i\leq n} Q_i^\pm(u)w_0=\prod_{1\leq i\leq n}\ms
Q_i^\pm(u)w_0&=\exp\bigg(-\sum_{t\geq
1}\frac{1}{[t]_\ttv}\bigg(\sum_{1\leq i\leq n}g_{i,\pm t}\bigg)
(u\ttv)^{\pm t}\bigg)w_0\\
&=\exp\bigg(-\sum_{1\leq i\leq p\atop 1\leq k\leq\mu_i}\sum_{t\geq
1}\frac1t(a_iu\ttv^{2k-1-\mu_i})^{\pm t}\bigg)w_0\\
&=\prod_{1\leq i\leq p\atop 1\leq k\leq\mu_i}\exp\bigg(-\sum_{t\geq
1}\frac1t(a_iu\ttv^{2k-1-\mu_i})^{\pm t}\bigg)w_0\\
&=\prod_{1\leq i\leq p\atop 1\leq k\leq\mu_i}\bigg(1-\big(
a_iu\ttv^{2k-1-\mu_i}\big)^{\pm 1}\bigg)w_0.
\end{split}
\end{equation*}
as $-\sum_{t\geq 1}\frac1t(a_iu\ttv^{2k-1-\mu_i})^{\pm
t}=\text{ln}\big(1-a_iu\ttv^{2k-1-\mu_i}\big)$. Hence,
\begin{equation}\label{one hand}
\prod_{1\leq i\leq n} Q_i^\pm(u)=\prod_{1\leq i\leq p\atop 1\leq
k\leq\mu_i}\bigg(1-\big( a_iu\ttv^{2k-1-\mu_i}\big)^{\pm 1}\bigg).
\end{equation}

On the other hand, by \eqref{Q P},
\begin{equation}\label{polynomial Qi}
Q_i^\pm(u)=P_i^\pm(uz^{-i+1})P_{i+1}^\pm(uz^{-i+2})\cdots
P_{n-1}^\pm(uz^{n-2i})Q_n^\pm(uz^{2(n-i)})
\end{equation}
for all $1\leq i\leq n$, and by \eqref{Drinfeld poly},
\begin{equation*}
\begin{split}
&\quad P_i^\pm(uz^{-i+1})P_i^\pm(uz^{-i+3})\cdots P_i^\pm(uz^{i-1}) \\
&=\prod_{1\leq j\leq p\atop \mu_j=i}(1-(a_juz^{-\mu_j+1})^{\pm 1})
(1-(a_juz^{-\mu_j+3})^{\pm 1})\cdots(1-(a_juz^{\mu_j-1})^{\pm 1}) \\
&=\prod_{1\leq j\leq p,\,\mu_j=i\atop 1\leq
k\leq\mu_j}(1-(a_juz^{2k-1-\mu_j})^{\pm 1}).
\end{split}
\end{equation*}
Thus,
\begin{equation*}
\prod_{1\leq i\leq
n-1}\bigg(P_i^\pm(uz^{-i+1})P_i^\pm(uz^{-i+3})\cdots
P_i^\pm(uz^{i-1})\bigg)=\prod_{1\leq j\leq p\atop 1\leq
k\leq\mu_j}\bigg(1-\big( a_ju\ttv^{2k-1-\mu_j}\big)^{\pm 1}\bigg)
\end{equation*}
Hence,
\begin{equation*}
\begin{split}
\prod_{1\leq i\leq n} Q_i^\pm(u) &=\prod_{1\leq i\leq n-1}\bigg( P_i^\pm(uz^{-i+1})P_i^\pm(uz^{-i+3})\cdots P_i^\pm(uz^{i-1})\bigg)\prod_{0\leq l\leq n-1}Q_n^\pm(u\ttv^{2l})\\
&=\prod_{1\leq i\leq p\atop 1\leq k\leq\mu_i}\bigg(1-\big(
a_iu\ttv^{2k-1-\mu_i}\big)^{\pm
1}\bigg)\prod_{0\leq l\leq n-1}Q_n^\pm(u\ttv^{2l}).\\
\end{split}
\end{equation*}
Combining this with \eqref{one hand} yields
$$
\prod_{0\leq l\leq n-1}Q_n^\pm(u\ttv^{2l})=1.
$$
So we have
$$
\exp\bigg(-\sum_{t\geq 1}\frac{1}{[t]_\ttv}g_{n,\pm
t}\bigg(\sum_{0\leq l\leq n-1}\ttv^{\pm 2lt}(u\ttv)^{\pm
t}\bigg)\bigg)w_0=\prod_{0\leq l\leq n-1}\ms
Q_n^\pm(u\ttv^{2l})w_0=w_0.
$$
It follows that
$$-\sum_{t\geq 1}\frac{1}{[t]_\ttv}g_{n,\pm
t}\bigg(\sum_{0\leq l\leq n-1}\ttv^{\pm 2lt}(u\ttv)^{\pm
t}\bigg)w_0=0.$$ This forces $g_{n,\pm t}w_0=0$ for all $t\geq 1.$
Consequently, $$Q_n^\pm(u)w_0=\ms Q_n^\pm(u)w_0=w_0,$$ and hence,
$Q_n^\pm(u)=1$. We conclude by \eqref{polynomial Qi} that all
$Q_i^{\pm}(u)$ are polynomials with constant term 1 and
$\bfQ_\bfs=(Q_1(u),\ldots,Q_n(u))\in\sQ(n)$. Moreover,
 $$\sum_{1\leq i\leq n}{\mathrm{deg}\,Q_i(u)}=\sum_{1\leq j\leq
n-1}j\mathrm{deg}\,P_j(u)=\sum_{1\leq j\leq
p}\mu_j=r=\sum_{i=1}^n\la_i.$$
 This forces $\la_n=0$ and consequently, $\text{deg}\, Q_i(u)=\la_i$
for all $1\leq i\leq n$. Therefore, $\OgC^{\ot r}\ot_{\afHrC}V_\bfs$
is a simple polynomial representation of $\afUglC$ and
 $\OgC^{\ot r}\ot_{\afHrC}V_\bfs\cong
L(\bfQ_\bfs)$, proving (1).

We now prove (2). For any $\bfQ=(Q_1(u),\ldots,Q_n(u))\in\sQ(n)$
such that $r=\sum_{1\leq i\leq n}\la_i$ with $\la_i=\mathrm{deg}
Q_i(u)$, we now prove that $L(\bfQ)$ is a simple $\afSrC$-module.

Since the polynomials
$$
P_j (u) =\frac{Q_j (u\ttv^{j-1})} {Q_{j+1} (u\ttv^{j+1})}\,\,(1\leq
j\leq n-1).
$$
have constant term $1$ and ${\rm deg\,}
P_j(u)=\la_j-\la_{j+1}=:\nu_j$, it follows that $\la\in\La^+(n,r)$
is a partition with at most $n$ parts. So $n>r$ implies $\la_n=0$.
Moreover, we may write, for $1\leq i\leq n-1$,
$$P_i(u)=(1-a_{\nu_1 +\cdots+\nu_{i-1}+1}u)(1-a_{\nu_1 +\cdots+\nu_{i-1}+2}u)
\cdots(1-a_{\nu_1 +\cdots+\nu_{i-1}+\nu_i}u),$$ where
$a_j^{-1}\in\mbc$, $1\leq j\leq p=\sum_i\nu_i$, are the roots of
$P_i(u)$. Let $\bfs=\{\sfs_1,\ldots,\sfs_p\}$, where
$$\sfs_i=(a_i\ttv^{-\mu_i+1},a_i\ttv^{-\mu_i+3},\ldots,a_i\ttv^{\mu_i-1})$$
and $(\mu_1,\ldots,\mu_p)=(1^{\nu_1},\ldots,(n-1)^{\nu_{n-1}})$, and
let $\bfa=(\sfs_1,\ldots,\sfs_p)$. Since
$$\sum_{1\leq j\leq p}\mu_j=\sum_{1\leq i\leq n-1}i\nu_i=\sum_{1\leq i\leq n}\la_i=r,$$
we have $\bfa\in(\mbc^*)^r$. By the first part of the proof, we see
$\bfQ=\bfQ_\bfs$, and hence, $\OgC^{\ot r}\ot_{\afHrC}V_\bfs\cong
L(\bfQ)$ as $\afUglC$-modules. In other words, $L(\bfQ)$ is a simple
$\afSrC$-module.
\end{proof}

\section{Classification of simple $\afSrC$-modules: the downwards approach}

We now  complete the classification of simple $\afSrC$-modules by
removing the condition $n>r$ in Theorem \ref{n>r representation}. We
will continue to use the downwards approach with a strategy
different from that in the previous section. Throughout this
section, we will identify $\afUglC$ with $\DC(n)$ via the
isomorphism $f$ in Proposition \ref{isom f} and will regard every
$\afSrC$-module as a $\afUglC$-module via the algebra homomorphism
$\zrC:\DC(n)\to\afSrC$; see \eqref{zrCf}.

Consider the $\afUglC$-module via $\xi_{1,\mbc}'$
$$\OgC(a):=\sfF(M_a)=\OgC\ot_{\afHoC}M_a$$
for $a\in\mbc^*$. By Proposition \ref{Chari-Presslry},
$\OgC(a)\cong\OgnC\ot_{\HoC}M_a\cong\OgnC$ as
$\sS(n,1)_\mbc$-modules. Hence, $\dim\OgC(a)=n$.

By Theorem \ref{category equiv}, $\OgC(a)$ is a simple
$\afUglC$-module since $\dim_\mbc M_a=1$ and $M_a=V_\bfs$ with
$\bfs=(a)\in\ms S_1$. Since $n>1$, by Theorem \ref{n>r
representation}, we have
\begin{equation}\label{OgC(a)}
\OgC(a)\cong L(\bfQ)\text{ with $Q_1(u)=1-au$ and $Q_i(u)=1$ for
$2\leq i\leq n$}.
\end{equation}

The $\afUglC$-module $\OgC(a)$ is very useful and we will prove that
every finite dimensional simple $\afSrC$-module is a quotient module
of $\OgC(a_1)\ot_\mbc\cdots\ot_\mbc\OgC(a_r)$ for some
$\bfa\in(\mbc^*)^r$ in Corollary \ref{quotient module of tensor
space}.

Let $\ti\og_i=\og_i\ot\bar 1\in\OgC(a)$ for all $i\in\mbz$.
\begin{Lem}\label{key for step1}
For any $\bfa=(a_1,\ldots,a_r)\in(\mbc^*)^r$, there is a
$\afUglC$-module isomorphism
$$\vi:\OgC(a_1)\ot_\mbc
\cdots\ot_\mbc\OgC(a_r)\lra\OgC^{\ot r}\ot_{\afHrC}M_\bfa$$ defined
by sending $\ti\og_{\bfi}$ to $\og_{\bfi}\ot\bar 1$ for
$\bfi\in\afInr$, where
$\ti\og_\bfi=\ti\og_{i_1}\ot\cdots\ot\ti\og_{i_r}$. Moreover, as an
$\SrC$-module, $\OgC(a_1)\ot_\mbc \cdots\ot_\mbc\OgC(a_r)$ is
isomorphic to the finite tensor space $\OgnC^{\ot r}$  for all
$\bfa\in(\mbc^*)^r$.
\end{Lem}
\begin{proof} The set
$$\{\ti\og_i\mid 1\leq i\leq n\}$$
forms a basis of $\OgC(a)$. Hence, the set
$$\{\ti\og_\bfi\mid\bfi\in I(n,r)\}$$
forms a basis of $\OgC(a_1)\ot_\mbc \cdots\ot_\mbc\OgC(a_r)$.

Similarly by Proposition \ref{Chari-Presslry}, we have
\begin{equation}\label{restriction of tensor space}
\OgC^{\ot r}\ot_{\afHrC}M_\bfa\cong\OgnC^{\ot
r}\ot_{\HrC}M_\bfa\cong\OgnC^{\ot r}
\end{equation}
as $\SrC$-modules. So the set
$$\{\og_\bfi\ot\bar 1\mid\bfi\in I(n,r)\}$$
forms a basis of $\OgC^{\ot r}\ot_{\afHrC}M_\bfa$. Hence, there is a
linear isomorphism $$\vi:\OgC(a_1)\ot_\mbc
\cdots\ot_\mbc\OgC(a_r)\lra\OgC^{\ot r}\ot_{\afHrC}M_\bfa$$ defined
by sending $\ti\og_{\bfi}$ to $\og_{\bfi}\ot\bar 1$ for $\bfi\in
I(n,r)$.

Now we assume $\bfi\in\afInr$. We write $\bfi=\bfj+n\bft$ with
$\bfj\in I(n,r)$ and $\bft\in\mbz^r$. Then
\begin{equation*}
\begin{split}
\vi(\ti\og_\bfi)&=\vi((a_1^{-t_1}
\ti\og_{j_1})\ot\cdots\ot(a_r^{-t_r}\ti\og_{j_r}))\\
&=a_1^{-t_1}a_2^{-t_2}\cdots a_r^{-t_r}\og_{\bfj}\ot\bar 1\\
&=\og_{\bfj}X_1^{-t_1}X_2^{-t_2}\cdots X_r^{-t_r}\ot\bar 1\\
&=\og_{\bfi}\ot\bar 1.
\end{split}
\end{equation*}
It follows easily that $\vi$ is a $\afUglC$-module isomorphism. The
last assertion is clear from \eqref{restriction of tensor space}.
\end{proof}

\begin{Coro}\label{quotient module of tensor space}
Let $V$ be a finite dimensional simple $\afUglC$-module. Then the
following conditions are equivalent
\begin{itemize}
\item[(1)] $V$ can be regarded as an $\afSrC$-module via $\zrC$;

\item[(2)] $V$ is a quotient module of the $\afUglC$-module $\OgC(a_1)\ot_\mbc
\cdots\ot_\mbc\OgC(a_r)$ for some $\bfa\in(\mbc^*)^r$;

\item[(3)] $V$ is a quotient module of $\OgC^{\ot r}$;

\item[(4)] $V$ is a subquotient module of $\OgC^{\ot r}$.
\end{itemize}
\end{Coro}
\begin{proof}
By the lemma above, the map
$$\OgC^{\ot r}\lra \OgC(a_1)\ot_\mbc
\cdots\ot_\mbc\OgC(a_r), \quad \og_\bfi\longmapsto\ti\og_\bfi,$$ is
a $\afUglC$-module epimorphism (say, induced by the natural
$\afHrC$-module epimorphism $\afHrC\to M_\bfa$). Hence, (2) implies
(3). Certainly,  (3) implies (4). Since $\OgC^{\ot r}$ is an
$\afSrC$-module, (4) implies (1).

If $V$ can be regarded as an $\afSrC$-module via $\zrC$, then, by
Theorem \ref{weakly category equiv}, $V\cong\OgC^{\ot
r}\ot_{\afHrC}V_\bfs$ for some $\bfs$. Since $V_\bfs$ is a
homomorphic image of some $M_\bfa$ (see \cite[3.4]{CP}, say), it
follows that $V$ is a homomorphic image of $V\cong\OgC^{\ot
r}\ot_{\afHrC}M_\bfa$, which is, by Lemma \ref{key for step1},
isomorphic to $\OgC(a_1)\ot_\mbc \cdots\ot_\mbc\OgC(a_r).$ Hence,
$V$ is a homomorphic image of $\OgC(a_1)\ot_\mbc
\cdots\ot_\mbc\OgC(a_r)$, proving (2).
\end{proof}

\begin{Rem} The algebras $\afUglC$, $\afSrC$, etc., under consideration are all defined over $\mbc$
with parameter $z$ which is not a root of unity. As an
$\sS(n,r)_\mbc$-module, the finite dimensional tensor space
$\Og_{n,\mbc}^{\ot r}$ is a semisimple module (i.e., is completely
reducible). In the affine case, however, it is unclear if the
infinite dimensional tensor space $\Og_{\mbc}^{\ot r}$ is completely
reducible. The fact that every simple $\afSrC$-module is a
homomorphic image of $\Og_{\mbc}^{\ot r}$ does reflect a certain
degree of the complete reducibility.
\end{Rem}

Using Corollary \ref{quotient module of tensor space}, we can prove
the first key result for the classification theorem.

\begin{Prop}\label{step1}
Every simple $\afSrC$-module is a polynomial representation of
$\afUglC$.
\end{Prop}
\begin{proof}
Let $V$ be a simple $\afSrC$-module. Then $V$ is finite dimensional
by Theorem \ref{fd simple modules}. Thus, $V$ is a quotient module
of $\OgC(a_1)\ot_\mbc \cdots\ot_\mbc\OgC(a_r)$ for some
$\bfa\in(\mbc^*)^r$,  by Corollary \ref{quotient module of tensor
space}. Now, \eqref{OgC(a)} implies that $\OgC(a)$ is a polynomial
representation of $\afUglC$. So, by \cite[4.3]{FM}, the tensor
product $\OgC(a_1)\ot_\mbc\cdots\ot_\mbc\OgC(a_r)$ is a polynomial
representation of $\afUglC$. Hence, $V$ is also a polynomial
representation of $\afUglC$.
\end{proof}

Now let us define the $\afUglC$-module $\Det_a$ for $a\in\mbc^*$ in
the following lemma. The $\afUglC$-module $\Det_a$ plays the same
role in the representation theory of $\afUglC$ as the quantum
determinant in the representation theory of $\UglC$. One can easily
see that by restriction $\Det_a$ is isomorphic to the quantum
determinant for $\UglC$ for any $a\in\mbc^*$.

\begin{Lem}
Fix $a\in\mbc^*$. Let
$\bfa=\bfs=(a,a\ttv^{2},\ldots,a\ttv^{2(n-1)})$ regarded as a single
segment, and let
$$\Det_{a}:=\OgC^{\ot n}\ot_{\afHnC}V_a,$$
where $V_a=V_\bfs$ is the submodule of $M_\bfa$ generated by
$\ol{y}_{(n)}$.
 Then $\dim\Det_a=1$ and
$$
\Det_{a}=\spann_{\mbc}\{\og_1\ot\cdots\ot\og_n\ot \ol{y}_{(n)}\}.
$$
Moreover, as a $\afUglC$-module, $\Det_{a}\cong L(\bfQ)$, where
$\bfQ=(Q_1(u),\ldots,Q_n(u))\in\sQ(n)$ with $Q_i(u)=1-az^{2(n-i)}u$
for all $i=1,2,\ldots,n$.
\end{Lem}
\begin{proof} Recall the notations used in \S4.3. We have $M_\bfa\cong\HnC$ and, by Theorem \ref{classification irr affine Hecke
algebra}, $V_a=\HnC\ol{y}_{(n)}=\mbc \ol{y}_{(n)}$, since
$T_iy_{(n)}=-y_{(n)}$ for all $1\leq i\leq n-1$.

By Theorem \ref{category equiv}, $\Det_a$ is a simple
$\afSnC$-module. By Proposition \ref{step1}, $\Det_a\cong L(\bfQ)$
for some $\bfQ\in\sQ(n)$. Since
$$\OgnC^{\ot n}=\bop_{\bfi\in\afJnn}\og_\bfi\HnC,$$ by Proposition
\ref{Chari-Presslry},
$$
\Det_a\cong \OgnC^{\ot
n}\ot_{\HnC}V_a\cong\bop_{\bfi\in\afJnn}\og_{\bfi}\HnC\ot_{\HnC}V_a.
$$
Hence, $ \Det_a=\spann_\mbc
\{\og_\bfi\ot\ol{y}_{(n)}\mid\bfi\in\afJnn\}. $ If $i_k=i_{k+1}$ in
an $\bfi\in\afJnn$ for some $1\leq k\leq n-1$, then
$$
\ttv^2\og_\bfi\ot\ol{y}_{(n)}=\og_\bfi T_k\ot \ol{y}_{(n)}=\og_\bfi
\ot T_k\ol{y}_{(n)}=-\og_\bfi\ot\ol{y}_{(n)}.
$$
This forces $\og_\bfi\ot\ol{y}_{(n)}=0$ as $z$ is not a root of
unity. The only $\bfi\in\afJnn$ with $i_k\not=i_{k+1}$ for  $1\leq
k\leq n-1$ is $\bfi=(1,2,\ldots,n)$. Hence,
$$\Det_a=\mbc\, \og_1\ot\cdots\ot\og_n\ot\ol{y}_{(n)}.$$  Since
$K_iw_0=\ttv w_0$, where $w_0=\og_1\ot\cdots\ot\og_n
\ot\ol{y}_{(n)},$ it follows that $\mathrm{deg}\, Q_i(u)=1$ for all
$1\leq i\leq n$.  On the other hand, since
$$P_i(u)=\frac{Q_i (u\ttv^{i-1})} {Q_{i+1} (u\ttv^{i+1})}$$ is a
polynomial, we must have $P_i(u)=1$ for all $1\leq i\leq n-1$.

Suppose $Q_n(u)=1-bu$ for some $b\in\mbc^*$. Then
$Q_i(u)=1-bz^{2(n-i)}u$ for all $i$. Thus, as in the proof of
Theorem \ref{n>r representation}, Lemma \ref{action1 central elts}
implies
$$
\frac{t\ttv^{\pm t}}{[t]_\ttv}\sum_{1\leq i\leq n}g_{i,\pm t}w_0
=\sfz_t^\pm w_0=\sum_{1\leq k\leq n} (a\ttv^{2(k-1)})^{\pm t}w_0
$$
for all $t\geq 1$, and hence,
\begin{equation*}
\begin{split}
\prod_{1\leq i\leq n}\bigg(1-\big(
b\ttv^{2(n-i)}u\big)^{\pm 1}\bigg)w_0&=\prod_{1\leq i\leq n} Q_i^\pm(u)w_0\\
&=\prod_{1\leq i\leq n}\ms Q_i^\pm(u)w_0=\prod_{1\leq k\leq
n}\bigg(1-\big( a\ttv^{2(k-1)}u\big)^{\pm 1}\bigg)w_0.
\end{split}
\end{equation*}
Equating the coefficients of $u$ forces $a=b$.
\end{proof}

\begin{Lem}\label{tensor product of irreducible module}
Let $V$ be a simple $\afSkC$-module and $W$ be a simple
$\afSlC$-module. Then $V\ot W$ is an $\afSklC$-module.
\end{Lem}
\begin{proof}
By Corollary \ref{quotient module of tensor space}, $V$ is a
quotient module of the $\afUglC$-module $\OgC(a_1)\ot_\mbc
\cdots\ot_\mbc\OgC(a_k)$ for some $\bfa\in(\mbc^*)^k$ and  $W$ is a
quotient module of $\OgC(b_1)\ot_\mbc \cdots\ot_\mbc\OgC(b_l)$ for
some $\bfb\in(\mbc^*)^l$. Thus, $V\ot W$ is a quotient module of the
$\afUglC$-module
$$
\OgC(a_1)\ot_\mbc\cdots\ot_\mbc
\OgC(a_k)\ot_\mbc\OgC(b_1)\ot_\mbc\cdots \ot_\mbc\OgC(b_l),
$$
which is isomorphic to $\OgC^{\ot(k+l)}\ot_{\afHrC}
M_{(\bfa,\bfb)}$, by Lemma~\ref{key for step1}. Hence, as a quotient
module of an $\afSklC$-module, $V\ot W$ is an $\afSklC$-module.
\end{proof}

We remark that it would be possible to embed
$\sH_\vtg(k)_\mbc\ot\sH_\vtg(l)_\mbc$ into $\sH_\vtg(k+l)_\mbc$ as a
subalgebra (see, e.g., \cite[3.2]{CP}), and hence, regard $\afSklC$
as a subalgebra of $\afSkC\ot\afSlC$. Thus, by restriction, the
$\afSkC\ot\afSlC$-module $V\ot W$ is an $\afSklC$-module.

For $1\leq i\leq n-1$ and $a\in(\mbc)^*$, define
$\bfQ_{i,a}\in\sQ(n)$ by setting $Q_n(u)=1$ and
$$\frac{Q_j(u\ttv^{j-1})}
{Q_{j+1}(u\ttv^{j+1})}=(1-au)^{\dt_{i,j}}$$ for $1\leq j\leq n-1$.
In other words,
$$\bfQ_{i,a}=(1-az^{i-1}u,\ldots,1-az^{-i+3}u,\underset{(i)}{1-az^{-i+1}u},1,\ldots,1).$$
Since $n>i$, by Theorem \ref{n>r representation}, the simple
$\afUglC$-module $L_{i,a}:=L(\bfQ)$ is also a simple
$\afSiC$-module. The weight of the pseudo-highest weight vector of $L_{i,a}$ is
$\la(i,a)=(1^i,0^{n-i})$. Now we can prove the second key result for
the classification theorem.

\begin{Prop}\label{step2}
If $\bfQ=(Q_i(u))\in\sQ(n)$ with $r=\sum_{1\leq i\leq
n}\mathrm{deg}(Q_i(u))$, then  $L(\bfQ)$ is (isomorphic to) a simple
$\afSrC$-module.
\end{Prop}
\begin{proof}
Let $\la=(\la_1,\ldots,\la_n)$ with $\la_j=\mathrm{deg}(Q_j(u))$ and
let
$$
P_i(u)=\frac{Q_i(u\ttv^{i-1})} {Q_{i+1}(u\ttv^{i+1})},\quad 1\leq
i\leq n-1.
$$
Write $Q_n(u)=(1-b_1u)\cdots(1-b_{\la_n}u)$ and
$$
P_i(u)=\prod_{1\leq j\leq\mu_i}(1-a_{i,j}u),
$$
where $\mu_i=\la_i-\la_{i+1}$. Let
$$
V=L_1\ot\cdots\ot L_{n-1}\ot\Det_{b_1}\ot\cdots\ot\Det_{b_{\la_n}}
$$
where $L_i=L_{i,a_{i,1}}\ot\cdots\ot L_{i,a_{i,\mu_i}}$ for $1\leq
i\leq n-1$.

 Let $w_{i,a_{i,k}}$ (resp. $v_j$) be a pseudo-highest weight vector of $L_{i,a_{i,k}}$ (resp., $\Det_{b_j}$), and
 let $w_0=w_1\ot w_2\ot\cdots\ot w_{n-1}$, where $w_i= w_{i,a_{i,1}}\ot w_{i,a_{i,2}}\ot\cdots\ot w_{i,a_{i,\mu_i}}$, and $v_0=v_1\ot\cdots\ot v_n$.
 Since $w_{i,a_{i,k}}$ has weight $(1^i,0^{n-i})$ and
 $$\aligned
 \ms Q_j^\pm(u)w_{i,a_{i,k}}&=\begin{cases}(1-(a_{i,k}z^{i-2j+1}u)^{\pm1})w_{i,a_{i,k}},&\text{if }1\leq j\leq i;\\
 w_{i,a_{i,k}},&\text{if }i< j\leq n,\end{cases}\text{ and }\\
 \ms Q_j^\pm(u)v_i&=(1-(b_iz^{2(n-j)}u)^{\pm1})v_i\;\text{ for }1\leq i\leq \la_n,\endaligned$$
it follows from \cite[Lem.~4.1]{FM} that
$$ \ms Q_j^\pm(u)(w_0\ot v_0)=Q^\pm_j(u)(w_0\ot v_0).$$
Moreover, the weight of $w_0\ot v_0$ is
$$\la=(\mu_1,0,\ldots,0)+(\mu_2,\mu_2,0,\ldots,0)+(\mu_{n-1},\ldots,\mu_{n-1},0)+(\la_n,\ldots,\la_n).$$

  Let  $W$ be the submodule of $V$ generated
by $w_0\ot v_0$. Then $W$ is a pseudo-highest weight module whose pseudo-highest
weight vector is a common eigenvector of $\ttk_i$ and $\ms Q_{i,s}$
with eigenvalues $z^{\la_i}$ and $Q_{i,s}$, respectively, where
$Q_{i,s}$ are the coefficients of $Q_i^\pm(u)$. So the simple
quotient module of $W$ is isomorphic to $L(\bfQ)$ (cf. the
construction in \cite[Lem.~4.8]{FM}). Since $\sum_{1\leq i\leq
n-1}i\mu_i+n\la_n=\sum_{1\leq i\leq n}\la_i=r$, by Lemma \ref{tensor
product of irreducible module},  $V$ is an $\afSrC$-module. Hence,
$L(\bfQ)$ has an $\afSrC$-module structure.
\end{proof}

Now using Propositions \ref{step1} and \ref{step2} we can prove the
following classification theorem.

\begin{Thm}\label{representation}
For any $n,r\geq1$, the set $\big\{L(\bfQ)\mid\bfQ\in\sQ(n)_r\big\}$
is a complete set of  nonisomorphic simple $\afSrC$-modules.
\end{Thm}
\begin{proof} By Proposition \ref{step2}, the set $\big\{L(\bfQ)\mid\bfQ\in\sQ(n)_r\big\}$
consists of nonisomorphic simple $\afSrC$-modules. It remains to
prove that every simple $\afSrC$-module is isomorphic to $L(\bfQ)$
for some $\bfQ\in\sQ(n)_r$.

Let $V$ be a simple $\afSrC$-module. Then $V\cong L(\bfQ)$ as a
$\afUglC$-module for some $\bfQ\in\sQ(n)$ by Proposition
\ref{step1}. Let $l=\sum_{1\leq i\leq n}\mathrm{deg}\,Q_i(u)$. Then,
by Proposition \ref{step2}, $L(\bfQ)$ is an $\afSlC$-module. Thus,
by restriction, $V$ is a module for the $q$-Schur algebra
$\sS(n,r)_\mbc$ and $V$ is also a module for the $q$-Schur algebra
$\sS(n,l)_\mbc$. Hence, $r=l$.
\end{proof}

\begin{Coro}\label{level r}
Let $V$ be a finite dimensional irreducible polynomial
representation of $\afUglC$. Then $V$ can be regarded as an
$\afSrC$-module via $\zrC$ if and only if $V$ is of level $r$ as a
$U(n)_\mbc$-module.
\end{Coro}
\begin{proof}
If $V$ can be regarded as an $\afSrC$-module via $\zrC$ then we may
view $V$ as an $\sS(n,r)_\mbc$-module by restriction, and hence, $V$
is of level $r$ as a $U_\vtg(n)_\mbc$-module.

Conversely, suppose that $V$ is of level $r$ as a $U(n)_\mbc$-module
and $V=L(\bfQ)$ for some $\bfQ\in\sQ(n)$. Then $V$ is an
$\sS_\vtg(n,r')_\mbc$-module by Theorem \ref{representation}, where
$r'=\sum_{1\leq i\leq n}\mathrm{deg}Q_i(u)$. Hence, $V$ is of level
$r'$ as a $U(n)_\mbc$-module. So $r=r'$ and $V$ is an
$\afSrC$-module.
\end{proof}

\begin{Rem}\label{level r2} It is reasonable to make the following definition.
A finite dimensional $\afUglC$-module is said to be of level $r$ if
it is an $\afSrC$-module via $\zrC$. Thus, if $V$ is a
$\afUglC$-module of level $r$, then its composition factors are all
homomorphic images of $\Og_\mbc^{\ot r}$. It would be interesting to
know if the converse is also true.
\end{Rem}

It is natural to make a comparison between the Classification
Theorems \ref{weakly category equiv} and \ref{representation} and to
raise the following problem.

\begin{Prob} \label{Prob-Identif-Thm} Generalize the Identification Theorem
\ref{n>r representation} to the case where $n\leq r$.
\end{Prob}

\section{Classification of simple $\afUnrC$-modules}

The homomorphic image $\afUnrC$ of the extended affine quantum
$\frak{sl}_n$, $\afUnC$, is a proper subalgebra of $\afSrC$ when
$n\leq r$. In other words, by restriction, the surjective algebra
homomorphism $\zrC:\afUglC\ra\afSrC$ induces a surjective algebra
homomorphism $\zrC:\afUnC\ra\afUnrC$. In this section, we will
classify simple $\afUnrC$-modules.

Let $\bfP\in\sP(n)$ and $\la\in\La^+(n,r)$ be such that
$\la_i-\la_{i+1}=\mathrm{deg}P_i$ for $1\leq i\leq n-1$. Define
$$
\Mcp(\bfP,\la)=\afUnC/\Icp(\bfP,\la),
$$
where $\bar I(\bfP,\la)$ is the left ideal of $\afUnC$ generated by
$\ttx^+_{i,s}$, $\msP_{i,s}-P_{i,s}$ and $\ttk_j-\ttv^{\la_j}$ for
$1\leq i\leq n-1$, $s\in\mbz$ and $1\leq j\leq n$, where $P_{i,s}$
is defined using \eqref{Pis}. The $\afUnC$-module $\Mcp(\bfP,\la)$
has a unique simple quotient $\afUnC$-module, which is denoted by
$\Lcp(\bfP,\la)$.

\begin{Lem}\label{restriction afUglC to afUnC}
For $\bfQ\in\sQ(n)$, let $\la=(\la_1,\ldots,\la_n)$ with
$\la_i=\mathrm{deg}(Q_i(u))$  for $1\leq i\leq n$ and
$\bfP=(P_1(u),\ldots,P_{n-1}(u))\in\sP(n)$ be such that
$$
P_j(u)=\frac{Q_j(u\ttv^{j-1})} {Q_{j+1}(u\ttv^{j+1})}
$$
for $1\leq j\leq n-1$. Then $ \Lcp(\bfP,\la)\cong L(\bfQ)|_{\afUnC}.
$
\end{Lem}
\begin{proof}
Let $w_0\in L(\bfQ)$ be a pseudo-highest weight vector. Since
$\afUglC$ is generated by $\afUnC$ and the central elements
$\sfz_t^\pm$ for $t\geq 1$, every simple $\afUglC$-module is a
simple $\afUnC$-module by restriction. In particular,
$L(\bfQ)|_{\afUnC}$ is simple. So we have $L(\bfQ)=\afUnC w_0$.
Hence, there is a surjective $\afUnC$-module homomorphism
$\vi:\Mcp(\bfP,\la)\ra L(\bfQ)$ defined by sending $\bar u$ to
$uw_0$ for $u\in\afUnC$. Thus, $L(\bfQ)\cong \Mcp(\bfP,\la)/{\rm
Ker}\,\vi$ as $\afUnC$-modules. Since $L(\bfQ)|_{\afUnC}$ is simple
and $\Mcp(\bfP,\la)$ has a unique simple quotient $\Lcp(\bfP,\la)$,
we have $\Lcp(\bfP,\la)\cong L(\bfQ)$ as $\afUnC$-modules.
\end{proof}
\begin{Coro}\label{lem1 for rep of UnrC}
Let $\bfP\in\sP(n)$ and $\la\in\La^+(n,r)$ with $\la_i-\la_{i+1}
=\mathrm{deg}P_i(u)$ for $1\leq i\leq n-1$. Then $\Lcp(\bfP,\la)$ is
a $\afUnrC$-module via $\zrC$ and
$\Lcp(\bfP,\la)|_{\afUslC}\cong\Lcp(\bfP)$.
\end{Coro}
\begin{proof}
Let $Q_n(u)=1+u^{\la_n}$. Using the formula
$$
P_j(u)=\frac{Q_j(u\ttv^{j-1})} {Q_{j+1}(u\ttv^{j+1})},
$$
we define the polynomials $Q_i(u)$ for $1\leq i\leq n-1$. Then we
have $\bfQ=(Q_1(u),\ldots,Q_n(u))\in\sQ(n)$ and
$\la_i=\mathrm{deg}Q_i(u)$ for $1\leq i\leq n$. By Theorem
\ref{representation}, $L(\bfQ)$ is an $\afSrC$-module. So, by Lemma
\ref{restriction afUglC to afUnC}, $\Lcp(\bfP,\la)\cong
L(\bfQ)|_{\afUnC}$ is a $\afUnrC$-module. Hence,
$\Lcp(\bfP,\la)|_{\afUslC}$ is simple since the algebra homomorphism
$\zrC:\afUslC\ra\afUnrC$ is surjective.
\end{proof}

Since $\afUnrC$ contains $\SrC$ as a subalgebra, it follows that, if
$\la\in\La^+(n,r)$, then $\Lcp(\bfP,\la)$ is an $\SrC$-module and
\begin{equation}\label{wt}
\Lcp(\bfP,\la)=\bop_{\mu\leq\la\atop\mu\in\La(n,r)}\Lcp(\bfP,\la)_\mu
\end{equation}
where $\Lcp(\bfP,\la)_\mu$ denote the weight space of
$\Lcp(\bfP,\la)$ as a $U(n)_\mbc$-module.

\begin{Lem}\label{lem2 for rep of UnrC} Let $\la,\ti\la$ be partitions and $\bfP,\ti\bfP\in\sP(n)$.
If $\Lcp(\bfP,\la)\cong\Lcp(\ti\bfP,\ti\la)$, then $\bfP=\ti\bfP$
and $\la=\ti\la$. In particular, for $\bfQ,\bfQ'\in\sQ(n)$,
$L(\bfQ)|_{\afUnC}\cong L(\bfQ')|_{\afUnC}$ if and only if
$\mathrm{deg}Q_i(u)=\mathrm{deg}Q_i'(u)$ and
$Q_j(u\ttv^{j-1})/Q_{j+1}(u\ttv^{j+1})=Q_j'(uz^{j-1})/Q_{j+1}'(u\ttv^{j+1})$
for all $1\leq i\leq n$ and $1\leq j\leq n-1$.
\end{Lem}
\begin{proof}
By \eqref{wt} we have $\la=\ti\la$. Since
$\Lcp(\bfP,\la)\cong\Lcp(\ti\bfP,\ti\la)$, it follows from Corollary
\ref{lem1 for rep of UnrC} that
$\Lcp(\bfP)\cong\Lcp(\bfP,\la)|_{\afUslC}\cong
\Lcp(\ti\bfP,\ti\la)|_{\afUslC}\cong\Lcp(\ti\bfP)$. Therefore,
$\bfP=\ti\bfP$.
\end{proof}
\begin{Lem}\label{lem3 for rep of UnrC}
Let $V$ be a finite dimensional simple $\afUnrC$-module. Then there
exist $\bfP=(P_1(u),\ldots,P_{n-1}(u))\in\sP(n)$ and $\la\in\La^+(n,r)$ with $\la_i-\la_{i+1}
=\mathrm{deg}P_i(u)$, for all $1\leq i\leq n-1$, such that $V\cong
\Lcp(\bfP,\la)$.
\end{Lem}
\begin{proof}
Since $\zrC:\afUslC\ra\afUnrC$ is surjective, $V$ is a simple
$\afUslC$-module. Let $w_0$ be a pseudo-highest weight vector satisfying
$$\ttx^+_{i,s}w_0=0, \msP_{i,s}w_0=P_{i,s}w_0, \text{ and }\ti\ttk_jw_0=\ttv^{\mu_j}w_0$$ for
all $1\leq i,j\leq n-1$ and $s\in\mbz$, where
$\mu_j=\mathrm{deg}P_j(u)$. Using the idempotent decomposition
$1=\sum_{\nu\in\La_\vtg(n,r)}\bfone_\nu$, \index{$\bfone_\la$,
idempotent $[\diag(\la)]$}
 $\sum_{\nu\in\La_\vtg(n,r)}\bfone_\nu
w_0=w_0\not=0$ implies that there exists $\la\in\La(n,r)$ such that
$\bfone_\la w_0\not=0$.

It is clear that $\bfone_\la w_0$ is also a pseudo-highest weight vector
satisfying
$$\ttx^+_{i,s}\bfone_\la w_0=0, \msP_{i,s}\bfone_\la w_0=P_{i,s}\bfone_\la w_0, \text{ and }\ti\ttk_j\bfone_\la w_0=\ttv^{\mu_j}\bfone_\la w_0.$$
 On the other hand, $\ttk_i\bfone_\la w_0=\ttv^{\la_i}\bfone_\la w_0$ for $1\leq i\leq
n$. Thus, $\ti\ttk_j\bfone_\la w_0=\ttv^{\la_j-\la_{j+1}}\bfone_\la
w_0$. So $\la_i-\la_{i+1}=\mu_i$ for $1\leq i\leq n-1$. Hence, there
is a surjective $\afUnC$-module homomorphism $\vi:\Mcp(\bfP,\la)\ra
V$ defined by sending $\bar u$ to $uw_0$ for all $u\in\afUnC$. This
surjection induces a $\afUnC$-module isomorphism
 $V\cong \Lcp(\bfP,\la)$.
\end{proof}


All together gives the following classification theorem.

\begin{Thm}\label{simple U(n,r)-modules}
The set
$$
\{\Lcp(\bfP,\la)\mid\bfP\in\sP(n),
\,\la\in\La^+(n,r),\,\la_i-\la_{i+1} =\mathrm{deg}P_i(u)\text{ for
$1\leq i\leq n-1$}\}
$$
is a complete set of nonisomorphic finite dimensional simple
$\afUnrC$-modules.
\end{Thm}

Define an equivalence relation $\sim$ on $\sQ(n)$ be setting for
$\bfQ,\bfQ'\in\sQ(n)$,
$$\aligned
\bfQ\sim\bfQ'\iff &\mathrm{deg}Q_i(u)=\mathrm{deg}Q_i'(u),1\leq
i\leq
n,\text{ and }\\ &\frac{Q_j(u\ttv^{j-1})}{Q_{j+1}(u\ttv^{j+1})}=\frac{Q_j'(uz^{j-1})}{Q_{j+1}'(u\ttv^{j+1})}, 1\leq j\leq n-1.\\
\endaligned$$

\begin{Coro} If $\Pi_r=\sQ(n)_r/\sim$ denotes the set of equivalence classes and choose a representative $\bfQ_\pi\in\pi$
for every $\pi\in\Pi_r$, then the set
$\{L(\bfQ_\pi)|_{\afUnC}:\pi\in\Pi_r\}$ is a complete set of
nonisomorphic finite dimensional simple $\afUnrC$-modules. Moreover,
if $n>r$, then $\Pi_r=\sQ(n)_r$.
\end{Coro}

\begin{proof} The last assertion follows from the fact that, if $n>r$, then $\afUnrC=\afSrC$.
\end{proof}

As seen in Theorem \ref{classification of simple afUglC-modules},
there is a rougher equivalence relation $\sim'$ on $\sQ(n)$ defined
by setting for $\bfQ,\bfQ'\in\sQ(n)$,
$$\bfQ\sim'\bfQ'\iff \frac{Q_j(u\ttv^{j-1})}{Q_{j+1}(u\ttv^{j+1})}=\frac{Q_j'(uz^{j-1})}{Q_{j+1}'(u\ttv^{j+1})}\;\,
\text{for $1\leq j\leq n-1$}$$
 such that the equivalence classes are in one-to-one correspondence
to simple $\afUslC$-modules.

\chapter{The presentation and realization problems}

As seen in Chapters 2 and 3, the double Ringel--Hall algebra
$\dHallr$ is presented by generators and relations, while the affine
quantum Schur algebra $\afbfSr$ is defined as an endomorphism algebra
which is a vector space with an explicitly defined multiplication. Now the algebra
epimorphism from $\dHallr$ to $\afbfSr$ raises two natural
questions: how to present affine quantum Schur algebras $\afbfSr$ in
terms of generators and relations and how to realize the double
Ringel--Hall algebra $\dHallr$ in terms of a vector space together
with an explicitly defined multiplication? In this and next
chapters, we will tackle these problems.

Since $\dHallr\cong \bfU_\vtg(n)\ot{\bf Z}_\vtg(n)$ by Remark
\ref{subalgebra quantum gl}(2), it follows that
$\afbfSr=\bfU_\vtg(n,r)\mathbf Z_\vtg(n,r)$, where $\bfU_\vtg(n,r)$
(resp., $\mathbf Z_\vtg(n,r)$) is the homomorphic image  of the
quantum group $\bfU_\vtg(n)$, the extended quantum affine
$\mathfrak{sl}_n$,\index{extended quantum affine $\frak{sl}_n$}
(resp. the central subalgebra ${\bf Z}_\vtg(n)$)\index{central
subalgebra of $\dHallr$} under the map $\xi_r$. We first review  in
\S5.1 a presentation of McGerty for $\bfU_\vtg(n,r)$. This is a
proper subalgebra of $\afbfSr$ if $n\leq r$. As a natural affine
analogue of the presentation given by Doty--Giaquinto \cite{DG}, we
will modify McGerty's presentation to obtain a Drinfeld--Jimbo type
presentation for $\bfU_\vtg(n,r)$ (Theorem \ref{2nd presentation for
im(afzr)}). We then determine the structure of the central
subalgebra $\mathbf Z_\vtg(n,r)$ of $\afbfSr$ (Proposition
\ref{basis-center}). However, it is almost impossible to combine the
two to give a presentation for $\afbfSr$. In \S5.3, we will use the
multiplication formulas given in \S3.4 to derive some extra
relations for an extra generator required for presenting
$\boldsymbol{\sS}_\vtg(r,r)$ for all $r\geq1$ (Theorem \ref{n,n
case}). In particular, we will then easily see why the Hopf algebra
$\widehat\bfU$ considered in \cite[3.1.1]{Gr99}\footnote{This
algebra $\widehat\bfU$ is denoted by $U(\widehat{gl}_n)$ in
\cite{Gr99}.} maps onto affine quantum Schur algebras; see Remark
\ref{RGreenQAgln}. Strictly speaking, $\widehat\bfU$ cannot be
regarded as a quantum enveloping algebra since it does not have a
triangular decomposition.

From \S5.4 onwards, we will discuss the realization problem. We
first formulate a realization conjecture in \S5.4, as suggested in
\cite[5.2(2)]{DF09}, and its classical ($\up=1$) version. In the
last section, we show that Lusztig's transfer maps are not
compatible with the map $\xi_r$ for the double Ringel--Hall algebra
${}'\dHallr$ considered in Remarks \ref{subalgebra quantum gl}(3).
This justifies that we cannot have a realization in terms of an
inverse limit of the transfer maps. We will then establish the
conjecture for the classical case in the next chapter.

\section{McGerty's presentation for $\bfU_\vtg(n,r)$}

The presentation problem for $\afbfSr$ when $n>r$ is relatively
easy. In this case, $\afbfSr=\bfU_\vtg(n,r)$ is a homomorphic image
of $\bfU_\vtg(n)$. By using McGerty's presentation for
$\bfU_\vtg(n,r)$, we obtain a new presentation for $\bfU_\vtg(n,r)$
similar to that for quantum Schur algebras given in \cite{DG} (cf.
\cite{DP03}). In particular, this gives Doty--Green's result
\cite{DGr} for $\afbfSr$ with $n>r$ (removing the condition $n\ge3$
required there).

Let $\xi_r:\dHallr\ra \afbfSr$ be the surjective homomorphism
defined in \eqref{xir}. For each $r\geq1$, define
$$\bfU_\vtg(n,r):=\xi_r(\bfU_\vtg(n)).$$\index{$\bfU_\vtg(n,r)$,
homomorphic image of $\bfU_\vtg(n)$} Clearly, $\bfU_\vtg(n,r)$ is
generated by the elements
$$\bffke_i:=\xi_r(E_i)=\afE_{i,i+1}(\bfl,r),\,
\bffkf_i:=\xi_r(F_i)=\afE_{i+1,i}(\bfl,r),\,\bffkk_i:=\xi_r(K_i)=0(\bfe_i,r),$$
for all $i\in I$.

Let $C=(c_{i,j})$ denote the generalized Cartan matrix of type $\ti
A_{n-1}$ as in \eqref{CMforAn-1}. Recall the elements $\bfone_\la$,
$\la\in\afLanr$, defined in Proposition
\ref{afSrz}(2).\index{$\bfone_\la$, idempotent $[\diag(\la)]$}
 The following result is taken from \cite[Prop.~6.4 \&
Lem.~6.6]{Mc07}.

\begin{Thm}\label{presentation for im(afzr)}
As a $\mbq(\up)$-algebra, $\bfU_\vtg(n,r)$ is generated by
$$\bffke_i,\ \bffkf_i,\ \bfone_\la\qquad(i\in I,\ \la\in\afLanr)$$
subject to the following relations
\begin{itemize}
\item[(1)] $\bfone_\la\bfone_\mu=\dt_{\la,\mu}\bfone_\la$,
$\sum_{\la\in\afLanr}\bfone_\la=1$;

\item[(2)] $\bffke_i\bfone_\la=\begin{cases}
\bfone_{\la+\afbse_i-\afbse_{i+1}}\bffke_i,\;\;\;&\text{if $\la+\afbse_i-\afbse_{i+1}\in\afLanr$};\\
0,\;\;&\text{otherwise;}\end{cases}$

\item[(3)] $\bffkf_i\bfone_\la=\begin{cases}
\bfone_{\la-\afbse_i+\afbse_{i+1}}\bffkf_i,\;\;\;&\text{if $\la-\afbse_i+\afbse_{i+1}\in\afLanr$};\\
0,\;\;&\text{otherwise;}\end{cases}$

\item[(4)] $\bffke_i\bffkf_j-\bffkf_j\bffke_i=\delta_{i,j}
\sum_{\la\in\afLanr}[\la_i-\la_{i+1}] \bfone_\la;$

\item[(5)] $\displaystyle\sum_{a+b=1-c_{i,j}}(-1)^a
\leb{1-c_{i,j}\atop a}\rib \bffke_i^{a}\bffke_j\bffke_i^{b}=0$ for
$i\not=j$;

\item[(6)] $\displaystyle\sum_{a+b=1-c_{i,j}}(-1)^a
\leb{1-c_{i,j}\atop a}\rib \bffkf_i^{a}\bffkf_j\bffkf_i^{b}=0$ for
$i\not=j$.
\end{itemize}
\end{Thm}

This theorem is the affine version of Theorem 3.4 in \cite{DG}.
Naturally, one expects the affinization of Theorem 3.1 in \cite{DG}
for a Drinfeld--Jimbo type presentation. In fact, this is an easy
consequence of the following result for the Laurent polynomial
 algebra $\bfU^0:=\mbq(\up)[K_1^{\pm1},\ldots,K_n^{\pm1}]=\bfU_\vtg(n)^0$,
whose proof can be found in \cite[Prop.~8.2,\,8.3]{DG} and
\cite[4.5,\,4.6]{DP03} in the context of quantum $\mathfrak{gl}_n$
and quantum Schur algebras (though the result itself was not
explicitly stated there).

For any $\mu\in\mbn^n$ and $1\leq i\leq n$, let
\begin{equation}\label{kk-lambda}
\aligned
\fL_\mu&=\prod\limits_{i=1}^n\left[{K_i;0 \atop
\mu_i}\right]\,\, \,\text{ and }\,\,\boldsymbol\up^\mu=(\up^{\mu_1},\ldots,\up^{\mu_n}),\\
[K_i;&r+1]^!=(K_i-1)(K_i-\up)\cdots(K_i-\up^r).\\\endaligned
\end{equation}
If we regard $\fL_\mu$ as a function from $\bfU^0$ to $\mbq(\up)$, then
 $\fL_\la(\boldsymbol\up^\mu)=\dt_{\la,\mu}$ for all $\la,\mu\in\La(n,r)$.

\begin{Lem}\label{le1} The ideals $\lan I_r\ran$ and $\lan J_r\ran$ of $\bfU^0$ generated by the sets
$$\aligned
I_r&=\{1-\Sigma_{\la\in\La(n,r)}\fL_\la\}\cup\{\fL_\la \fL_\mu-\dt_{\la,\mu}\fL_\la\mid
\la,\mu\in\La(n,r)\}\\
&\quad\,\cup\{K_i\fL_\la-\up^{\la_i}\fL_\la\mid
1\leq i\leq n,\,\la\in\La(n,r)\},\quad\text{ and }\\
J_r&=\{\kappa:=K_1\cdots K_n-\up^r,[K_i;r+1]^!\mid 1\leq i\leq n\},\endaligned$$
are the same.
\end{Lem}

\begin{proof} For completeness, we provide here a direct proof. Consider the algebra epimorphism
$$\varphi:\bfU^0\lra\bigoplus_{\mu\in\La(n,r)}\bfU^0/\lan K_1-\up^{\mu_1},\ldots,K_n-\up^{\mu_n}\ran,\,\, f\longmapsto (f(\boldsymbol\up^\mu))_{\mu\in\La(n,r)}.$$
It is clear from the relations
$\fL_\la(\boldsymbol\up^\mu)=\dt_{\la,\mu}$ that $\lan I_r\ran={\rm
Ker}\,\varphi$. Applying the Chinese Remainder Theorem
yields\footnote{The ideal $J$ appeared in the proof of
\cite[Lem.~13.36]{DDPW} should be
$$\bigcap_{0\leq\mu_1,\ldots,\mu_n\leq r}\lan
x_1-\up^{\mu_1},\ldots,x_n-\up^{\mu_n}\ran,$$ while the quotient
algebra $R/J$ should have dimension $(r+1)^n$.}
$$\aligned
\lan I_r\ran&=\bigcap_{\mu\in\La(n,r)}\lan K_1-\up^{\mu_1},\ldots,K_n-\up^{\mu_n}\ran\\
&=\bigcap_{0\leq\mu_1,\ldots,\mu_n\leq r}(\lan\kappa\ran+\lan K_1-\up^{\mu_1},\ldots,K_n-\up^{\mu_n}\ran)\\
&=\lan \kappa\ran+\bigcap_{0\leq\mu_1,\ldots,\mu_n\leq r}\lan
K_1-\up^{\mu_1},\ldots,K_n-\up^{\mu_n}\ran\\
&=\lan \kappa\ran+\lan [K_1;r+1]^!,\ldots,[K_n;r+1]^!\ran\\
&=\lan J_r \ran,\\\endaligned$$
since $\kappa\in \lan K_1-\up^{\mu_1},\ldots,K_n-\up^{\mu_n}\ran\iff \mu_1+\cdots+\mu_n=r$.\end{proof}

 The above lemma together with Theorem \ref{presentation for im(afzr)}
gives the following result, which, as mentioned above, was proved in
\cite{DGr} under the assumption that $n\geq 3$ and $n>r$.

\begin{Thm} \label{2nd presentation for im(afzr)}
The algebra $\bfU_\vtg(n,r)$ is generated by the elements
$$\bffke_i,\ \bffkf_i,\ \bffkk_i\ (i\in I=\mbz/n\mbz)$$
subject to the relations$:$
\begin{itemize}
\item[(QS1)] $\bffkk_{i}\bffkk_{j}=\bffkk_{j}\bffkk_{i};$\index{defining relations!({\rm QS1})--({\rm QS6})}

\item[(QS2)] $\bffkk_{i}\bffke_j=\up^{\dt_{i,j}-\dt_{i,j+1}}\bffke_j\bffkk_{i},\
\bffkk_{i}\bffkf_j=\up^{-\dt_{i,j}+\dt_{i,j+1}}
\bffkf_j\bffkk_i;$

\item[(QS3)] $\bffke_i\bffkf_j-\bffkf_j\bffke_i=\delta_{i,j}\frac
{\ti\bffkk_{i}-\ti\bffkk_{i}^{-1}}{\up-\up^{-1}},\ where \
\ti\bffkk_i =\bffkk_{i}\bffkk_{i+1}^{-1};$

\item[(QS4)] $\displaystyle\sum_{a+b=1-c_{i,j}}(-1)^a
\leb{1-c_{i,j}\atop a}\rib \bffke_i^{a}\bffke_j\bffke_i^{b}=0$ for
$i\not=j$;

\item[(QS5)] $\displaystyle\sum_{a+b=1-c_{i,j}}(-1)^a
\leb{1-c_{i,j}\atop a}\rib \bffkf_i^{a}\bffkf_j\bffkf_i^{b}=0$ for
$i\not=j$;

\item[(QS6)] $[\bffkk_i;r+1]^!=0,$ $\bffkk_1\cdots\bffkk_n=\up^r$.
\end{itemize}
\end{Thm}

\begin{proof} First, by the lemma, the ideal $\sI_r$ of $\afbfU$ generated by $I_r$ is
the same as the ideal $\sJ_r$  generated by $J_r$. Second, if
$\xi'_r:\bfU_\vtg(n)\ra\afbfSr$ denotes the restriction of $\xi_r$
to $\bfU_\vtg(n)$, then it is clear that $\sI_r\han{\rm
Ker}\,\xi'_r$ (see the proof of Proposition \ref{afSrz}). Thus, we
obtain an epimorphism
$$\varphi:\afbfU/\sI_r\lra \afbfU/{\rm Ker}\,\xi'_r\cong\bfU_\vtg(n,r)$$
satisfying $\varphi(E_i+\sI_r)=\bffke_i$, $\varphi(F_i+\sI_r)=\bffkf_i$, and
$\varphi(\fL_\la+\sI_r)=\bfone_\la$ for $i\in I$ and $\la\in\afLanr$.

On the other hand, it is direct to check that all relations given in
Theorem \ref{presentation for im(afzr)} hold in $\afbfU/\sI_r$ (see,
e.g., the proof of \cite[Lem~13.40]{DDPW}). Thus, applying Theorem
\ref{presentation for im(afzr)} yields a natural algebra
homomorphism
$$\psi:\bfU_\vtg(n,r)\lra\afbfU/\sI_r$$
satisfying $\bffke_i\mapsto E_i+\sI_r$, $\bffkf_i\mapsto F_i+\sI_r$, and
$\bfone_\la\mapsto \fL_\la+\sI_r$.
Therefore, $\varphi$ has to be an isomorphism, forcing ${\rm
Ker}\,\xi'_r=\sI_r=\sJ_r$.
\end{proof}

\begin{Rems} (1) It is possible to replace the relations in (QS6) by the relations

\vspace{.2cm}

\qquad $[\bffkk_1;\mu_1]^![\bffkk_2;\mu_2]^!
\cdots[\bffkk_n;\mu_n]^!=0$ for all $\mu\in\mbn_\vtg^n$ with
$\sigma(\mu)=r+1$,

\vspace{.2cm}

\noindent and replace $\bffkk_n$ used in (QS1)--(QS5) by
$\bffkk_n:=\up^r\bffkk_1^{-1}\cdots\bffkk_{n-1}^{-1}$. The new
presentation uses only $3n-1$ generators. For more details, see
\cite{DP03} or \cite[\S13.10]{DDPW}. It is interesting to point out
that the homomorphic image $\bfU(\infty,r)$ of
$\bfU(\mathfrak{gl}_\infty)$, which is a proper subalgebra of the
infinite quantum Schur algebra $\boldsymbol\sS(\infty,r)$ for all
$r\geq 1$, has only a presentation of this type; see
\cite[4.7,5.4]{DF09}.

(2) If $\ti\sI_r$ denotes the ideal of $\dHallr$ generated by $I_r$, then
$$\dHallr/\ti\sI_r\cong \bfU_\vtg(n,r)\ot{\bf Z}_\vtg(n).$$
Thus, adding (QS6) to relations (QGL1)--(QGL8) in Theorem
\ref{presentation dHallAlg} gives a presentation for this algebra.
\end{Rems}

The relations (QS1)--(QS6) in Theorem \ref{2nd presentation for
im(afzr)} will form part of the relations in a presentation for
$\boldsymbol\sS_\vtg(r,r)$, $r\geq 1$; see \S5.3.

\section{Structure of affine quantum Schur algebras}

When $n\leq r$, $\bfU_\vtg(n,r)$ is a proper subalgebra of $\afbfSr$.
The next two sections are devoted to the study of the structure of
affine quantum Schur algebras $\afbfSr$ in this case. We will first
see in this section a general structure of $\afbfSr$ inherited from
$\dHallr$ and give an explicit presentation for
${\bds\sS}_\vtg(r,r)$ in \S5.3.


We first endow $\afbfSr$ with a $\mbz$-grading through the surjective
algebra homomorphism
$$\xi_r:\dHallr\lra \afbfSr=\End_{\afbfHr}(\bdOg^{\otimes r}).$$

If we assign to each $u_A^+$ (resp., $u_A^-$, $K_i$) the degree
$\fkd(A)=\dim M(A)$ (resp., $-\fkd(A)$, $0$), then $\dHallr$ admits
a $\mbz$-grading $\dHallr=\oplus_{m\in\mbz}\dHallr_m$. By
definition, we have for each $x\in\dHallr_m$ and
$\og_\bfi=\og_{i_1}\ot\cdots\ot\og_{i_r}\in\bdOg^{\ot r}$,
$$x\cdot \og_\bfi=\sum_{p=1}^l a_p\og_{\bfj^{(p)}},$$
 where $a_p\in\mbq(\up)$ and
 $\bfj^{(p)}=(j_{1,p},\ldots,j_{r,p})\in\mbz^r$ satisfy
 $\sum_{s=1}^rj_{s,p}=\sum_{s=1}^r i_s-m$ for all $1\leq p\leq l$.
 Thus, letting ${\bds\sS}_\vtg(n,r)_m=\xi_r(\dHallr_m)$ gives a
 decomposition
$${\bds \sS}_\vtg(n,r)=\bigoplus_{m\in\mbz}{\bds
\sS}_\vtg(n,r)_m$$
 of ${\bds \sS}_\vtg(n,r)$, i.e., ${\bds \sS}_\vtg(n,r)$ is
$\mbz$-graded, too.

By Remark \ref{subalgebra quantum gl}(2), 
there is a central subalgebra ${\bf
Z}_\vtg(n)=\mbq(\up)[\sfz_m^+,\sfz_m^-]_{m\geq 1}$ of $\dHallr$
such that
$\dHallr=\bfU_\vtg(n)\otimes_{\mbq(\up)}{\bf Z}_\vtg(n).$
 This gives another subalgebra of $\afbfSr$
 $${\bf Z}_\vtg(n,r):=\xi_r({\bf Z}_\vtg(n))$$
such that $\afbfSr=\bfU_\vtg(n,r) {\bf Z}_\vtg(n,r)$.  In other
words, $\xi_r$ induces a surjective algebra homomorphism
$$\bfU_\vtg(n,r)\otimes {\bf Z}_\vtg(n)\lra \afbfSr,\;x\otimes y\lm
x\xi_r(y).$$ Clearly, ${\bf Z}_\vtg(n,r)$ is contained in the center
of $\afbfSr$.

By \cite[Th.~7.10 \& 8.4]{Lu99} (see also Corollary \ref{n>r}), we
have the following result.

\begin{Lem} \label{S(n,r)=U(n,r)}  The equality $\bfU_\vtg(n,r)=\afbfSr$
holds if and only if $n>r$. In other words, ${\bf
Z}_\vtg(n,r)\subseteq\bfU_\vtg(n,r)$ if and only if $n>r$.
\end{Lem}

Moreover, for each $m\geq 1$, $\xi_r(\sfz_m^+)\in\afbfSr_{mn}$ and
$\xi_r(\sfz_m^-)\in\afbfSr_{-mn}$, and the $\mbz$-grading of
$\afbfSr$ induces a $\mbz$-grading
$$\bfU_\vtg(n,r)=\bigoplus_{m\in\mbz}\bfU_\vtg(n,r)_m$$
of $\bfU_\vtg(n,r)$, where $\bfU_\vtg(n,r)_m=\bfU_\vtg(n,r)\cap
 {\bds\sS}_\vtg(n,r)_m$.

 We are now going to determine the structure of both
 ${\bf Z}_\vtg(n,r)$ and $\bfU_\vtg(n,r)$. For all $1\leq s\leq r$, define commuting $\mbq(\up)$-linear
 maps
$$\phi_s:\bdOg^{\otimes r}\lra\bdOg^{\otimes
r},\;\omega_{\bfi}\longmapsto \omega_{\bfi-n\bfe_s}=\og_{i_1}\ot
\cdots\ot\og_{i_s-n}\ot\cdots \ot \og_{i_r}$$
 and set for each $m\geq 1$,
$${\frak p}_m=\sum_{s=1}^r \phi_s^m\;\;\text{and}\;\;{\frak q}_m=\sum_{s=1}^r \phi_s^{-m}.$$
 By \eqref{action central elts}, ${\frak p}_m=\xi_r(\sfz_m^+)$ and
${\frak q}_m=\xi_r(\sfz_m^-)$. Thus, they both lie in $\afbfSr$.
 Moreover,
$${\bf Z}_\vtg(n,r)={\mbq}(\up)[{\frak p}_m,{\frak q}_m]_{m\geq 1}.$$

\begin{Rem} \label{images-central-elements} The right action \eqref{afH action}
of $\afbfHr$ on $\bdOg^{\ot r}$ induces an algebra homomorphism
$\psi_r:\afbfHr\ra\End_{\mbq(\up)}(\bdOg^{\ot r})$. It is easy to
see from the definition that for each $m\geq 1$,
$${\frak p}_m=\psi_r(X_1^m+\cdots+X_r^m)\;\;\text{and}\;\;
{\frak q}_m=\psi_r(X_1^{-m}+\cdots+X_r^{-m}).$$
 It is well known that for each $m\in\mbz$, the element
 $X_1^m+\cdots+X_r^m$ is central in $\afbfHr$.
\end{Rem}

 Now let $\sg_1,\ldots,\sg_r$ (resp., $\tau_1,\ldots,\tau_r$)
denote the elementary symmetric polynomials in
$\phi_1,\ldots,\phi_r$ (resp., $\phi_1^{-1},\ldots,\phi_r^{-1}$),
i.e., for $1\leq s\leq r$,
$$\sg_s=\sum_{1\leq t_1<\cdots <t_s\leq
r}\phi_{t_1}\cdots\phi_{t_s}\;\;({\rm resp.}, \tau_s=\sum_{1\leq
t_1<\cdots <t_s\leq r}\phi^{-1}_{t_1}\cdots\phi^{-1}_{t_s}).$$
 Then $\sg_s,\tau_s\in {\bf Z}_\vtg(n,r)$ and
$${\bf Z}_\vtg(n,r)={\mbq}(v)[\sg_1,\ldots,\sg_r,\tau_1,\ldots,\tau_r].$$
 Since
$$\tau_r=\sg_r^{-1}\;\;\text{and}\;\;\tau_s=\sg_{r-s}\tau_r\;\;\text{for each $1\leq s<r$,}$$
 this implies that
$${\bf Z}_\vtg(n,r)={\mbq}(v)[\sg_1,\ldots,\sg_r,\sg_r^{-1}].$$

\begin{Prop} \label{basis-center} The set
$${\mathcal X}:=\{\sg_1^{\la_1}\cdots \sg_{r-1}^{\la_{r-1}}\sg_r^{\la_r}\mid
\la_1,\ldots,\la_{r-1}\in\mbn,\la_r\in\mbz\}$$
 forms a $\mbq(\up)$-basis of ${\bf Z}_\vtg(n,r)$. In other words,
${\bf Z}_\vtg(n,r)$ is a (Laurent) polynomial ring in
$\sg_1,\ldots,\sg_r,\sg_r^{-1}$.
\end{Prop}

\begin{proof} It is obvious that ${\bf Z}_\vtg(n,r)$ is
spanned by $\mathcal X$. It remains to show that $\mathcal X$ is
linearly independent.

For $\bfi=(i_1,\ldots,i_r),\bfj=(j_1,\ldots,j_r)\in\mbz^r$, we
define the lexicographic ordering $\bfi\leq_{\rm lex}\bfj$ if
$\bfi=\bfj$ or there exists $1<s\leq r$ such that
$$i_r=j_r,\ldots,i_s=j_s,i_{s-1}<j_{s-1}.$$
 Clearly, this gives a linear ordering on $\mbz^r$.

For each
$\la=(\la_1,\ldots,\la_{r-1},\la_r)\in\mbn^{r-1}\times\mbz$, define
$$\sg^\la=\sg_1^{\la_1}\cdots
\sg_{r-1}^{\la_{r-1}}\sg_r^{\la_r}.$$
 Then, by definition, for $\bfi=(i_1,\ldots,i_r)\in\mbz^r$, we
 obtain that\footnote{The linear maps $\sg_i$ and $\sg^\la$ should not be confused with the sum function
 $\sg$ defined on $M_{\vtg,n}(\mbz)$ and $\mbz_\vtg^n$ in \S1.1.}
$$\sg^{\la}(\og_{\bfi})=\og_{\bfi\ast\la}+\sum_{\bfi\ast\la<_{\rm lex}\bfj}b_\bfj \og_\bfj,$$
 where all but finitely many $b_\bfj\in\mbn$ are zero and
$$\bfi\ast\la=\bigl(i_1-\la_r
n,i_2-(\la_{r-1}+\la_r)n,\ldots,i_r-(\la_1+\cdots+\la_r)n\bigr).$$
 Note that for $\la,\mu\in\mbn^{r-1}\times\mbz$,
$$\bfi\ast\la=\bfi\ast\mu\Longleftrightarrow \la=\mu.$$

Now suppose
$$\sum_{t=1}^m a_t\sg^{\la^{(t)}}=0,$$
 where $a_t\in\mbq(\up)$ and
 $\la^{(t)}\in\mbn^{r-1}\times\mbz$ for $1\leq t\leq m$. Fix an $\bfi\in\mbz^r$.
Without loss of generality, we may suppose
$\bfi\ast{\la^{(1)}}<_{\rm lex}\bfi\ast{\la^{(t)}}$ for all $2\leq t
\leq m$. Then
$$0=\sum_{t=1}^m a_t \sg^{\la^{(t)}}(\og_\bfi)=
a_1\og_{\bfi\ast{\la^{(1)}}}+x',$$
 where $x'$ is a linear combination of $\og_\bfj$ with
 $\bfi\ast{\la^{(1)}}<_{\rm lex}\bfj$. Hence, $a_1=0$. Inductively,
 we deduce that $a_t=0$ for all $1\leq t\leq m$.
This finishes the proof.
\end{proof}

\begin{Rem} It has been proved in \cite{Yang2} that the center of the affine Schur
algebra of type $A$ is isomorphic to the polynomial algebra
${\mbq}[\sg_1,\ldots,\sg_r,\sg_r^{-1}]$. It is natural to conjecture
that ${\bf Z}_\vtg(n,r)={\mbq}(v)[\sg_1,\ldots,\sg_r,\sg_r^{-1}]$ is
exactly the center of $\afbfSr$.
\end{Rem}

Recall that $\bffke_i=\xi_r(E_i),\,
\bffkf_i=\xi_r(F_i),\,\bffkk_i=\xi_r(K_i)$ for each $1\leq i\leq n$.
 Let ${\bds \sS}_\vtg(\cycn,r)^0$ denote the subalgebra of
$\afbfSr$ generated by $\bffkk_i$ for $1\leq i\leq n$. Then, by
\cite[Cor.~4.7(1)]{DP03} (see also \cite[Lem.~3.29]{DDPW}), ${\bds
\sS}_\vtg(\cycn,r)^0$ has a basis
$$\{\bffkk_{\bfj}=\bffkk_1^{j_1}\cdots \bffkk_{n-1}^{j_{n-1}}\mid
\bfj=(j_1,\ldots,j_{n-1})\in(\mbn^{n-1})_{\leq r}\},$$
 where $(\mbn^{n-1})_{\leq r}=\{\bfj=(j_1,\ldots,j_{n-1})\in\mbn^{n-1}\mid j_1+\cdots+j_{n-1}\leq r\}$.
Let further $\wh{\bf Z}_\vtg(n,r)$ be the subalgebra of $\afbfSr$
generated by ${\bf Z}_\vtg(n,r)$ and ${\bds \sS}_\vtg(\cycn,r)^0$.

\begin{Prop} The multiplication map
$${\bf Z}_\vtg(n,r)\otimes_{\mbq(\up)}{\bds \sS}_\vtg(\cycn,r)^0
\lra\wh{\bf Z}_\vtg(n,r)$$
 is a $\mbq(\up)$-algebra isomorphism.
\end{Prop}

\begin{proof} It remains to show that the set
$$\{\sg^{\la}\bffkk_{\bfj}\mid \la\in\mbn^{n-1}\times\mbz,
\bfj\in(\mbn^{n-1})_{\leq r}\}$$
 is linearly independent. By definition, for each $\bfi\in\mbz^r$,
 $\bffkk_\bfj(\og_\bfi)=\up^{f(\bfi,\bfj)}\og_\bfi$, where $f(\bfi,\bfj)\in\mbz$
is determined by $\bfi$ and $\bfj$. Thus, by the proof of
Proposition \ref{basis-center}, we infer that for
$\la\in\mbn^{n-1}\times\mbz, \bfj\in(\mbn^{n-1})_{\leq r}$, and
$\bfi\in\mbz^r$,
$$\sg^{\la}\bffkk_{\bfj}(\og_\bfi)=\up^{f(\bfi,\bfj)}\og_{\bfi\ast\la}+
\sum_{\bfi\ast\la<_{\rm lex}\bfj}c_\bfj \og_\bfj,$$
 where all but finitely many $c_\bfj\in\mbq(\up)$ are zero.

Since $\bfi\ast\la=\bfi\ast\mu$ if and only if $\la=\mu$, it
suffices to show that for each fixed $\la$,
$\{\sg^{\la}\bffkk_{\bfj}\mid \bfj\in(\mbn^{n-1})_{\leq r}\}$ is a
linearly independent set. This follows from the fact that
$\{\bffkk_{\bfj}\mid \bfj\in(\mbn^{n-1})_{\leq r}\}$ is a linearly
independent set.
\end{proof}

By the discussion above, the center of the positive part ${\bds
\sS}_\vtg(\cycn,r)^+$ of $\afbfSr$ contains the polynomial algebra
$\mbq(\up)[\sg_1,\ldots,\sg_r]:={\bf Z}_\vtg(n,r)^+$, and the center
of the negative part ${\bds \sS}_\vtg(\cycn,r)^-$ contains the
polynomial algebra $\mbq(\up)[\tau_1,\ldots,\tau_r]:={\bf
Z}_\vtg(n,r)^-$. Moreover,
$${\bds \sS}_\vtg(\cycn,r)^+=\bfU_\vtg(n,r)^+
\cdot{\bf Z}_\vtg(n,r)^+\;\;\text{and}\;\;{\bds
\sS}_\vtg(\cycn,r)^-=\bfU_\vtg(n,r)^- \cdot{\bf Z}_\vtg(n,r)^-.$$

For each $m\geq 0$, set
$${\bds \sS}_\vtg(\cycn,r)^\pm_m={\bds
\sS}_\vtg(\cycn,r)^\pm\cap\afbfSr_m\;\;\text{and}\;\;
\bfU_\vtg(n,r)^\pm_m=\bfU_\vtg(n,r)^\pm\cap \bfU_\vtg(n,r)_m.$$
  Then we have
$${\bds \sS}_\vtg(\cycn,r)^\pm=\bigoplus_{m\geq 0}
{\bds \sS}_\vtg(\cycn,r)^\pm_m\;\;\text{and}\;\;
\bfU_\vtg(n,r)^\pm=\bigoplus_{m\geq 0}\bfU_\vtg(n,r)^\pm_m.$$
 Note that for $m\geq 1$, ${\frak p}_m\in {\bds\sS}_\vtg(\cycn,r)^+_{mn}$ and
${\frak q}_m\in {\bds\sS}_\vtg(\cycn,r)^-_{mn}$, and for $1\leq
s\leq r$, $\sg_s\in {\bds\sS}_\vtg(\cycn,r)^+_{sn}$ and $\tau_s\in
{\bds\sS}_\vtg(\cycn,r)^-_{sn}$.

\begin{Prop}\label{dim-difference} For each $m\geq 0$,
$$\dim {\bds \sS}_\vtg(\cycn,r)^\pm_m/\bfU_\vtg(n,r)^\pm_m=|\{A\in\afThnp\mid
A\;\text{is periodic with}\; \sg(A)\leq r,{\frak d}(A)=m\}|.$$
\end{Prop}

\begin{proof} We only prove the assertion for the ``$+$'' case, and the
``$-$'' case is similar.

By \cite[\S8]{DDX}, the composition subalgebra
${\boldsymbol{\mathfrak C}_\vtg(n)^+}$ of $\boldsymbol{\frak
D}_\vtg(n)^+$ has a basis
$$\{E^+_A\mid A\in\afThnp\;\text{is aperiodic}\},$$
 where $E_A^+=\ti u_A^++\sum_{B<_\deg A} \eta_B^A\,\ti u_B^+$ with
all $B$ periodic and $\eta_B^A\in\mbq(\up)$. Furthermore, by
\cite[\S6]{DD05}, for $A,B\in\afThnp$,
$$B\leq_\deg A \Longrightarrow \sg(B)\geq \sg(A).$$
 Thus, if $A\in\afThnp$ is aperiodic with $\sg(A)>r$, then by
 \eqref{def-A(j,r)},
$$\xi_r(E^+_A)=A({\bf0},r)+\sum_{B<_\deg A} \eta_B^A\, B({\bf0},r)=0.$$
Hence, for each $m\geq 0$, the set
$${\scr X}_m:=\{\xi_r(E^+_A)\mid A\in\afThnp\;\text{is aperiodic with}\; \sg(A)\leq r,{\frak d}(A)=m\}$$
 is a spanning set for $\bfU_\vtg(n,r)^+_m$. By Proposition \ref{afSrz}(1), the set
$$\{A({\bf0},r)\mid A\in\afThnp, \sg(A)\leq r,{\frak d}(A)=m\}$$
 is a basis of ${\bds \sS}_\vtg(\cycn,r)^+_m$. Hence, ${\scr X_m}$
 is a linearly independent set and, thus, is a basis of
 $\bfU_\vtg(n,r)^+_m$. This gives the desired assertion.
\end{proof}

\begin{Coro}\label{n=r case} For each $r\geq 2$,
$${\bds \sS}_\vtg(r,r)^+=\bfU_\vtg(r,r)^+\oplus\bigoplus_{t\geq 1}\mbq(\up)\sg_1^t\;\;\text{and}\;\;{\bds
\sS}_\vtg(r,r)^-=\bfU_\vtg(r,r)^-\oplus\bigoplus_{t\geq
1}\mbq(\up)\tau_1^t.$$
\end{Coro}

\begin{proof} By Proposition \ref{dim-difference},
$$\dim {\bds \sS}_\vtg(r,r)^+_m/\bfU_\vtg(r,r)^+_m
=\begin{cases}1,&\text{if $m\not=0$ and $m\equiv 0$ mod $r$,}\\
0,&\text{otherwise}.\end{cases}$$
 Since $\sg_1={\frak p}_1=\xi_r(\sfz_1^+)$, it follows from Theorem
 \ref{surjective-dHall-aff}(2) that ${\bds \sS}_\vtg(r,r)^+$
is generated by $\bffke_i$ ($1\leq i\leq r$) and $\sg_1$. This
implies the first decomposition. The second one can be proved
similarly.
\end{proof}

\section{A presentation for ${\bds\sS}_\vtg(r,r)$}

In this section we give a presentation for ${\bds\sS}_\vtg(r,r)$ by
describing explicitly a complement of $\bfU_\vtg(r,r)$ in
${\bds\sS}_\vtg(r,r)$.

We first consider the general case and recall the surjective algebra
homomorphism $\xi_r:\dHallr\ra \afbfSr$. For each
$\bfa=(a_i)\in\mbn^n$, we write in $\dHallr$,
$$u^\pm_\bfa=u^\pm_{[S_\bfa]}\;\;\text{and}\;\;\ti
u^\pm_\bfa=\up^{\sum_ia_i(a_i-1)}u^\pm_\bfa.$$
 By \eqref{non-sincere elements}, if $\bfa$ is not sincere, say $a_i=0$, then
$$\ti u^\pm_\bfa=\up^{\sum_ja_j(a_j-1)}u^\pm_\bfa=(u_{i-1}^\pm)^{(a_{i-1})}\cdots (u_1^\pm)^{(a_1)}(u_n^\pm)^{(a_n)}\cdots
(u_{i+1}^\pm)^{(a_{i+1})}\in \bfU_\vtg(n).$$
 Thus, if $\bfa$ is not sincere, then both ${\bffke}_\bfa:=\xi_r(\ti u^+_\bfa)$ and
${\bffkf}_\bfa:=\xi_r(\ti u^-_\bfa)$ lie in $\bfU_\vtg(n,r)$.

Now consider the following element $\rho$ in
$\End_{\mbq(\up)}(\bdOg^{\otimes r})$:
 $$\rho:\bdOg^{\otimes r}\lra\bdOg^{\otimes
r},\;\omega_{\bfi}\longmapsto \omega_{\bfi-\bfe_1-\cdots-\bfe_r},$$
 i.e., $\rho(\og_{i_1}\otimes\cdots \otimes\og_{i_r})=\og_{i_1-1}\otimes\cdots \otimes\og_{i_r-1}$.
 It is clear that $\rho^r=\sg_r\in{\bds\sS}_\vtg(n,r)$.

\begin{Lem} \label{description-rho} Suppose $n\geq r$. Then
$$\rho=\sum_{\bfa\in\mbn^n,\, \sg(\bfa)=r}{\bffke}_\bfa\;\;\text{and}\;\;
\rho^{-1}=\sum_{\bfa\in\mbn^n,\, \sg(\bfa)=r}{\bffkf}_\bfa.$$
 In particular, if $n>r$, then $\rho,\rho^{-1}\in\bfU_\vtg(n,r)$.
\end{Lem}

\begin{proof}  For each $\bfa\in\mbn^n$, let $Y_\bfa$ denote the set of the sequences
 $\bfj=(j_1,\ldots,j_r)$ satisfying that for each $1\leq i\leq n$,
$$a_i=|\{1\leq s\leq r\mid j_s=i\}|.$$
 By Corollary \ref{action-ss}, we have for each $\og_{i_1}\otimes\cdots
\otimes\og_{i_r}\in\bdOg^{\otimes r}$,
$$\ti u_\bfa^+\cdot(\og_{i_1}\otimes\cdots
\otimes\og_{i_r}) =\sum_{\bfj\in
Y_\bfa}\dt_{\ol{j_1},\ol{i_1-1}}\cdots\dt_{\ol{j_r},\ol{i_r-1}}\og_{i_1-1}\otimes\cdots
\otimes\og_{i_r-1},$$
 where for $m\in\mbz$, $\ol m$ denotes its residue class in $\mbz/n\mbz$. Therefore,
$$\bigl(\sum_{\bfa\in\mbn^n,\, \sg(\bfa)=r}\ti u^+_\bfa\bigr)\cdot (\og_{i_1}\otimes\cdots
\otimes\og_{i_r})=\og_{i_1-1}\otimes\cdots \otimes\og_{i_r-1},$$
 that is, the first equality holds. The second one can be proved
 similarly.

\end{proof}

 The above lemma implies that
\begin{equation}\label{relation-rho-semisimple}
{\bffke}_\dt':=\rho-{\bffke}_\dt,
\,{\bffkf}_\dt':=\rho^{-1}-{\bffkf}_\dt\in\bfU_\vtg(r,r),\end{equation}
 where $\dt=(1,\ldots,1)$. More precisely,
 \begin{equation}\label{edeltaf}
 \aligned
 {\bffke}_\dt'&=\sum_{i=1}^n\sum_{{\bfa\in\mbn^n}\atop {\sg(\bfa)=r,a_i=0}} (\bffke_{i-1})^{(a_{i-1})}\cdots (\bffke_1)^{(a_1)}(\bffke_n)^{(a_n)}\cdots
(\bffke_{i+1})^{(a_{i+1})}\\
{\bffkf}_\dt'&=\sum_{i=1}^n\sum_{{\bfa\in\mbn^n}\atop {\sg(\bfa)=r,a_i=0}} (\bffkf_{i-1})^{(a_{i-1})}\cdots (\bffkf_1)^{(a_1)}(\bffkf_n)^{(a_n)}\cdots
(\bffkf_{i+1})^{(a_{i+1})}.\endaligned
 \end{equation}

 From now on, we assume $n=r$. By \S1.4 and Theorem
 \ref{surjective-dHall-aff}(2), ${\bds \sS}_\vtg(r,r)$
is generated by ${\bffke_i},{\bffkf}_i,{\bffkk}_i$ ($1\leq i\leq
n=r$), and ${\bffke}_\dt, {\bffkf}_\dt$. It follows from
\eqref{relation-rho-semisimple} and \eqref{edeltaf} that ${\bds
\sS}_\vtg(r,r)$ is also generated by the
${\bffke_i},{\bffkf}_i,{\bffkk}_i$, and $\rho, \rho^{-1}$.
Since
$$\aligned
 {\bds \sS}_\vtg(r,r)^+_r&=\bfU_\vtg(r,r)^+_r\oplus\mbq(\up)\rho=\bfU_\vtg(r,r)^+_r\oplus\mbq(\up)\sg_1\;\;\text{and}\\
 {\bds
 \sS}_\vtg(r,r)^-_r&=\bfU_\vtg(r,r)^-_r\oplus\mbq(\up)\rho^{-1}=\bfU_\vtg(r,r)^-_r\oplus\mbq(\up)\tau_1,
\endaligned$$
 it follows from Corollary \ref{n=r case} that
\begin{equation}\label{decomposition for tri parts n=r}
{\bds \sS}_\vtg(r,r)^+=\bfU_\vtg(r,r)^+\oplus\bigoplus_{t\geq
1}\mbq(\up)\rho^t\;\;\text{and}\;\;{\bds
\sS}_\vtg(r,r)^-=\bfU_\vtg(r,r)^-\oplus\bigoplus_{t\geq
1}\mbq(\up)\rho^{-t}.\end{equation}

\begin{Rem}\label{RGreenQAgln} As a consequence of the above discussion,
we obtain \cite[Th.~3.4.8]{Gr99} which states that (for $n\geq r$
with $n\geq 3$) there is a surjective algebra homomorphism
$\al_r:{\widehat\bfU}\ra {\bds \sS}_\vtg(n,r)$ taking $R$ to
$\rho^{-1}$, where $\widehat\bfU$ is a Hopf algebra obtained from
$\bfU_\vtg(n)$ by adding primitive elements $R, R^{-1}$ with
$RR^{-1}=R^{-1}R=1$. It seems to us that $\widehat\bfU$ is not a
quantum group in a strict sense since it does not admit a triangular
decomposition.
\end{Rem}

\begin{Prop} For each $1\leq i\leq r$, we have in ${\bds \sS}_\vtg(r,r)$,
\begin{itemize}
 \item[(1)] ${\bffkk}_i{\bffke}_\dt=\up{\bffke}_\dt$;
 \item[(2)] ${\bffke}_i{\bffke}_\dt=\frac{1}{\up+\up^{-1}}
      {\bffke}_i^2{\bffke}_{i-1}\cdots{\bffke}_1{\bffke}_r\cdots {\bffke}_{i+1}$;
 \item[(3)] ${\bffkf}_i{\bffke}_\dt=\frac{1}{1-\up^{-2}}{\bffke}_{i-1}\cdots{\bffke}_1{\bffke}_r\cdots
      {\bffke}_{i+1}{\bffkk}_i^{-1}({\bffkk}_{i+1}-{\bffkk}_{i+1}^{-1})$;
 \item[(4)] ${\bffkk}_i{\bffkf}_\dt=\up{\bffkf}_\dt$;
 \item[(5)] ${\bffkf}_i{\bffkf}_\dt=\frac{1}{\up+\up^{-1}}
      {\bffkf}_i^2{\bffkf}_{i+1}\cdots{\bffkf}_r{\bffkf}_1\cdots {\bffkf}_{i-1}$;
 \item[(6)] ${\bffke}_i{\bffkf}_\dt=\frac{1}{1-\up^{-2}}{\bffkf}_{i+1}\cdots{\bffkf}_r{\bffkf}_1\cdots
      {\bffkf}_{i-1}{\bffkk}_{i+1}^{-1}({\bffkk}_i-{\bffkk}_i^{-1})$.
\end{itemize}
\end{Prop}

\begin{proof} All these relations can be deduced from the multiplication formulas in Theorem \ref{multiplication formulas in affine q-Schur algebra}.
However, we provide here a direct proof for (1) and (2). The
relations (4) and (5) can be proved in a similar manner.

By definition, we have in $\dHallrr$,
$$u_i^+u_\dt^+=u^+_{[M]}+(\up^2+1)u^+_{\dt+\bfe_i},$$
 where $M=S_1\oplus \cdots \oplus S_{i-1}\oplus S_i[2]\oplus S_{i+2}\oplus\cdots \oplus S_r.$
On the other hand,
$$\aligned
(u_i^+)^2u_{i-1}^+\cdots u^+_1u^+_r\cdots
u_{i+1}^+&=\up(\up^2+1)u^+_{[2S_i]}
u^+_{[S_1\oplus\cdots S_{i-1}\oplus S_{i+1}\oplus \cdots\oplus S_r]}\\
&=(\up+\up^{-1})(u^+_{[M]}+u^+_{\dt+\bfe_i}).
\endaligned$$
 Since $S_{\dt+\bfe_i}$ is semisimple of dimension $r+1$, it follows that
 $\xi_r(u^+_{\dt+\bfe_i})=0$. Hence,
$$\aligned
{\bffke}_i{\bffke}_\dt&=\xi_r(u_i^+u_\dt^+)=\xi_r(u^+_{[M]})
=\frac{1}{\up+\up^{-1}}\xi_r((u_i^+)^2u_{i-1}^+\cdots u^+_1u^+_r\cdots u_{i+1}^+)\\
&=\frac{1}{\up+\up^{-1}}{\bffke}_i^2{\bffke}_{i-1}\cdots{\bffke}_1{\bffke}_r\cdots
{\bffke}_{i+1},
\endaligned$$
 which gives the relation (2).

The relation (1) follows from the fact that for each
$\og_{i_1}\otimes\cdots \otimes\og_{i_r}\in\bdOg^{\otimes r}$,
$$\aligned
&u_\dt\cdot(\og_{i_1}\otimes\cdots \otimes\og_{i_r})=\sum_{\bfj\in
Y_\dt}\dt_{\ol{j_1},\ol{i_1-1}}\cdots\dt_{\ol{j_r},\ol{i_r-1}}\og_{i_1-1}\otimes\cdots
\otimes\og_{i_r-1}\\
=&\begin{cases}\og_{i_1-1}\otimes\cdots
\otimes\og_{i_r-1},\;\;\;&\text{if ${\ol i_1},\ldots,{\ol i_r}$ are pairwise distinct;}\\
0, &\text{otherwise.}\end{cases}
\endaligned$$
 Note that if ${\ol i_1},\ldots,{\ol i_r}$ are pairwise distinct, then
$$K_i\cdot(\og_{i_1}\otimes\cdots
\otimes\og_{i_r})=v\og_{i_1}\otimes\cdots \otimes\og_{i_r}.$$
\end{proof}

The above proposition together with
$\rho={\bffke}_\dt'+{\bffke}_\dt$ and
$\rho^{-1}={\bffkf}_\dt'+{\bffkf}_\dt$ gives the following result.

\begin{Coro} For each $1\leq i\leq r$, we have in ${\bds \sS}_\vtg(r,r)$,
\begin{itemize}
 \item[(QS1$'$)] $({\bffkk}_i-\up)\rho=({\bffkk}_i-\up){\bffke}'_\dt$;
 \item[(QS2$'$)] ${\bffke}_i\rho={\bffke}_i{\bffke}_\dt'+\frac{1}{\up+\up^{-1}}
      {\bffke}_i^2{\bffke}_{i-1}\cdots{\bffke}_1{\bffke}_r\cdots{\bffke}_{i+1}$;
 \item[(QS3$'$)] ${\bffkf}_i\rho={\bffkf}_i{\bffke}_\dt'+\frac{1}{1-\up^{-2}}
      {\bffke}_{i-1}\cdots{\bffke}_1{\bffke}_r\cdots
      {\bffke}_{i+1}{\bffkk}_i^{-1}({\bffkk}_{i+1}-{\bffkk}_{i+1}^{-1})$;
 \item[(QS4$'$)] $({\bffkk}_i-\up)\rho^{-1}=({\bffkk}_i-v){\bffkf}'_\dt$;
 \item[(QS5$'$)] ${\bffkf}_i\rho^{-1}={\bffkf}_i{\bffkf}_\dt'+\frac{1}{\up+\up^{-1}}
      {\bffkf}_i^2{\bffkf}_{i+1}\cdots{\bffkf}_r{\bffkf}_1\cdots {\bffkf}_{i-1}$;
 \item [(QS6$'$)] ${\bffke}_i\rho^{-1}={\bffke}_i{\bffkf}_\dt'+\frac{1}{1-\up^{-2}}
      {\bffkf}_{i+1}\cdots{\bffkf}_r{\bffkf}_1\cdots
      {\bffkf}_{i-1}{\bffkk}_{i+1}^{-1}({\bffkk}_i-{\bffkk}_i^{-1})$, \index{defining relations!({\rm QS1$'$})--({\rm QS6$'$}), ({\rm QS0$'$})}
\end{itemize}
where ${\bffke}_\dt'$ and ${\bffkf}_\dt'$ are given in \eqref{edeltaf}.
\end{Coro}

\begin{Thm}\label{n,n case} The $\mbq(\up)$-algebra $\bds{\sS}_\vtg(r,r)$ is generated by
${\bffke_i},{\bffkf}_i,{\bffkk}_i^{\pm1}$ ($1\leq i\leq r$) and
$\rho^{\pm 1}$ subject to the relations
$${\rm (QS0')}\qquad\rho\rho^{-1}=\rho^{-1}\rho=1,\,\rho{\bffke}_i\rho^{-1}
={\bffke}_{i-1},\,\rho{\bffkf}_i\rho^{-1}={\bffkf}_{i-1},\,\rho{\bffkk}_i\rho^{-1}={\bffkk}_{i-1}\hspace{3.8cm}$$
 together with the relations {\rm (QS1$'$)--(QS6$'$)} and the relations
{\rm (QS1)--(QS6)} in Theorem \ref{2nd presentation for im(afzr)}.
In particular,
$${\bds \sS}_\vtg(r,r)=\bfU_\vtg(r,r)\oplus
\bigoplus_{0\not=m\in\mbz}{\mbq(\up)}\rho^m.$$
\end{Thm}

\begin{proof} Let $\bds\sS'$ be the $\mbq(\up)$-algebra
generated by ${\frak x}_i,{\frak y}_i,{\frak z}_i^{\pm1}$ ($1\leq
i\leq r$) and $\eta^{\pm 1}$ with the relations (QS1)--(QS7) and
(QS0$'$)--(QS6$'$) (Here we replace
${\bffke_i},{\bffkf}_i,{\bffkk}_i^{\pm1}$, and $\rho^{\pm 1}$ by
${\frak x}_i,{\frak y}_i,{\frak z}_i^{\pm1}$, and $\eta^{\pm 1}$,
respectively). Thus, there is a surjective $\mbq(\up)$-algebra
homomorphism
$$\begin{array}{rl}
\Upsilon:{\bds\sS'}&\lra \bds{\sS}_\vtg(r,r)\\
{\frak x}_i &\lm {\bffke}_i\\
{\frak y}_i &\lm {\bffkf}_i\\
{\frak z}_i^\pm &\lm {\bffkk}_i^\pm\\
\eta^{\pm1}&\lm \rho^{\pm 1} .
 \end{array}$$
 Let $\bfU'$ be the $\mbq(\up)$-subalgebra of $\bds\sS'$
generated by ${\frak x}_i,{\frak y}_i$, and ${\frak z}_i^{\pm1}$
($1\leq i\leq r$). Then $\Upsilon$ induces a surjective homomorphism
$\Upsilon_1:\bfU'\ra\bfU_\vtg(r,r)$. Since the relations
(QS1)--(QS7) are the generating relations for $\bfU_\vtg(r,r)$,
there is also a natural surjective homomorphism
$\Phi:\bfU_\vtg(r,r)\ra\bfU'$. Clearly, the compositions
$\Upsilon_1\Phi$ and $\Phi\Upsilon_1$ are identity maps. Thus, both
$\Upsilon_1$ and $\Phi$ are isomorphisms.

It is clear that for each $0\not=t\in\mbz$, $\rho^t$ lies in ${\bds
\sS}_\vtg(r,r)_{rt}$. We claim that $\rho^t$ does not lie in
$\bfU_\vtg(r,r)_{rt}$. Otherwise, applying relations
(QS0$'$)--(QS6$'$) would give that
$$\rho=\rho^t\rho^{-t+1}\in \bigl(\bfU_\vtg(r,r)_{rt}\bigr)\rho^{-t+1}\subseteq
\bfU_\vtg(r,r)_r.$$
 This implies ${\bds \sS}_\vtg(r,r)=\bfU_\vtg(r,r)$, which
contradicts Lemma \ref{S(n,r)=U(n,r)}. Hence,
$\rho^t\not\in\bfU_\vtg(r,r)_{rt}$ for all $0\not=t\in\mbz$.

For each $m\in\mbz$, choose a $\mbq(\up)$-basis ${\sB}_m$ of
$\bfU_\vtg(r,r)_m$. Then the set ${\sB}:=\cup_{m\in\mbz}{\sB}_m$
forms a basis of $\bfU_\vtg(r,r)$. Moreover, by the discussion
above, the set ${\sB}\cup\{\rho^t\mid 0\not=t\in\mbz\}$ is linearly
independent in $\bds{\sS}_\vtg(r,r)$. On the other hand, by the
definition of $\bds{\sS}'$, the set
$$\{\Phi(c)\mid c\in\sB\}\cup\{\eta^m\mid 0\not=m\in\mbz\}$$
 is a spanning set of $\bds{\sS}'$. Since $\Upsilon(\Phi(c))=c$ and
$\Upsilon(\eta^m)=\rho^m$, it follows that the above spanning set is
a basis of $\bds{\sS}'$. Therefore, $\Upsilon$ is an isomorphism.
This finishes the proof.
\end{proof}

\begin{Rem} We can also present ${\bds\sS}_\vtg(r, r)$ by
using generators ${\bffke_i},{\bffkf}_i,{\bffkk}_i^{\pm1}$ ($1\leq
i\leq r$), $\bffke_\dt$ and $\bffkf_\dt$, but the relation between $\bffke_\dt$ and $\bffkf_\dt$
is not clear. 
Furthermore, if we use generators $\sg_1$ and $\tau_1$
instead of $\rho$ and $\rho^{-1}$, the relations would be much more
complicated; see the example below.
 From the relations obtained in
the case $n=r$, it seems very hard to get the relations for the
general case $n<r$.
\end{Rem}

\begin{Example} \label{presentation-n=r=2} We consider the special
case $n=r=2$. In this case, $\bfU_\vtg(2,2)=\xi_2(\bfU_\vtg(2))$ is
generated by $\bffke_i,\bffkf_i,\bffkk_i$ for $i=1,2$, and
${\boldsymbol {\bf Z}}_\vtg(2,2)={\mathbb
Q}(\up)[\sg_1,\sg_2,\sg_2^{-1}].$ Moreover, $\rho\in\End_{\mathbb
Q(\up)}(\bdOg^{\otimes 2})$ is defined by
$\rho(\omega_s\ot\omega_t)=\omega_{s-1}\ot\omega_{t-1}$ and
satisfies $\rho^2=\sg_2$. A direct calculation shows that
$$\rho=\frac{1}{\up+\up^{-1}}(\bffke_1^2+\bffke_1\bffke_2+\bffke_2\bffke_1+\bffke_2^2-
\sg_1).$$
 By \eqref{decomposition for tri parts n=r}, we obtain that
$${\bds \sS}_\vtg(2,2)^+=\bfU_\vtg(2,2)^+\oplus
\bigoplus_{m\geq 1}{\mathbb Q(\up)}\rho^m=\bfU_\vtg(2,2)^+\oplus
\bigoplus_{m\geq 1}{\mathbb Q(\up)}\sg_1^m.$$
 Similarly,
$$\rho^{-1}=\frac{1}{\up+\up^{-1}}(\bffkf_1^2+\bffkf_1\bffkf_2+\bffkf_2\bffkf_1+\bffkf
_2^2-\tau_1)$$
 and
$${\bds \sS}_\vtg(2,2)^-=\bfU_\vtg(2,2)^-\oplus
\bigoplus_{m\geq 1}{\mathbb Q(\up)}\rho^{-m}=\bfU_\vtg(2,2)^-\oplus
\bigoplus_{m\geq 1}{\mathbb Q(\up)}\tau_1^m.$$
 Hence, $\sg_1\tau_1\in\bfU_\vtg(2,2)$ and
$${\bds \sS}_\vtg(2,2)=\bfU_\vtg(2,2)\oplus
\bigoplus_{0\not=m\in\mbz}{\mathbb
Q(\up)}\rho^m=\bfU_\vtg(2,2)\oplus \bigoplus_{m\geq 1}\bigl({\mathbb
Q(\up)}\sg_1^m \oplus {\mathbb Q(\up)}\tau_1^m\bigr).$$

Furthermore, the following relations hold in ${\bds \sS}_\vtg(2,2)$:
\begin{equation}\label{n=r=2 relations}
\aligned
&\sg_1\bffke_1=\bffke_1\bffke_2\bffke_1,\;\sg_1\bffke_2=\bffke_2\bffke_1\bffke_2,\\
&\tau_1\bffkf_1=\bffkf_1\bffkf_2\bffkf_1,\;\tau_1\bffkf_2=\bffkf_2\bffkf_1\bffkf_2
,\\
 &\sg_1\bffkf_1=\bffke_1\bffke_2\bffkf_1+\bffkf_1\bffke_2\bffke_1,\;
 \sg_1\bffkf_2=\bffke_2\bffke_1\bffkf_2+\bffkf_2\bffke_1\bffke_2,\\
 &\tau_1\bffke_1=\bffkf_1\bffkf_2\bffke_1+\bffke_1\bffkf_2\bffkf_1,\;
 \tau_1\bffke_2=\bffkf_2\bffkf_1\bffke_2+\bffke_2\bffkf_1\bffkf_2,\\
 &\sg_1(\bffkk_i-\up)=(\bffke_1\bffke_2+\bffke_2\bffke_1)(\bffkk_i-\up),\;i=1,2,\\
 &\tau_1(\bffkk_i-\up)=(\bffkf_1\bffkf_2+\bffkf_2\bffkf_1)(\bffkk_i-\up),\;i=1,2.
\endaligned
\end{equation}

\end{Example}

\begin{Rem} \label{Counterexample}
Recall from \S2.2 that $\fD_\vtg(n)$ is the $\sZ$-subalgebra of $\dHallr$
generated by $K_i^{\pm1}$, $\big[{K_i;0\atop t}\big]$, $\sfz_s^+$,
$\sfz^-_s$, $(u_i^+)^{(m)}$ and $(u_i^-)^{(m)}$ for $i\in I$ and
$s,t,m\geq 1$. By Proposition \ref{afzrpm} and Corollary
\ref{central-elem-integral}, all $\xi_r(\sfz^\pm_s)$ lie in $\afSr$.
Thus, the $\mbq(\up)$-algebra homomorphism $\xi_r:\dHallr\to\afbfSr$
induces a $\sZ$-algebra homomorphism
$\xi_{r,\sZ}:\fD_\vtg(n)\to\afSr$. By Corollary
\ref{n>r-integral-surjectivity}, $\xi_{r,\sZ}$ is surjective in case
$n>r$. However, it is in general not surjective as shown below.

Let
$n=r=2$. By Theorem \ref{n,n case},
$${\bds \sS}_\vtg(2,2)=\bfU_\vtg(2,2)\oplus
\bigoplus_{0\not=m\in\mbz}{\mbq(\up)}\rho^m,$$
 where $\rho=\sum_{\bfa\in\mbn^2,\, \sg(\bfa)=2}{\bffke}_\bfa$ and
$\rho^{-1}=\sum_{\bfa\in\mbn^2,\, \sg(\bfa)=2}{\bffkf}_\bfa$; see
Lemma \ref{description-rho}. By definition, ${\bffke}_\bfa=\xi_r(\ti
u^+_\bfa)$ and ${\bffkf}_\bfa=\xi_r(\ti u^-_\bfa)$. It follows that
$\rho$ and $\rho^{-1}$ lie in $\sS_\vtg(2,2)$. Using an argument
similar to the proof of Corollary \ref{n>r-integral-surjectivity},
we can show that $\sS_\vtg(2,2)$ can be generated by
$\bffke_i^{(m)},\bffkf_i^{(m)}$ ($i=1,2, m\geq 1$), $\bfone_\la$
($\la\in\La_\vtg(2,2)$), $\rho$ and $\rho^{-1}$. Thus, we obtain
that
$$\sS_\vtg(2,2)=U_\vtg(2,2)\oplus
\bigoplus_{0\not=m\in\mbz}\sZ\rho^m,$$
 where $U_\vtg(2,2)$ is the $\sZ$-subalgebra of ${\bds \sS}_\vtg(2,2)$
generated by $\bffke_i^{(m)},\bffkf_i^{(m)},\bfone_\la$ for $i=1,2$,
$m\geq 1$, and $\la\in\La_\vtg(2,2)$.

On the other hand, the image ${\rm Im}\,\xi$ for
$\xi:=\xi_{2,\sZ}:\fD_\vtg(2)\to\sS_\vtg(2,2)$ is the
$\sZ$-subalgebra of $\sS_\vtg(2,2)$ generated by $\bffkk_i^{\pm1}$,
$\big[{\bffkk_i;0\atop t}\big]$, $\xi(\sfz_s^+)$, $\xi(\sfz^-_s)$,
$\bffke_i^{(m)}$ and $\bffkf_i^{(m)}$ for $i=1,2$ and $s,t,m\geq 1$.
Since $\sg_1=\xi(\sfz_1^+)$ and $\tau_1=\xi(\sfz_1^-)$, we have by
the example above that
$$\xi(\sfz_1^+)=-(\up+\up^{-1})\rho+(\bffke_1^2+\bffke_1\bffke_2+\bffke_2\bffke_1+\bffke_2^2)$$ and
$$\xi(\sfz_1^-)=-(\up+\up^{-1})\rho^{-1}+(\bffkf_1^2+\bffkf_1\bffkf_2+\bffkf_2\bffkf_1+\bffkf
_2^2).$$
 Furthermore, for each $s\geq 1$,
$$\xi(\sfz_s^\pm)=\left\{\begin{array}{ll}
                 (\xi(\sfz_1^\pm))^s-\sum_{t=1}^{(s-1)/2}\big({s\atop t}\big)\xi(\sfz_{s-2t}^\pm)\rho^{\pm 2t},\;\;& \text{if $s$ is odd;}\\
                 (\xi(\sfz_1^\pm))^s-\sum_{t=1}^{s/2}\big({s\atop t}\big)\xi(\sfz_{s-2t}^\pm)\rho^{\pm 2t},\;\;& \text{if $s$ is even}.
                 \end{array}\right.$$
 By \eqref{n=r=2 relations}, all the elements
$$\bffke_i\rho^{\pm 1},\, \bffkf_i\rho^{\pm 1},\, \rho^{\pm
1}\bffke_i, \,\rho^{\pm 1}\bffke_i$$
 lie in $\bfU_\vtg(2,2)$ for $i=1,2$. Thus,
an inductive argument implies that for each $s\geq 1$,
$$\xi(\sfz_s^\pm)\equiv (-1)^s(\up^s+\up^{-s})\rho^{\pm s}\;\,{\rm
mod}\;\,U_\vtg(2,2).$$
 From \eqref{n=r=2 relations} it also follows that
$$\bffkk_i\rho^{\pm 1}\equiv \up \rho^{\pm 1}\equiv \rho^{\pm 1}\bffkk_i\;\,{\rm
mod}\;\,\bfU_\vtg(2,2)\;\;\text{for $i=1,2$.}$$
 We conclude that neither of $\rho$ and $\rho^{-1}$ lies in ${\rm Im}\,\xi$.
Therefore, the $\sZ$-algebra homomorphism
$\xi=\xi_{2,\sZ}:\fD_\vtg(2)\to\sS_\vtg(2,2)$ is not surjective.
\end{Rem}

\section{The realization conjecture}

We now look at the realization problem for quantum affine
$\frak{gl}_n$. In the non-affine case,
Beilinson--Lusztig--MacPherson \cite{BLM} provided a construction
for quantum $\frak{gl}_n$ via quantum Schur algebras. In order to
generalize the BLM approach to the affine case, a modified BLM
approach has been introduced in \cite{DF09}. On the one hand,
this approach produces a realization of quantum $\frak{gl}_n$
which serves as a ``cut-down'' version of the original BLM
realization. On the other hand, most of the constructions in this
approach can be generalized to the affine case. The conjecture
proposed below is a natural outcome from this consideration; see \cite[5.2]{DF09}.

Let  $\afsK$ be the $\sZ$-algebra which has $\sZ$-basis
$\{[A]\}_{A\in\afThn}$ and multiplication defined by $[A]\cdot[B]=0$
if $\co(A)\neq\ro(B)$, and $[A]\cdot[B]$ as given in $\afbfSr$ if
$\co(A)=\ro(B)$ and $r=\sg(A)$, where $\co(A)$ (resp., $\ro(A)$) is
the column (resp., row) sum vector associated with $A$ (see \S1.1).
\index{$\co(A)$, column sum vector} \index{$\ro(A)$, row sum vector}
This algebra has no identity but infinitely many idempotents
$[\diag(\la)]$ for all $\la\in\afmbnn$. Moreover,
$\afsK\cong\oplus_{r\ge 0}\afSr$. (Note that $\sS_\vtg(n,0)=\sZ$
with  a basis labeled by the zero matrix.)

Let $\afbsK=\afsK_{\mbq(\up)}$ and let $\afhbsK$ be the vector space
of all formal (possibly infinite) $\mbq(\up)$-linear combinations
$\sum_{A\in\afThn}\beta_A[A]$ which have the following properties:
for any ${\bf x}\in\afmbnn$,
\begin{equation}\label{(F)}\text{the sets }
\aligned &\{A\in\afThn \mid  \beta_A\neq0,\ \ro(A)={\bf
x}\}\\&\{A\in\afThn \mid  \beta_A\neq0,\ \co(A)={\bf
x}\}\\
\endaligned \text{ are finite.}
\end{equation}
In other words, for $\la,\mu\in\mbn^n_\vtg$, both
$$\sum_{A\in\afThn}\beta_A [\diag(\la)]\cdot[A]\;\;\text{and}\;\;
\sum_{A\in\afThn}\beta_A [A]\cdot[\diag(\mu)]$$
 are finite sums. Thus, there is a well-defined multiplication on
$\afhbsK$ by setting
$$\bigl(\sum_{A\in\afThn}\al_A [A]\bigr)\cdot\bigl(\sum_{B\in\afThn}\beta_B [B]\bigr):
=\sum_{A,B\in\afThn}\al_A\beta_B [A][B].$$
 This defines an associative algebra structure on $\afhbsK$. This
algebra has an identity element
$\sum_{\la\in\mbn^n_\vtg}[\diag(\la)]$, the sum of all $[D]$ with
$D$ a diagonal matrix in $\afThn$, and contains $\afbsK$ as a
natural subalgebra without identity. Note that $\afhbsK$ is
isomorphic to the direct product algebra $\prod_{r\geq0}\afbfSr$,
and that the anti-involutions $\tau_r$ given in \eqref{taur} induce the algebra anti-involution
\begin{equation}\label{tau}
\tau:=\sum_r\tau_r:\afhbsK\lra\afhbsK,
\qquad\sum_{A\in\afThn}\beta_A[A]\longmapsto\sum_{A\in\afThn}\beta_A[\tA].
\end{equation}

For
$A\in\afThnpm$ and $\bfj\in\afmbzn$, define  $A(\bfj)\in \afhbsK$ by
$$
A(\bfj)=\sum_{\la\in\afmbnn}\up^{\la\centerdot\bfj}[A+\diag(\la)],
$$ where $\la\centerdot\bfj=\sum_{1\leq i\leq n}\la_ij_i$.
Clearly, $A(\bfj)=\sum_{r\ge0}A(\bfj,r)$.

Let $\bsA_\vtg(n)$ be the subspace of $\afhbsK$ spanned by $A(\bfj)$
for all $A\in\afThnpm$ and  $\bfj\in\afmbzn$. Since the structure
constants (with respect to the BLM basis) appeared in the
multiplication formulas given in Theorem \ref{multiplication
formulas in affine q-Schur algebra} are independent of $r$, we
obtain immediately the following similar formulas in $\bsA_\vtg(n)$.
Here, again for the simplicity of the statement, we extend the
definition of $A(\bfj)$ to all the matrices in $M_{n,\,\vtg}(\mbz)$
by setting $A(\bfj)=0$ if some off-diagonal entries of $A$ are
negative.

\begin{Thm}\label{multiplication formulas for A(j)}
Maintain the notations used in Theorem \ref{multiplication formulas
in affine q-Schur algebra}. The following multiplication formulas
hold in $\bsA_\vtg(n)$:
\begin{equation}\label{formula1 (1) in the affine q-Schur algebra}
\begin{split}
0(\bfj)A(\bfj')&=\up^{\bfj\centerdot\ro(A)}A(\bfj+\bfj'),\\
A(\bfj')0(\bfj)&=\up^{\bfj\centerdot\co(A)}A(\bfj+\bfj'),
\end{split}
\end{equation}
where $0$ stands for the zero matrix,

\begin{equation}\label{formula1 (2) in the affine q-Schur algebra}
\begin{split}
 \afE_{h,h+1}(\bfl)&A(\bfj)=\sum_{i<h;a_{h+1,i}\geq 1}\up^{f(i)}\ol{\left[\!\!\left[{a_{h,i}+1
\atop 1}\right]\!\!\right]}(A+\afE_{h,i}-\afE_{h+1,i})(\bfj+\afal_h)\\
&+\sum_{i>h+1;a_{h+1,i}\geq
1}\up^{f(i)}\ol{\left[\!\!\left[{a_{h,i}+1\atop 1
}\right]\!\!\right]}(A+\afE_{h,i}-\afE_{h+1,i})(\bfj)\\
&+\up^{f(h)-j_h-1}\frac{(A-\afE_{h+1,h})(\bfj+\afal_h)-(A-\afE_{h+1,h})(\bfj+\afbt_h)}{1-\up^{-2}}\\
&+\up^{f(h+1)+j_{h+1}}\ol{\left[\!\!\left[{a_{h,h+1}+1\atop 1
}\right]\!\!\right]}(A+\afE_{h,h+1})(\bfj),
\end{split}
\end{equation}
\begin{equation}\label{formula1 (3) in the affine q-Schur algebra}
\begin{split}
 \afE_{h+1,h}(\bfl)&A(\bfj)=\sum_{i<h;a_{h,i}\geq
 1}\up^{f'(i)}\ol{\left[\!\!\left[{a_{h+1,i}+1\atop 1
}\right]\!\!\right]}(A-\afE_{h,i}+\afE_{h+1,i})(\bfj)\\
&+\sum_{i>h+1;a_{h,i}\geq
1}\up^{f'(i)}\ol{\left[\!\!\left[{a_{h+1,i}+1\atop 1
}\right]\!\!\right]}(A-\afE_{h,i}+\afE_{h+1,i})(\bfj-\afal_h)\\
&+\up^{f'(h+1)-j_{h+1}-1}\frac{(A-\afE_{h,h+1})
(\bfj-\afal_h)-(A-\afE_{h,h+1})(\bfj+\afbt_h)}{1-\up^{-2}}\\
&+\up^{f'(h)+j_{h}}\ol{\left[\!\!\left[{a_{h+1,h}+1\atop 1
}\right]\!\!\right]}(A+\afE_{h+1,h})(\bfj).
\end{split}
\end{equation}
\end{Thm}

There is a parallel construction in the non-affine case as discussed
in \cite[\S8]{DF09}. By removing the sub/superscripts $\vtg$, we
obtain the corresponding objects and multiplication formulas in this
(non-affine) case. Since the quantum $\mathfrak{gl}_n$, denoted by
$\bfU(n)$, is generated by $E_i,F_i$ and $K_j^{\pm1}$ ($1\leq i\leq
n-1,1\leq j\leq n$), the corresponding multiplication formulas
define an algebra isomorphism $\bfU(n)\to{\mathfrak A}(n)$ sending
$E_h$ to $E_{h,h+1}(\bfl)$, $F_h$ to $E_{h+1,h}(\bfl)$ and $K_j\to
0(\boldsymbol e_j)$. In this way, we obtain a realization of the
quantum $\mathfrak{gl}_n$. This is the so-called modified approach
introduced in \cite{DF09}. In this approach, the algebra
$\bds\sK(n)$ as a direct sum of all quantum Schur algebras
$\bds\sS(n,r)$ is just a homomorphic image of the BLM algebra
$\mathbf K$ constructed in \cite{BLM}, which is isomorphic to the
modified quantum group $\dot\bfU(n)$. However, since we avoid using
the stabilization property (see \cite[\S4]{BLM}) required in the
construction of $\mathbf K$, it can be generalized to obtain the
affine construction above.

The $\mbq(\up)$-space $\bsA_\vtg(n)$ is a natural candidate for a
realization of $\dHallr\cong\bfU(\afgl)$. Since $\dHallr$ has
generators other than simple generators, the formulas in Theorem
\ref{multiplication formulas for A(j)} are not sufficient to show
that $\bsA_\vtg(n)$ is a subalgebra. However, as seen in \cite[5.4
\&7.5]{DF09}, these formulas are sufficient to embed the subalgebras
$\bfU_\vtg(n)$, $\Hallpi$ and $\Hallmi$ into $\bsA_\vtg(n)$. Thus,
it is natural to formulate the following conjecture.

\begin{Conj}\label{realization conjecture} The $\mbq(\up)$-space
$\bsA_\vtg(n)$ is a subalgebra of $\afhbsK$ which is isomorphic to
 $\dHallr$, the quantum $\afgl$.\index{realization conjecture}
\end{Conj}

 As seen in \S1.4, there are three types of extra generators for $\dHallr$. We expect to
 derive more multiplication formulas between these extra generators and the BLM basis elements
 for $\afbfSr$ (see, e.g., Problem \ref{Prob-realization}) and, hence, to prove the conjecture.
As a first test, we will establish the conjecture for the classical
($\up=1$) case in the next chapter. We end this section with a
formulation of the conjecture in this case.

 Let $\sZ\to\mbq$ be the
specialization by sending $\up$ to 1, and let $\afhsKq$ be the
vector space of all formal (possibly infinite) $\mbq$-linear
combinations $\sum_{A\in\afThn}\beta_A[A]_1$ satisfying \eqref{(F)}.
Here $[A]_1$ denote the image of $[A]$ in $\afSr_\mbq$. For any
$r>0$, $A\in\afThnpm$ and ${\bf j}\in \afmbnn$, define in
$\afSr_\mbq$ (cf. \cite[(3.0.3)]{Fu09})
\begin{equation}\label{A[j,r]}
\begin{split}
A[{\bf j},r]&=\begin{cases} \sum_{\la\in\La_\vtg(n,r-\sg(A))}
\la^\bfj[A+\diag(\la)]_1,&\text{ if }\sg(A)\leq r;\\
0,&\text{ otherwise.}\end{cases}\\
\end{split}
\end{equation}
where $\la^\bfj=\prod_{i=1}^n\la_i^{j_i}$. (These elements play a role similar to the elements
$A(\bfj,r)$ for affine quantum Schur algebras. See \S6.2 below for more discussion of the elements.)
Let
\begin{equation}\label{A[j]}
A[{\bf j}]=\sum_{r=0}^\infty A[\bfj,r]\in \afhsKq.
\end{equation}
Then the classical version of Conjecture \ref{realization conjecture} claims that
the $\mbq$-span $\fA_{\vtg,\mbq}[n]$ of all $A[\bfj]$ is a subalgebra
of $\afhsKq$ which is isomorphic to the algebra
$$\overline{\fD_\vtg(n)}_\mbq:=\fD_\vtg(n)_\mbq/\lan K_i-1\mid 1\leq i\leq
n\ran,$$
where $\fD_\vtg(n)$ is the integral form defined at the end of \S2.2. We will prove in \S6.1 that
the specialized algebra $\overline{\fD_\vtg(n)}_\mbq$ is isomorphic to the universal enveloping algebra
$\sU(\afgl)$ over $\mbq$ introduced in \S1.1.

\section{Lusztig's transfer maps and semisimple generators}

In \cite{L00} Lusztig defined an algebra homomorphism
$\phi_{r,r-n}:\afbfS(n,r)\ra \afbfS(n,r-n)$ for $r\geq n$, called
the {\it transfer map}\index{transfer map}. In this section, we
describe the images of $A(\bfl,r)$ and $\tA(\bfl,r)$ under
$\phi_{r,r-n}$ for all $A=A_\la=\sum_{i=1}^n\la_iE^\vtg_{i,i+1}
\in\afTh(n)$.

Throughout this section, for a prime power $q$,
$\sS_\vtg(n,r)_{\mbc}=\sS_\vtg(n,r)\otimes_{\sZ}\mbc$ always denotes
the specialization of $\sS_\vtg(n,r)$ by at $\up=\sqrt{q}$. In other
words, $\sS_\vtg(n,r)_{\mbc}$ is a $\mbc$-vector space with a basis
$\{e_A=e_A\ot 1\mid A\in\afThnr\}$; see Definition \ref{generic
affine quantum Schur algebra}.

We first recall from \cite[\S1]{L00} the definition of
$\phi_{r,r-n}$.

 As in \S3.1, let $V$ be a free
$\field[\ep,\ep^{-1}]$-module of rank $r$ and let $\afFn={\scr
F}_{\vtg,n}(V)$ be the set of all cyclic flags
$\bfL=(L_i)_{i\in\mbz}$, where each $L_i$ is a lattice in $V$ such
that $L_{i-1}\han L_i$ and $L_{i-n}=\ep L_i$ for all $i\in\mbz$. If
$\field$ is the finite field with $q$ elements, then we use the
notation $\scrY=\afFn(q)$ for $\afFn$ and $\sS_\vtg(n,r)_{\mbc}$ is
identified with $\mbc_G(\scrY\times\scrY)$, where $G=G_V$ is the
group of automorphisms of the $\field[\ep,\ep^{-1}]$-module $V$.
Given $\bfL=(L_i),\widetilde\bfL=(\widetilde L_i)\in\afFn$ with
$\widetilde\bfL\subseteq\bfL$, i.e., $\widetilde L_i\subseteq L_i$
for all $i\in\mbz$, $\bfL/\widetilde\bfL$ can be viewed as a
nilpotent representation of the cyclic quiver $\triangle(n)$; see
\S3.6.

As in \eqref{semisimple module}, for each $\la=(\la_i)_{i\in\mbz}\in
\mbn_\vtg^n$, let $S_\la=\oplus_{1\leq i\leq n}\la_iS_i$, where
$S_i$ denotes the simple representation of $\triangle(n)$
corresponding to the vertex $i$, and set $A_\la=\sum_{1\leq i\leq
n}\la_i E^\vtg_{i,i+1}\in\Th_\vtg(n)$.

For $A\in\afThnpm$ and $\bfj\in\afmbzn$, let $A(\bfj,r)\in
\afSr$ be the BLM basis elements defined as in \eqref{def-A(j,r)}.
 We also view $A(\bfj,r)$ as an element in $\sS_\vtg(n,r)_{\mbc}$
via specializing $\up$ to $\sqrt{q}$.

By Corollary \ref{A(bfj,r) acts on (bfL,bfL')}, for all
$\la\in\mbn_\vtg^n$ and $\bfL=(L_i),\widetilde\bfL=(\widetilde
L_i)\in\afFn$, we have
\begin{equation}\label{action1}
A_\la(\bfl,r)(\bfL,\widetilde\bfL)=\left\{\begin{array}{ll}
            \up^{-c(\bfL,\widetilde\bfL)},\;\;&\text{if $\widetilde\bfL\subseteq\bfL$
            and $\bfL/\widetilde\bfL\cong S_\la$;}\\
            0, & \text{otherwise,}\end{array}\right.
\end{equation}
 where
$$\aligned
c(\bfL,\widetilde\bfL)&=\sum_{1\leq i\leq n}\dim (L_{i+1}/\widetilde
L_{i+1})
\bigl(\dim (\widetilde L_{i+1}/\widetilde L_i)-\dim (L_i/\widetilde L_i)\bigr)\\
&=\sum_{{1\leq i\leq n}}\dim (L_i/\widetilde L_i) \bigl(\dim
(L_i/L_{i-1})-\dim (L_i/\widetilde L_i)\bigr).
\endaligned$$
Moreover,
\begin{equation}\label{action2}
0(\la,r)(\bfL,\widetilde\bfL)=\left\{\begin{array}{ll}
            \up^{\la\centerdot\mu},\;\;&\text{if $\widetilde\bfL=\bfL$;}\\
            0, & \text{otherwise,}\end{array}\right.
\end{equation}
 where $\mu=(\mu_i)\in\mbn_\vtg^n$ with $\mu_i=\dim L_i/L_{i-1}$.

Now let $r,r',r''\geq 0$ satisfy $r=r'+r''$. Let $V''$ be a direct
summand of $V$ of rank $r''$. Set $V'=V/V''$. Define $\pi':{\scr
F}_{\vtg,n}(V)\ra {\scr F}_{\vtg,n}(V'), \bf L\mapsto\bf L'$ (resp.,
$\pi'':{\scr F}_{\vtg,n}(V)\ra {\scr F}_{\vtg,n}(V''), \bf
L\mapsto\bf L''$) by setting
$$L_i'=(L_i+V'')/V''\;\; \text{(resp., $L_i''=L_i\cap V''$)}$$
for all $i\in\mbz$. Thus, for each $i\in\mbz$, there is an exact
sequence
$$0\lra L_i''/L_{i-1}''\lra L_i/L_{i-1}\lra L_i'/L_{i-1}'\lra 0$$
of $\field$-vector spaces.

Let $\field$ be the finite field with $q$ elements. Following
\cite[1.2]{L00}, let $\Xi:\sS_\vtg(n,r)_{\mbc}\ra
\sS_\vtg(n,r')_{\mbc}\otimes \sS_\vtg(n,r'')_{\mbc}$ be the map
defined by\footnote{The map $\Xi$ is denoted by $\Dt$ in
\cite{L00}.}
$$\Xi(f)(\bfL',\widetilde\bfL',\bfL'',\widetilde\bfL'')=
\sum_{{\widetilde\bfL}\in{\scr
F}_{\vtg,n}(V)}f(\bfL,\widetilde\bfL)$$
 for $f\in \sS_\vtg(n,r)_{\mbc}$, where
$\bfL',\widetilde\bfL'\in {\scr F}_{\vtg,n}(V')$,
$\bfL'',\widetilde\bfL''\in{\scr F}_{\vtg,n}(V'')$, and $\bfL$ is a
fixed element in ${\scr F}_{\vtg,n}(V)$ satisfying
$\pi'(\bfL)=\bfL'$ and $\pi''(\bfL)=\bfL''$, and the sum is taken
over all $\widetilde\bfL\in {\scr F}_{\vtg,n}(V)$ such that
$\pi'(\widetilde\bfL)=\widetilde\bfL'$ and
$\pi''(\widetilde\bfL)=\widetilde\bfL''$.

For two lattices $L,L'$ in $V$, define
$$(L:L')=\dim_\field (L/\widetilde L)-\dim_\field(L'/\widetilde L),$$
 where $\widetilde L$ is a lattice contained in $L\cap L'$. For
$\bfL=(L_i),\bfL'=(L'_i)\in{\scr F}_{\vtg,n}(V)$, set
$$(\bfL:\bfL')=\sum_{i=1}^n (L_i:L_i').$$
 We finally define a $\mbc$-linear isomorphism
$\xi:\sS_\vtg(n,r)_{\mbc}\ra\sS_\vtg(n,r)_{\mbc}$ by
$$\xi(f)(\bfL,\bfL')=q^{(\bfL:\bfL')/2}f(\bfL,\bfL')$$
for $f\in \sS_\vtg(n,r)_{\mbc}$ and $\bfL,\bfL'\in{\scr
F}_{\vtg,n}(V)$. Then $\xi$ is an algebra isomorphism; see
\cite[1.7]{L00}.

As before, let $\dt=(\dt_i)\in\mbn^n_\vtg$ with all $\dt_i=1$. For
each $A=(a_{i,j})\in\afTh(n,n)$, if $\ro(A)\not=\dt$ or
$\co(A)\not=\dt$, then we set ${\rm sgn}_A=0$. Suppose now
$\ro(A)=\co(A)=\dt$. Then there is a unique permutation
$w:\mbz\ra\mbz$ such that $a_{i,j}=\dt_{w(j),i}$. In this case, we
set ${\rm sgn}_A=(-1)^{{\rm Inv}(w)}$, where
$${\rm Inv}(w)=|\{(i,j)\in\mbz^2\mid 1\leq i\leq n, i<j,w(i)>w(j)\}|$$
(cf. \S3.2).
 By \cite[1.8]{L00}, there is an algebra homomorphism
$\chi:\sS_\vtg(n,n)_{\mbc}\ra\mbc$ such that $\chi(e_{A})={\rm
sgn}_A$. In particular, for $1\leq i\leq n$ and $\la\in\mbz^n_\vtg$,
$$\chi(E_{i,i+1}^\vtg(\bfl,n))=0,\;\chi(E_{i+1,i}^\vtg(\bfl,n))=0,\;
\chi(0(\la,n))=q^{(\la\centerdot\dt)/2}.$$
 Suppose $r\geq n$ and let $\phi^q_{r,r-n}$ denote the composition
$$\sS_\vtg(n,r)_{\mbc}\stackrel{\Xi}{\lra}
\sS_\vtg(n,r-n)_{\mbc}\otimes
\sS_\vtg(n,n)_{\mbc}\stackrel{\xi\otimes\chi}{\lra}\sS_\vtg(n,r-n)_{\mbc}\otimes
\mbc=\sS_\vtg(n,r-n)_{\mbc}.$$

Lusztig \cite{L00} showed that for each pair $A\in \afThnr$ and
$B\in\Th_\vtg(n,r-n)$, there is a uniquely determined polynomial
$f_{A,B}(\up,\up^{-1})\in\sZ=\mbz[\up,\up^{-1}]$ such that for each
finite field $\field$ of $q$ elements,
$$\phi_{r,r-n}^q(q^{-d_A/2}e_{A})=\sum_{B\in \Th_\vtg(n,r-n)}
f_{A,B}(q^{1/2},q^{-1/2})q^{-d_B/2}e_{B}.$$
 This gives a $\mathbb Q(\up)$-algebra homomorphism
$\phi_{r,r-n}:\afbfS(n,r)\ra\afbfS(n,r-n)$ defined by setting
$$\phi_{r,r-n}([A])=\sum_{B\in
\Th_\vtg(n,r-n)}f_{A,B}(\up,\up^{-1})[B],$$
 which is called the {\it transfer map}. Moreover,
$\phi_{r,r-n}$ takes
$$\aligned
E_{i,i+1}^\vtg(\bfl,r)&\longmapsto E_{i,i+1}^\vtg(\bfl,r-n),\\
E_{i+1,i}^\vtg(\bfl,r)&\longmapsto E_{i+1,i}^\vtg(\bfl,r-n),\\
0(\la,r)&\longmapsto \up^{\la\centerdot\dt}0(\la,r-n)
\endaligned$$
 for $i\in\mbz$ and $\la\in\mbz^n_\vtg$.

By the definition, for each $i\in I$, we have
$$\phi_{r,r-n}\xi_r(K_i)=\up\, 0(\afbse_i,n-r)\not=0(\afbse_i,n-r)=\xi_{r-n}(K_i).$$
 Thus, $\phi_{r,r-n}\xi_r\not=\xi_{r-n}$. However, if we view
$\bfU(\afsl)$ as a subalgebra of $\dHallr$ generated by $E_i,F_i,
\ti K_i$ for $i\in I$ and denote the restriction of $\xi_r$ (resp.,
$\xi_{r-n}$) to $\bfU(\afsl)$ still by $\xi_r$ (resp., $\xi_{r-n}$),
then $\phi_{r,r-n}\xi_r=\xi_{r-n}$, i.~e., the following triangle
commutes
\begin{center}
\begin{pspicture}(0,-0.1)(4,2.2)
\psset{xunit=.8cm,yunit=.7cm} \uput[l](1,1.4){$\bfU(\afsl)$}
\uput[u](4.5,2){$\afbfS(n,r)$} \uput[d](4.8,0.8){$\afbfS(n,r-n)$}
\uput[r](4.65,1.4){$\phi_{r,r-n}$} \psline{->}(4.6,2)(4.6,0.75)
\psline{->}(1.1,1.5)(3.5,2.4)
\psline{->}(1.1,1.3)(3.5,0.45)\uput[u](2.4,2){$\xi_r$}
\uput[d](2.4,0.95){$\xi_{r-n}$}
\end{pspicture}
\end{center}
 As a result, there is an induced algebra homomorphism
 $\bfU(\afsl)\ra\underset{\longleftarrow}\lim\afbfS(n,n+m)$; see \cite[3.4]{L00}.
As seen above (a fact pointed out by Lusztig in \cite{L00}), if
$\bfU(\afsl)$ is replaced by $\bfU_\vtg(n)$, then the diagram above
is not commutative. Hence, the homomorphism
$\bfU(\afsl)\ra\underset{\longleftarrow}\lim\afbfS(n,n+m)$ cannot be
extended to $\bfU_\vtg(n)$, nor to $\dHallr$. However, it is natural
to ask if this homomorphism can be extended to the double
Ringel--Hall algebra $'\dHallr$ introduced in Remarks
\ref{subalgebra quantum gl}(3), which has the same 0-part as
$\bfU(\afsl)$. We now show below that this is not the case.

It is known from \S3.7 that $\afbfS(n,r)$ is generated by the
$A_\la(\bfl,r)$, $\tA_\la(\bfl,r)$ and $0(\afbse_i,r)$ for
$\la\in\mbn^n_\vtg$ and $i\in I$. In the following we describe the
images of $A_\la(\bfl,r)$ and $\tA_\la(\bfl,r)$ under
$\phi_{r,r-n}$. As in \S2.1, for $\la=(\la_i)\in\mbz^n_\vtg$, define
$\tau\la\in \mbz_\vtg^n$ by setting $(\tau\la)_i=\la_{i-1}$ for all
$i\in \mbz$.

\begin{Prop} Keep the notation above. Let $\field$ be the finite field
with $q$ elements. For each $\la=(\la_i)\in\mbn_\vtg^n$, we have
$$\Xi(A_\la(\bfl,r))=\sum_{{\mu,\nu\in\mbn^n_\vtg}\atop{\mu+\nu=\la}}
q^{(\mu\centerdot\tau\nu-\tau\mu\centerdot\nu)/2}A_\mu(\bfl,r')0(\nu,r')\otimes
A_\nu(\bfl,r'')0(-\mu,r'').$$
 Dually, we have
$$\Xi(\tA_\la(\bfl,r))=\sum_{{\mu,\nu\in\mbn^n_\vtg}\atop{\mu+\nu=\la}}
q^{(\mu\centerdot\tau\nu-\tau\mu\centerdot\nu)/2}\cdot\tA_\mu(\bfl,r')0(\tau\nu,r')\otimes
\tA_\nu(\bfl,r'')0(-\tau\mu,r'').$$
\end{Prop}

\begin{proof} We only prove the first formula. The second one can be
proved similarly.

For $\mu+\nu=\la$ with $\mu,\nu\in\mbn^n_\vtg$, we write
$$\Psi_{\mu,\nu}=A_\mu(\bfl,r')0(\nu,r')\otimes
A_\nu(\bfl,r'')0(-\mu,r'').$$
 Thus, it suffices to show that for all fixed $\bfL',\widetilde\bfL'\in{\scr F}_{\vtg,n}(V')$
and $\bfL'',\widetilde\bfL''\in {\scr F}_{\vtg,n}(V'')$,
$$\Xi(A_\la(\bfl,r))(\bfL',\widetilde\bfL',\bfL'',\widetilde\bfL'')
=\sum_{{\mu,\nu\in\mbn^n_\vtg}\atop{\mu+\nu=\la}}
q^{(\mu\centerdot\tau\nu-\tau\mu\centerdot\nu)/2}\Psi_{\mu,\nu}(\bfL',\widetilde\bfL',\bfL'',\widetilde\bfL'').$$

By the definition and \eqref{action1}, we obtain that
$$\Xi(A_\la(\bfl,r))(\bfL',\widetilde\bfL',\bfL'',\widetilde\bfL'')=
\sum_{\widetilde\bfL}A_\la(\bfl,r)(\bfL,\widetilde\bfL)
=q^{(-\sum_{1\leq i\leq n}\la_i(\dim L_i/L_{i-1}-\la_i))/2}\ell,$$
 where $\ell$ denotes the cardinality of the set
 $${\scr L}:=\{\widetilde\bfL\in{\scr F}_{\vtg,n}(V)\mid
\widetilde\bfL\subseteq\bfL,\, \bfL/\widetilde\bfL\cong S_\la,\,
\pi'(\widetilde\bfL)=\widetilde\bfL',\,
\pi''(\widetilde\bfL)=\widetilde\bfL''\}.$$
 Now let $\widetilde\bfL\in\scr L$.  For each $i\in\mbz$, consider the
projection
$$\th_i:L_i\lra L_i'=(L_i+V'')/V''.$$
 Thus, $\th_i^{-1}(\widetilde L_i)=\widetilde L_i+L_i\cap V''=\widetilde L_i+L_i''$ and
$L_i/\th_i^{-1}(\widetilde L_i)\cong L'_i/\widetilde L_i'$. This
implies that
$$\dim \th_i^{-1}(\widetilde L_i)/\widetilde L_i=\dim
L_i/\widetilde L_i-\dim L'_i/\widetilde L_i'=\dim L''_i/\widetilde
L_i''.$$
 The semisimplicity of $\bfL/\widetilde\bfL\cong S_\la$ shows
$L_{i-1}\subseteq \widetilde L_i$. Hence, we have the inclusions
$$L_{i-1}+\widetilde L_i''\subseteq \widetilde L_i\subseteq
\th_i^{-1}(\widetilde L_i)=\widetilde L_i+L_i''\subseteq L_i.$$
 Then
$$\th_i^{-1}(\widetilde L_i)/(L_{i-1}+\widetilde L_i'')
=\widetilde L_i/(L_{i-1}+\widetilde
L_i'')+(L_{i-1}+L_i'')/(L_{i-1}+\widetilde L_i'')\;\;\text{and}$$

$$\dim \th_i^{-1}(\widetilde L_i)/(L_{i-1}+\widetilde L_i'')
=\dim\widetilde L_i/(L_{i-1}+\widetilde L_i'')+\dim L''_i/\widetilde
L_i''.$$
 The inequality $\dim(L_{i-1}+L_i'')/(L_{i-1}+\widetilde L_i'')\leq
\dim L''_i/\widetilde L_i''$ gives that
$$\dim(L_{i-1}+L_i'')/(L_{i-1}+\widetilde L_i'')= \dim
L''_i/\widetilde L_i''\;\;\text{and}$$
$$\th_i^{-1}(\widetilde
L_i)/(L_{i-1}+\widetilde L_i'') =\widetilde L_i/(L_{i-1}+\widetilde
L_i'')\oplus (L_{i-1}+L_i'')/(L_{i-1}+\widetilde L_i'').$$
 Thus, $\widetilde L_i/(L_{i-1}+\widetilde
L_i'')$ is a complement of $(L_{i-1}+L_i'')/(L_{i-1}+\widetilde
L_i'')$.

Consequently,
$$\ell=\prod_{1\leq i\leq n}\ell_i,$$
 where $\ell_i$ is
the number of subspaces in $\th_i^{-1}(\widetilde
L_i)/(L_{i-1}+\widetilde L_i'')$ which are complementary to
$(L_{i-1}+L_i'')/(L_{i-1}+\widetilde L_i'')$. Further,
$$\begin{array}{rl}
&\dim \th_i^{-1}(\widetilde L_i)/(L_{i-1}+{\widetilde L_i''})=\dim
L_i/(L_{i-1}+{\widetilde L_i''})- \dim L_i/\th_i^{-1}(\widetilde L_i)\\
=&\dim L_i/(L_{i-1}+L_i'')+\dim (L_{i-1}+L_i'')/(L_{i-1}+{\widetilde
L_i''})-\dim L'_i/\widetilde L_i' \\
=&\dim L_i/L_{i-1}-\dim(L_{i-1}+L_i'')/L_{i-1}+\dim L_i''/\widetilde
L_i''-\dim
L'_i/\widetilde L_i'\\
=&\dim L'_i/L'_{i-1}+\dim L_i''/\widetilde L_i''-\dim
L'_i/\widetilde L_i'.
\end{array}$$
 Therefore,
$$\ell_i=q^{\dim L_i''/\widetilde L_i''(\dim L'_i/L'_{i-1}-\dim
L'_i/\widetilde L_i')}.$$
 We finally get that
$$\Xi(A_\la(\bfl,r))(\bfL',\widetilde\bfL',\bfL'',\widetilde\bfL'')\\
=q^{a/2},$$
 where $a=\sum_{1\leq i\leq n}\bigl(-\la_i(\dim L_i/L_{i-1}-\la_i)+2\dim
L_i''/\widetilde L_i''(\dim L'_i/L'_{i-1}-\dim L'_i/\widetilde
L_i')\bigr)$.

On the other hand, by \eqref{action1} and \eqref{action2},
$$\begin{array}{rl}
&\qquad\Psi_{\mu,\nu}(\bfL',\widetilde\bfL',\bfL'',\widetilde\bfL'')\not=0\\
&\Longleftrightarrow
\widetilde\bfL'\subseteq\bfL',\,\widetilde\bfL''\subseteq\bfL'',\,
\bfL'/\widetilde\bfL'\cong S_\mu,\,\bfL''/\widetilde\bfL''\cong
S_\nu \\
&\Longleftrightarrow \widetilde L_i'\subseteq L_i',\,\widetilde
L_i''\subseteq L_i'',\, L_i'/\widetilde L_i'\cong
\field^{\mu_i},\,L_i''/\widetilde L_i''\cong
\field^{\nu_i},\;\forall\,i\in\mbz.
\end{array}$$
 Moreover, if this is the case, then
$$\aligned
&\Psi_{\mu,\nu}(\bfL',\widetilde\bfL',\bfL'',\widetilde\bfL'')\\
=&A_\mu(\bfl,r')(\bfL',\widetilde\bfL')0(\nu,r')(\widetilde\bfL',\widetilde\bfL')
A_\nu(\bfl,r'')(\bfL'',\widetilde\bfL'')0(-\mu,r'')(\widetilde\bfL'',\widetilde\bfL'')\\
=&q^{b/2}
\endaligned$$
 where
$b=-c(\bfL',\widetilde\bfL')-c(\bfL'',\widetilde\bfL'')+\sum_{1\leq
i\leq n}(\nu_i\dim \widetilde L'_i/\widetilde L_{i-1}'-\mu_i\dim
\widetilde L''_i/\widetilde L_{i-1}'')$. Since
$$\begin{array}{ll}
c(\bfL',\widetilde\bfL')=\sum_{1\leq i\leq n}\mu_i(\dim
L'_i/L_{i-1}'-\mu_i),\; & c(\bfL'',\widetilde\bfL'')=\sum_{1\leq
i\leq n}\nu_i(\dim L''_i/L_{i-1}''-\nu_i),\\
\dim \widetilde L'_i/\widetilde L_{i-1}'=\dim
L'_i/L_{i-1}'+\mu_{i-1}-\mu_i, & \dim \widetilde L''_i/\widetilde
L_{i-1}''=\dim L''_i/L_{i-1}''+\nu_{i-1}-\nu_i,
\end{array}$$
 it follows that
$$\aligned
b=&\sum_{1\leq i\leq n} \bigl(-\mu_i\dim L'_i/L_{i-1}'-\nu_i\dim
L''_i/L_{i-1}''+\mu_i^2+\nu_i^2\\
&+ \nu_i L'_i/L_{i-1}'-\mu_i\dim
L''_i/L_{i-1}''+\nu_i\mu_{i-1}-\mu_i\nu_{i-1}\bigr)\\
=&a+\tau\mu\centerdot\nu-\mu\centerdot\tau\nu.
\endaligned$$
 Note that $\la_i=\mu_i+\nu_i$ and $\dim L_i/L_{i-1}=\dim
 L_i'/L_{i-1}'+\dim L_i''/L_{i-1}''$ for all $i\in\mbz$.
 Therefore,
$$\Xi(A_\la(\bfl,r))(\bfL',\widetilde\bfL',\bfL'',\widetilde\bfL'')=\sum_{\mu+\nu=\la}
q^{(\mu\centerdot\tau\nu-\tau\mu\centerdot\nu)/2}\Psi_{\mu,\nu}(\bfL',\widetilde\bfL',\bfL'',\widetilde\bfL''),$$
 which finishes the proof.
\end{proof}

\begin{Coro}\label{5.5.2} Suppose $r\geq n$. If $\la=(\la_i)\in\mbn_\vtg^n$ satisfies
$\la_i\geq 1$ for all $i\in\mbz$, then
$$\phi_{r,r-n}(A_\la(\bfl,r))=A_\la(\bfl,r-n)+A_{\la-\delta}(\bfl,r-n)\cdot0(\dt,r-n)$$
and
$$\phi_{r,r-n}(\tA_\la(\bfl,r))=\tA_\la(\bfl,r-n)+\tA_{\la-\delta}(\bfl,r-n)\cdot0(\dt,r-n).$$
\end{Coro}

\begin{proof} Write $r'=r-n$. Fix a finite field $\field$ with $q$ elements. Then
applying the above proposition gives
$$\Xi(A_\la(\bfl,r))=\sum_{{\mu,\nu\in\mbn^n_\vtg}\atop{\mu+\nu=\la}}
q^{(\mu\centerdot\tau\nu-\tau\mu\centerdot\nu)/2}A_\mu(\bfl,r')0(\nu,r')\otimes
A_\nu(\bfl,n)0(-\mu,n).$$
 If $\sg(\nu)>n$, then $A_\nu(\bfl,n)=0$. If $1\leq\sg(\nu)\leq n$
 and $\nu\not=\dt$, then $\chi(A_\nu(\bfl,n))=0$. Thus,
$$\aligned
&\phi_{r,r'}^q(A_\la(\bfl,r))=(\xi\otimes\chi)\Xi(A_\la(\bfl,r))\\
=&\xi(A_\la(\bfl,r'))\chi(0(-\la,n))
+\xi(A_{\la-\dt}(\bfl,r')0(\dt,r'))\chi(A_\dt(\bfl,n)0(-\la+\dt,n))\\
=&A_\la(\bfl,r')+A_{\la-\dt}(\bfl,r')0(\dt,r').
\endaligned$$
 The second formula can be proved similarly.

\end{proof}

By the above corollary, if $\la=(\la_i)\in\mbn_\vtg^n$ satisfies
$\la_i\geq 1$ for all $i\in\mbz$ and $\sg(\la)\leq r$, then
$$\aligned
&\phi_{r,r-n}\xi_r(A_\la(\bfl,r))\not=\xi_{r-n}(A_\la(\bfl,r))\;\;\text{and}\\
&\phi_{r,r-n}\xi_r(\tA_\la(\bfl,r))\not=\xi_{r-n}(\tA_\la(\bfl,r)).
\endaligned$$
This shows that the transfer maps $\phi_{r,r-n}$ are not compatible
with the homomorphisms $\xi_r:{}'\dHallr\to\afbfSr$.


\chapter{The classical ($\up=1$) case}

Let $\sU(\afgl)$ be the universal enveloping
algebra\index{$\sU(\afgl)$, universal enveloping algebra of $\afgl$}
of the loop algebra $\afgl(\mbq)$\index{loop algebra of
$\frak{gl}_n$} as mentioned in \S1.1. We will establish, on the one
hand, a surjective homomorphism from $\sU(\afgl)$  to the affine
Schur algebra $\afSr_\mbq$ via the natural action of $\afgl(\mbq)$
on the $\mbq$-space $\Og_\mbq$, and prove, on the other hand, that
$\sU(\afgl)$ is isomorphic to the specialization
$\overline{\fD_\vtg(n)}_\mbq$ of the double Ringel--Hall algebra
$\dHallr$ at $\up=1$ and $K_i=1$. In this way, we obtain a
surjective algebra homomorphism
$\eta_r:\overline{\fD_\vtg(n)}_\mbq\to\afSr_\mbq$ which is regarded
as the the classical ($\up=1$) version of the surjective
homomorphism $\xi_r:\dHallr\to\afSr$. We then prove the classical
version of Conjecture \ref{realization conjecture} via the
homomorphisms  $\eta_r$. A crucial step to establish the conjecture
in this case is the extension of the multiplication formulas given
in \S3.4 to formulas between indecomposable modules of imaginary
roots and arbitrary BLM basis elements. This is done in \S6.2. The
conjecture in the classical case is proved in \S6.3.

In order to distinguish the specializations at non-roots of unity
$\up=z\in\mbc$ considered in previous chapters from the
specialization at $\up=1$, we will particularly consider the
specialization $\sZ\to \mbq$ by sending $\up$ to 1 throughout the
chapter. Thus, $\afHr_\mbq$ identifies the group algebra
$\mbq\fS_{\vtg,r}$, $\Og_\mbq$ is the $\mbq$-space with basis
$\{\og_i\}_{i\in\mbz}$, and $\afSr_\mbq$ identifies the classical
affine Schur algebra. In other words,
$\afSr_\mbq\cong\text{End}_{\mbq\fS_{\vtg,r}}(\Og_\mbq^{\otimes
r})$.

\section{The universal enveloping algebra $\sU(\afgl)$}
Recall from \S1.1 the loop algebra $\afgl(\mbq)$ has a basis
$\{E_{i,j}^\vtg\}_{1\leq i\leq n, j\in\mbz}$. Thus, the natural
action of these basis elements on $\Og_\mbq$ defined by
\begin{equation}\label{natural module}
\afE_{i,j}\og_{k}=
\begin{cases}
\og_{i+tn} &\text{if $k=j+tn$;}\\
0&\text{otherwise}
\end{cases}
\end{equation}
gives rise to a $\sU(\afgl)$-module structure on $\Og_\mbq$, and
hence, on the $r$-fold tensor product $\Og_\mbq^{\ot r}$. Thus, we
obtain an algebra homomorphism
$$\eta_r:\sU(\afgl)\lra\afSr_\mbq.$$

The surjectivity of $\eta_r$ was established in \cite[Th.~6.6(i)]{Yang1} by a coordinate algebra approach.
We now present a different proof by identifying the above action with the Hall algebra action
as discussed in \S\S3.5-6. We first interpret the $\pm$-part of $\sU(\afgl)$ as the specialization of
Hall algebras. We need the following in this interpretation.

\begin{Lem}\label{Hall polynomial evaluation at 1}
For $1\leq k\leq t$, let
$A_k=(a_{i,j}^{(k)}),B=(b_{i,j})\in\afThnp$.
\begin{itemize}
\item[(1)] If $\sg(B)\geq\sum_{k=1}^t\sg(A_k)$ and
$B\not=\sum_{k=1}^tA_k$, then
$(\up^2-1)|\vi_{A_1,A_2,\ldots,A_t}^B$.
\item[(2)] If $B=\sum_{k=1}^tA_k$, then
$$(\up^2-1)\bigg|\bigg(\vi_{A_1,A_2,\ldots,A_t}^B-
\prod_{1\leq i\leq n,\,j\in\mbz\atop
i<j}\frac{b_{i,j}!}{a_{i,j}^{(1)}!a_{i,j}^{(2)}!\cdots
a_{i,j}^{(t)}!}\bigg).$$
\end{itemize}
\end{Lem}

\begin{proof}
By using \cite[p.~441, Lem.]{Ri90} and \cite[Prop.~ 3]{Ri92}, the
assertion can be proved in a way similar to the proof of
\cite[Prop.~4.1]{Peng}.
\end{proof}

Since $\Hall^+_\mbq\cong \Hall_\mbq$ and $\Hall^-_\mbq\cong
\Hall_\mbq^{\text{op}}$, it follows that $\Hall^\pm_\mbq$ is
generated by $u_{_{\afE_{i,j},1}}^\pm$ for all $i<j$. {\it Here and
below, the subscripts $1$ indicate those elements in
$\Hall^\pm_\mbq$ obtained from elements in $\Hall^\pm$ by
specializing $\up$ to $1$}.

\begin{Thm}\label{positive}
Suppose $\sU(\afgl)^+$ (resp., $\sU(\afgl)^-$) is the subalgebra of
$\sU(\afgl)$ generated by $\afE_{i,j}$ (resp., $\afE_{j,i}$) for all
$i<j$. Then, for each $\ep\in\{+,-\}$,  there is a unique algebra
isomorphism $\sU(\afgl)^\ep\to\Hall_\mbq^\ep$ taking
$\afE_{i,j}\mapsto u_{_{\afE_{i,j},1}}^+$ (resp.,$\afE_{j,i}\mapsto
u_{_{\afE_{i,j},1}}^-$) for all $i<j$. In particular, there are
algebra embeddings
$$\iota^\pm:\Hall_\mbq^\pm\lra \sU(\afgl).$$
\end{Thm}

\begin{proof} It suffices to prove the $+$ case. Let $\afgl^+$ be the Lie subalgebra of $\afgl$
generated by $\afE_{i,j}$ ($i<j$).
 Then $\sU(\afgl)^+$ is isomorphic to the enveloping algebra
of $\afgl^+$ and in $\sU(\afgl)^+$,
$$[\afE_{i,j},\afE_{k,l}]=\dt_{\bar j,\bar
k}\afE_{i,l+j-k}-\dt_{\bar l,\bar i}\afE_{k,j+l-i}\;\;\text{for
$i,j,k,l\in\mbz$}.$$ On the other hand, by Lemma \ref{Hall
polynomial evaluation at 1}, for $i<j$ and $k<l$, we have in
$\Hall^+_\mbq$,
\begin{equation}\label{commute formula}
u^+_{_{\afE_{i,j},1}}u^+_{_{\afE_{k,l},1}}-u^+_{_{\afE_{k,l},1}}u^+_{_{\afE_{i,j},1}}=\dt_{\bar
j,\bar k}u^+_{_{\afE_{i,l+j-k},1}}-\dt_{\bar l,\bar
i}u^+_{_{\afE_{k,j+l-i},1}}.
\end{equation}
Thus,  there is an algebra homomorphism
$f:\sU(\afgl)^+\ra\Hall^+_\mbq$ such that
$f(\afE_{i,j})=u^+_{_{\afE_{i,j},1}}$ for $i<j$. Let
$\sL=\{(i,j)\mid 1\leq i\leq n,\ j\in\mbz,\ i<j\}$. By Lemma
\ref{Hall polynomial evaluation at 1} again, we see for any
$A=\sum_{(i,j)\in\sL}a_{i,j}E_{i,j}^\vtg\in\afThnp$,
$$\prod_{(i,j)\in\sL}(u^+_{_{\afE_{i,j},1}})^{a_{i,j}}
=\bigg(\prod_{(i,j)\in\sL}a_{i,j}!\bigg)u^+_{A,1}
+\sum_{B\in\afThnp,\,\sg(B)<\sg(A)}f_B(1)u^+_{B,1},$$
 where $f_{B}\in\sZ$ and the products are taken with respect to a fixed total
order on $\sL$. Hence, the set
$$\bigg\{\prod_{(i,j)\in\sL}(u^+_{_{\afE_{i,j},1}})^{a_{i,j}}\,\bigg|\,
A=(a_{i,j})\in\afThnp\bigg\}$$
 is a $\mbq$-basis of $\Hall^+_\mbq$. Thus, $f$
sends a PBW-basis of $\sU(\afgl)^+$ to a basis of $\Hall^+_\mbq$.
Consequently, $f$ is an isomorphism.
\end{proof}

Recall from Proposition \ref{afzrpm} that $\Hall^\pm$ act on the $\sZ$-free module $\Og$ via the maps $\zeta_r^\pm$.
These induce actions of $\Hall_\mbq^\pm$ on $\Og_\mbq$ via the maps
$$\zeta_{r,\mbq}^\pm:\Hall_\mbq^\pm\lra \afSr_\mbq.$$
 Again, in the following the subscripts $1$ indicate those elements in
$\afSr_\mbq$ obtained from elements in $\afSr$ by specializing $\up$
to $1$.

\begin{Lem}\label{etar(Eij)} For any $r\geq 0$, we have $\eta_r\circ\iota^\pm=\zeta_{r,\mbq}^\pm$.
More precisely, we have $\eta_r(\afE_{i,j})=\afE_{i,j}({\bf 0},r)_1$
for $i\not = j$ and $\eta_r(\afE_{i,i})=\big[{\ttk_{i,r};0\atop
1}\big]_1$, where $\ttk_{i,r}:=0(\afbse_i,r)\in\afSr$ and
$\big[{\ttk_{i,r};0\atop 1}\big]$ is defined as in \eqref{Gauss
poly2}. Moreover, the elements $\big({\afE_{i,i}\atop
t}\big):=\frac{\afE_{i,i}(\afE_{i,i}-1)\cdots(\afE_{i,i}-t+1)}{t!}$
in $\sU(\afgl)$ have the images $\eta_r\big({\afE_{i,i}\atop
t}\big)=\big[{\ttk_{i,r};0\atop t}\big]_1$ for $t\geq 0$.
\end{Lem}

\begin{proof} By Proposition \ref{compareson with VV's action} and Theorem \ref{xirl},
$\Hall_\mbq^\pm$ acts on $\Og_\mbq$ via $\zeta^\pm_{1,\mbq}$: for
$i<j$,
$$u_{_{\afE_{i,j},1}}^+\og_{k}=
\begin{cases}
\og_{i+tn}, &\text{if $k=j+tn$,}\\
0,&\text{otherwise}
\end{cases}\;\;\text{and}\;\;
u_{_{\afE_{i,j},1}}^-\og_{k}=
\begin{cases}
\og_{j+tn}, &\text{if $k=i+tn$,}\\
0,&\text{otherwise.}
\end{cases}
$$
 This action agrees with the action of $\afE_{i,j}$ on $\Og_\mbq$ defined in
\eqref{natural module}. On the other hand, by Propositions
\ref{Green Xiao}(b) and \ref{Green Xiao 2}(b$'$), we have
$\Delta(u_{_{\afE_{i,j},1}}^\pm)=u_{_{\afE_{i,j},1}}^\pm\otimes
1+1\otimes u_{_{\afE_{i,j},1}}^\pm$. Thus, the action of
$\afE_{i,j}$ on $\Og_\mbq^{\otimes r}$ agrees with the action of
$u_{_{\afE_{i,j},1}}^\pm$, proving the first statement.

For $\la\in\afLanr$, let $\bfi_\la$ be defined as in the proof of
Proposition \ref{bimodule-isom}. Since
$\afE_{i,i}(\og_{\bfi_\la}w)=\la_i
\og_{\bfi_\la}w=\big[{\ttk_{i,r};0\atop 1}\big]_1(\og_{\bfi_\la}w)$
for all $w\in\affSr$, we have
$\eta_r(\afE_{i,i})=\big[{\ttk_{i,r};0\atop 1}\big]_1$. Hence,
\begin{equation*}
\begin{split}
\eta_r\bigg({\afE_{i,i}\atop t}\bigg)&=
\frac{1}{t!}\prod_{s=0}^{t-1}\bigg(\sum_{\la\in\afLanr}(\la_i-s)[\diag(\la)]_1\bigg)\\
&=\sum_{\la\in\afLanr}\bigg({\la_i\atop t}\bigg)[\diag(\la)]_1=\bigg[{\ttk_{i,r};0\atop t}\bigg]_1,
\end{split}
\end{equation*}
as required.
\end{proof}

\begin{Thm}\label{surjective in the case of q=1}
Identifying $\Hall^\pm_\mbq$ with $\sU(\afgl)^\pm$ via $\iota^\pm$,
the map $\eta_r$ is surjective and $\eta_r(u_{A,1}^+)=A(\bfl,r)_1$
and $\eta_r(u_{A,1}^-)=\tA(\bfl,r)_1$ for $A\in\afThnp$.
\end{Thm}

\begin{proof}
The first assertion follows from Proposition \ref{afzrpm} and Theorem \ref{PBW
basis of affine q-Schur algebras}. Since $\Hall^\pm_\mbq$ is generated
by $u_{\afE_{i,j},1}^\pm$ for $i<j$, the second assertion follows from
Lemma \ref{etar(Eij)} and Proposition \ref{afzrpm}.
\end{proof}

Let $\fD_\vtg(n)$ be the integral form of $\dHallr$ defined right
after Remark \ref{freeness of U0}. We now prove that the
specialization $\overline{\fD_\vtg(n)}_\mbq$ of $\dHallr$ at $v=1$
and $K_i=1$ is isomorphic to the universal enveloping algebra
$\sU(\afgl)$, and thus, via the theorem above, we establish a
surjective homomorphism $\bar\xi_r$ from
$\overline{\fD_\vtg(n)}_\mbq$ to the affine Schur algebra; cf.
\cite{Yang1}. Clearly, $\bar\xi_r$ is a specialization of the
homomorphism $\xi_r$ defined in \eqref{xir}.

By specializing $\up$ to $1$, $\mbq$ is regarded as a $\sZ$-module. Consider the $\mbq$-algebra
\begin{equation}\label{bar integral form}
\aligned
 \overline{\fD_\vtg(n)}_\mbq&:=\fD_\vtg(n)_\mbq/\lan K_i-1\mid 1\leq i\leq
n\ran\\
&\cong(\fD_\vtg(n)\ot_\sZ\mbq[\up,\up^{-1}])/\lan\up-1,K_i-1\mid
1\leq i\leq n\ran.\endaligned
\end{equation}
We have used and will use the same notation for elements in $\fD_\vtg(n)_\mbq$ and $\overline{\fD_\vtg(n)}_\mbq$.

Let $\afsl(\mbq)=\frak{sl}_n(\mbq)\ot\mbq[t,t^{-1}]$ and set
$\afgl^\rmL=\afgl^\rmL(\mbq)=\afsl(\mbq)\op\mbq\afE$, where
$\afE=\sum_{1\leq i\leq n}\afE_{i,i}$ is in the center of
$\afgl^\rmL(\mbq)$. Then the set
\[
\begin{split}
X:=&\{\afE_{i,j+ln}\mid 1\leq i,j\leq n,\,i\not
=j,\,l\in\mbz\}\cup\{\afE_{i,i}\mid 1\leq i\leq n\}\\
&\quad\cup\{\afE_{i,i+ln}-\afE_{i+1,i+1+ln}\mid 1\leq i\leq
n-1,\,l\in\mbz,\,l\not=0\}
\end{split}
\]
forms a $\mbq$-basis for $\afgl^\rmL$. Let $\sfz_m=\sum_{1\leq h\leq
n}\afE_{h,h+mn}$ for $m\not=0$. Then $X\cup\{\sfz_{m}\mid
m\in\mbz,\,m\not=0\}$ forms a $\mbq$-basis for $\afgl=\afgl(\mbq)$.
By the PBW theorem,
\begin{equation}\label{decompositon of afbfUgl}
\sU(\afgl)\cong\sU(\afgl^\rmL)\ot\mbq [\sfz_m]_{m\in\mbz\backslash\{0\}},
\end{equation}
where  $\sU(\afgl^\rmL)$ is the enveloping algebra of $\afgl^\rmL$. Note
that the $\sfz_m$ are central elements in $\sU(\afgl)$.

\begin{Thm} \label{HalltoLoopAlgebra at v=1}
 There is an algebra isomorphism $\phi:\overline{\fD_\vtg(n)}_\mbq\ra\sU(\afgl)$
defined by sending $\big[{K_i;0\atop t}\big]$ to
$\big({\afE_{i,i}\atop t}\big)$, $(u_i^+)^{(m)}$ to
$(\afE_{i,i+1})^m/m!$, $(u_i^-)^{(m)}$ to $(\afE_{i+1,i})^m/m!$,
$\sfz_s^+$ to $\sfz_s$ and $\sfz_s^-$ to $\sfz_{-s}$ for $1\leq
i\leq n$ and $m,s,t\geq 1$.
\end{Thm}

\begin{proof}
Let $U_\vtg(n)$ be the $\sZ$-subalgebra of $\dHallr$ generated
by $K_i^{\pm1}$, $\big[{K_i;0\atop t}\big]$, $(u_i^+)^{(m)}$ and
$(u_i^-)^{(m)}$ for $1\leq i\leq n$ and $t,m\geq 1$.   By
\cite[6.7]{Lu90} and \cite{Lu93}, specializing $\up$ to 1 induces an algebra isomorphism
$$\al:\overline{U_\vtg(n)}_\mbq:=U_\vtg(n)_\mbq/\lan K_i-1\mid 1\leq i\leq
n\ran\lra\sU(\afgl^\rmL)
$$
defined by taking $\big[{K_i;0\atop
t}\big]\mapsto\big({\afE_{i,i}\atop t}\big)$,
$(u_i^+)^{(m)}\mapsto(\afE_{i,i+1})^m/m!$,
$(u_i^-)^{(m)}\mapsto(\afE_{i+1,i})^m/m!$ for $1\leq i\leq n$ and
$m,t\geq 1$.

Let $\iota$ be the natural algebra homomorphism induced by the
inclusion $U_\vtg(n)\subset \fD_\vtg(n)$:
$$\iota:\overline{U_\vtg(n)}_\mbq\lra\overline{\fD_\vtg(n)}_\mbq=\fD_\vtg(n)_\mbq/\lan K_i-1\mid 1\leq i\leq
n\ran.$$ Since $\fD_\vtg(n)\cong U_\vtg(n)\ot\sZ[\sfz_m^{\pm}\mid
m>0]$, by \eqref{decompositon of afbfUgl}, the map $\al$ induces an
algebra homomorphism
$$\phi:\overline{\fD_\vtg(n)}_\mbq\lra\sU(\afgl)$$ such that
$\phi\iota(x)=\al(x)$ for $x\in \overline{U_\vtg(n)}_\mbq$ and
$\phi(\sfz_s^+)=\sfz_s$ and $\phi(\sfz_s^-)=\sfz_{-s}$ for $s>0$. We
choose a $\mbq$-basis $\{x_j\mid j\in J\}$ of
$\overline{U_\vtg(n)}_\mbq$. Then $\overline{\fD_\vtg(n)}_\mbq$ is
spanned by the set
$$Y:=\bigg\{\iota(x_j)\prod_{i=1}^k(\sfz_i^+)^{a_i}\prod_{i=1}^l(\sfz_i^-)^{b_i}\,\bigg|\,
j\in J,\,k,l\geq 1,\,a_i,b_i\in\mbn,\,\forall i\bigg\}.$$ By
\eqref{decompositon of afbfUgl}, the set $\phi(Y)$ forms a
$\mbq$-basis for $\sU(\afgl)$. Thus, $\phi$ is an isomorphism.
\end{proof}

Put $\bar\xi_r=\eta_r\circ\phi$. Theorem \ref{HalltoLoopAlgebra at
v=1} gives rise to an algebra epimorphism
$$\bar\xi_r:\overline{\fD_\vtg(n)}_\mbq\lra \afSr_\mbq,$$
which is the classical version of the map $\xi_r:\dHallr\to\afbfSr$.
In the next two sections, we will describe explicitly the image of the map
$$\bar\xi=\prod_{r\geq
0}\bar\xi_r:\overline{\fD_\vtg(n)}_\mbq\lra\afhsKq,$$
 or equivalently, the map
\begin{equation}\label{map eta}
\eta=\prod_{r\geq
0}\eta_r:\sU(\afgl)\lra\afhsKq.
\end{equation}
 Here we have identified the algebra $\afhsKq$ defined at the end of \S5.4 with the direct product $\prod_{r\geq 0}\afSr_\mbq$.

\section{More multiplication formulas in affine Schur algebras}

In order to prove Conjecture \ref{realization conjecture} in the classical case, we use
the elements $A[\bfj,r]$ for $A\in\afThnpm,{\bf j}\in \afmbnn$ defined in \eqref{A[j,r]}:
$$A[{\bf j},r]=\begin{cases} \sum_{\la\in\La_\vtg(n,r-\sg(A))}
\la^\bfj[A+\diag(\la)]_1,&\text{ if }\sg(A)\leq r;\\
0,&\text{ otherwise.}\end{cases}$$\index{$A[{\bf j},r]$, BLM basis
element for affine Schur algebra}
 These elements can not be obtained by specializing $\up$ to 1 from
the elements $A(\bfj,r)$ defined in \eqref{def-A(j,r)}. However, we
have $A[\bfl,r]=A(\bfl,r)_1$ for $A\in\afThnpm$ (assuming $0^0=1$).
We also point out another difference when $\sg(A)=r$. In this case,
$A[\bfj,r]=\dt_{\bfl,\bfj}[A]_1$ while $A(\bfj,r)_1=[A]_1$ for all
$\bfj\in \afmbnn$.

We will first show that, for a given $r>0$, the set
$\{A[\bfj,r]\}_{A\in\afThnpm,{\bf j}\in \afmbnn}$ spans the affine
Schur algebra $\afSr_\mbq$. Then we derive some multiplication
formulas between $A[\bfj,r]$ and certain generators corresponding to
simple and homogeneous indecomposable representations of the cyclic
quiver $\tri$. We will leave the proof of the conjecture to the next
section.


The following result is the classical counterpart of \cite[Prop.~4.1]{DF09}. Its proof is similar to the proof there; cf. \cite[4.3,4.2]{Fu09}.

\begin{Prop} For any fixed $1\leq i_0\leq n$, the set
$$\{A[\bfj,r]\mid A\in\afThnpm,{\bf j}\in \afmbnn,j_{i_0}=0,\; \sg(A)+\sg(\bfj)\leq r\}$$
 is a basis for $\afSr_\mbq$. In particular, the set $\{A[\bfj,r]\mid A\in\afThnpm,{\bf j}\in \afmbnn, \sg(A)\leq r\}$ forms a spanning set of $\afSr_\mbq$.
\end{Prop}

The first two of the following multiplication formulas in affine
Schur algebras are a natural generalization of the multiplication
formulas for Schur algebras given in \cite[Prop.~3.1]{Fu09}.  The
third formula is new and is the key to the proof of Conjecture
\ref{realization conjecture} in the classical case. It would be
interesting to find the corresponding formula for affine quantum
Schur algebras. For the simplicity of the statement in the next
result, we also set $A[\bfj,r]=0$ if some off-diagonal entries of
$A$ are negative.

\begin{Thm}\label{Multiplication Formulas at v=1}
Assume $1\leq h,t\leq n$, $\bfj=(j_k)\in\afmbnn$ and $A=(a_{i,j})\in\afThnpm$. The
following multiplication formulas hold in $\afSr_\mbq$:
\begin{itemize}
\item[(1)] $0[\afbse_t,r]  A[\bfj,r]=A[\bfj+\afbse_t,r]
+\bigl(\sum_{s\in\mbz}a_{t,s}\bigr)A[\bfj,r]$;
\item[(2)] for $\varepsilon\in\{1,-1\}$,
\begin{equation*}\label{MF1}
\begin{split}
 E^\vtg_{h,h+\varepsilon}[\bfl,r]  A[\bfj,r] =
 &\sum_{a_{h+\varepsilon,i}\geq 1\atop\forall i\not=h,h+\varepsilon}
(a_{h,i}+1)(A+E^\vtg_{h,i}-E^\vtg_{h+\varepsilon,i})[\bfj,r] \\
&+\sum_{0\leq i\leq j_h}(-1)^i\bigg({j_h\atop i}\bigg)
(A-E^\vtg_{h+\varepsilon,h})[\bfj+(1-i)\afbse_h,r] \\
&+(a_{h,h+\varepsilon}+1)\sum_{0\leq i\leq j_{h+\varepsilon}}\bigg({j_{h+\varepsilon}\atop
i}\bigg)(A+E^\vtg_{h,h+\varepsilon})[\bfj-i\afbse_{h+\varepsilon},r];
\end{split}
\end{equation*}

\item[(3)] for $m\in \mbz\backslash\{0\}$,
\begin{equation*}
\begin{split}
\afE_{h,h+mn}[\bfl,r]A[\bfj,r] &=\sum_{s\not\in\{h,h-mn\}
\atop a_{h,s}\geq 1}(a_{h,s+mn}+1)(A+\afE_{h,s+mn}-\afE_{h,s})[\bfl,r]\\
&\quad+\sum_{0\leq t\leq j_h}(a_{h,h+mn}+1)\left({j_h\atop
t}\right)(A+\afE_{h,h+mn})[\bfj-t\afbse_h,r]\\
&\quad+\sum_{0\leq t\leq j_h}(-1)^t\left({j_h\atop t}
\right)(A-\afE_{h,h-mn})[\bfj+(1-t)\afbse_h,r].
\end{split}
\end{equation*}
\end{itemize}
\end{Thm}

It is natural to compare Theorem \ref{Multiplication Formulas at
v=1}(1)--(2) with \cite[(4.2.1-3)]{DF09} or Theorem
\ref{multiplication formulas in affine q-Schur algebra} which
generalize the corresponding ones for quantum Schur algebras. They
are {\it not} obtained from the quantum counterpart by specializing
$\up$ to 1. For example, the second sum in the right hand side of
(2) above is slightly different from the quantum version. In the
case where $\sg(A)=r+1$ and $a_{h+\varepsilon,h}\geq1$, the left
hand side of Theorem \ref{Multiplication Formulas at v=1}(2) is
zero. By the remark at the beginning of the section, the right hand
side is also 0 since $\bfj+(1-i)\afbse_h\neq\bfl$ for all $0\leq
i\leq j_h$.

The proof of the following result will be given at the end of the
chapter as an appendix; see \S6.4. It should be pointed out that the
first formula can also be obtained from \cite[3.5]{Lu99} by
specializing $\up$ to 1, while the second formula is the key to the
proof of part (3) of the theorem above.

\begin{Prop}\label{MFforSBE}Let $1\leq h\leq n$, $B=(b_{i,j})\in\afThnr$ and $\la=\ro(B)$.
\begin{itemize}
\item[(1)] If $\varepsilon\in\{1,-1\}$ and $\la\geq\afbse_{h+\varepsilon}$, then
$$[E^\vtg_{h,h+\varepsilon}+\diag(\la-\bse^\vtg_{h+\varepsilon})]_1[B]_1=\sum_{{i\in\mbz},{b_{h+\varepsilon,i}\geq1}}
(b_{h,i}+1)[B+E_{h,i}^\vtg-E_{h+\varepsilon,i}^\vtg]_1.$$

\item[(2)] If $m\in\mbz\backslash\{0\}$ and $\la\geq\afbse_h$, then
$$[\afE_{h,h+mn}+\diag(\la-\afbse_h)]_1[B]_1=\sum_{s\in\mbz\atop
b_{h,s}\geq 1}(b_{h,s+mn}+1)[B+\afE_{h,s+mn}-\afE_{h,s}]_1.$$
\end{itemize}
\end{Prop}

We now use these formulas to prove the theorem.

\begin{proof} The proof of formula (1) is straightforward. Since for any
$A\in\afThnr$ and $\la\in\afLanr$,
$$[\diag(\la)]_1[A]_1=\begin{cases} [A]_1,&\text{ if }\la=\ro(A);\\
0,&\text{ otherwise,}\end{cases}$$ it follows that
$$\aligned
0[\afbse_t,r]  A[\bfj,r]&=\sum_{\la\in\La(n,r-1)}\la_t[\diag(\la)]_1\sum_{\mu\in\La(n,r-\sg(A))}\mu^\bfj[A+\diag(\mu)]_1\\
&=\sum_{\mu\in\La(n,r-\sg(A))}\la_t\mu^\bfj[A+\diag(\mu)]_1\;\;\;\text{(where }\la=\ro(A)+\mu)\\
&=\sum_{\mu\in\La(n,r-\sg(A))}(\sum_{j\in\mbz}a_{t,j}+\mu_t)\mu^\bfj[A+\diag(\mu)]_1\\
&=\sum_{\mu\in\La(n,r-\sg(A))}\mu^{\bfj+\bse_t}[A+\diag(\mu)]_1+\biggl(\sum_{j\in\mbz}a_{t,j}\biggr)A[\bfj,r]=\text{RHS}.\\
\endaligned$$

We now prove formula (2). For convenience, we set $[B]_1=0$ if one
of the entries of $B$ is negative. Since $[A]_1[B]_1=0$ whenever
$\co(A)\neq\ro(B)$, we have
$$E^\vtg_{h,h+\varepsilon}[\bfl,r]  A[\bfj,r] =\sum_{\mu\in\La(n,r-\sg(A))}\mu^\bfj[E^\vtg_{h,h+\ep}+\diag(\mu+\ro(A)-\afbse_{h+\ep})]_1[A+\diag(\mu)]_1.$$
By Proposition \ref{MFforSBE}(1),
$$\aligned
{}[E^\vtg_{h,h+\ep}&+\diag(\mu+\ro(A)-\afbse_{h+\ep})]_1[A+\diag(\mu)]_1\\
&=\sum_{{i\neq h,h+\ep}\atop{a_{h+\varepsilon,i}\geq1}}
(a_{h,i}+1)[A+E_{h,i}^\vtg-E_{h+\varepsilon,i}^\vtg+\diag(\mu)]_1\\
&\qquad+(\mu_h+1)[A-E_{h+\ep,h}^\vtg+\diag(\mu+\afbse_h)]_1\\
&\qquad+(a_{h,h+\ep}+1)[A+E_{h,h+\ep}+\diag(\mu-\afbse_{h+\ep})]_1.\\
\endaligned
$$
Thus,
$$E^\vtg_{h,h+\varepsilon}[\bfl,r]  A[\bfj,r]=\sum_{{i\neq h,h+\ep}\atop{a_{h+\varepsilon,i}\geq1}}
(a_{h,i}+1)(A+E_{h,i}^\vtg-E_{h+\varepsilon,i}^\vtg)[\bfj,r]+\sY_h+\sY_{h+\ep},$$
where
$$\aligned
\sY_h&=\sum_{\mu\in\La(n,r-\sg(A))}\mu^\bfj(\mu_h+1)[A-E_{h+\ep,h}^\vtg+\diag(\mu+\afbse_h)]_1\\
&=\sum_{\mu\in\La(n,r-\sg(A))}\bigl(\prod_{i\neq h}\mu_i^{j_i}\bigr)(\mu_h+1-1)^{j_h}(\mu_h+1)[A-E_{h+\ep,h}^\vtg+\diag(\mu+\afbse_h)]_1\\
&=\sum_{0\leq i\leq j_h}(-1)^i{j_h\choose i}(A-E_{h+\ep,h}^\vtg)[\bfj+(1-i)\afbse_h,r],\\
\endaligned$$
and
$$\aligned
&\quad\,\sY_{h+\ep}=\sum_{\mu\in\La(n,r-\sg(A))}\mu^\bfj(a_{h,h+\ep}+1)[A+E_{h,h+\ep}+\diag(\mu-\afbse_{h+\ep})]_1\\
&=(a_{h,h+\ep}+1)\sum_{\mu\in\La(n,r-\sg(A))}\bigl(\prod_{i\neq h+\ep}\mu_i^{j_i}\bigr)(\mu_{h+\ep}-1+1)^{j_{h+\ep}}[A+E_{h,h+\ep}+\diag(\mu-\afbse_{h+\ep})]_1\\
&=(a_{h,h+\ep}+1)\sum_{0\leq i\leq j_{h+\ep}}{j_{h+\ep}\choose i}\sum_{\mu\in\La(n,r-\sg(A)-1)}\mu^{\bfj-i\afbse_{h+\ep}}[A+E_{h,h+\ep}+\diag(\mu)]_1\\
&=(a_{h,h+\ep}+1)\sum_{0\leq i\leq j_{h+\varepsilon}}\bigg({j_{h+\varepsilon}\atop
i}\bigg)(A+E^\vtg_{h,h+\varepsilon})[\bfj-i\afbse_{h+\varepsilon},r].
\endaligned$$
Substituting gives (2).

Finally, we prove formula (3).  The proof is similar to that of (2).
First, with a same reasoning,
\[
\begin{split}
\afE_{h,h+mn}[0,r]A[\bfj,r]&=\sum_{\mu\in\afLa(n,r-\sg(A))}
\mu^\bfj[\afE_{h,h+mn}+\diag(\mu)+ro(A)-\afbse_h]_1\cdot[A+\diag(\mu)]_1.\\
\end{split}
\]
Applying Proposition \ref{MFforSBE}(2) yields
\[
\begin{split}
&\qquad[\afE_{h,h+mn}+\diag(\mu)+ro(A)-\afbse_h]_1
\cdot[A+\diag(\mu)]_1\\
&=\sum_{s\not\in\{h,h-mn\}}
(a_{h,s+mn}+1)[A+\afE_{h,s+mn}-\afE_{h,s}+\diag(\mu)]_1\\
&\qquad +(a_{h,h+mn}+1)[(A+\afE_{h,h+mn})+\diag(\mu-\afbse_h)]_1
\\&\qquad+
(\mu_h+1)[(A-E_{h,h-mn})+\diag(\mu+\afbse_h)]_1.\\
\end{split}
\]
Thus,
\[
\begin{split}
\afE_{h,h+mn}[0,r]A[\bfj,r]&=\sum_{s\not\in\{h,h-mn\}}
(a_{h,s+mn}+1)(A+\afE_{h,s+mn}-\afE_{h,s})[\bfl,r]+\sX_1+\sX_2
\end{split}
\]
where
\[\begin{split}
\sX_1&=(a_{h,h+mn}+1)\sum_{\mu\in\afLa(n,r-\sg(A))}\mu^\bfj
[(A+\afE_{h,h+mn})+\diag(\mu-\afbse_h)]_1\\
&=(a_{h,h+mn}+1)\sum_{\mu\in\afLa(n,r-\sg(A))}\prod_{s\not=h\atop
1\leq s\leq
n}\mu_s^{j_s}(\mu_h-1+1)^{j_h}[(A+\afE_{h,h+mn})+\diag(\mu-\afbse_h)]_1\\
&=(a_{h,h+mn}+1)\sum_{\mu\in\afLa(n,r-\sg(A))\atop 0\leq t\leq
j_h}\bigg({j_h\atop
t}\bigg)(\mu-\afbse_h)^{\bfj-t\afbse_h}[A+\afE_{h,h+mn}+\diag(\mu-\afbse_h)]_1\\
&=(a_{h,h+mn}+1)\sum_{0\leq t\leq j_h}\bigg({j_h\atop
t}\bigg)(A+\afE_{h,h+mn})[\bfj-t\afbse_h,r]
\end{split}
\]
and
\[
\begin{split}
\sX_2&=\sum_{\mu\in\afLa(n,r-\sg(A))}\mu^\bfj(\mu_h+1)
[(A-E_{h,h-mn})+\diag(\mu+\afbse_h)]_1\\
&=\sum_{\mu\in\afLa(n,r-\sg(A))}\prod_{s\not=h}\mu_s^{j_s}(\mu_h+1-1)^{j_h}(\mu_h+1)
[(A-E_{h,h-mn})+\diag(\mu+\afbse_h)]_1\\
&=\sum_{\mu\in\afLa(n,r-\sg(A))\atop 0\leq t\leq
j_h}(-1)^t\bigg({j_h\atop
t}\bigg)(\mu+\afbse_h)^{\bfj+(1-t)\afbse_h}[(A-E_{h,h-mn})+\diag(\mu+\afbse_h)]_1\\
&=\sum_{0\leq t\leq j_h}(-1)^t\bigg({j_h\atop
t}\bigg)(A-E_{h,h-mn})[\bfj+(1-t)\afbse_h,r],
\end{split}
\]
proving (3). This completes the proof of the theorem.
\end{proof}

\section{Proof of Conjecture 5.4.1 at $v=1$}

We now use Theorem \ref{Multiplication Formulas at v=1} to prove
Conjecture \ref{realization conjecture} in the classical case.

Recall from the proof of Theorem \ref{positive}, the specialized
Ringel--Hall algebra $\Hall_\mbq$ is generated by
$u_{i,j}:=u_{\afE_{i,j},1}$ for all $i<j\in\mbz$. As seen from
Proposition \ref{indecomposable basis}, the Ringel--Hall algebra
$\bfHall$ over $\mbq(\up)$ can be generated by the elements
associated with simple and homogeneous indecomposable
representations of $\tri$. We first prove that this is also true for
$\Hall_\mbq$.


\begin{Lem}\label{generators}
The $\Hall_\mbq$ is generated by the elements $u_{i,i+1}$ and
$u_{i,i+mn}$ for $i\in\mbz$ and $m\in\mbn$. In particular, the
subalgebras $\sS^\pm_\vtg(n,r)_\mbq$ spanned by
$A[\bfl,r]=A(\bfl,r)_1$ for all $A\in\Theta^\pm(n)$ can be generated
by $\afE_{h,h\pm1}[\bfl,r]$ and $\afE_{h,h\pm mn}[\bfl,r]$ for all
$1\leq h\leq n$ and $m\geq 1$, respectively.
\end{Lem}

\begin{proof}
Let $\frak K$ be the subalgebra of $\Hall_\mbq$ generated by the
elements $u_{i,i+1}$ and $u_{i,i+mn}$ for $i\in\mbz$ and $m\in\mbn$.
It is enough to prove $u_{i,j}\in \frak K$ for all $i<j$. Write
$j-i=mn+k$, where $m\in\mbz$ and $1\leq k\leq n$. Then it is clear
that $u_{i,j}\in\Hall_\mbq$ for $k=n$. Now assume $1\leq k<n$. We
apply induction on $k$. If $k=1$, then by \eqref{commute formula},
\[
u_{i,mn+i+1}=u_{i,i+1}u_{i+1,mn+1+i}-u_{i+1,mn+i+1}u_{i,i+1}\in
\frak K.
\]
Now suppose $u_{i,j}\in \frak K$ for all $i<j$ with $j-i=mn+k-1$.
Then by \eqref{commute formula} and the induction hypothesis,
\[
u_{i,mn+k+i}=u_{i,i+1}u_{i+1,mn+k+i}-u_{i+1,mn+k+i}u_{i,i+1}\in
\frak K,
\]
proving the first assertion.

The last assertion follows from Proposition \ref{afzrpm}.
\end{proof}

As in \S5.4, let $\afhsKq$ be the vector space of all formal
(possibly infinite) $\mbq$-linear combinations
$\sum_{A\in\afThn}\beta_A[A]_1$ satisfying \eqref{(F)}. Recall from
\eqref{A[j]} that for $A\in\afThnpm$ and ${\bf j}\in \afmbnn$,
$$A[{\bf j}]=\sum_{r=0}^\infty A[\bfj,r]\in \afhsKq.$$
 Furthermore, let $\leq$ and $\pr$ be the orders on $\mbz_\vtg^n$ and $\afThn$
defined in \eqref{order on afmbzn} and \eqref{order preceq},
respectively.
\begin{Prop}\label{triangular formula in A[bfj]}
For $A\in\afThnpm$ and $\bfj\in\afmbnn$, we have
\[
A^+[\bfl]0[\bfj]A^-[\bfl]=A[\bfj]+\sum_{\bfj'<\bfj
\atop\bfj'\in\afmbnn}f_{A,\bfj}^{\bfj'}A[\bfj'] +\sum_{B\in\afThnpm
\atop B\p A,\,\bfj'\in\afmbnn}f_{A,\bfj}^{B,\bfj'}B[\bfj'],
\]
 where $f_{A,\bfj}^{\bfj'},f_{A,\bfj}^{B,\bfj'}\in\mbq$.
\end{Prop}

\begin{proof} For each $r\geq 0$, by Lemma \ref{generators}, we may
write $A^+[\bfl,r]$ as a linear combination of monomials
in $\afE_{h,h+1}[\bfl,r]$ and $\afE_{h,h+mn}[\bfl,r]$. By Theorem \ref{Multiplication Formulas at v=1},
there exist $f_{A,\bfj}^{B,\bfj'}\in\mbq$ (independent of $r$) such that
\begin{equation}\label{independent of r}
A^+[\bfl,r]0[\bfj,r]A^-[\bfl,r]=\sum_{B\in\afThnpm \atop
\bfj'\in\afmbnn}f_{A,\bfj}^{B,\bfj'}B[\bfj',r]
\end{equation}
for all $r\geq 0$. On the other hand, using an argument similar to the second display for the computation of $\ttp_A$ in the proof of
Theorem \ref{PBW basis of affine q-Schur algebras} yields
\[\begin{split}
A^+[\bfl,r]0[\bfj,r]A^-[\bfl,r] &=\sum_{\la\in\afLanr}\la^\bfj
A^+[\bfl,r][\diag(\la)]_1 A^-[\bfl,r]\\
&=\sum_{\la\in\afLanr\atop\la\geq\bfsg(A)}
\la^\bfj[A+\diag(\la-\bfsg(A))]_1+g
\end{split}\]
where $\bfsg(A)=\co(A^+)+\ro(A^-)$ and $g$ is a $\mbq$-linear
combination of $[B]_1$ with $B\in\afThnr$ and $B\p A$.  Since
$$\la^\bfj=(\la-\bfsg(A)+\bfsg(A))^\bfj=(\la-\bfsg(A))^\bfj+
\sum_{\bfj'<\bfj\atop
\bfj'\in\afmbnn}f_{A,\bfj}^{\bfj'}(\la-\bfsg(A))^{\bfj'}$$ where
$f_{A,\bfj}^{\bfj'}\in\mbz$ are independent of $\la$ and $r$, it follows that
\[
A^+[\bfl,r]0[\bfj,r]A^-[\bfl,r]=A[\bfj,r]+\sum_{\bfj'<\bfj
\atop\bfj'\in\afmbnn}f_{A,\bfj}^{\bfj'}A[\bfj',r] +g.
\]
Combining this with \eqref{independent of r} proves the assertion.
\end{proof}

For integers $a\geq -1$ and $l\geq 1$, let $\mbz_{a,l}=\{a+i\mid i=1,2,\ldots,l\}$
and $$(\mbz_{a,l})^n=\{\la\in\mbn^n\mid \la_i\in \mbz_{a,l}\,\forall i\}.$$

\begin{Lem}\label{Vandermonde  determinant}
For fixed integers $a\geq -1$ and $n,l\geq 1$, if we order
$(\mbz_{a,l})^n$ lexicographically and form an $l^n\times l^n$
matrix $B_n=(\la^\mu)_{\la,\mu\in(\mbz_{a,l})^n}$, where
$\la^\mu=\la_1^{\mu_1}\la_2^{\mu_2}\cdots\la_n^{\mu_n}$, then
$\det(B_n)\neq 0$.
\end{Lem}

\begin{proof} Write $(\mbz_{a,l})^n=\{\bsa_1,\bsa_2,\ldots,\bsa_{l^n}\}$ with $\bsa_i<_{\text{lx}}\bsa_{i+1}$ for all $i$
under the lexicographical order $<_{\text{lx}}$. If the $(i,j)$ entry of $B_n$ is $\bsa_i^{\bsa_j}$, then $B_n$ has the form
$$B_n=\begin{pmatrix}
(a+1)^{(a+1)}B_{n-1}&(a+1)^{a+2}B_{n-1}&\cdots&(a+1)^{a+l}B_{n-1}\\
(a+2)^{(a+1)}B_{n-1}&(a+2)^{(a+2)} B_{n-1}&\cdots&(a+2)^{(a+l)}B_{n-1}\\
\hdotsfor{4}\\
(a+l)^{(a+1)}B_{n-1}&(a+l)^{(a+2)}B_{n-1}&\cdots&(a+l)^{(a+l)}B_{n-1}
\end{pmatrix}.$$
Thus,
$$\aligned
&\quad\,\det(B_n)=\prod_{i=1}^l(a+i)^{(a+1)l^{n-1}}\det\begin{pmatrix}
B_{n-1}&(a+1)B_{n-1}&\cdots&(a+1)^{l-1}B_{n-1}\\
B_{n-1}&(a+2)B_{n-1}&\cdots&(a+2)^{l-1}B_{n-1}\\
\hdotsfor{4}\\
B_{n-1}&(a+l)B_{n-1}&\cdots&(a+l)^{l-1}B_{n-1}
\end{pmatrix}\\
&=\prod_{i=1}^l(a\!+\!i)^{(a+1)l^{n-1}}\!\!\!\det\!\begin{pmatrix}
B_{n-1}&0&\cdots&0\\
B_{n-1}&(2\!-\!1)B_{n-1}&\cdots&((a\!+\!2)^{l-1}\!-\!(a\!+\!2)^{l-2}(a\!+\!1))B_{n-1}\\
\hdotsfor{4}\\
B_{n-1}&(l\!-\!1)B_{n-1}&\cdots&((a\!+\!l)^{l-1}\!-\!(a\!+\!l)^{l-2}(a\!+\!1))B_{n-1}
\end{pmatrix}\\
&=\prod_{i=1}^l(a+i)^{(a+1)l^{n-1}}\det(B_{n-1})\prod_{j=2}^l(j-1)^{l^{n-1}}\det\begin{pmatrix}
B_{n-1}&\cdots&(a+2)^{l-2}B_{n-1}\\
\hdotsfor{3}\\
B_{n-1}&\cdots&(a+l)^{l-2}B_{n-1}
\end{pmatrix}\\
\endaligned
$$
Hence,
$$\det(B_n)=\det(B_{n-1})^l\prod_{i=1}^l(a+i)^{(a+1)l^{n-1}}\prod_{1\leq j<i\leq
l}(i-j)^{l^{n-1}}\neq0,$$
by induction.
\end{proof}

Note that, in the proof above, if $a=-1$, then
$$\mbz_l:=\mbz_{-1,l}=\{0,1,\ldots,l-1\}$$ and so the product
$\prod_{i=1}^l(a+i)^{(a+1)l^{n-1}}=1$.

As introduced at the end of \S5.4, let $\fA_{\vtg,\mbq}[n]$ be the
subspace of $\afhsKq$ spanned by the elements $A[\bfj]$ for all
$A\in\afThnpm$ and $\bfj\in\afmbnn$. The following theorem gives a
realization of the universal enveloping algebra $\sU(\afgl)$. Let
$(\mbz_l)_\vtg^n=\flat_2^{-1}((\mbz_l)^n)$, where $\flat_2$ is
defined in \eqref{flat2}.
\begin{Thm}\label{classical BLM basis}
The $\mbq$-space $\fA_{\vtg,\mbq}[n]$ is a subalgebra of $\afhsKq$
with $\mbq$-basis
 $${\mathfrak B}=\{A[\bfj]\mid A\in\afThnpm,\,\bfj\in\afmbnn\}.$$
Moreover, the map  $\eta:=\prod_{r\geq 0}\eta_r$ defined in \eqref{map eta} is injective
and induces a $\mbq$-algebra isomorphism $\sU(\afgl)\overset\eta\cong\fA_{\vtg,\mbq}[n]$.
\end{Thm}
\begin{proof} We first prove the linear independence of $\mathfrak B$. Suppose
\[
\sum_{A\in\afThnpm,\,\bfj\in\afmbnn} f_{A,\bfj}A[\bfj]=0
\]
for some $f_{A,\bfj}\in\mbq$. Then
\[
\begin{split}
0=\sum_{A\in\afThnpm\atop\bfj\in\afmbnn}
f_{A,\bfj}A[\bfj]&=\sum_{r\geq 0}\sum_{A\in\afThnpm
\atop\la\in\afLa(n,r-\sg(A))}\bigg(\sum_{\bfj\in\afmbnn}\la^\bfj
f_{A,\bfj}\bigg)[A+\diag(\la)].
\end{split}
\]
Thus, $\sum_{\bfj\in\afmbnn}\la^\bfj f_{A,\bfj}=0$ for all
$A\in\afThnpm$, $\la\in\afLa(n,r-\sg(A))$ and $r\geq\sg(A)$. So,
when $A$ is arbitrarily fixed, there is a finite subset $J$ of
$\afmbnn$ satisfying $f_{A,\bfj}\neq0$ for all $\bfj\in J$. Choose
$\l\geq1$ such that $J$ is a subset of $(\mbz_l)^n_\vtg$ and set
$f_{A,\bfj}=0$ if $\bfj\in (\mbz_l)^n_\vtg\backslash J$. Since
$\cup_{r\geq\sg(A)}\afLa(n,r-\sg(A))=\afmbnn$ which contains
$(\mbz_l)^n_\vtg$, it follows that
 $\sum_{\bfj\in (\mbz_l)^n_\vtg}\la^\bfj
f_{A,\bfj}=0$ for all $\la\in (\mbz_l)^n_\vtg$. Applying Lemma
\ref{Vandermonde determinant} gives $f_{A,\bfj}=0$ for all $\bfj$. Hence, $\mathfrak B$ forms a basis for
$\fA_{\vtg,\mbq}[n]$.

Since
$\sU(\afgl)\cong\sU(\afgl)^+\sU(\afgl)^0\sU(\afgl)^-$, by identifying $\sU(\afgl)^\pm$ with $\Hall_\mbq^\pm$, the set
$$\{u_{A,1}^+(\afE_{1,1})^{j_1}\cdots(\afE_{n,n})^{j_n}u_{B,1}^-\mid
A,B\in\afThnp,\bfj\in\afmbnn\}$$
 forms a basis of $\sU(\afgl)$. Now Lemma \ref{etar(Eij)} and
Theorem \ref{surjective in the case of q=1} imply
$$\eta(u^+_{A^+,1}(\afE_{1,1})^{j_1}\cdots(\afE_{n,n})^{j_n}u^-_{A^-,1})
=A^+[\bfl]0[\bfj]A^-[\bfl].$$ By Proposition \ref{triangular formula in A[bfj]},
the set $$\{A^+[\bfl]0[\bfj]A^-[\bfl]\mid A\in\afThnpm,\,\bfj\in\afmbnn\}$$
forms another basis for $\fA_{\vtg,\mbq}[n]$.
Hence, $\eta$ is injective and $\fA_{\vtg,\mbq}[n]$ is exactly the image of $\eta$. This completes the
proof of the theorem.
\end{proof}

This theorem together with Theorem \ref{Multiplication Formulas at v=1} implies immediately the
following multiplication formulas in $\sU(\afgl)$.

\begin{Coro}\label{classical MFs} The universal enveloping algebra
$\sU(\afgl)$ of the loop algebra $\afgl(\mbq)$ has a basis $\{A[\bfj]\mid A\in\afThnpm,\bfj\in\mbn_\vtg^n\}$
which satisfies the following multiplication formulas: for $1\leq h,t\leq n$, $\bfj=(j_k)\in\afmbnn$ and $A=(a_{i,j})\in\afThnpm$,
\begin{itemize}
\item[(1)] $0[\afbse_t]  A[\bfj]=A[\bfj+\afbse_t]
+\bigl(\sum_{s\in\mbz}a_{t,s}\bigr)A[\bfj]$;
\item[(2)] for $\varepsilon\in\{1,-1\}$,
\begin{equation*}\label{MF1}
\begin{split}
 E^\vtg_{h,h+\varepsilon}[\bfl]  A[\bfj] =
 &\sum_{a_{h+\varepsilon,i}\geq 1\atop\forall i\not=h,h+\varepsilon}
(a_{h,i}+1)(A+E^\vtg_{h,i}-E^\vtg_{h+\varepsilon,i})[\bfj] \\
&+\sum_{0\leq i\leq j_h}(-1)^i\bigg({j_h\atop i}\bigg)
(A-E^\vtg_{h+\varepsilon,h})[\bfj+(1-i)\afbse_h] \\
&+(a_{h,h+\varepsilon}+1)\sum_{0\leq i\leq j_{h+\varepsilon}}\bigg({j_{h+\varepsilon}\atop
i}\bigg)(A+E^\vtg_{h,h+\varepsilon})[\bfj-i\afbse_{h+\varepsilon}];
\end{split}
\end{equation*}

\item[(3)] for $m\in \mbz\backslash\{0\}$,
\begin{equation*}
\begin{split}
\afE_{h,h+mn}[\bfl]A[\bfj] &=\sum_{s\not\in\{h,h-mn\}
\atop a_{h,s}\geq 1}(a_{h,s+mn}+1)(A+\afE_{h,s+mn}-\afE_{h,s})[\bfl]\\
&\quad+\sum_{0\leq t\leq j_h}(a_{h,h+mn}+1)\left({j_h\atop
t}\right)(A+\afE_{h,h+mn})[\bfj-t\afbse_h]\\
&\quad+\sum_{0\leq t\leq j_h}(-1)^t\left({j_h\atop t}
\right)(A-\afE_{h,h-mn})[\bfj+(1-t)\afbse_h].
\end{split}
\end{equation*}
\end{itemize}
\end{Coro}

\begin{Rem} There should be applications of these multiplication formulas. For example,
one may define the $\mbz$-subalgebra $\sU(\afgl)_\mbz$ generated by the divided powers of $E^\vtg_{h,h+\varepsilon}[\bfl]$
together with $\afE_{h,h+mn}[\bfl]$ for all $1\leq h\leq n$, $\ep\in\{1,-1\}$ and $m\in\mbz\backslash\{0\}$.
This should serve as the Kostant $\mbz$-form of $\sU(\afgl)$.
\end{Rem}

\section{Appendix: Proof of Proposition 6.2.3}

For a finite subset $X\han\affSr$, define $\ul X=\sum_{x\in
X}x\in\mbq\affSr$. If $A=\jmath_\vtg(\la,d,\mu)$ is the matrix
corresponding to the double coset $\fS_\la d\fS_\mu$ for
$\la,\mu\in\afLanr$, $d\in{\mathscr D}_{\la,\mu}^\vtg$,  then the
element $[A]_1\in\afSr_\mbq$ is the map $\phi_{\la,\mu}^d$ (at
$\up=1$) as defined in \eqref{def of standard basis}. Note also
that, if $\up=1$, then $x_\la=\ul{\fS_\la}$. Thus,
$[A]_1(\ul{\fS_\nu})=0$ for $\nu\neq\mu$ and
$[A]_1(\ul{\fS_\mu})=\ul {\fS_\la
d\fS_\mu}=\ul{\fS_\la}d\ul{{\mathscr D}_\nu^\vtg\cap\fS_\mu}$ by
Lemma \ref{decomposition of double coset}, where $\nu$ is the
composition defined by $\fS_\nu=d^{-1}\fS_\la d\cap \fS_\mu$ (see
Corollary \ref{double coset} for a precise description of $\nu$).
Lemma \ref{decomposition of double coset} implies also that for
$\la,\mu\in\afLanr$ and $w\in\affSr$,
\begin{equation}\label{conjugate intersection}
\ul{{\frak S}_\la} w\ul{\frak S_\mu}=|w^{-1} \frak S_\la
w\cap\frak S_\mu|\ul{{\frak S}_\la w\frak S_\mu}.
\end{equation}
This fact will be used frequently in the proofs below.

\vspace{.3cm}

\textsc{Proposition \ref{MFforSBE}(1).}
{\it Let $h\in[1,n]$ and $A=(a_{i,j})=\jmath_\vtg(\la,d,\mu)\in\afThnr$ with $\ro(A)=\la$ . If $\varepsilon\in\{1,-1\}$ and $\la\geq\afbse_{h+\varepsilon}$, then
$$[E^\vtg_{h,h+\varepsilon}+\diag(\la-\bse^\vtg_{h+\varepsilon})]_1[A]_1=
\sum_{{i\in\mbz},\;{a_{h+\varepsilon,i}\geq1}}
(a_{h,i}+1)[A+E_{h,i}^\vtg-E_{h+\varepsilon,i}^\vtg]_1.$$}

\begin{proof} 
Observe $\ro(E^\vtg_{h,h+\ep})=\bse^\vtg_h$ and
$\co(E^\vtg_{h,h+\ep})=\bse^\vtg_{h+\ep}$. Thus, applying
\eqref{Ajmath} yields
$$\jmath_\vtg(\la+{\boldsymbol\al}_{h,\ep}^\vtg,
1,\la)=E^\vtg_{h,h+\ep}+\text{diag}(\la-\bse^\vtg_{h+\ep}),$$ where
${\boldsymbol\al}_{h,\ep}^\vtg=\bse^\vtg_h-\bse^\vtg_{h+\ep}$. In
other words, the matrix
$E^\vtg_{h,h+\ep}+\diag(\la-\bse^\vtg_{h+\ep})$ is defined by the
double coset $\fS_{\la-{\boldsymbol\al}_{h,\ep}^\vtg}\fS_\la$.
Hence, putting $\sL=[1,n]\times\mbz$, we have by Corollary
\ref{double coset} that
$$\aligned{}
[E^\vtg_{h,h+\ep}+\diag(\la-\bse^\vtg_{h+\ep})]_1[A]_1(\ul{\fS_\mu})
&=[E^\vtg_{h,h+\ep}+\diag(\la-\bse^\vtg_{h+\ep})]_1(\ul{\fS_\la d\fS_\mu})\\
&=\prod_{s,t\in\sL}\frac1{a_{s,t}!}\ul{\fS_{\la+{\boldsymbol\al}_{h,\ep}^\vtg}\fS_\la}d\ul{\fS_\mu}\\
&=\prod_{s,t\in\sL}\frac1{a_{s,t}!}\ul{\fS_{\la+{\boldsymbol\al}_{h,\ep}^\vtg}}\ul{\afmsD_\ga\cap\fS_\la}d\ul{\fS_\mu},\endaligned$$
where
$$\ga=\ga(\ep)=\begin{cases} (\la_1,\ldots,\la_h,1,\la_{h+1}-1,\la_{h+2},\ldots,\la_n),&\text{ if }\ep=1;\\
(\la_1,\ldots,\la_{h-1}-1,1,\la_h,\la_{h+1},\ldots,\la_n),&\text{ if }\ep=-1,\\
\end{cases}$$
For $i\in R_{h+\ep}^\la$, define $w_{\ep,i}\in\fS_\la$ by setting
$$w_{\ep,i}=
\biggl(\begin{matrix}
1\cdots\la_{0,h}&\la_{0,h}+1&\cdots&i-1&     i       &i+1&\cdots\la_{0,h+1}\cdots r\\
1\cdots\la_{0,h}&\la_{0,h}+2&\cdots& i &\la_{0,h}+1&i+1&\cdots\la_{0,h+1}\cdots r\\
\end{matrix}\biggr), $$
if $\ep=1$, and
$$w_{\ep,i}=
\biggl(\begin{matrix}
1\cdots\la_{0,h-2}\cdots i-1&   i     &i+1&\cdots& \la_{0,h-1}  &\la_{0,h-1}+1\cdots r\\
1\cdots\la_{0,h-2}\cdots i-1&\la_{0,h-1}& i &\cdots& \la_{0,h-1}-1&\la_{0,h-1}+1\cdots r\\
\end{matrix}\biggr)$$
if $\ep=-1$. Then $\afmsD_\ga\cap\fS_\la=\{w_{\ep,i}\mid i\in
R_{h+\ep}^\la\}$. Hence,
$$[E^\vtg_{h,h+\ep}+\diag(\la-\bse^\vtg_{h+\ep})]_1[A]_1(\ul{\fS_\mu})
=\prod_{s,t\in\sL}\frac1{a_{s,t}!}\sum_{i\in
R_{h+\ep}^\la}\ul{\fS_{\la+{\boldsymbol\al}_{h,\ep}^\vtg}}w_{\ep,i}d\ul{\fS_\mu}.
$$
Let $B^{(\ep,i)}=(b_{s,t}^{(\ep,i)})\in\afThnr$ be the matrix
associated with $\la+{\boldsymbol\al}_{h,\ep}^\vtg,\mu$ and the
double coset $\fS_{\la+{\boldsymbol\al}_{h,\ep}^\vtg}w_{\ep,i}d\frak
S_\mu$. Since
\[
w_{\ep,i}^{-1}(R_s^{\la+{\boldsymbol\al}_{h,\ep}^\vtg})=
\begin{cases}
R_s^\la,&\text{if $1\leq s\leq n$ but $s\not=h,h+\ep$};\\
R_h^\la\cup\{i\},&\text{if $s=h$};\\
R_{h+\ep}^\la\backslash\{i\},&\text{if $s=h+\ep$},\\
\end{cases}
\]
it follows that for $i\in R_{h+\ep}^\la$,
\[
\begin{split}
b_{s,t}^{(\ep,i)} &=|d^{-1}w_{\ep,i}^{-1}
R_s^{\la+{\boldsymbol\al}_{h,\ep}^\vtg}\cap
R_t^\mu|\\
&=
\begin{cases}
a_{s,t},&\text{if $1\leq s\leq n$ but $s\not=h,h+\ep$};\\
a_{h,t}+|\{d^{-1}(i)\cap R_t^\mu\}|,&\text{if $s=h$};\\
a_{h+\ep,t}-|\{d^{-1}(i)\cap R_t^\mu\}|,&\text{if $s=h+\ep$}.\\
\end{cases}
\end{split}
\]
If $t_i\in\mbz$ is the unique integer such that $d^{-1}(i)\in
R_{t_i}^\mu$ (and $a_{h+\ep,t_i}\geq1$), then
\[
b_{s,t}^{(\ep,i)}=
\begin{cases}
a_{s,t},&\text{if $1\leq s\leq n$, $s\neq h,h+\ep$ or $t\not=t_i$};\\
a_{h,t}+1,&\text{if $s=h$, $t=t_i$};\\
a_{h+\ep,t}-1,&\text{if $s=h+\ep$, $t=t_i$}.
\end{cases}
\]
This implies that $B^{(\ep,i)}=A+\afE_{h,t_i}-\afE_{h+\ep,t_i}$ for all $i\in R_{h+\ep}^\la$. By
Corollary \ref{double coset} again,
$$\aligned{}
[E^\vtg_{h,h+\ep}+\diag(\la-\bse^\vtg_{h+\ep})]_1[A]_1(\ul{\fS_\mu})
&=\prod_{s,t\in\sL}\frac1{a_{s,t}!}\sum_{i\in R_{h+\ep}^\la}\prod_{s,t\in\sL}b_{s,t}^{(\ep,i)}!\ul{\fS_{\la+{\boldsymbol\al}_{h,\ep}^\vtg}w_{\ep,i}d\fS_\mu}\\
&=\sum_{i\in R_{h+\ep}^\la}\prod_{s,t\in\sL}\frac{b_{s,t}^{(\ep,i)}!}{a_{s,t}}[B^{(\ep,i)}]_1(\ul{\fS_\mu})\\
&=\sum_{i\in R_{h+\ep}^\la}\frac{a_{h,t_i}+1}{a_{h+\ep,t_i}}[B^{(\ep,i)}]_1(\ul{\fS_\mu}).\\
\endaligned.
$$
Finally,
$$\aligned{}
[E^\vtg_{h,h+\ep}&+\diag(\la-\bse^\vtg_{h+\ep})]_1[A]_1\\
&=\sum_{t\in\mbz,a_{h+\ep,t}\geq1}|\{i\in\mbz\mid i\in R_{h+\ep}^\la,t=t_i\}|\frac{a_{h,t}+1}{a_{h+\ep,t}}[A+\afE_{h,t}-\afE_{h+\ep,t}]_1\\
&=\sum_{t\in\mbz,a_{h+\ep,t}\geq1}(a_{h,t}+1)[A+\afE_{h,t}-\afE_{h+\ep,t}]_1,\\
\endaligned$$
as $|\{i\in\mbz\mid i\in
R_{h+\ep}^\la,t=t_i\}|=|d^{-1}R_{h+\ep}^\la\cap
R_t^\mu|=a_{h+\ep,t}$.
\end{proof}

We need some preparation before proving Proposition \ref{MFforSBE}(2). We follow the notation used in \S3.2. Thus, for $1\leq i\leq r$,
$\bse_i=(0,\ldots,0,\underset
i1,0,\ldots,0)\in\affSr$ is the permutation sending $i$ to $i+r$ and $j$ to $j$ for all $1\leq j\leq r$
 with $j\neq i$, and $\bse_i=\rho
s_{r+i-2}\cdots s_{i+1}s_i$ as seen in the proof of Proposition \ref{affine symmetric groups}.
Note that $s_{j+1}\rho=\rho s_j$ for all $j\in\mbz$.

Recall also from \eqref{set R} the sets
$R_{i+kn}^\la=\{\la_{k,i-1}+1,\la_{k,i-1}+2,\ldots,\la_{k,i-1}+\la_i\}$\linebreak
associated with $\la\in\afLanr$, where
$\la_{k,i-1}:=kr+\sum_{j=1}^{i-1}\la_j$. If $i\in R_h^{\la}$ for
some (unique) $1\leq h\leq n$, then removing those simple
reflections $s_j$ from $\bse_i$ indexed by the numbers
$i,i+1,\ldots,\la_{0,h}-1$ and
$r+\la_{0,h-1},r+\la_{0,h-1}+1,\ldots,r+i-2$ yields the shortest
representative $\rho s_{r+\la_{0,h-1}-1}\cdots
s_{\la_{0,h}+1}s_{\la_{0,h}}$ of the double coset
$\fS_\la\bse_i\fS_\la$.

We now determine the shortest representative in the double coset
$\fS_\la\bse_i^m\fS_\la$ for every $0\not=m\in\mbz$, $1\leq i\leq r$
and $\la\in\afLanr$. Suppose $i\in R_h^{\la}$ as above and define
$u_{m,i}^\la\in\affSr$ by
\[
u_{m,i}^\la=
\begin{cases}
\rho s_{r+\la_{0,h-1}-1}\cdots s_{\la_{0,h}+1}s_{\la_{0,h}}, \;&\text{if $m=1$};\\
\rho s_{r+\la_{0,h-1}-1}\cdots s_{i+1}s_i\bse_i^{m-2}\rho s_{r+i-2}\cdots
s_{\la_{0,h}+1}s_{\la_{0,h}}, \;&\text{if $m\geq 2$};\\
(u_{-m,i}^{\la})^{-1},&\text{if $m<0$}.
\end{cases}
\]

\begin{Lem}\label{power of e_i}Maintain the notation introduced above.
\begin{itemize}
\item[(1)] $u_{m,i}^\la$ is the shortest element in $\frak S_\la\bse_i^m\frak
S_\la$;
\item[(2)] $\jmath_\vtg(\la,u_{m,i}^\la,\la)=\afE_{h,h-mn}+\diag(\la-\afbse_i)$;
\item[(3)]
 $(u_{m,i}^\la)^{-1}\frak S_\la u_{m,i}^{\la}\cap\fS_\la=\frak S_\ga$, where
\[
\ga=
\begin{cases}
(\la_1,\ldots,\la_{h-1},
\la_h-1,1,\la_{h+1},\ldots,\la_n),\;&\text{if
$m>0$;}\\
(\la_1,\ldots,\la_{h-1},1,\la_h-1,\la_{h+1},\ldots,\la_n),&\text{if
$m<0$}.
\end{cases}
\]
\end{itemize}
\end{Lem}

\begin{proof}Since $\bse_i^m=s_{r+i-1}\cdots s_{\la_{0,h-1}+2}s_{r+\la_{0,h-1}+1}u_{m,i}^\la s_{\la_{0,h}-1}\cdots s_{i+1}s_i$,
$\frak S_\la\bse_i^m\frak S_\la=\frak S_\la
u_{m,i}^\la\frak S_\la$. If $m>0$, then $u_{m,i}^\la\in\affSr$ is determined by its action on
$\{1,2,\ldots,r\}$:
$$u_{m,i}^\la\!=\!\left(\begin{matrix}
1\cdots\la_{0,h-1}\,\la_{0,h-1}+1\,\la_{0,h-1}+2&\cdots&\la_{0,h}-1  &\la_{0,h}&\la_{0,h}+1\cdots r\\
1\cdots\la_{0,h-1}\,\la_{0,h-1}+2\,\la_{0,h-1}+3&\cdots&\la_{0,h}&mr\!+\!\la_{0,h-1}\!+\!1&\la_{0,h}+1\cdots r\\
\end{matrix}\right)$$
Hence,
$$u_{-m,i}^\la\!=\!(u_{m,i}^\la)^{-1}\!=\!\left(\begin{matrix}
1\cdots\la_{0,h-1}&\la_{0,h-1}+1&\la_{0,h-1}+2&\cdots &\la_{0,h} &\la_{0,h}+1\cdots r\\
1\cdots\la_{0,h-1}&-mr+\la_{0,h}&\la_{0,h-1}+1&\cdots&\la_{0,h}-1&\la_{0,h}+1\cdots r\\
\end{matrix}\right)$$
Thus, $u_{m,i}^\la(i)<u_{m,i}^\la(i+1)$ for any $i$ with
$s_i\in\frak S_\la$ and $m\in\mbz\backslash\{0\}$. Hence, by
\eqref{minimal coset representative},
$u_{m,i}^\la\in\afmsD_{\la,\la}$, proving (1). The assertion (2)
follows from Lemma \ref{the map jmath}. The assertion (3) is a
consequence of (2) and Corollary \ref{double coset}.
\end{proof}


\textsc{Proposition \ref{MFforSBE}(2).} {\it Let $h\in[1,n]$ and $A=(a_{i,j})\in\afThnr$ with
$ro(A)=\la$.
If $m\in\mbz\backslash\{0\}$ and $\la\geq\afbse_h$, then
$$[\afE_{h,h+mn}+\diag(\la-\afbse_h)]_1[A]_1=\sum_{s\in\mbz,\;
a_{h,s}\geq 1}(a_{h,s+mn}+1)[A+\afE_{h,s+mn}-\afE_{h,s}]_1.$$}

\begin{proof} As above, let
$$\la_{0,h}=\la_1+\cdots+\la_h\quad (\text{and $\la_{0,0}=0$}).$$
Assume $\mu=\co(A)$ and $d\in\msD^\vtg_{\la,\mu}$ such that
$\jmath_\vtg(\la, d,\mu)=A$. By Lemma \ref{power of e_i} and
Corollary \ref{double coset} (and \eqref{conjugate intersection}),
\[
\begin{split}
[\afE_{h,h+mn}+\diag(\la-\afbse_h)]_1[A]_1(\ul{\frak S_\mu})&=\ul{\frak S_\la(u_{m,i}^\la)^{-1}\frak S_\la}\cdot d\cdot\ul{\msD^\vtg_\al\cap\frak S_\mu}\\
&=\frac{1}{|\frak S_\la|}\ul{\frak S_\la(u_{m,i}^\la)^{-1}\frak
S_\la}\cdot\ul{\frak S_\la d\frak S_\mu}\\
&=\frac{1}{|\frak S_\la|}\prod_{1\leq s\leq n\atop
t\in\mbz}\frac{1}{a_{s,t}!}\ul{\frak S_\la(u_{m,i}^\la)^{-1}\frak
S_\la}\cdot\ul{\frak S_\la}\cdot d\cdot\ul{\frak S_\mu}\\
&=\prod_{1\leq s\leq n\atop t\in\mbz}\frac{1}{a_{s,t}!}\ul{\frak
S_\la(u_{m,i}^\la)^{-1}\frak S_\la}\cdot d\cdot\ul{\frak S_\mu}\\
&=\prod_{1\leq s\leq n\atop t\in\mbz}\frac{1}{a_{s,t}!}\ul{\frak
S_\la}\cdot(u_{m,i}^\la)^{-1}\cdot\ul{\msD^\vtg_\dt\cap\frak
S_\la}\cdot d\cdot\ul{\frak S_\mu}
\end{split}
\]
where $\frak S_\al=d^{-1}\frak S_\la d\cap\frak S_\mu$ and
$\frak S_\dt=u_{m,i}^\la\frak S_\la (u_{m,i}^{\la})^{-1}\cap\fS_\la$ with
\[ \dt=
\begin{cases}
(\la_1,\ldots,\la_{h-1},1,\la_h-1,\la_{h+1},\ldots,\la_n)&\text{if
$m>0$,}\\
(\la_1,\ldots,\la_{h-1}, \la_h-1,1,\la_{h+1},\ldots,\la_n),&\text{if
$m<0$}.
\end{cases}
\]
We now compute $\msD^\vtg_\dt\cap\frak S_\la$. For $m\not=0$ and
$\la_{0,h-1}+1\leq i\leq \la_{0,h}$,\footnote{This condition is
equivalent to $i\in R_h^\la$.} define $w_{m,i}\in\fS_\la$ as
follows. If $m>0$, then
$$w_{m,i}:=\left(\begin{matrix}
1\cdots\la_{0,h-1}&\la_{0,h-1}+1&\cdots&i-1&     i       &i+1&\cdots\la_{0,h}\cdots r\\
1\cdots\la_{0,h-1}&\la_{0,h-1}+2&\cdots& i &\la_{0,h-1}+1&i+1&\cdots\la_{0,h}\cdots r\\
\end{matrix}\right)$$
If $m<0$, then
$$w_{m,i}:=\left(\begin{matrix}
1\cdots\la_{0,h-1}\cdots i-1&   i     &i+1&\cdots& \la_{0,h}  &\la_{0,h}+1\cdots r\\
1\cdots\la_{0,h-1}\cdots i-1&\la_{0,h}& i &\cdots& \la_{0,h}-1&\la_{0,h}+1\cdots r\\
\end{matrix}\right)$$
From the definition, it is clear that $w_{m,\la_{0,h-1}+1}=1$ for
$m>0$ and $w_{m,\la_{0,h}}=1$ for $m<0$. It is also clear that
$$\msD^\vtg_\dt\cap\frak
S_\la=\{w_{m,i}\mid\la_{0,h-1}+1\leq
i\leq\la_{0,h}\}\text{ for all }m\in\mbz\backslash\{0\}.$$
Hence, $[\afE_{h,h+mn}+\diag(\la-\afbse_h)]_1[A]_1(\ul{\frak S_\mu})
= \prod_{1\leq s\leq
n\atop t\in\mbz}\frac{1}{a_{s,t}!}\sum_{i\in R_h^\la}\ul{\frak S_\la}\cdot(u_{m,i}^\la)^{-1}
w_{m,i}  d\cdot\ul{\frak S_\mu}$.

Let $B^{(m,i)}=(b_{s,t}^{(m,i)})\in\afThnr$ be the matrix associated
with $\la,\mu$ and the double coset
$\fS_\la(u_{m,i}^\la)^{-1}w_{m,i}d\frak S_\mu$, where
$m\in\mbz\backslash\{0\}$ and $i\in R_h^\la$. Since
\[
w_{m,i}^{-1}u_{m,i}^\la(R_s^\la)=
\begin{cases}
R_s^\la,&\text{if $1\leq s\leq n$ but $s\not=h$};\\
(R_h^\la\backslash\{i\})\cup\{mr+i\},&\text{if $s=h$},
\end{cases}
\]
it follows that for $i\in R_h^\la$,
\[
\begin{split}
b_{s,t}^{(m,i)} &=|d^{-1}w_{m,i}^{-1}u_{m,i}^\la R_s^\la\cap
R_t^\mu|\\
&=
\begin{cases}
a_{s,t},&\text{if $1\leq s\leq n$ but $s\not=h$};\\
a_{s,t}-|\{d^{-1}(i)\cap R_t^\mu\}|+|d^{-1}(mr+i)\cap
R_t^\mu|,&\text{if $s=h$}.
\end{cases}
\end{split}
\]
If $t_i\in\mbz$ is the unique integer such that $d^{-1}(i)\in
R_{t_i}^\mu$ (and $a_{h,t_i}\geq1$), then
\[
b_{h,t}^{(m,i)}=
\begin{cases}
a_{h,t},&\text{if $t\not\in\{t_i,mn+t_i\}$};\\
a_{h,t}-1,&\text{if $t=t_i$};\\
a_{h,t}+1,&\text{if $t=mn+t_i$}.
\end{cases}
\]
This implies that $B^{(m,i)}=A+\afE_{h,mn+t_i}-\afE_{h,t_i}$ for all $i\in R_h^\la$.
Thus, applying Corollary \ref{double coset} again yields
\[
\begin{split}
 [\afE_{h,h+mn}+&\diag(\la-\afbse_h)]_1[A]_1(\ul{\frak S_\mu}) \\
&=
\prod_{1\leq s\leq n\atop
t\in\mbz}\frac{1}{a_{s,t}!}\sum_{i\in R_h^\la}\prod_{1\leq s\leq n\atop
t\in\mbz}b_{s,t}^{(m,i)}!\ul{\fS_\la(u_{m,i}^\la)^{-1}w_{m,i}d\frak S_\mu}\\
&=\sum_{i\in R_h^\la}
\prod_{1\leq s\leq n\atop
t\in\mbz}\frac{b_{s,t}^{(m,i)}!}{a_{s,t}!}[B^{(m,i)}]_1(\ul{\frak S_\mu})\\
&=\sum_{i\in R_h^\la,a_{h,t_i}\geq1}\frac{a_{h,mn+t_i}+1}{a_{h,t_i}}[B^{(m,i)}]_1(\ul{\frak
S_\mu}).
\end{split}
\]
Therefore,
\[
\begin{split}
&\qquad[\afE_{h,h+mn}+\diag(\la-\afbse_h)]_1[A]_1\\
&=\sum_{t\in\mbz\atop a_{h,t}\geq 1}|\{i\in\mbz\mid
i\in R_h^\la,\,
t=t_i\}|\frac{a_{h,mn+t}+1}{a_{h,t}}[A+\afE_{h,mn+t}-\afE_{h,t}]_1\\
&=\sum_{t\in\mbz\atop a_{h,t}\geq
1}(a_{h,t+mn}+1)[A+\afE_{h,t+mn}-\afE_{h,t}]_1
\end{split}
\]
since $a_{h,t}=|d^{-1}R_h^\la\cap R_t^\mu|=|\{i\in\mbz\mid i\in
R_h^\la,\, t=t_i\}|$ by Lemma \ref{the map jmath}.
\end{proof}

Proposition \ref{MFforSBE}(2) is the key to the establishment of the
multiplication formulas in Theorem \ref{Multiplication Formulas at
v=1}, which in turn play a decisive role in the proof of the
conjecture in the classical case (see Proposition \ref{triangular
formula in A[bfj]}). It would be natural to raise the following
question which is the key to Problem \ref{Problem for MF}:

\begin{Prob} \label{Prob-realization} Find the quantum version of the multiplication formulas given in Proposition \ref{MFforSBE}(2)
for affine quantum Schur algebras.
\end{Prob}

\backmatter

\printindex
\end{document}